\documentclass[12pt,a4paper,times]{elsarticle}

\usepackage{amsmath,amssymb,amsthm,latexsym,color,fullpage}
\usepackage{yfonts,tikz,pgf,shadow}
\usepackage[colorlinks = true,
            linkcolor = blue,
            urlcolor  = blue,
            citecolor = blue,
            anchorcolor = blue]{hyperref}
\usepackage{booktabs,xcolor,siunitx,wrapfig}
\usepackage[skip=0pt]{caption}
\usepackage{mathrsfs,longtable}
\usepackage[font={color=black,it},figurename=Figure,labelfont={it}]{caption}
\usepackage{xparse,colortbl,multirow}
\usepackage{makecell,collcell,datatool}
\usepackage{dirtytalk}
\usepackage{ragged2e}
\NewDocumentCommand{\eulerian}{omm}
 {%
  \IfNoValueTF{#1}
   {\genfrac<>{0pt}{}{#2}{#3}}%
   {\mathchoice{\genfrac<>{0pt}{}{#2}{#3}_{\!\!#1}}
               {\genfrac<>{0pt}{}{#2}{#3}_{\mkern-2mu#1}}
               {\genfrac<>{0pt}{}{#2}{#3}_{\mkern-2mu#1}}
               {\genfrac<>{0pt}{}{#2}{#3}_{\mkern-2mu#1}}%
   }
 }
\NewDocumentCommand{\stirling}{omm}
 {%
  \IfNoValueTF{#1}
   {\genfrac[]{0pt}{}{#2}{#3}}%
   {\mathchoice{\genfrac[]{0pt}{}{#2}{#3}_{\!\!#1}}
               {\genfrac[]{0pt}{}{#2}{#3}_{\mkern-2mu#1}}
               {\genfrac[]{0pt}{}{#2}{#3}_{\mkern-2mu#1}}
               {\genfrac[]{0pt}{}{#2}{#3}_{\mkern-2mu#1}}%
   }
 }
\DeclareRobustCommand{\Stirling}{\genfrac\{\}{0pt}{}} 

\newcommand{\ve}{{\varepsilon}}
\def\dd#1{\,\mathrm{d}{#1}}

\def\le{\leqslant}
\def\ge{\geqslant}
\def\tr#1{\lfloor #1\rfloor}
\def\ltr#1{\bigl\lfloor #1\bigr\rfloor}
\def\cl#1{\lceil #1\rceil}

\def\lpa#1{\bigl({#1}\bigr)}
\def\Lpa#1{\Bigl({#1}\Bigr)}

\def\llpa#1{\biggl({#1}\biggr)}
\newtheorem{thm}{Theorem}
\newtheorem{lmm}{Lemma}
\newtheorem{prop}{Proposition}
\newtheorem{cor}{Corollary}
\newtheorem{defi}{Definition}
\definecolor{yg}{RGB}{235,255,204}
\definecolor{lightgray}{gray}{0.5}
\def\ET#1{\mathscr{E}\langle\!\langle #1\rangle\!\rangle}
\def\EET#1#2{\mathscr{E}\left\langle\!\left\langle
    #1,#2\right\rangle\!\right\rangle}
\def\GET#1#2{\mathscr{E}_{#1}\left\langle\!\left\langle
    #2\right\rangle\!\right\rangle}
\def\GT#1#2{\mathscr{E}_{#1}\langle\!\langle #2\rangle\!\rangle}    
\newcommand{\overbar}[1]{\mkern 1.5mu\overline{\mkern-1.5mu#1\mkern-1.5mu}\mkern 1.5mu}

\newcommand{\nom}[2]{$ \mathscr{N}( #1 , #2) $ }
\def\deq{\stackrel{d}{\thickapprox}}
\def\cid{\stackrel{d}{\longrightarrow}}

\makeatletter
\newsavebox\myboxA
\newsavebox\myboxB
\newlength\mylenA

\newcommand*\xbar[2][0.7]{%
    \sbox{\myboxA}{$\m@th#2$}%
    \setbox\myboxB\null
    \ht\myboxB=\ht\myboxA%
    \dp\myboxB=\dp\myboxA%
    \wd\myboxB=#1\wd\myboxA
    \sbox\myboxB{$\m@th\overline{\copy\myboxB}$}
    \setlength\mylenA{\the\wd\myboxA}
    \addtolength\mylenA{-\the\wd\myboxB}%
    \ifdim\wd\myboxB<\wd\myboxA%
       \rlap{\hskip 0.5\mylenA\usebox\myboxB}{\usebox\myboxA}%
    \else
        \hskip -0.5\mylenA\rlap{\usebox\myboxA}{\hskip 0.5\mylenA\usebox\myboxB}%
    \fi}
\makeatother

\begin{document}

\title{An asymptotic distribution theory for Eulerian 
recurrences\\ with applications}

\author[1]{Hsien-Kuei Hwang\corref{cor1}\fnref{fn1}}
\ead{hkhwang@stat.sinica.edu.tw} 
\author[2]{Hua-Huai Chern\fnref{fn2}}
\ead{felix@mail.ntou.edu.tw}
\author[3]{Guan-Huei Duh}
\ead{arthurduh1@gmail.com}   

\cortext[cor1]{Corresponding author}
\fntext[fn1]{Partially supported by an Investigator Award from  
Academia Sinica under the Grant AS-IA-104-M03.}
\fntext[fn2]{Partially supported by Ministry of Science and Technology (Taiwan) under the Grant MOST-106-2115-M-019-001.}
\address[1]{Institute of Statistical Science, Academia Sinica,
    Taipei 115, Taiwan}
\address[2]{Department of Computer Science, National Taiwan Ocean 
    University, Keelung 202, Taiwan}
\address[3]{Institute of Statistical Science, Academia Sinica,
    Taipei 115, Taiwan}


\begin{abstract}
We study linear recurrences of Eulerian type of the form 
\[
    P_n(v) = (\alpha(v)n+\gamma(v))P_{n-1}(v)
    +\beta(v)(1-v)P_{n-1}'(v)\qquad(n\ge1),
\]
with $P_0(v)$ given, where $\alpha(v), \beta(v)$ and $\gamma(v)$ are
in most cases polynomials of low degrees. We characterize the various
limit laws of the coefficients of $P_n(v)$ for large $n$ using the
method of moments and analytic combinatorial tools under varying
$\alpha(v), \beta(v)$ and $\gamma(v)$, and apply our results to more
than two hundred of concrete examples when $\beta(v)\ne0$ and more
than three hundred when $\beta(v)=0$ that we gathered from the
literature and from Sloane's OEIS database. The limit laws and the 
convergence rates we worked out are almost all new and include 
normal, half-normal, Rayleigh, beta, Poisson, negative binomial,
Mittag-Leffler, Bernoulli, etc., showing the surprising richness and
diversity of such a simple framework, as well as the power of the
approaches used.
\end{abstract}


\maketitle

\tableofcontents

\section{Introduction}

The Eulerian numbers, first introduced and presented by Leonhard
Euler in 1736 (and published in 1741; see \cite{Euler1741} and
\cite[Art.\ 173--175]{Euler1755}) in series summations, have been
widely studied because of their natural occurrence in many different
contexts, ranging from finite differences to combinatorial
enumeration, from probability distribution to numerical analysis,
from spline approximation to algorithmics, etc.; see the books
\cite{Bona2004, Foata1970, Knuth1998, Petersen2015, Sandor2004,
Sobolev1997, Stanley2012} and the references therein for more
information. See also the historical accounts in the papers
\cite{Carlitz1958, Janson2013, Takacs1979, Warren1996}. Among the
large number of definitions and properties of the Eulerian numbers
$\eulerian{n}{k}$, the one on which we base our analysis is the 
recurrence
\begin{align}\label{A173018}
    P_n(v) = (vn+1-v)P_{n-1}(v)
    +v(1-v)P_{n-1}'(v)\qquad(n\ge1),
\end{align}
with $P_0(v)=1$, where $P_n(v)=\sum_{0\le k\le n}\eulerian{n}{k}v^k$. 
In terms of the coefficients, this recurrence translates into 
\begin{align}\label{eulerian-nk}
    \eulerian{n}{k}
    = (k+1)\eulerian{n-1}{k} +(n-k)\eulerian{n-1}{k-1}
    \qquad(n,k\ge1),
\end{align}
with $\eulerian{n}{k}=0$ for $k<0$ or $k\ge n$ except that  
$\eulerian{0}{0}:=1$. We extend the recurrence \eqref{A173018} by 
considering the more general \emph{Eulerian recurrence}
\begin{align}\label{Pnv-general}
    P_n(v) = (\alpha(v)n+\gamma(v))P_{n-1}(v)
    +\beta(v)(1-v)P_{n-1}'(v)\qquad(n\ge1),
\end{align}
with $P_0(v)$, $\alpha(v), \beta(v)$ and $\gamma(v)$ given (they are
often but not limited to polynomials). We are concerned with the
limiting distribution of the coefficients of $P_n(v)$ for large $n$
when the coefficients are nonnegative. Both normal and non-normal
limit laws will be mostly derived by the \emph{method of moments}
under varying $\alpha(v), \beta(v)$ and $\gamma(v)$. While the
extension \eqref{Pnv-general} seems straightforward, the study of the
limit laws is justified by the large number of applications and
various extensions. We will also solve the corresponding partial
differential equation (PDE) satisfied by the exponential generating
function (EGF) of $P_n$ whenever possible, and show how the use of
EGFs largely simplifies the classification of the extensive list of
examples we compiled, as well as the finer \emph{approximation
theorems} established by the complex analysis, in addition to the 
quick \emph{limit theorems} offered by the method of moments.

The history of Eulerian numbers is notably marked by many
\emph{rediscoveries} of previously known results, often in different
guises, which is indicative of their importance and usefulness. In
particular, Carlitz pointed out in his 1959 paper \cite{Carlitz1958}
that ``\emph{an examination of Mathematical Reviews for the past ten
years will indicate that they [Eulerian numbers and polynomials] have
been frequently rediscovered.}'' Later Schoenberg \cite[p.\
22]{Schoenberg1973} even described in his book on spline
interpolation that ``\emph{[Eulerian-Frobenius polynomials] were
rediscovered more recently by nearly everyone working on spline
interpolation.}'' We will give a simple synthesis of the approaches
used in the literature capable of establishing the asymptotic
normality of the Eulerian numbers, showing partly why rediscoveries
are common. We do not aim to be exhaustive in this synthesis of
approaches (very difficult due to the large literature), but will
rather content ourselves with a methodological and comparative
discussion.


\begin{wraptable}[11]{r}{7cm}
\begin{tabular}{c|ccccccc}
	$n\backslash k$
	& $0$ & $1$ & $2$ & $3$ & $4$ & $5$ \\ \hline
    $0$ & $1$ &&&&& \\
    $1$ & $1$ &&&&& \\
    $2$ & $1$ & $1$ &&&& \\
    $3$ & $1$ & $4$ & $1$ &&& \\
	$4$ & $1$ & $11$ & $11$ & $1$ && \\ 
    $5$ & $1$ & $26$ & $66$ & $26$ & $1$ & \\ 
    $6$ & $1$ & $57$ & $302$ & $302$ & $57$ & $1$ \\
\end{tabular}
\medskip
\caption{The first few rows of $\eulerian{n}{k}$.}
\label{tab-eulerian}
\end{wraptable}

In addition to their first appearance in series summation or
successive differentiation
\[
    \sum_{j\ge0}j^n v^j 
    = (v\mathbb{D}_v)^n\frac1{1-v} 
    = \frac{vP_n(v)}{(1-v)^{n+1}},
\]
the Eulerian numbers also emerge in many statistics on permutations
such as the number of descents (or runs) whose first few rows are
given on the right table; see \cite{Comtet1974, Graham1994,
Stanley2012} and Sloane's OEIS pages on
\href{https://oeis.org/A008292}{A008292},
\href{https://oeis.org/A123125}{A123125} and
\href{https://oeis.org/A173018}{A173018} for more information and
references. The earliest reference we found dealing with descents
(called ``inversions \'el\'ementaires'') in permutations is Andr\'e's
1906 paper \cite{Andre1906}; see also \cite{MacMahon1908,
Schrutka1941}. On the other hand, von Schrutka's 1941 paper
\cite{Schrutka1941} mentions the connection between descents in
permutations and a few other known expressions for Eulerian numbers; 
although he does not cite explicitly Euler's work, the references 
given there, notably Frobenius's 1910 paper \cite{Frobenius1910} and
Saalsch\"utz's 1893 book \cite{Saalschutz1893}, indicate the 
connecting link, which was later made explicit in Carlitz and 
Riordan's 1953 paper \cite{Carlitz1953}. Moreover, Carlitz and his 
collaborators have made broad contributions to Eulerian numbers and
permutation statistics, leading to more unified and extensive
developments of modern theory of Eulerian numbers; see
\cite{Petersen2015,Stanley2012}.

Each row sum in Table~\ref{tab-eulerian} is equal to $n!$. It is 
natural to define the random variable $X_n$ by
\[
    \mathbb{P}(X_n=k) = \frac{1}{n!}\,\eulerian{n}{k},
	\quad\text{or}\quad
	\mathbb{E}\lpa{v^{X_n}} = \frac{P_n(v)}{P_n(1)},
\]
where $P_n(v)$ satisfies \eqref{A173018}. Here $\mathbb{E}(v^{X_n})$
denotes the probability generating function of $X_n$. From a
distributional point of view, we observe a distinctive feature of
Eulerian numbers here: \emph{they have a higher concentration near
the middle} when compared for example with the binomial coefficients
(which is also symmetric). In particular, the fifth row (in the above 
table) of the probability distribution reads $(\frac1{24}, 
\frac{11}{24}, \frac{11}{24}, \frac1{24})$, while that of the 
corresponding binomial distribution reads 
$(\frac18,\frac38,\frac38,\frac18)$; see Figure~\ref{eulerian-bino} 
for a graphical illustration.

\begin{figure}[!ht]
\begin{center}
\renewcommand{\arraystretch}{1.7}
\begin{tabular}{c|c} \hline
Eulerian distribution $\frac{1}{n!}\,\eulerian{n}{k}$ &
Binomial distribution $\frac1{2^{n-1}}\binom{n-1}{k}$\\ \hline
$\eulerian{n}{k} = (k+1)\eulerian{n-1}{k}
 +(n-k)\eulerian{n-1}{k-1}$ &
$\binom{n}{k} = \binom{n-1}{k}+\binom{n-1}{k-1}$\\ 
\includegraphics[trim=0 0 0 -1cm,height=3cm]{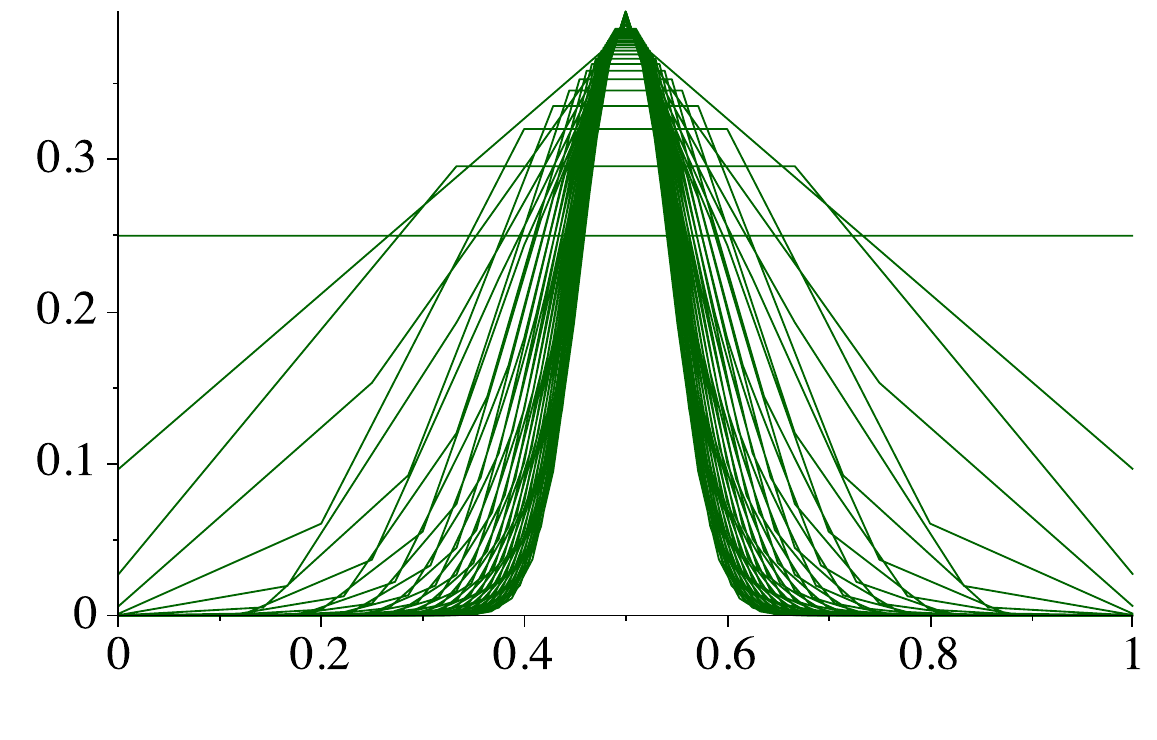} 
&
\includegraphics[trim=0 0 0 -1cm,height=3cm]{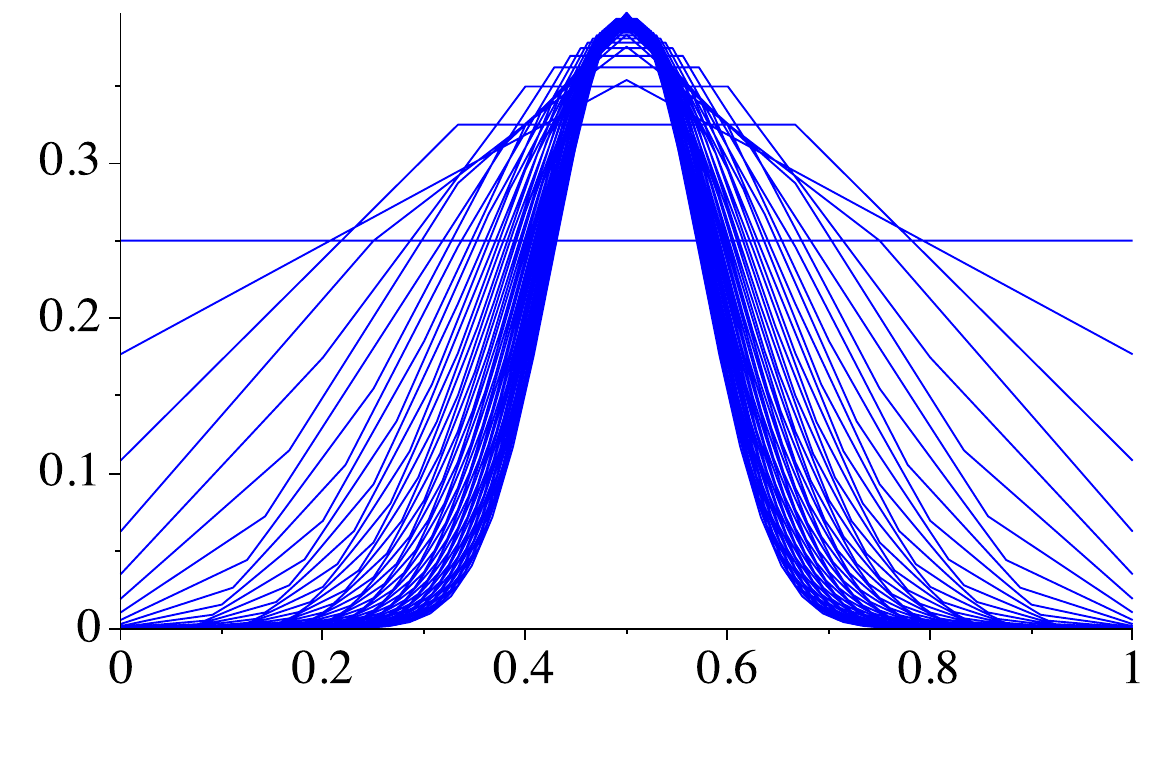} 
\\ \hline
\end{tabular}
\end{center}
\caption{A comparison between Eulerian and binomial distributions for
$2,\dots, 50$ (magnified by standard variations and normalized into
the unit interval). The higher concentration of Eulerian
distributions near their mean values is visible.}
\label{eulerian-bino}
\end{figure}

Such a high concentration in distribution may be ascribed to the
large multiplicative factors $k+1$ and $n-k$ when $k$ is near
$\frac12n$ in \eqref{eulerian-nk}, leading to the ``rich gets richer"
effect for terms near the mode of the distribution. More precisely,
it is known that $X_n$ is asymptotically normally distributed (in the
sense of convergence in distribution) with mean asymptotic to
$\frac12 n$ and variance to $\frac1{12}n$; the variance is smaller
than the binomial variance $\frac14 n$, which partially reflects the 
high concentration. For brevity, we will write (CLT standing for 
\emph{central limit theorem})
\begin{align}\label{eulerian-clt}
    X_n \sim \mathscr{N}\lpa{\tfrac12n,\tfrac1{12}n} 
	\quad\text{for the CLT}\quad
	\sup_{x\in\mathbb{R}}\left|
	\mathbb{P}\left(\frac{X_n-\frac12n}{\sqrt{\frac1{12} n}}
	\le x\right)-\Phi(x)\right| \to 0,
\end{align}
and $\mathbb{E}(X_n)\sim \frac12n$ and 
$\mathbb{V}(X_n)\sim\frac1{12}n$, 
where $\Phi(x)$ denotes the standard normal distribution function
\[
	\Phi(x) := \frac1{\sqrt{2\pi}}
	\int_{-\infty}^x e^{-\frac12t^2}\dd t
    \qquad(x\in\mathbb{R}).
\]
Such an \emph{asymptotic normality with small variance} will be 
constantly observed throughout the examples we will examine.

Due to the multifaceted appearance of Eulerian numbers, it is no 
wonder that the limit result \eqref{eulerian-clt} has been proved by 
many different approaches in miscellaneous guises; see 
Table~\ref{clt-approaches} for some of them. 
\begin{table}[!ht]
\begin{center}
	\begin{tabular}{llll}
		\multicolumn{4}{c}{{}} \\
		\multicolumn{1}{c}{Approach} &
		\multicolumn{1}{c}{First reference} &
		\multicolumn{1}{c}{Year} &
		\multicolumn{1}{c}{See also} \\ \hline
		Sum of Uniform$[0,1]$ 
		&  Laplace \cite{Laplace1812} & 1812
		&  \cite{Hensley1982,Tanny1973} \\ 
		Sum of $\nearrow$ or $\searrow$ indicators 
		& Wolfowitz \cite{Wolfowitz1944} & 1944
		&  \cite{Dwass1973,Esseen1985} \\ 
		Method of moments 
		& Mann \cite{Mann1945} & 1945
		& \cite{David1962} \\ 
		Spline \& characteristic functions
		& Curry \& Schoenberg 
		\cite{Curry1966} & 1966
		&\cite{Chen2004,Xu2011} \\  
		Real-rootedness  &
		Carlitz et al. \cite{Carlitz1972} & 1972
		& \cite{Pitman1997,Warren1996}\\ 
		Complex-analytic 
		& Bender \cite{Bender1973} & 1973
		& \cite{Flajolet1993,Hwang1994} \\ 
		Stein's method
		& Chao et al.\ \cite{Chao1996} & 1996
		&  \cite{Chuntee2017,Conger2007,Fulman2004}	\\ \hline
	\end{tabular} 
\end{center}
\caption{A list of some approaches used to establish the asymptotic 
normality \eqref{eulerian-clt} of Eulerian numbers.} 
\label{clt-approaches}
\end{table}


The normal limit law \eqref{eulerian-clt} in the form of descents in
permutations appeared first in 1945 by Mann \cite{Mann1945} where a
method of moments based on the recurrence \eqref{eulerian-nk}
was employed, proving the empirical observation made in
\cite{Moore1943}. A similar approach was worked out in David and
Barton \cite{David1962} where they showed that all cumulants of $X_n$
are linear with explicit leading coefficients. A more general
treatment of runs up and down in permutations had already been given
by Wolfowitz \cite{Wolfowitz1944} in 1944, where he relied 
instead his analysis on decomposing the random variables $X_n$ into a
\emph{sum of indicators} and then on applying \emph{Lyapunov's
criteria for CLT} by computing the fourth central moments; see 
\cite{Fischer2011}. These publications have remained little known in 
combinatorics literature mainly because they were published in a 
statistical journal.


On the other hand, the asymptotic normality \eqref{eulerian-clt} had
been established earlier than 1944 in other forms, although the links 
to Eulerian numbers were only known later. The earliest connection we
found is in Laplace's \emph{Th\'eorie analytique des probabilit\'es},
first version published in 1812 \cite{Laplace1812}. The connection is
through the expression (already known to Euler \cite[Art.\ 
173]{Euler1755})
\[
    \frac{1}{n!}\, \eulerian{n}{k}
	= \frac1{n!}\sum_{0\le j\le k+1}\binom{n+1}{j}(-1)^j
	(k+1-j)^n\qquad(n\ge0),
\]
and the distribution of the sum of $n$ independent and identically 
distributed uniform $[0,1]$ random variables $U_1,\dots,U_n$:
\begin{align}\label{Snx}
    \mathbb{P}(U_1+\cdots+U_n\le t)
	=\frac1{n!}\sum_{0\le j\le t} \binom{n}{j}(-1)^j
	(t-j)^n.
\end{align}
It then follows that (see \cite{Hensley1982, Janson2013,
Pitman1997, Stanley1977, Tanny1973})
\[
    \mathbb{P}(X_n\le t) = \mathbb{P}(U_1+\cdots+U_n\le t+1),
\]
and the asymptotic normality of $X_n$ follows from that of the sum of 
uniform random variables, which was first derived by Laplace in 
\cite{Laplace1812} by large powers of characteristic functions,  
Fourier inversion and a saddle-point approximation (or Laplace's 
method). 

Concerning the expression \eqref{Snx} (the sum on the right-hand side
already appeared in \cite{Euler1755}), sometimes referred to as
Laplace's formula (see for example \cite{Diaconis1987}), we found
that it appears (up to a minor normalization) in Simpson's 1756 paper
\cite{Simpson1756} where the sum of continuous uniforms is treated as
the limit of sum of discrete uniforms; see also his book
\cite{Simpson1757}. The underlying question, closely connected to the
counts of repeated tossing of a general dice, has a very long history
and rich literature in the early development of probability theory.
In particular, Simpson's treatment finds its roots in de Moivre's
extension of Bernoulli's binomial distribution, ``\emph{which in turn
was derived from Newton's binomial theorem and before that from
Pascal's arithmetic triangle---this approach may have the most
impressive provenance of any in probability theory}'' (quoted from
Stigler \cite[P.\ 92]{Stigler1986}). Interestingly, de Moivre's
approach also constitutes one of the very early uses of generating
functions; see \cite[Ch.\ 2]{Stigler1986}. The same expression
\eqref{Snx} was derived in the 1770s by Lagrange, Laplace and later
by many others, notably in spline and related areas; see
\cite{Chui1992, Schoenberg1973}. See also the books
\cite{Fischer2011, Hald1998, Petersen2015} for more information.
Coincidentally, expressions very similar to \eqref{Snx} also emerged
in Laplace's analysis of series expansions; see \cite{Laplace1777}.
But he did not mention the connection to Eulerian numbers.

The sum-of-indicators approach used by Wolfowitz is very useful 
due to its simplicity but the more classical Lyapunov 
condition is later replaced by limit theorems for $2$-dependent 
indicators; see \cite{Dwass1973,Esseen1985,Hoeffding1948}. Also it is 
possible to derive finer properties such as large deviations; see 
\cite{Esseen1985}. 

Instead of decomposing the Eulerian distribution as a sum of
\emph{dependent} Bernoulli variates, a much more successful and
fruitful approach in combinatorics is to express it as a sum of
\emph{independent} Bernoullis based on the property that all roots of
its generating polynomial $P_n(v)$ (see \eqref{A173018}) are real and
negative; see \cite{Carlitz1972,Frobenius1910,Warren1996}.
More precisely, $P_n(v)$ has the decomposition \cite{Frobenius1910}
\[
    P_n(v) = \prod_{1\le j\le n}(\zeta_{n,j}+v),
\]
where $\zeta_{n,j}\in\mathbb{R}^+$. It follows that $X_n = \sum_{1\le
j\le n}\xi_{n,j}$, where $\xi_{n,j}$ is a Bernoulli with probability
$\frac1{1+\zeta_{n,j}}$ of assuming $1$. Then Harper's approach
\cite{Harper1967} to establishing the asymptotic normality
\eqref{eulerian-clt} consists in showing that the variance tends to
infinity, which amounts to checking Lyapunov's condition because the
summands are bounded. This was carried out for Eulerian distribution
by Carlitz et al.\ in \cite{Carlitz1972}. For a slightly more general
context (all roots lying in the negative half-plane), see Hayman's
influential paper \cite{Hayman1956} and R\'enyi's synthesis
\cite{Renyi1967,Renyi1967a}. See also the surveys \cite{Branden2015,
Brenti1989, Brenti1994a, Liu2007, Pitman1997, Stanley1989} for the
usefulness of this real-rootedness approach.

We describe two other approaches listed in Table~\ref{clt-approaches}
that are closely connected to our study here, leaving aside other
ones such as spline functions, matched asymptotics, and Stein's
method; see \cite{Chao1996, Chen2004, Chuntee2017, Conger2007,
Curry1966, Fulman2004, Giladi1994, Xu2011} for more information. For
the connection to P\'olya's urn models, see \cite{Flajolet2006,
Freedman1965, Oden2006} and Section~\ref{ss-polya}. See also the very
recent papers \cite{Fulman2019,Kim2018,Kim2019} for a kind of
saddle-point approach and \cite{Ozdemir2019} for an approach via
martingales.

A general study of asymptotic normality based on complex-analytic 
approach was initiated by Bender \cite{Bender1973} where in the 
particular case of Eulerian numbers he used the relation for the 
exponential generating function (EGF)
\begin{align}\label{eulerian-egf}
    F(z,v) := \sum_{n\ge0}\frac{P_n(v)}{n!}\, z^n
    = \frac{1-v}{e^{(v-1)z}-v},
\end{align}
and observes that the dominant simple pole $z=\rho(v) :=
\frac1{1-v}\log\frac1v$ ($\rho(1) := 1$) provides the essential
information we need for establishing the asymptotic normality
\eqref{eulerian-clt} since for large $n$
\[
    \frac{P_n(e^s)}{n!} 
    =  e^{-s}\left(\frac{e^s-1}{s}\right)^{n+1}
    +\text{exponentially smaller terms},
\]
uniformly for $|s|\le \ve$. The uniformity then guarantees that the
characteristic functions of the centered and normalized random
variables tend to that of the standard normal distribution, implying
\eqref{eulerian-clt} by \emph{L\'evy's continuity theorem} (see
\cite[\S~C.5]{Flajolet2009}). This approach provides not only a limit
theorem, but also much finer properties such as local limit theorems
and large deviations in many situations, as already clarified in
\cite{Bender1973} and later publications such as \cite{Flajolet1993, 
Gao1992, Hwang1994}. In general, the characterization of limit
laws or other stochastic properties through a detailed study of the
singularities of the corresponding generating functions, coupling
with suitable analytic tools, proved very powerful and successful;
see \cite{Canfield2015, Flajolet2009, Gao1992, Hwang1994, 
Odlyzko1995} for more information. Note that $F$ satisfies the PDE
\[
     (1-vz)\partial_z F -v(1-v)\partial_v F = F,
\]
the resolution of which adding another interesting dimension to the 
richness of Eulerian recurrences, which we will briefly explore 
in Section~\ref{ss-pde}.  

While each of these approaches has its own strengths and weaknesses,
a large portion of the asymptotic normality results for recursively
defined polynomials in the combinatorics literature rely on Harper's
real-rootedness approach. Also many powerful criteria for justifying
the real-rootedness of a sequence of polynomials have been developed
over the years; see for example \cite{Branden2015,
Brenti1989, Brenti1994a, Liu2007, Pitman1997, Stanley1989}. However, 
the real-rootedness property is an exact one and is very sensitive to
minor changes. For example, if we change the factor $vn+1-v$ to
$vn+(1+v)^2$ in the recurrence \eqref{A173018}, then all coefficients
remain positive but complex roots are abundant as can be seen from
Figure~\ref{fig-complex-roots}. On the other hand, by our theorem
below, the coefficients still follow the same CLT
\eqref{eulerian-clt} (with the same asymptotic mean and asymptotic
variance). Historically, the proof of the first moment convergence
theorem by Markov relies on the (real) zeros of Hermite polynomials;
see \cite{Frechet1931}.
\begin{figure}[!ht]
\begin{center}
\includegraphics[height=4cm]{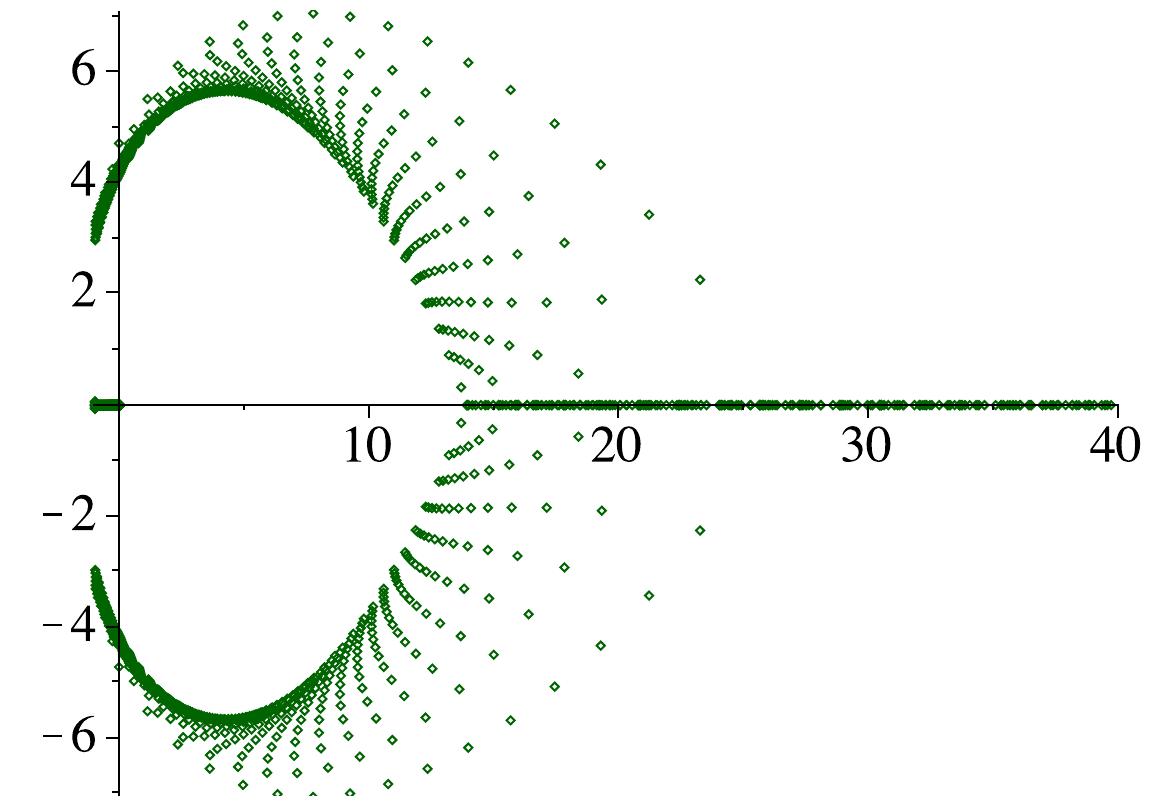}
\includegraphics[height=4cm]{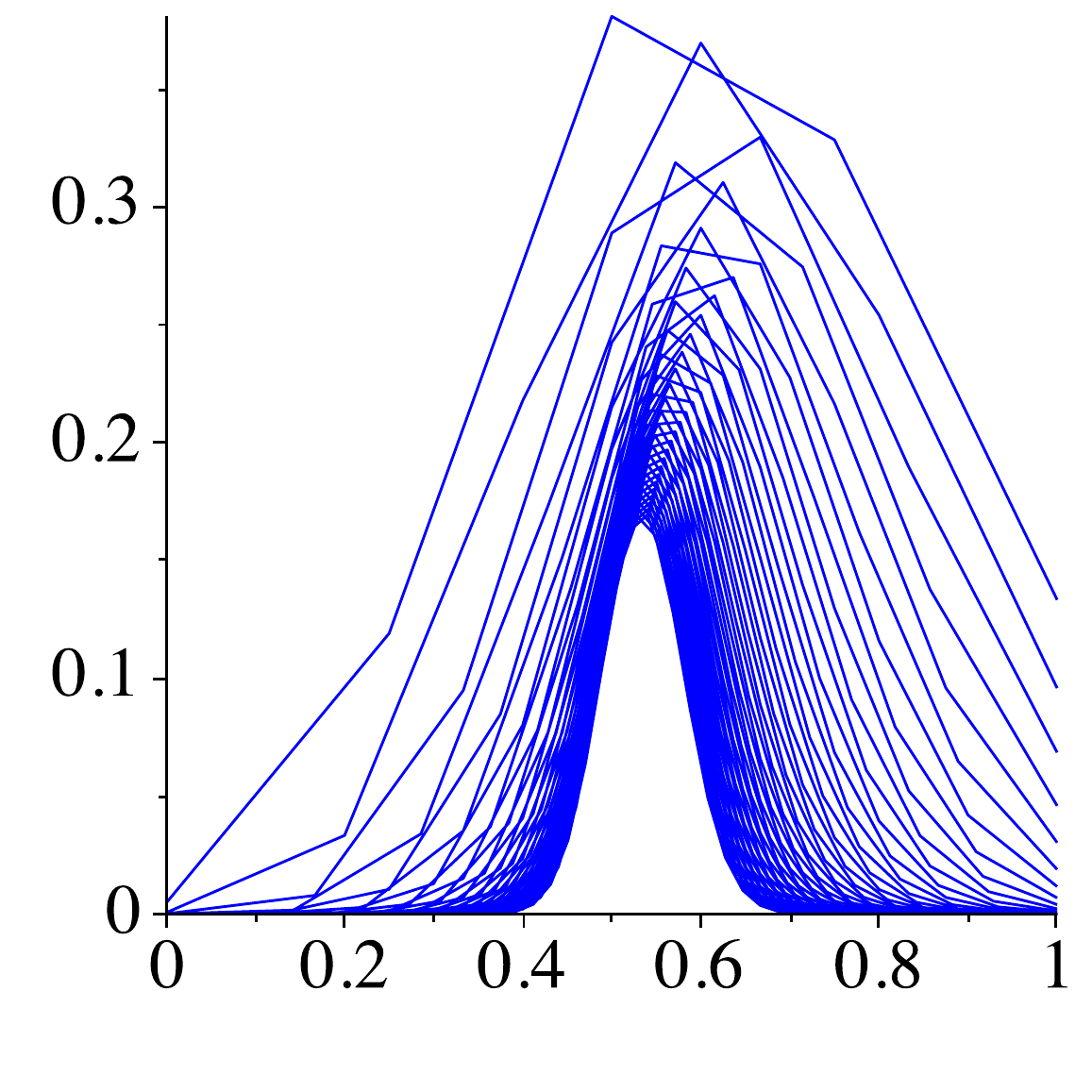}
\end{center}
\caption{Left: zero distributions of the polynomials $R_n(v) = 
(vn+(1+v)^2)R_{n-1}(v) + v(1-v)R_{n-1}'(v)$ for $n\ge1$ with 
$R_0(v)=1$; right: the corresponding histograms. Here 
$n=2,3,\dots,50$. The EGF equals $\left(\frac{1-v}
{1-ve^{(1-v)z}}\right)^5e^{(1-v)z-ve^{(1-v)z}+v}$; see 
Section~\ref{ss-pde}.} \label{fig-complex-roots}
\end{figure}

On the other hand, the closed-form expression \eqref{eulerian-egf}
for the EGF represents another exact property and may not be
available in more general cases \eqref{Pnv-general}, especially when
the corresponding PDE is difficult to solve. A simple example is the
sequence OEIS \href{https://oeis.org/A244312}{A244312} for which 
\begin{align}\label{A244312}
    P_n(v) = 
	\begin{cases}
		(vn-1)P_{n-1}(v)+v(1-v)P_{n-1}'(v), 
		&\text{if $n$ is even},\\
		(vn-v)P_{n-1}(v)+v(1-v)P_{n-1}'(v),
		&\text{if $n$ is odd},
	\end{cases}
	\qquad(n\ge2),
\end{align}
with $P_1(v)=v$. The same $\mathscr{N}\lpa{\frac12n,\frac1{12}n}$ 
can be proved by the method of moments (see 
Section~\ref{ss-euler-clt}), but it is less clear how to solve the 
corresponding PDE ($F$ being the EGF of $P_n$)
\begin{align}\label{A244312-egf}
    (1-vz)\partial_z F(z,v) +(1-v)\frac{F(z,v)-F(-z,v)}{2}
    =v(1-v)\partial_v F(z,v)+v.
\end{align}

One of our aims of this paper is to show the usefulness of the method
of moments for general recurrences such as \eqref{Pnv-general}. More
precisely, we will derive in the next section a CLT for
\eqref{Pnv-general} under reasonably weak conditions on $\alpha(v),
\beta(v)$ and $\gamma(v)$. While our limit result seems conceptually
less deep (when compared with, say the real-rootedness properties),
it is very effective and easy to apply; indeed, its effectiveness 
will be testified by more than three hundred of polynomials in
later sections. The list of examples we compiled is by far the most
comprehensive one (although not exhaustive).

On the other hand, although the method of moments has been employed
before in similar contexts (see \cite{Bagchi1985, David1962,
Freedman1965, Mann1945}), our manipulation of the recurrence (via
developing the ``asymptotic transfer") is simpler and more
systematic; see also \cite{Hwang2003} for the developments for other
divide-and-conquer recurrences. In addition to the method of moments,
we will also explore the usefulness of the complex-analytic
approach for Eulerian recurrences. In particular, we obtain optimal
convergence rates in the CLTs, using tools developed in Flajolet and
Sedgewick's authoritative book \cite{Flajolet2009} on Analytic
Combinatorics. We will then extend the same method of moments to
characterize non-normal limit laws in Sections~\ref{sec-nnll} with
applications given in later sections. Extensions along many different
directions are discussed in Section~\ref{sec-extensions}, and the
simpler framework when $\beta(v)=0$ (in \eqref{Pnv-general}) in
Section~\ref{sec-beta-is-0} for completeness, some examples of this
framework being collected in Appendix~\ref{App-Bino}. 
Section~\ref{sec-conclusions} concludes this paper.


\medskip

\noindent \textbf{Notations.} Throughout this paper, $P_n(v)$ is a
generic symbol whose expression may differ from one occurrence to
another, and $Q_n(v)$ always denotes the reciprocal polynomial
(reading each row coefficients of $P_n(v)$ from right to left) of
$P_n(v)$, except in Section~\ref{sec:multi}. The EGF of $P_n$ is 
always denoted by $F(z,v)$. For convenience, the Eulerian recurrence
\begin{align}\label{Pnv-gen}
	\begin{cases}
	    P_n(v) = a_n(v) P_{n-1}(v) + b_n(v) (1-v) P_{n-1}'(v)
		\qquad(n\ge1)\\
		P_0(v) \text{ given}
	\end{cases}	
\end{align}
will be abbreviated as $P_n\in\ET{a_n(v),b_n(v)}$ or 
$P_n\in\ET{a_n(v),b_n(v); P_0(v)}$ if we want to specify the initial 
condition. When the initial condition on $P_r(v)$, say $P_r(v)=1+v$, 
is given with $r\ge1$, we write $P_n\in\GT{r}{a_n(v),b_n(v);1+v}$, 
with the understanding that the recurrence starts from $n\ge r+1$. 

\medskip

\noindent \textbf{Web forms.} All examples in this paper (with a
total of 628 items in which 594 are in OEIS) are compiled and
maintained at the two webpages \cite{Hwang2019} 
(\eqref{Pnv-gen} with $b_n(v)\ne0$) and \cite{Hwang2019-2}
(\eqref{Pnv-gen} with $b_n(v)=0$), with types, links, numerical 
tables and other properties. 

\section{A normal limit theorem}
\label{sec-clt}

We consider in this section the limiting distribution (for large $n$)
of the coefficients of linear type Eulerian recurrence $P_n(v)$:
\[
    \ET{\alpha(v)n+\gamma(v), \beta(v);P_0(v)},
\]
where $\alpha(v), \beta(v)$, $\gamma(v)$ and $P_0(v)$ are any 
functions analytic in $|v|\le 1$, and we assume that all Taylor 
coefficients $[v^k]P_n(v)$ are nonnegative for $k, n\ge0$. If 
$[v^k]P_n(v)\ge0$ for $n\ge n_0$ with $n_0>0$, then we can consider 
the shifted functions $R_n(v) := P_{n+n_0}(v)$, which satisfy the 
same form \eqref{Pnv-gen} but with $\gamma(v)$ replaced by
$n_0\alpha(v)+\gamma(v)$. So without loss of generality, we assume
that $n_0=0$ and $P_n(1)>0$ for $n\ge0$ for which a sufficient 
condition is $[v^k]P_n(v)\ge0$ and $P_n(v)\not\equiv0$ for $k,n\ge0$. 

For simplicity, we write $\alpha=\alpha(1)$ and similarly for $\beta$ 
and $\gamma$. By \eqref{Pnv-general}, we see that 
\[
	P_n(1) 
	= (\alpha n+\gamma)P_{n-1}(1)
	= P_0(1)\prod_{1\le j\le n}(\alpha j+\gamma)
	= P_0(1) \alpha^n\frac{\Gamma\lpa{n+1+\frac\gamma\alpha}}
	{\Gamma\lpa{1+\frac\gamma\alpha}};
\]
thus $P_n(1)$ is independent of $\beta(v)$, and the factor
``$1-v$'' in front of $P_{n-1}'(v)$ in \eqref{Pnv-gen} makes the
recurrence satisfied by the moments easier to handle. Note that the
assumption that $P_n(1)>0$ for $n\ge0$ implies that $\alpha+\gamma>0$.

Define the random variables $X_n$ by
\begin{align}\label{Xnk-general}
   \mathbb{P}(X_n=k) = \frac{[v^k]P_n(v)}{P_n(1)}
\qquad(k,n\ge0).
\end{align}

\begin{thm}[Asymptotic normality of $X_n$] \label{thm-clt}
Assume that the sequence of functions $P_n(v)$ is defined 
recursively by \eqref{Pnv-gen} satisfying (i) $[v^k]P_n(v)\ge0$ and 
$P_n(v)\not\equiv0$ for $k,n\ge0$, and (ii) $P_0(v)$, $\alpha(v)$, 
$\beta(v)$ and $\gamma(v)$ analytic in $|v|\le1$. If, furthermore,
\begin{align} \label{thm-clt-positivity}
    \alpha+2\beta>0\quad\text{and}\quad \sigma^2>0, 
\end{align}
where
\begin{align}\label{mu-var}
    \mu := \frac{\alpha'(1)}{\alpha+\beta}
    \quad\text{and}\quad
    \sigma^2 
    := \mu+\frac{\alpha''(1)-2\mu\beta'(1)-\alpha\mu^2}
    {\alpha+2\beta},
\end{align}
then the sequence of random variables $X_n$, defined by 
\eqref{Xnk-general}, satisfies $X_n \sim \mathscr{N}(\mu n, 
\sigma^2n)$, namely, $X_n$ is asymptotically normally distributed 
with the mean and the variance asymptotic to $\mu n$ and $\sigma^2 
n$, respectively. 
\end{thm}
Indeed, we will prove convergence of all moments. 

Observe first that $P_0(v)$, $\alpha(v), \beta(v)$ and $\gamma(v)$
need not be polynomials, although in almost all our examples they
are; see \S~\ref{sss-conway} for an example with
$\gamma(v)=\frac{1-v}{1+v}$. Also the two constants $\mu$ and
$\sigma^2$ depend only on $\alpha(v)$ and $\beta(v)$, but not on
$\gamma(v)$; neither do they depend on the initial condition
$P_0(v)$. This offers the flexibility of varying $\gamma(v)$ without
changing the normal limit law, as we did in Introduction 
(Figure~\ref{fig-complex-roots}), provided that $[v^k]P_n(v)\ge0$. 
Furthermore, our conditions are very easy to check in all cases we 
will discuss. Finally, recurrences similar to ours have been studied 
in the literature; see for example \cite{Dominici2011, Dubeau1995,
Hitczenko2018, Warren1999} and the references therein.

The same method of proof can be extended to the cases when
the factor $\alpha(v) n+\gamma(v)$ of $P_{n-1}(v)$ in \eqref{Pnv-gen}
also contains higher powers of $n$. See Section~\ref{sec-extensions} 
for extensions along many different lines. 

In connection with the inequalities in \eqref{thm-clt-positivity}, we 
have the order relations for the mean and the variance:
\[
    \begin{cases}
    	\text{if }\alpha+\beta<0 
		\text{ or }-\frac{\beta}{\alpha}>1,
		\text{ then } 
		\mathbb{E}(X_n)\sim C n^{-\frac{\beta}{\alpha}},\\
    	\text{if }\alpha+2\beta<0 
		\text{ or }-\frac{\beta}{\alpha}>\frac12,
		\text{ then } 
		\mathbb{V}(X_n)\sim C' n^{-\frac{2\beta}{\alpha}},
    \end{cases}
\]
where $C$ and $C'$ are constants depending on $P_0(v), \alpha(v),
\beta(v)$ and $\gamma(v)$. In general, we expect that the limit law 
is no more normal when $\alpha+2\beta<0$. The same moments approach 
can be extended to such a case, but we leave this aside in this paper 
for simplicity of presentation (also because of few examples). For 
similar contexts in urn models, see \cite{Bagchi1985, Janson2004, 
Mahmoud2008}. 

We will prove Theorem~\ref{thm-clt} by the method of moments. We 
assume, throughout this section, that $\alpha>0$. 

\subsection{Mean value of $X_n$}
Consider now the moment generating function
\[
    M_n(s) := \frac{P_n(e^s)}{P_n(1)}.
\]
By \eqref{Pnv-gen}, for $n\ge1$
\begin{align}\label{Qns}
	M_n(s) = \frac{\alpha(e^s)n+\gamma(e^s)}
	{\alpha n+\gamma}\, M_{n-1}(s)
	-\frac{\beta(e^s)(1-e^{-s})}
	{\alpha n+\gamma}\,M_{n-1}'(s),
\end{align}
with $M_0(s) = \frac{P_0(e^s)}{P_0(1)}$. The mean value can then be 
computed by the recurrence
\begin{align}\label{mu-rr}
    \mu_n := M_n'(0) =
    \left(1-\frac{\beta}{\alpha n+\gamma}\right)\mu_{n-1}
    +\frac{\alpha'(1)n+\gamma'(1)}{\alpha n+\gamma}
    \qquad(n\ge1),
\end{align}
with $\mu_0=M_0'(0)=\frac{P_0'(1)}{P_0(1)}$. 

For our asymptotic purpose, we will use the following approximations. 
\begin{prop}[Asymptotics of $\mu_n$] The mean $\mu_n$ of $X_n$ 
can be approximated as follows.
\begin{itemize}
\item If $-\frac\beta\alpha<1$, then 
\begin{align} \label{mu-as}
    \mu_n 
    = \frac{\alpha'(1)}
    {\alpha+\beta}\, n +
    \begin{cases}
        O\lpa{1+n^{-\frac\beta\alpha}}, & \text{if }\beta\ne0;\\
        O\lpa{\log n}, &\text{if }\beta=0.
    \end{cases}    
\end{align}

\item If $-\frac\beta\alpha=1$, then
\[
    \mu_n = \frac{\alpha'(1)}{\alpha}\,n\log n 
    +C_0n + O(\log n),
\]
where ($\psi$ denoting the digamma function)
\[
    C_0 := \frac{u_0\alpha+\gamma'(1)}{\alpha+\gamma}
    -\frac{\alpha'(1)}{\gamma}
    -\frac{\alpha'(1)}{\alpha}\left(1+\psi
    \left(\frac\gamma\alpha\right)\right).
\]

\item If $-\frac\beta\alpha>1$, then 
\[
    \mu_n = C_1 n^{-\frac\beta\alpha}
    \left(1+O\lpa{n^{-1}}\right)+O(n), 
\]
where
\[
    C_1 := \frac{\Gamma\lpa{1+\frac{\gamma}{\alpha}}}
    {\Gamma\lpa{1+\frac{\gamma-\beta}{\alpha}}}
    \left(\mu_0- \frac{\gamma'(1)}{\beta} 
    -\frac{\alpha'(1)(\beta-\gamma)}
    {\beta(\alpha+\beta)}\right).
\]
\end{itemize}	
\end{prop}
\begin{proof}
We can solve the first-order difference equation \eqref{mu-rr} and 
obtain for $n\ge0$: 
\begin{itemize} 
\item if $\beta(\alpha+\beta)\ne0$, then  
\begin{align} \label{mun-linear}
    \begin{split}
    \mu_n &= \frac{\alpha'(1)}{\alpha+\beta}\, n
    + \frac{\gamma'(1)}{\beta}
    +\frac{\alpha'(1)(\beta-\gamma)}{\beta(\alpha+\beta)}\\
    &\qquad + \frac{\Gamma\lpa{1+\frac{\gamma}{\alpha}}
    \Gamma\lpa{n+1+\frac{\gamma-\beta}{\alpha}}}
    {\Gamma\lpa{1+\frac{\gamma-\beta}{\alpha}}
    \Gamma\lpa{n+1+\frac{\gamma}{\alpha}}}
    \left(\mu_0- \frac{\gamma'(1)}{\beta} 
    -\frac{\alpha'(1)(\beta-\gamma)}
    {\beta(\alpha+\beta)}\right);
    \end{split}
\end{align}
\item if $\beta=0$, then 
\[
    \mu_n = \frac{\alpha'(1)}{\alpha}\, n
    +\frac{\alpha\gamma'(1)-\alpha'(1)\gamma}{\alpha^2}
    \left(\psi\left(n+1+\frac\gamma\alpha\right)
    -\psi\left(1+\frac\gamma\alpha\right)\right) +\mu_0;
\]

\item if $\alpha+\beta=0$, then 
\begin{align*}
    \mu_n &= (\alpha(n+1)+\gamma)
    \left(\frac{\alpha'(1)}{\alpha^2}\left(
    \psi\left(n+1+\frac\gamma\alpha\right)
    -\psi\left(1+\frac\gamma\alpha\right)\right)
    +\frac{\mu_0}{\alpha+\gamma}\right)\\
    &\qquad+\left(\frac{\gamma'(1)}{\alpha+\gamma}-
    \frac{\alpha'(1)}{\alpha}\right)n.
\end{align*}
\end{itemize}
The asymptotic approximations of the Proposition then follow 
from these relations. Note that $-\frac\beta\alpha\gtreqqless 1$ is 
equivalent to $\alpha+\beta\lesseqqgtr 0$. 
\end{proof}

\begin{cor} \label{cor-D-bdd} 
The asymptotic estimate $\mu_n\sim \mu n$ is equivalent to  
$\mu_n-\mu_{n-1}\sim\mu$. 
\end{cor}
\begin{proof}
Note that in general situations $\mu_n-\mu_{n-1}\sim \mu$ implies 
$\mu_n\sim \mu n$ but not vice versa. In our setting, this follows 
from rewriting \eqref{mu-rr} as
\[
    \mu_n-\mu_{n-1}
    = -\frac{\beta \mu_{n-1}}{\alpha n+\gamma}
    +\frac{\alpha'(1)n+\gamma'(1)}{\alpha n+\gamma},
\]
which, by the assumption $\mu_n\sim\mu n$, yields
\[
    \mu_n-\mu_{n-1}
    \sim -\frac{\beta}\alpha\,\mu+\frac{\alpha'(1)}{\alpha}
    =\mu.
\]
\end{proof}

\subsection{Recurrence relation for higher central moments}

Assume from now on $\alpha+2\beta>0$. Then $\alpha+\beta>0$ (since 
$\alpha>0$), so that $\mu_n$ is linear by \eqref{mu-as} with 
$\mu_n-\mu_{n-1}=O(1)$. The higher moments can then be computed 
through the moment generating function of the centered random 
variables
\[
    \xbar{M}_n(s) := M_n(s)e^{-\mu_ns},
\]
which, by \eqref{Qns}, satisfies the recurrence 
\begin{align} \label{bar-Mns}
    \xbar{M}_n(s)
	&= \frac{e^{-\Delta_n s}}{\alpha n+\gamma}\left(
	\left(\begin{array}{l}
	    \alpha(e^s)n+\gamma(e^s)\\
		-\mu_{n-1}\beta(e^s)(1-e^{-s})
	\end{array}\right) \xbar{M}_{n-1}(s)
	-\beta(e^s)(1-e^{-s})
	\, \xbar{M}_{n-1}'(s)\right),
\end{align}
for $n\ge1$, where $\Delta_n := \mu_n-\mu_{n-1}=O(1)$ by 
Corollary~\ref{cor-D-bdd}. Write now
\[
    \xbar{M}_n(s) = \sum_{m\ge0}\frac{M_{n,m}}{m!}\,s^m,
\]
where $M_{n,m} = \mathbb{E}(X_n-\mu_n)^m$, and 
\begin{align}\label{abc}
    e^{-\Delta_n s}\alpha(e^s) 
	= \sum_{j\ge0}\frac{\alpha_j}{j!}\,s^j,\quad
	e^{-\Delta_n s}
	\beta(e^s)(1-e^{-s}) = \sum_{j\ge1}\frac{\beta_j}{j!}\,s^j,
	\quad e^{-\Delta_n s}
	\gamma(e^s) = \sum_{j\ge0}\frac{\gamma_j}{j!}\,s^j,
\end{align}
where all the coefficients depend on $n$ and are bounded. Note that 
we have the relations $M_{n,0}=1$, $M_{n,1}=0$, $\alpha_0=\alpha, 
\beta_1=\beta$ and $\gamma_0=\gamma$.

\begin{lmm} The $m$th central moment $M_{n,m}$ of $X_n$ satisfies the 
recurrence 
\begin{align}\label{Mnm-rr}
    M_{n,m} = \left(1-\frac{m\beta}{\alpha n+\gamma}\right)
	M_{n-1,m} + N_{n,m} \qquad(m\ge2),
\end{align}
where
\begin{align}\label{N-nm}
    \begin{split}
        N_{n,m} := \frac1{\alpha n+\gamma}
    	&\left(\sum_{2\le j\le m}\binom{m}{j}\lpa{
    	(\alpha_j n+\gamma_j-\beta_j \mu_{n-1}}
    	M_{n-1,m-j} \right.\\ &\left.
    	-\sum_{2\le j<m}\binom{m}{j}\beta_j
    	M_{n-1,m+1-j}  \right).
    \end{split}
\end{align}
\end{lmm}
\begin{proof} 
By extracting the coefficient of $s^m$ on both sides of 
\eqref{bar-Mns}, we obtain \eqref{Mnm-rr} with 
\begin{align*}
    N_{n,m} = \frac1{\alpha n+\gamma}
	&\left(\sum_{1\le j\le m}\binom{m}{j}
	(\alpha_j n-\beta_j \mu_{n-1}+\gamma_j) 
    M_{n-1,m-j} \right.\\ &\left.
	-\sum_{2\le j< m}\binom{m}{j}\beta_j
	M_{n-1,m+1-j} \right).
\end{align*}
Since $N_{n,1}=0$, we have the relation
\[
    \alpha_1n-\beta_1 \mu_{n-1}+\gamma_1
	=(\alpha'(1) -\alpha\Delta_n)n 
    -\gamma\Delta_n-\beta\mu_{n-1}+ 
	\gamma'(1)=0,
\]
which is nothing but \eqref{mu-rr}. Then \eqref{N-nm} follows.
\end{proof}

We now consider the general recurrence
\begin{align}\label{xy-yn}
    x_n = \left(1-\frac{m\beta}{\alpha n+\gamma}\right)
	x_{n-1}+y_n\qquad(n\ge n_0+1),
\end{align}
with $x_{n_0}\ne0$ and $\{y_n\}_{n>n_0}$ given. Without loss of 
generality, we assume that 
\[
    j\alpha-m\beta+\gamma\ne0 
    \qquad(j>n_0). 
\]
If this fails, then we can find a larger $n_0$ such that this 
holds. The solution of this recurrence is easily obtained by
iteration. 

\begin{lmm} \label{lmm-xnyn}
The solution to the recurrence \eqref{xy-yn} is given by 
\begin{align*}
    x_n 
    &= x_{n_0} \,\frac{\Gamma\lpa{n_0+1+\frac{\gamma}{\alpha}}
    \Gamma\lpa{n+1+\frac{\gamma-m\beta}{\alpha}}}
    {\Gamma\lpa{n_0+1+\frac{\gamma-m\beta}{\alpha}}
    \Gamma\lpa{n+1+\frac{\gamma}{\alpha}}}
	+\frac{\Gamma\lpa{n+1+\frac{\gamma-m\beta}{\alpha}}}
    {\Gamma\lpa{n+1+\frac{\gamma}{\alpha}}}
	\sum_{n_0< k\le n}
	\frac{\Gamma\lpa{k+1+\frac{\gamma}{\alpha}} }
	{\Gamma\lpa{k+1+\frac{\gamma-m\beta}{\alpha}}}
    \, y_k,
\end{align*}
for $n\ge n_0$. 
\end{lmm}

\begin{cor} Assume $m\ge1$. If $y_n\sim c n^\tau$, where $c\ne0$,
then 
\begin{align}\label{xnyn-at}
	x_n \sim \begin{cases}\displaystyle
        \frac{c}{1+\tau+\frac{m\beta}\alpha}
    	\, n^{1+\tau},& \text{ if }
        \tau>-1-\tfrac{m\beta}\alpha,\\ \displaystyle
        x_{n_0} \frac{\Gamma\lpa{n_0+1+\frac{\gamma}{\alpha}}}
        {\Gamma\lpa{n_0+1+\frac{\gamma-m\beta}{\alpha}}}
        \, n^{-\frac{m\beta}{\alpha}},&
        \text{ if }\tau<-1-\tfrac{m\beta}\alpha.
    \end{cases}    
\end{align}
\end{cor}
\begin{proof}
By \eqref{lmm-xnyn} using the asymptotic approximation to the ratio 
of Gamma functions (see \cite[\S~1.18]{Erdelyi1981})
\begin{align}\label{gamma-ratio}
    \frac{\Gamma(n+x)}{\Gamma(n+y)}
    = n^{x-y}\left(1+O\lpa{n^{-1}}\right),
\end{align}
for large $n$ and bounded $x$ and $y$. 
\end{proof}

\subsection{Asymptotics of $\mathbb{V}(X_n)$}

To prove Theorem~\ref{thm-clt}, we assume that condition
\eqref{thm-clt-positivity} holds. Consider the variance. We examine 
first the term (\eqref{N-nm} with $m=2$)
\[
    N_{n,2} = \frac{\alpha_2n+\gamma_2-\beta_2\mu_{n-1}}
	{\alpha n+\gamma}
	\sim \frac{\alpha_2-\beta_2\mu}\alpha ,
\]
where, by the definition \eqref{abc},
\begin{align*}
    \alpha_2 
    &= \alpha''(1) -(2\Delta_n-1)\alpha'(1)+\Delta_n^2\alpha,\\
    \beta_2 
    &= 2\beta'(1)-(2\Delta_n+1)\beta.
\end{align*}
Since we assume that $\alpha+2\beta>0$ (condition 
\eqref{thm-clt-positivity}), we can apply the asymptotic transfer 
\eqref{xnyn-at} (first case with $\tau=0$), and obtain 
\[
    M_{n,2} = \mathbb{V}(X_n)
	\sim \sigma^2 n,
\]
where, by Corollary~\ref{cor-D-bdd},
\begin{align*}
	\sigma^2 := \lim_{n\to\infty}
	\frac{\alpha_2-\beta_2\mu}{\alpha+2\beta}
	= \mu+\frac{\alpha''(1)-2\mu\beta'(1)-\alpha\mu^2}
    {\alpha+2\beta}.
\end{align*}

Note that the condition $\sigma^2>0$ is equivalent to 
\[
    \beta\lpa{\alpha''(1)+2\alpha'(1)
    (\alpha'(1)+\beta'(1))}>0,
\]
because $\alpha+2\beta>0$.

\subsection{Asymptotics of higher central moments}
We now prove by induction that 
\begin{equation}\label{mm}
	\left\{
	\begin{split}
	M_{n,2\ell} &\sim \frac{(2\ell)!}{\ell!2^\ell}
	\, \sigma^{2\ell} n^\ell, \\
	M_{n,2\ell-1} &= O\lpa{n^{\ell-1}}, 
	\end{split}
	\right.
\end{equation}
for $\ell\ge1$. This will imply particularly that $M_{n,m}=
O\lpa{n^{\tr{\frac m2}}}$ for $m\ge0$. Since \eqref{mm} with $\ell=1$
has already been proved, we now prove \eqref{mm} for $\ell\ge2$. 
Consider first the odd case $m=2\ell+1$. By \eqref{N-nm} and induction
hypothesis,
\[
    N_{n,2\ell+1} 
	= O\llpa{\sum_{2\le j\le 2\ell+1} 
    n^{\tr{\frac{2\ell+1-j}2}}}
	=O\lpa{n^{\ell-1}},
\]
implying that $M_{n,2\ell+1} = O\lpa{n^\ell}$. When $m=2\ell$, only 
the term with $j=2$ in the first sum on the right-hand side of 
\eqref{N-nm} is dominant, and we see that
\begin{align*}
	N_{n,2\ell} 
    &\sim \binom{2\ell}2
	\frac{\alpha_2 n-\beta_2\mu_{n-1}}{\alpha n} \,
    M_{n-1,2\ell-2}\\
	&\sim \binom{2\ell}2
	\frac{(2\ell-2)!}{2^{\ell-1}(\ell-1)!}\cdot
    \frac{\alpha_2-\beta_2\mu}\alpha \,
	\sigma^{2\ell-2} n^{\ell-1}\\
	&= \frac{(2\ell)!}{2^\ell (\ell-1)!}\cdot
    \frac{\alpha_2-\beta_2\mu}\alpha
	\,\sigma^{2\ell-2}n^{\ell-1}.
\end{align*}
By the asymptotic transfer \eqref{xnyn-at} with $m=2\ell$ and 
$\tau=\ell-1$, we then have
\[
	M_{n,2\ell} \sim \frac{\alpha_2-\beta_2\mu}
    {\ell(\alpha +2\beta)}\cdot
	\frac{(2\ell)!}{2^\ell (\ell-1)!}
	\,\sigma^{2\ell-2}n^{\ell},
\]
which proves the first claim in \eqref{mm}. This completes the proof
of \eqref{mm} and Theorem~\ref{thm-clt} by Frechet-Shohat's
convergence theorem (see \cite{Chow1988,Frechet1931}), which, for the
reader's convenience, is included here: it states that \emph{if the
$k$th moment of a sequence of random variables $Z_n$ tends to a
finite limit $\nu_k$ as $n\to\infty$, and the $\{\nu_k\}$'s are the
moments of a uniquely determined distribution function $Z$, then
$Z_n$ converges in distribution to $Z$.} This completes the proof of
\eqref{mm}, and in turn that of Theorem~\ref{thm-clt}. \qed

From the proof it is obvious that the analyticity of $\alpha(v), 
\beta(v)$ and $\gamma(v)$ on $|v|\le 1$ can be replaced by that in 
$|v|<1$ and the existence of all derivatives at unity. This will be 
needed in Section~\ref{sss-conway}. 

\subsection{Mean and variance in a more general setting}
\label{ss-mv-gen}
In general, for the framework \eqref{Pnv-gen} 
$P_n\in\ET{a_n(v),b_n(v);P_0(v)}$, we have 
\[
    P_n(1) = P_0(1) \prod_{1\le j\le n}a_j(1),
\]
(assuming each factors positive). Normalizing both sides by $P_n(1)$ 
gives
\[
    \bar P_n(v) := \frac{P_n(v)}{P_n(1)}
    = \frac{a_n(v)}{a_n(1)}\bar P_{n-1}(v)
    +\frac{b_n(v)}{a_n(1)}(1-v) \bar P_{n-1}'(v).
\]
Then the mean $\mu_n := \bar P_n'(1)$ satisfies
\[
    \mu_n 
    = \left(1-\frac{b_n(1)}{a_n(1)}\right)\mu_{n-1}
    +\frac{a_n'(1)}{a_n(1)},
\]
and the variance $\sigma_n^2$ satisfies, by the same 
shifting-the-mean technique used above, 
\[
   \sigma_n^2 
   = \left(1-\frac{2b_n(1)}{a_n(1)}\right)\sigma_{n-1}^2
   +\frac{a_n''(1)+2a_n'(1)-2b_n'(1)\mu_{n-1}}{a_n(1)}
   -\Delta_n^2-\Delta_n,
\]
where $\Delta_n := \mu_n-\mu_{n-1}$. These will be used later (see 
Section~\ref{ss-ext-ce} when $a_n(v)$ is not a linear function of 
$n$). 

\section{A complex-analytic approach}

In addition to the method of moments, which is elementary in nature,
we describe briefly a complex-analytic approach in this section,
which is equally useful in proving most of the CLTs we derive in this
paper but has remained less explored in the combinatorics literature.
Following Bender's pioneering work \cite{Bender1973}, this approach
is based on the EGF $F(z,v)$ of $P_n(v)$ (satisfying \eqref{Pnv-gen})
and relies on complex analysis (notably the singularity analysis
\cite{Flajolet1990}). It turns out that a simple asymptotic framework
in the form of quasi-powers \cite[\S~IX.5]{Flajolet2009}
\cite{Hwang1998} proves particularly useful for establishing the
asymptotic normality of the coefficients of $P_n(v)$.

\subsection{The partial differential equation and its resolution}
\label{ss-pde}
We begin with the PDE satisfied by the EGF of $P_n(v)$ (defined in   
\eqref{Pnv-gen}) 
\begin{align}\label{F-pde}
    \begin{cases}
        (1-\alpha(v)z)\partial_z F 
        -\beta(v)(1-v)\partial_v F
        -(\alpha(v)+\gamma(v))F = 0,\\
        F(0,v)= P_0(v).
    \end{cases}    
\end{align}
Such a first-order equation can often be solved by the method of
characteristics (see \cite{Evans2000,Myint-U2007}), which first
reduces a PDE to a family of ordinary DEs and then integrate the
solutions with the initial or boundary conditions. For \eqref{F-pde},
we start with the characteristic equation
\begin{align}\label{pde-ch-eq}
    \frac{\dd z}{1-\alpha(v)z}
    = -\frac{\dd v}{\beta(v)(1-v)}
	= \frac{\dd F}{(\alpha(v)+\gamma(v))F}.
\end{align}
The first equation can be written as 
\begin{align}\label{pde-zv}
    \frac{\dd z}{\dd v} -\frac{\alpha(v)}{\beta(v)(1-v)}\,z
    +\frac1{\beta(v)(1-v)}=0,
\end{align}
which is not always exactly solvable. In the special case 
when $\alpha(v)=\beta(v)$ (as in Sections~\ref{sec-abv} and 
\ref{sec-quadratic}), the above DE becomes 
\[
    (1-v)\frac{\dd z}{\dd v} -z
    =\frac{\dd{}}{\dd v}((1-v)z)
    =-\frac1{\beta(v)}. 
\]
Since $\beta(v)$ is in most cases a polynomial of low degree, this DE 
can often be solved explicitly. Such a simplification does not apply 
in general when $\alpha(v)\ne \beta(v)$, but we can still follow  
the standard procedure to characterize the solution (mostly in 
implicit forms).

From \eqref{pde-zv}, we see that either we have an ODE of separable 
type, or we have an explicit form for the integrating factor
\[
    I(v) := \exp\left(-\int
    \frac{\alpha(v)}{\beta(v)(1-v)}\dd v\right),
\]
the function in the exponent is taken as an antiderivative (or 
indefinite integral), which is then used to solve the DE 
\eqref{pde-zv} by quadrature as
\[
    \frac{\mbox{d}}{\dd v} 
    \left(I(v)z +\int \frac{I(v)}
    {\beta(v)(1-v)}\,\dd v \right)=0
    \Longleftrightarrow \xi(z,v)= C.
\]
Here the \emph{first integral} $\xi(z,v)$ can be made explicit in 
many cases we study in this paper. For example, when 
$\alpha(v)=\beta(v)$, we have 
\begin{align}\label{xi-z}
    \xi(z,v) = (1-v)z + \int \frac{\dd v}{\beta(v)},
\end{align}
where the integral is again an antiderivative. We then have the first
characteristics, which, after the changes of variables $u=\xi(z,v)$,
$w=v$ and $H(u,w)=F(z,v)$, leads to the ODE
\[
    \frac{\partial}{\partial w}\,H(u,w)
    +\frac{\alpha(w)+\gamma(w)}{\beta(w)(1-w)}\, 
    H(u,w)=0,
\]
which is the second equation of \eqref{pde-ch-eq}. This first-order 
DE is then solved and we obtain the general relations 
\[
    g(w) H(u,w) = G(u)
    \Longleftrightarrow g(v)F(z,v) = G(\xi(z,v)),
\]
where the integrating factor $g$ has the form 
\[
    g(v) = \exp\left(\int\frac{\alpha(v)+\gamma(v)}
	{\beta(v)(1-v)}\,\dd v\right).
\]
The last step is to specify $G$ by using the initial value at $z=0$:
\[
    g(v) P_0(v) = G(\xi(0,v)).
\]
We then conclude that 
\begin{align}\label{FJG}
    F(z,v) = \frac{G(\xi(z,v))}{g(v)}. 
\end{align}
This standard approach works for almost all cases we examine in 
this paper and has also been used in the combinatorics literature; 
see for example, \cite{Andre1895, Barbero2015, Chow2014a, Wilf2004}. 

Consider for example the Eulerian recurrence of type 
$\ET{qvn +p+(qr-p-q)v, qv;1}$; see \eqref{Pnv-Eabc} below. Then we 
have 
\begin{align*}
    I(v) &= \exp\left(-\int \frac{\dd v}
    {1-v}\right) = 1-v\\
    g(v) &= \exp\left(\int \frac{p(1-v)+qrv}
    {qv(1-v)}\dd v\right) = v^{\frac pq}(1-v)^{-r},
\end{align*}
and, by $P_0(v)=1$, 
\[
    G\lpa{q^{-1}\log v} = g(v), \quad
    \text{or}\quad G(w) = e^{pw}\lpa{1-e^{qw}}^{-r}.
\]
Finally, by \eqref{FJG},
\[
    F(z,v) = v^{-\frac pq}(1-v)^r
    e^{p(1-v)z+\frac1q\log v}\lpa{1-ve^{q(1-v)z}}^{-r}
    = e^{p(1-v)z}\left(\frac{1-v}{1-ve^{q(1-v)z}}\right)^r.
\]

When the integrals involved have no explicit forms such as the 
recurrence $\ET{(p+qv)n+1-p-qv,v;1}$ (see \cite{Rzadkowski2019} 
or Section~\ref{ss-ru} below), we can still apply the same procedure 
and get a solution in implicit form:
\begin{align}\label{Fzv-ru}
    F(z,v) 
    = \frac{1-v}{v}\cdot \frac{T\lpa{S(v)
    +\frac{(1-v)^{p+q}z}{v^p}}}{1-T\lpa{S(v)
    +\frac{(1-v)^{p+q}z}{v^p}}},
\end{align}
where $T(S(v))=v$ and 
\begin{align}\label{Rz-Ur-T}
    S(v) = \int v^{-p-1}(1-v)^{p+q-1}\dd v.
\end{align}
The form \eqref{Fzv-ru} is understood in the following formal power 
series sense:
\[
    T\left(S(v)+\frac{(1-v)^{p+q}z}{v^p}\right)
	= \sum_{m\ge0}\frac{T^{(m)}(S(v))}{m!}\,
	\left(\frac{(1-v)^{p+q}}{v^p}\right)^m
	z^m,
\]
where $T(S(v))=v$ and $T^{(m)}(S(v))$ are expressible in terms 
of $S^{(j)}(v)$ for $m, j\ge1$, which in turn are well-specified
by 
\[
    S'(v) = v^{-p-1}(1-v)^{p+q-1},
\]
and then $S^{(m)}=(S^{(m-1)})'$ for $m\ge2$. 

It is also possible to extend the approach when the non-homogeneous
terms are present; see the examples in Sections~\ref{sec-v2v},
\ref{ss-ru}, \ref{sec-1plusv}, \ref{sss-v2}, \ref{sss-chebikin},
\ref{sec-2v-1plusv}, and \ref{sec-3v2v}.

For ease of reference, we list the first integrals $\xi(z,v)$ in 
Table~\ref{tab-pde} for most examples (leading to asymptotic 
normality) studied in this paper. 

\begin{small}
\begin{center}
\def\arraystretch{1.5}
\begin{longtable}{lll}
	\multicolumn{3}{c}{{}} \\
	\multicolumn{1}{c}{Section} &
	\multicolumn{1}{c}{$(\alpha(v),\beta(v))$} &
	\multicolumn{1}{c}{$\xi(z,v)$} \\ \hline
\S~\ref{sec-abv} & $(qv,qv)$ & $(1-v)z+q^{-1}\log v$ \\
\S~\ref{sec-barbero} & $(qv,v)$ & $(1-v)^qz+\int v^{-1}
(1-v)^{q-1}\dd v$ \\ 
\S~\ref{sec-v2v} & $(\frac12v,v)$ &
$\sqrt{1-v}z+\frac12\log v - \log\lpa{1+\sqrt{1-v}}$ \\
\S~\ref{ss-ru} & $(p+qv,v)$ &
$\frac{(1-v)^{p+q}}{v^p}\,z+\int v^{-p-1}(1-v)^{p+q-1}\dd v$ \\
\S~\ref{sec-1plusv} & $\lpa{\frac12(1+v),\frac12(3+v)}$ &
$\sqrt{(1-v)(3+v)}\,z+2\arcsin(\frac12(1+v))$ \\
\S~\ref{sss-v2} & $(v,1+v)$ &
$\sqrt{1-v^2}\,z+\arcsin(v)$\\
\S~\ref{sss-v2} & $(v^2,v(1+v))$ &
$\sqrt{1-v^2}\,z-\text{arctanh}\lpa{\sqrt{1-v^2}}$\\ 
\S~\ref{sss-chebikin} & $\lpa{\frac12(1+v^2),
\frac12(1+v^2)}$ & $(1-v)z+2\arctan(v)$ \\
\S~\ref{sss-v1pv} & $(v(1+v),v(1+v))$ &
$(1-v)z+\log\frac{v}{1+v}$ \\ 
\S~\ref{sss-2v2} & $(2v^2,v(1+v))$ &
$(1-v^2)z-\log v$\\ 
\S~\ref{sec-2v-1plusv} & $(2qv,q(1+v))$ & $(1-v^2)z-\frac1q v$\\
\S~\ref{sss-23} & $(2(1+v),3+v)$ & $(1-v)(3+v)z+v$ \\
\S~\ref{sec-3v2v} & $(q(1+3v), 2qv)$ & 
$\frac{(1-v)^2}{\sqrt{v}}\,z - \frac{1+v}{q\sqrt{v}}$ \\
\S~\ref{sss-5p3v} & $(5+3v,2(1+v))$ & 
$\frac{(1-v)^2}{\sqrt{1+v}}\,z-\frac{3+v}{\sqrt{1+v}}$ \\
\S~\ref{sss-7p2v} & $(\frac13(7+2v), \frac13(5+4v)$ & 
$\frac{1-v}{\sqrt{5+4v}} z - \frac{3}{2\sqrt{5+4v}}$ \\
\S~\ref{sss-1p3v2} & $(1+3v^2,v(1+v))$ &
$\frac{(1-v^2)^2}{v}z - \frac{1+v^2}v$\\
\S~\ref{ss-ck} & $(-1+(q+1)v,qv)$ & 
$v^{\frac1q}(1-v)z +v^{\frac1q}$\\ \hline
\caption{The first integrals in some exactly solvable cases of 
\eqref{F-pde}.}
\label{tab-pde}   
\end{longtable}
\end{center}
\end{small}

\subsection{Singularity analysis and quasi-powers theorem for 
CLT}\label{ss-qpa}

Most EGFs in this paper have either algebraic or logarithmic
singularities and it is possible to study the limit laws of the
coefficients by examining the singular behavior of the EGF near its
dominant singularity; see \cite{Bender1973, Gao1992, Flajolet1993,
Hwang1994}. The following theorem, from Flajolet and Sedgewick's book
\cite[p.\ 676, \S~IX.7.2]{Flajolet2009}, is very useful for all
Eulerian recurrences we study in this paper and leads to a CLT with
optimal convergence rate; see also \cite{Bender1973} for the original
meromorphic version. The proof relies on the uniformity provided by
the singularity analysis \cite{Flajolet1990} coupling with the
quasi-powers theorems \cite[\S~IX.5]{Flajolet2009}.

\textbf{Notation.}
For notational convenience, we will write $X_n\sim
\mathscr{N}\lpa{\mu n,\sigma^2n; \ve_n}$, which means $X_n \sim
\mathscr{N}(\mu n,\sigma^2n)$ with the convergence rate $\ve_n$:
\[
    \sup_{x\in\mathbb{R}}\left|
    \mathbb{P}\left(\frac{X_n-\mu n}{\sigma \sqrt{n}}
    \le x\right) -\Phi(x)\right| = O(\ve_n),
\]
where $\ve_n\to0$. The convergence rate in the CLT is often referred 
to as the \emph{Berry-Esseen bound} in the probability literature. We 
will use interchangeably both terms. 

\begin{thm}[Algebraic Singularity Schema] \label{thm-saqp} Let
$F(z,v)$ be an analytic function at $(z,v)=(0,0)$ with nonnegative
coefficients. Under the following three conditions, the random
variables $X_n$ defined via the coefficients of $F$:
\[
    \mathbb{E}\lpa{v^{X_n}} := \frac{[z^n]F(z,v)}
    {[z^n]F(z,1)}
\]
satisfy $X_n \sim \mathscr{N}(\mu n,\sigma^2 n;n^{-\frac12})$, 
where the convergence rate is, modulo the implied constant, optimal. 
The three conditions are:
\begin{enumerate}
    \item Analytic perturbation: there exist three functions 
    $\Lambda, K, \Psi$, analytic in a domain 
    $D:=\{|z|\le\zeta\}\times 
    \{|v-1|\le\ve\}$, such that, for some $\zeta_0$ with 
    $0<\zeta_0\le \zeta$, and $\ve>0$, the following representation 
    holds, $\kappa\not\in\mathbb{Z}_{\le0}$, 
    \begin{align}\label{F-ABC}
        F(z,v)=\Lambda(z,v)+K(z,v)\Psi(z,v)^{-\kappa}; 
    \end{align}
    furthermore, assume that, in $|z| \le \zeta$, there exists a 
    unique root $\rho>0$ of the equation $\Psi(z, 1) = 0$, that this 
    root is simple, and that $K(\rho, 1)\not= 0$.

    \item Non-degeneracy: one has $\partial_z
    \Psi(\rho,1)\cdot\partial_v \Psi(\rho,1)\ne 0$, ensuring the 
    existence of a non-constant $\rho(v)$ analytic at $v = 1$, such 
    that $\Psi(\rho(v),v) = 0$ and $\rho(1) = \rho$.

    \item Variability: $\sigma^2(\rho):= \frac{\rho''(1)}{\rho(1)}
    +\frac{\rho'(1)}{\rho(1)}-\lpa{\frac{\rho'(1)}{\rho(1)}}^2\ne0 $.
\end{enumerate}
\end{thm}

For our purpose, we show how the two constants $(\mu,\sigma^2)$ can 
be computed from the dominant singularity $\rho(v)$. By the  
asymptotic approximation (see \cite[Eq.\ (64), p.\ 678]{Flajolet2009})
\begin{align}\label{qpa}
    [z^n]F(z,v) = g(v) n^{\kappa-1}\rho(v)^{-n}
    \left(1+O\lpa{n^{-1}}\right),
\end{align}
where the $O$-term holds uniformly in a neighborhood of $v=1$, we 
see that 
\[
    \mathbb{E}\lpa{v^{X_n}}
    = \frac{g(v)}{g(1)}\, 
    \exp\left(n\log \frac{\rho(1)}{\rho(v)}\right)
    \left(1+O\lpa{n^{-1}}\right),
\]
uniformly for $|v-1|\le\ve$. Thus 
\begin{align}\label{mu-sigma}
    \mu = -[s]\log\rho(e^s) = -\frac{\rho'(1)}{\rho}
    \quad\text{and}\quad
    \sigma^2(\rho) = 2[s^2]\log\rho(e^s) .
\end{align}
Note also that 
\[
    \rho'(1) = -\frac{\partial_v \Psi(\rho,1)}
	{\partial_z \Psi(\rho,1)}, 
\]
and it is often simpler to replace the second condition (of the 
Theorem) by $\rho'(1)\ne0$ or $\mu\ne0$. 

We illustrate the use of these expressions by the simplest example 
when $F$ has the form (see \eqref{Epqr})
\[
    F(z,v) = e^{p(1-v)z}
    \left(\frac{1-v}{1-ve^{q(1-v)z}}\right)^r,
\]
where $q,r>0$ and $p\le qr$ (implying that 
$[z^nv^k]F(,v)\ge0$). With the notations of \eqref{F-ABC}, 
we take $\kappa=r$, $\Lambda=0$, $K(z,v) = e^{p(1-v)z}$ and 
\[
    \Psi(z,v) := \frac{1-ve^{q(1-v)z}}{1-v}. 
\]
Then the dominant singularity $\rho(v)$ solves the equation 
$1=ve^{q(1-v)z}$ and $\rho(1)=q^{-1}$, namely,
\[
    \rho(v) = \frac{\log v}{q(v-1)}.
\]
One checks that $-\rho'(1)=\frac1{2q}\ne0$. Also by the Taylor 
expansion
\begin{align*}
    -\log\rho(e^s) &= \log q + \frac s2+
    \sum_{k\ge1}\frac{\text{Bernoulli}_{2k}}{(2k)\cdot (2k)!}
    \, s^{2k}\\
    &=\log q + \frac s2+\frac{s^2}{24}
    -\frac{s^4}{2880}+\frac{s^6}{181400}+O\lpa{|s|^8},
\end{align*}
we then obtain $(\mu,\sigma^2)=\lpa{\frac12,\frac1{12}}$. We see that
the variance constant does not require the calculation of the second
moment and the square of the mean, making it a
\emph{cancellation-free} approach for computing the variance; see
\cite{Hwang1994} for more information on quasi-powers framework.
Furthermore, finer results such as cumulants of higher orders and
more effective asymptotic approximations can be derived. For example,
in the above case, we see that all odd cumulants are bounded, and all
even cumulants are asymptotically linear; in particular, the fourth
and sixth cumulants are asymptotic to $-\frac1{120}n$ and
$\frac1{252}n$, respectively.

In Table~\ref{tab-qp}, we list the mean and the variance constants of 
a few cases to be discussed below. 

\begin{small}
\begin{center}
\def\arraystretch{1.5}
\begin{longtable}{lllll}
	\multicolumn{5}{c}{{}} \\
	\multicolumn{1}{c}{Section} &
	\multicolumn{1}{c}{$(\alpha(v),\beta(v))$} &
	\multicolumn{1}{c}{$F(z,v)$} &
	\multicolumn{1}{c}{$\rho(v)$} &
	\multicolumn{1}{c}{$(\mu,\sigma^2)$} \\ \hline	
\S~\ref{sec-abv} & $(qv,qv)$ & \eqref{Epqr} 
& $\frac{\log v}{q(v-1)}$ & $\lpa{\frac12,\frac1{12}}$ \\
\S~\ref{sec-barbero} & $(qv,v)$ & 
\eqref{barbero-egf} & 
$\frac{\int_v^1t^{-1}(1-t)^{q-1}\dd t}
{(1-v)^q}$ & $\lpa{\frac q{q+1}, 
\frac{q^2}{(q+1)^2(q+2)}}$ \\ 
\S~\ref{sec-v2v} & $(v,2v)$ & \eqref{v2v}
& $\frac1{2\sqrt{1-v}}\log\frac{1+\sqrt{1-v}}{1-\sqrt{1-v}}$
& $\lpa{\frac13,\frac2{45}}$ \\
\S~\ref{ss-ru} & $(p+qv,v)$ & \eqref{Fzv-ru}
& $\frac{\int_v^1t^{-p-1}(1-t)^{p+q-1}\dd t}{v^{-p}(1-t)^{p+q}}$
& $(\star)$ \\
\S~\ref{sec-1plusv} & $\lpa{\frac12(1+v),\frac12(3+v)}$ 
& \eqref{dd-rho} & $\frac{2\arccos\frac{1+v}2}
	{\sqrt{(1-v)(3+v)}}$ & $\lpa{\frac16,\frac{23}{180}}$ \\
\S~\ref{sss-v2} & $(v,1+v)$ 
& \eqref{Qpq} & $\frac{\arccos(v)}{\sqrt{1-v^2}}$
&  $\lpa{\frac13,\frac{8}{45}}$\\ 
\S~\ref{sss-v2} & $(v^2,v(1+v))$ 
& \eqref{Qpq} & $\frac{\log(1+\sqrt{1-v^2})-\log v}{\sqrt{1-v^2}}$
&  $\lpa{\frac23,\frac{8}{45}}$\\ 
\S~\ref{sss-chebikin} & $\lpa{\frac12(1+v^2),
\frac12(1+v^2)}$  & \eqref{egf-chebikin} 
& $\frac{\arccos\lpa{\frac{2v}{1+v^2}}}{v-1}$ 
& $\lpa{\frac12,\frac5{12}}$\\
\S~\ref{sss-v1pv} & $(v(1+v),v(1+v))$ 
& \eqref{egf-v1pv} & $\frac1{1-v}\log\frac{1+v}{2v}$
& $\lpa{\frac34,\frac{7}{48}}$  \\ 
\S~\ref{sss-2v2} & $(2v^2,v(1+v))$ & 
\eqref{egf-2v2} & $\frac{-\log v}{1-v^2}$
& $\lpa{1,\frac13}$ \\
\S~\ref{sec-2v-1plusv} & $(2qv,q(1+v))$ & \eqref{rho-2qv}
& $\frac1{q(1+v)}$ & $\lpa{\frac12,\frac14}$\\
\S~\ref{sss-23} & $(2(1+v),3+v)$ & \eqref{rho-21pv}
& $\frac1{3+v}$ & $\lpa{\frac14,\frac3{16}}$\\
\S~\ref{sec-3v2v} & $(q(1+3v), 2qv)$ & \eqref{ogf-narayana}
& $\frac1{q(1+\sqrt{v})^2}$ & $\lpa{\frac12,\frac18}$\\
\S~\ref{sss-5p3v} & $(5+3v,2(1+v))$ & \eqref{rho-5p3v}
& $\frac1{(\sqrt{2}+\sqrt{1+v})^2}$ & $\lpa{\frac14,\frac5{32}}$\\
\S~\ref{sss-7p2v} & $(\frac13(7+2v), \frac13(5+4v)$ & 
\eqref{rho-7p2v}
& $\frac2{3+\sqrt{5+4v}}$ & $\lpa{\frac19,\frac2{27}}$\\
\S~\ref{sss-1p3v2} & $(1+3v^2,v(1+v))$ & \eqref{rho-1p3v2}
& $\frac1{(1+v)^2}$ & $\lpa{1,\frac12}$\\
\S~\ref{ss-ck} & $(-1+(q+1)v,qv)$
& \eqref{egf-ck} & $\frac{v^{-\frac1q}-1}{1-v}$
& $\lpa{\frac{q+1}{2q},\frac{q^2-1}{12q^2}}$\\ \hline
\caption{The dominant singularity $\rho(v)$ and the corresponding 
mean and variance constants in some exactly solvable cases of 
\eqref{F-pde}. Here $(\star) = \lpa{\frac q{p+q+1}, 
\frac{q(p+1)(p+q)}{(p+q+1)^2(p+q+2)}}$; see~\S~\ref{ss-ru}.}
\label{tab-qp}   
\end{longtable}
\end{center}
\end{small}

In the next two sections (and in Section~\ref{sec-extensions}), we
will apply both Theorem~\ref{thm-clt} and Theorem~\ref{thm-saqp} to
polynomials whose coefficients follow asymptotically normal limit
laws. The main differences between the two theorems when
\emph{specializing to Eulerian recurrences} are similar to those
between an elementary and an analytic approach to asymptotics (see
\cite{Chern2002, Odlyzko1995}): Theorem~\ref{thm-clt} is more general
but gives weaker results, while Theorem~\ref{thm-saqp} gives stronger
approximations but needs the availability of tractable EGFs (often
from solving the corresponding PDEs). Note that both theorems are not
limited to Eulerian recurrences.

\begin{center}
\begin{tabular}{c|cc}
 & Theorem~ \ref{thm-clt} & Theorem~\ref{thm-saqp}  \\ \hline
nature & elementary & complex-analytic \\
based on & recurrence \eqref{Pnv-gen} & generating function \\
CLT & no rate & with optimum rate \\
\end{tabular}    
\end{center}    

\section{Applications I: $(\alpha(v),\beta(v))=(qv,qv)
\Longrightarrow\mathscr{N}\lpa{\frac12n,\frac1{12}n}$}
\label{sec-abv}

We gather in this section many applications of Theorems~\ref{thm-clt}
and \ref{thm-saqp}, grouping them according to the pair
$(\alpha(v),\beta(v))=(qv,qv)$; other pairs with $\alpha(v)\neq
\beta(v)$ or nonlinear $\alpha(v), \beta(v)$ are further categorized
in the next section. Despite our efforts to be comprehensive,
omissions may still remain in view of the large literature on
Eulerian numbers and their applications.

Before our discussions, we observe that the following three simple
transformations on polynomials do not change essentially the
distribution of the coefficients: 
\begin{itemize}
    \item \emph{shift}: $P_n(v) \mapsto P_{n+m}(v)$,
    \item \emph{translation}: $P_n(v)\mapsto v^mP_n(v)$, and 
    \item \emph{reciprocity (or row-reverse)}: $P_n(v) 
    \mapsto Q_n(v) := v^{n+m} P_n\lpa{\frac1v}$, where $m$ is  
	properly chosen so that $Q_n(v)$ is a polynomial in $v$ and 
    is referred to as the \emph{reciprocal polynomial} of $P_n$. 
\end{itemize}
In particular, the polynomials $Q_n(v) := v^{n+m}P_n(\frac1v)$ of
$P_n$ (defined in \eqref{Pnv-gen}) satisfy the recurrence
\[
    Q_n\in\EET{\left(v\alpha\lpa{\tfrac1v}
	-v(1-v)\beta\lpa{\tfrac1v}\right)n
	+v\gamma\lpa{\tfrac1v}
	-(m-1)v(1-v)\beta\lpa{\tfrac1v}}
    {v^2(1-v)\beta\lpa{\tfrac1v}}.
\]
Note specially that if $X_n$ (and $Y_n$) is defined by the 
coefficients of $P_n$ (and $Q_n$) as in \eqref{Xnk-general}, then 
$X_n+Y_n=n+m$. These operations sometimes provide additional 
computational efficiencies. In particular, we may assume in many 
cases that $P_0(v)=1$ and start the recurrence \eqref{Pnv-gen} from 
$n=1$.

For an easier classification of the examples, we introduce further 
the following definition. 
\begin{defi}[Equivalence of distributions] \label{rr-def}
Two random variables $X_n$ and $Y_n$ are said to be equivalent (or 
have the same distribution) if $X_n+dY_{n+m}=c_n$ for $n\ge n_0$ 
for some constant $d\ne0$, integers $m$ and $n_0$ and a deterministic 
sequence $c_n$.
\end{defi}

Eulerian numbers are the source prototype of our framework
\eqref{Pnv-gen}, and we saw in Introduction that they satisfy
\eqref{Pnv-gen} with $\alpha(v)=\beta(v)=v$. Theorem~\ref{thm-clt}
applies since $\alpha=\beta=\alpha'(1)=\beta'(1)=1$, and, by
\eqref{mu-var}, $\mu=\frac12$ and $\sigma^2=\frac1{12}$. The
literature abounds with diverse extensions and generalizations of
Eulerian numbers. It turns out that exactly the same limiting
$\mathscr{N}\lpa{\frac12 n, \frac1{12}{n}}$ behavior appears in a
large number of variants, extensions, and generalizations of Eulerian
numbers (by a direct application of Theorem~\ref{thm-clt}), which we
examine below. Furthermore, in almost all cases, the stronger result
$\mathscr{N}\lpa{\frac12 n, \frac1{12}{n}; n^{-\frac12}}$ also
follows from a direct use of Theorem~\ref{thm-saqp}.

\subsection{The class $\mathscr{A}(p,q,r)$}

One of the most common patterns we found with very rich combinatorial
properties among the extensions of Eulerian numbers is of the form
\begin{align}\label{Pnv-Eabc}
    P_n \in \ET{qvn+p+(qr-q-p)v, qv;1},
\end{align}
which covers more than 60 examples in OEIS (and many other non-OEIS 
ones) and leads always to the same 
$\mathscr{N}\lpa{\frac12n,\frac1{12}n;
n^{-\frac12}}$ behavior. The EGF of $P_n$ satisfies the PDE
\[
    (1-qvz)\partial_z F -qv(1-v) \partial_v F 
    =(p+(qr-p)v)F,
\]
with $F(0,v)=1$, which has the closed-form solution (see 
Section~\ref{ss-pde})
\begin{align}\label{Epqr}
    F(z,v) = e^{p(1-v)z}
    \left(\frac{1-v}{1-ve^{q(1-v)z}}\right)^r.
\end{align}
For convenience, we will write this form as $F\in\mathscr{A}(p,q,r)$.
We also write $c\mathscr{A}(p,q,r)$ to denote the class of
polynomials whose EGFs are of the form $cF(z,v)$. Although it is
possible to restrict our consideration to only the case $q=1$ by a
simple change of variables, we keep the form of three parameters
($p,q,r$) for a more natural presentation of the diverse examples.

For later reference, we state the following result. 
\begin{thm} \label{thm-Apqr}
Assume that the EGF $F$ of $P_n$ is of type $F\in\mathscr{A}(p,q,r)$.
If $q,r>0$ and $0\le p\le qr$, then the random variables $X_n$ 
defined on the coefficients of $P_n$ (\eqref{Xnk-general}) satisfies 
$X_n\sim \mathscr{N}\lpa{\frac12n,\frac1{12}n;n^{-\frac12}}$. More 
precise approximations to the mean and the variance are given by 
\begin{align}\label{Apqr-mv}
    \mathbb{E}(X_n)=\frac{n+r}{2}-\frac pq +O(n^{-1}),
    \;\text{ and }\;
    \mathbb{V}(X_n) = \frac{n+r}{12}+O(n^{-2}). 
\end{align}
\end{thm}
\begin{proof}
Observe that $q,r>0$ and $p\le qr$ imply $P_n(1)>0$ for $n\ge0$ and
$[v^k]P_n(v)\ge0$ for $k,n\ge0$. The CLT without rate
$\mathscr{N}\lpa{\frac12n,\frac1{12}n}$ follows easily from
Theorem~\ref{thm-clt}. The stronger version with optimal rate is
proved by applying Theorem~\ref{thm-saqp} (as already discussed in
Section~\ref{ss-qpa}). The finer estimates for $\mathbb{E}(X_n)$ and
$\mathbb{V}(X_n)$ are obtained by a direct calculation using either
the recurrence $\ET{qvn+p+(qr-q-p)v, qv;1}$ or the EGF (by computing 
$[z^nt]F(z,1+t)$ for the mean and $2[z^nt^2]F(z,1+t)$ for the second 
factorial moment). Note specially the smaller error term in the 
variance approximation in \eqref{Apqr-mv}; also when $r=1$, both 
$O$-terms in \eqref{Apqr-mv} are identically zero for $n\ge2$.  
\end{proof}

\begin{lmm} \label{Dabc} If $F\in\mathscr{A}(p,q,r)$, then
$\stackrel{\leftarrow}{F}\in\mathscr{A}(qr-p,q,r)$, where
$\stackrel{\leftarrow}{F}(z,v) := F(vz,\frac1v)$ denotes the EGF of
the reciprocal polynomial of $P_n$, and if $p=qr$, then $\partial_z
F\in p\mathscr{A}(p,q,r+1)$.
\end{lmm}
The proof is straightforward and omitted. Note that 
$\partial_zF$ corresponds to the EGF of $P_{n+1}$. 

\begin{cor}
If $F\in\mathscr{A}(p,q,r)$ with $p=\frac12qr$, then $P_n$ is 
symmetric or palindromic, namely, $P_n(v)=v^nP_n\lpa{\frac1v}$.	
\end{cor}

\begin{defi} \label{rr-def2}
We write $X_n(p,q,r)\deq X_n(p',q',r')$ if the random variables
associated with the two types $\mathscr{A}(p,q,r)$ and
$\mathscr{A}(p',q',r')$ (defined as in \eqref{Xnk-general}), 
respectively, are equivalent in the sense of Definition~\ref{rr-def}.
\end{defi}

\begin{cor} \label{cor-eq} If $p\ne qr$, then $X_n(p,q,r)\deq 
X_n(qr-p,q,r)$; if $p=qr$, then 
\begin{align}
    X_n(qr,q,r) 
    \deq X_n(0,q,r)
    \deq X_n(qr,q,r+1) 
    \deq X_n(q,q,r+1).
\end{align}
\end{cor}
This shows partly the advantages of considering the framework 
\eqref{Pnv-Eabc} and the EGF \eqref{Epqr}. 

We now discuss some concrete examples grouped according to increasing 
values of $q$. Most CLTs and their optimal Berry-Esseen bounds are 
new. 

\subsection{$q=1$}

\paragraph{Eulerian numbers}
By \eqref{eulerian-egf}, the Eulerian numbers are of type 
$\mathscr{A}(1,1,1)$, and, by Lemma~\ref{Dabc}, also of types 
$\mathscr{A}(1,1,2)$ and $\mathscr{A}(0,1,1)$. The correspondence to 
OEIS sequences is as follows. 
\begin{center}
\begin{tabular}{llll}
	\multicolumn{4}{c}{{}} \\
	\multicolumn{1}{c}{Description} &
	\multicolumn{1}{c}{OEIS} &
	\multicolumn{1}{c}{Type (in $\mathscr{A}$)} &
	\multicolumn{1}{c}{Type (in $\mathscr{E}$)}  \\ \hline		
Eulerian numbers ($1\le k\le n$) & 
\href{https://oeis.org/A008292}{A008292} 
& $\mathscr{A}(0,1,1)-1$ & $\GT{1}{vn,v;v}$ \\ 
Eulerian numbers ($1\le k\le n$)& 
\href{https://oeis.org/A123125}{A123125} 
& $\mathscr{A}(0,1,1)$ & $\ET{vn,v;1}$\\ 
Eulerian numbers ($0\le k< n$)& 
\href{https://oeis.org/A173018}{A173018} 
& $\mathscr{A}(1,1,1)$ & $\ET{vn+1-v,v;1}$\\ \hline
\end{tabular}	
\end{center}
Note that $v\mathscr{A}(1,1,1) = \mathscr{A}(0,1,1)+v-1$.
In addition to these, with $P_n$ defined by 
\href{https://oeis.org/A123125}{A123125}, the sequence 
\href{https://oeis.org/A113607}{A113607} equals $v^{n+1}+1+P_n(v)$ 
(with $1$'s at both ends of each row); we obtain the same CLT. 

\paragraph{LI Shanlan numbers} LI Shanlan\footnote{This author's
name appeared in the western literature ``under a bewildering variety
of fanciful spellings such as Li Zsen-Su or Shoo Le-Jen" (quoted from
\cite[Ch.\ 18]{Martzloff2006}) or Le Jen Shoo or Li Jen-Shu or Li
Renshu. We capitalize his family name to avoid confusion.}
(1810--1882) in his 1867 book \emph{Duoji Bilei}\footnote{In LI's 
context, ``Duo'' means some binomial coefficients, ``Ji'' means
summation, ``Bi'' is ``to compare'' and ``Lei'' is to classify (and 
``Bilei" means to compile and compare by types).} \cite[Ch.\ 
4]{Li1867} (\emph{Series Summations by Analogies}) studied
$\mathscr{A}(1,1,r+1)$, where $r=0,1,\dots$; see
\cite{Luo1982,Zhang1939} (in Chinese), \cite[p.\ 350]{Martzloff2006},
and \cite[Part II]{Wilson2013} for more modern accounts. In our
format, $P_n$ satisfies
\begin{align}\label{li}
    P_n\in \ET{vn+1+(r-1)v,v;1}.
\end{align}

The first few rows of these \emph{LI Shanlan numbers} are given in 
Table~\ref{tab-li}.
\begin{table}[!ht]
\begin{center}
\begin{tabular}{c|ccccc}
	$n\backslash k$
	& $0$ & $1$ & $2$ & $3$ & $4$  \\ \hline
    $0$ & $1$ &&&& \\
    $1$ & $1$ & $r$ &&& \\
    $2$ & $1$ & $1+3r$ & $r^2$ &&\\
    $3$ & $1$ & $4+7r$ & $1+4r+6r^2$ & $r^3$ &\\
	$4$ & $1$ & $11+15r$ & $11+30r+25r^2$ & $1+5r+10r^2+10r^3$ & 
	$r^4$ \\ 
\end{tabular}    
\end{center}
\caption{The first few rows of the polynomial $\ET{vn+1+(r-1)v,v;1}$.}
\label{tab-li}
\end{table}

Indeed, LI derived in \cite{Li1867} the identity
\[
    \sum_{1\le j\le m}j^n\binom{j+r-1}{j-1}
    = \sum_{0\le k\le n}\binom{m+n-k+r}{m-1-k}
    [v^k]P_n(v) 
\]
only for $n=1, 2, 3$ (generalizing a version of the identity later
often named after Worpitzky \cite{Worpitzky1883}), and mentioned the 
straightforward extension to higher powers, which was later carried 
out in detail by Zhang \cite{Zhang1939}, who also obtained many 
interesting expressions for $P_n(v)$. 

By Corollary~\ref{cor-eq}, we see that 
\begin{align}\label{A01r}
    X_n(1,1,r+1) 
    \deq X_n(0,1,r)
    \deq X_n(r,1,r) 
    \deq X_n(r,1,r+1).
\end{align}
Also by a change of variables, we have for any $p>0$
\begin{align}\label{Arrr}
    X_n(1,1,r+1)  \deq X_n(p,p,r+1).
\end{align}
In particular, the cases $r=0,1$ correspond to Eulerian numbers (so
that $\mathscr{A}(2,2,2)$ also leads to the same Eulerian
distribution \href{https://oeis.org/A008292}{A008292}), and the cases
$r=2,\dots,5$ appear in OEIS with suitable
offsets (see the table below), where they are referred to as
$r$-Eulerian numbers whose generating polynomials satisfy
$P_n\in\GT{r}{vn+1-v,v;1}$, which equals \eqref{li} by shifting
$n$ to $n-r$; see also Section~\ref{sss-r-E}.

\centering
\begin{tabular}{llll}
	\multicolumn{4}{c}{{}} \\
	\multicolumn{1}{c}{Description} &
	\multicolumn{1}{c}{OEIS} &
	\multicolumn{1}{c}{Type} &
	\multicolumn{1}{c}{Equivalent types} \\ \hline		
$2$-Eulerian & \href{https://oeis.org/A144696}{A144696} 
& $\mathscr{A}(1,1,3)$ & $\mathscr{A}(0,1,2)$,
$\mathscr{A}(2,1,2), \mathscr{A}(2,1,3)$ \\
$3$-Eulerian & \href{https://oeis.org/A144697}{A144697} 
& $\mathscr{A}(1,1,4)$ & $ \mathscr{A}(0,1,3)$, 
$ \mathscr{A}(3,1,3), \mathscr{A}(3,1,4)$ \\
$4$-Eulerian & \href{https://oeis.org/A144698}{A144698} 
& $\mathscr{A}(1,1,5)$ & $ \mathscr{A}(0,1,4)$, 
$ \mathscr{A}(4,1,4), \mathscr{A}(4,1,5)$ \\
$5$-Eulerian & \href{https://oeis.org/A144699}{A144699} 
& $\mathscr{A}(1,1,6)$ & $ \mathscr{A}(0,1,5)$, 
$ \mathscr{A}(5,1,5), \mathscr{A}(5,1,6)$ \\
$6$-Eulerian & \href{https://oeis.org/A152249}{A152249} 
& $\mathscr{A}(1,1,7)$ 
& $ \mathscr{A}(0,1,6)$, $\mathscr{A}(6,1,6), \mathscr{A}(6,1,7)$\\ 
\hline
\end{tabular}

\justifying

\medskip
These numbers found their later use in data smoothing techniques; see 
\cite[\S 4.3]{Montgomery1990}. For more information on $r$-Eulerian
numbers, see \cite{Bona2004, Ma2013b, Mezo2014} and the corresponding
OEIS pages. Combinatorial interpretation of the polynomials of type
$\mathscr{A}(1,1,r)$ was discussed by Carlitz in \cite{Carlitz1973}; 
these polynomials were also examined in the recent paper
\cite{Caro-Lopera2015} (without mentioning Eulerian numbers). The
distribution associated with $\mathscr{A}(0,1,p)$ appeared in
\cite{Dillon1968} and later in a random walk model
\cite{Janardan1993}.

The type $\mathscr{A}(q,1,q)$ (switching from $r$ to $q$ for 
convention) has also been studied in the combinatorics literature, 
corresponding to the recurrence satisfied by the $q$-analogue of 
Eulerian numbers ($\mathcal{S}_n$ being the set of all permutations 
of $n$ elements)
\[
    P_n(v) 
	= \sum_{\pi \in \mathcal{S}_n} q^{\text{cycle}(\pi)}
	v^{\text{exceedance}(\pi)+1},
\]
which is of type 
\begin{align}\label{riordan}
    P_n\in \ET{vn+q-v,v; 1};
\end{align}
see Foata and Sch\"utzenberger's book \cite[Ch.\ IV]{Foata1970}
for a detailed study. See also \cite[p.\ 235]{Riordan1958} and 
\cite{Carlitz1960, Dillon1968, Ikollo-Ndoumbe2016, Magagnosc1980}.
The type $\mathscr{A}(2,1,1)$ (with the different initial condition 
$P_2(v)=2$) enumerates big ($\ge2$) descents in permutations: 
\begin{center}
\begin{tabular}{llll}\hline
Big descents in perms. & \href{https://oeis.org/A120434}{A120434} 
& $\mathscr{A}(2,1,2)$ & $\GT{1}{vn+2-v,v;2}$ \\ 
Reciprocal of \href{https://oeis.org/A120434}{A120434} 
& \href{https://oeis.org/A199335}{A199335} 
& $\mathscr{A}(0,1,2)$ & $\ET{vn+v,v;1}$\\ \hline
\end{tabular}	
\end{center}
As already indicated above, these two distributions are also 
equivalent to those of $2$-Eulerian numbers and of 
$\mathscr{A}(2,1,3)$. 

By Theorem~\ref{thm-Apqr}, the polynomials \eqref{riordan} with any 
real $q>0$ lead to the same $\mathscr{N}\lpa{\frac12n,\frac1{12}n; 
n^{-\frac12}}$ asymptotic behavior. 

\paragraph{Generalized Eulerian numbers \cite{Carlitz1974, 
Morisita1971}} 
Morisita \cite{Morisita1971} introduced in 1971 in statistical 
ecology a class of distributions, which corresponds to 
$\mathscr{A}(p,1,p+q)$ in our notation, or
\begin{align}\label{morisita}
    P_n \in \ET{vn+p+(q-1)v,v;1}.
\end{align}

By Corollary~\ref{cor-eq}, $X_n(p,1,p+q)\deq X_n(q,1,p+q)$. Such
polynomials were also independently studied in 1974 by Carlitz and
Scoville \cite{Carlitz1974}, and are referred to as the
\emph{generalized Eulerian numbers}; see \cite{Charalambides1991,
Janardan1988, Janardan1993}.

The CLT for the coefficients of \eqref{morisita} was later derived in 
\cite{Charalambides1991} in a statistical context by checking the 
real-rootedness property and Lindeberg's condition, as motivated by 
\cite{Janardan1988, Morisita1971}, where the usefulness of these 
numbers is further highlighted via a few concrete models. See also 
\cite{Janardan1993} for more models leading to $X_n(p,1,p+q)$. 

In the context of random staircase tableaux, these polynomials were 
also examined in detail by Hitczenko and Janson \cite{Hitczenko2014}, 
where they derived not only a CLT but also an LLT. Moreover, they 
also address the situation when $p$ and $q$ may become large with 
$n$. 

\paragraph{Euler-Frobenius numbers} Dwyer \cite{Dwyer1940} studied
$\mathscr{A}(p,1,1)$, referred to as the ``cumulative numbers'' but 
better known later as the Euler-Frobenius numbers; see for example
\cite{Gawronski2013, Harris1994, Janson2013, Riordan1958} and the 
references therein. They are called \emph{non-central Eulerian 
numbers} in \cite[p.\ 538]{Charalambides2002}. The coefficients of 
such polynomials are nonnegative if $p\in[0,1]$; see also
\cite{Foulkes1980,Kaplansky1946}. The asymptotic normality
$\mathscr{N}\lpa{\frac12n,\frac1{12}n}$ of the coefficients is first 
proved in \cite{Harris1994} and later in \cite{Clark1998,
Gawronski2013, Janson2013} by different approaches; see also
\cite{Gawronski2013, Harris1994, Hwang1994, Janson2013} for local
limit theorems. In particular, an asymptotic expansion for $p=0$
(Eulerian numbers) was derived in the Ph.D. Thesis of the first
author \cite[p.\ 76]{Hwang1994}, the approach there being based on a
framework of quasi-powers \cite{Flajolet2009, Hwang1998} and a direct
Fourier analysis.

This class of polynomials is more useful than it seems because the 
coefficients of any polynomial of type $\mathscr{A}(p,q,1)$ with 
$q>0$ have the same distribution as $\mathscr{A}(\frac pq,1,1)$, 
which has nonnegative coefficients when $0\le p\le q$; see 
\cite{Janson2013} for details. 

\subsection{$q=2$}
\paragraph{Eulerian numbers} The sequence of polynomials 
\href{https://oeis.org/A296229}{A296229}, which corresponds to 
$2^n\eulerian{n}{k}$, is of type (shifting $n$ by $1$) 
$2\mathscr{A}(2,2,2)$, which has the same distribution as Eulerian 
numbers; see \eqref{Arrr}.

\paragraph{MacMahon numbers (or Eulerian numbers of type $B$)}
MacMahon numbers (first introduced in \cite{MacMahon1921}) are 
generated by the recurrence $P_n\in\ET{2vn+1-v,2v;1}$, which is of
type $\mathscr{A}(1,2,1)$; see Figure~\ref{fig-2v2v}. Their signed 
version is \href{https://oeis.org/A138076}{A138076}, and a
doubled-power version (with a zero between every two entries) is
\href{https://oeis.org/A158781}{A158781}. The CLT 
$\mathscr{N}\lpa{\frac12 n, \frac1{12}n}$ was proved
in \cite{Chen2009,Dasse-Hartaut2013,Janson2013}; see also
\cite{Diaconis2009,Schmidt1997}. The stronger results  
$\mathscr{N}\lpa{\frac12 n, \frac1{12}n; n^{-\frac12}}$ for these 
numbers follow readily from Theorem~\ref{thm-Apqr}.

\begin{center}
\begin{tabular}{llll}\hline
Eulerian numbers of type $B$ & 
\href{https://oeis.org/A060187}{A060187} 
& $\mathscr{A}(1,2,1)$ & $\ET{2vn+1-v,2v;1}$\\ 
\href{https://oeis.org/A060187}{A060187}: $v\mapsto v^2$ & 
\href{https://oeis.org/A158781}{A158781} & 
& $\ET{2v^2n+1-v^2,v(1+v);1}$\\ 
Signed version of \href{https://oeis.org/A060187}{A060187} & 
\href{https://oeis.org/A138076}{A138076} & \\ \hline
\end{tabular}	
\end{center}
The signed version \href{https://oeis.org/A138076}{A138076} can on 
the other hand be generated by $P_0(v)=1$ and 
\[
    P_n(v) = (2vn-1-v)P_{n-1}(v) -2v(1+v)P_{n-1}'(v)
    \qquad(n\ge1),
\]
whose EGF has the closed form expression $\mathscr{A}(1,2,1)$ but 
with $v\mapsto -v$ and $z\mapsto -z$. 

\paragraph{Polynomials arising from higher order derivatives} Many 
polynomials of the Eulerian type \eqref{Pnv-gen} are generated by 
successive differentiations of a given base function. Indeed, this is 
the very first genesis of Eulerian numbers (see \cite{Euler1755}):
\[
    (x\mathbb{D}_x)^n \frac1{1-x} 
    = \frac{P_n(x)}{(1-x)^{n+1}},
    \quad\text{where }
    P_n \text{ is of type } \mathscr{A}(0,1,1).
\]
For type $B$
\[
    \mathbb{D}_x^n \frac{e^x}{1-e^{2x}}
    = \frac{e^xP_n(e^{2x})}{(1-e^{2x})^{n+1}},
    \quad\text{where }
    P_n \text{ is of type } \mathscr{A}(1,2,1).
\]
Changing the base function to $\frac1{\sqrt{1-x}}$ gives 
\[
    (x\mathbb{D}_x)^n \frac1{\sqrt{1-x}}
    = \frac{P_n(x)}{2^n(1-x)^{n+\frac12}},
    \quad\text{where }
    P_n \text{ is of type } \mathscr{A}(0,2,\tfrac12).
\]
The last $P_n = $ \href{https://oeis.org/A156919}{A156919}$(n) =
v$\href{https://oeis.org/A185411}{A185411}$(n+1)$. (The former is
$\mathscr{A}(2,2,\frac32)$ while the latter is
$\mathscr{A}(0,2,\frac12)$). The same polynomials also appear in
\cite{Ma2013} in the form
\[
    (\tan(x) \mathbb{D}_x)^n \sec x
    = (\sec x)^{2n+1}P_n(\sin^2x),
    \quad\text{where }
    P_n \text{ is of type } \mathscr{A}(0,2,\tfrac12).
\]
By Corollary~\ref{cor-eq}
\[
    X_n(0,2,\tfrac12)
    \deq X_n(1,2,\tfrac12)  
    \deq X_n(1,2,\tfrac32) 
    \deq X_n(2,2,\tfrac32).
\]
In particular, $\mathscr{A}\lpa{1,2,\frac12}$ (the reciprocal of 
\href{https://oeis.org/A156919}{A156919}) also appears in
\cite{Savage2012} and corresponds to 
\href{https://oeis.org/A185410}{A185410}. 

More generally, we have 
\[
    (x\mathbb{D}_x)^n (1-x)^{-r}
    = \frac{P_n(x)}{(1-x)^{n+r}},
    \quad\text{where }
    P_n \text{ is of type } \mathscr{A}(0,1,r),
\]
and we have the equivalence relations \eqref{A01r}. 

On the other hand, Lehmer \cite{Lehmer1985} shows that, with $g(x) :=
\frac{x\arcsin x}{\sqrt{1-x^2}}$,
\begin{align}\label{lehmer}
    (x\mathbb{D}_x)^n g(x) 
    = \frac{P_n(x^2)g(x)+x^2R_n(x^2)}{(1-x^2)^n}, 
    \quad\text{where }
    P_n \text{ is of type } \mathscr{A}(1,2,\tfrac12),
\end{align}
and $R_n$ is Eulerian with a non-homogeneous term:
\begin{align}\label{lehmer2}
    R_n(v) = (2vn+2-4v)R_{n-1}(v) +2v(1-v)R_{n-1}'(v)
    +P_{n-1}(v)\qquad(n\ge1),
\end{align}
with $R_0(v)=0$. The EGF of $R_n(v)$ can be solved to be (by the 
approach described in Section~\ref{ss-pde})
\[
    e^{(1-v)z}\frac{\arcsin\lpa{2ve^{2(1-v)z}-1}
    -\arcsin(2v-1)}
    {2\sqrt{v(1-ve^{2(1-v)z})}}.
\]
The optimal CLT $\mathscr{N}(\frac12n,\frac1{12}n;n^{-\frac12})$ for
the coefficients of Lehmer's polynomials $P_n$ \eqref{lehmer} and
$R_n$ follows from an application of Theorem~\ref{thm-saqp}; see
Figure~\ref{fig-2v2v} for an illustration of the histograms. The CLT
$\mathscr{N}(\frac12n,\frac1{12}n)$ for this $P_n$ or
$\mathscr{A}(0,2,\frac12)$ was previously derived in \cite{Ma2013} by
the real-rootedness and unbounded variance approach. An LLT was also
established by Bender \cite{Bender1973}. See \cite{Ma2013b} for a
general treatment of derivative polynomials generated by context-free
grammars.

\begin{center}
\begin{tabular}{lll}\hline
$(x\mathbb{D}_x)^n \frac1{\sqrt{1-x}}$ & 
\href{https://oeis.org/A185411}{A185411} &
$\mathscr{A}(0,2,\frac12)$\\ 
$v$\href{https://oeis.org/A185411}{A185411}$(n+1)$ & 
\href{https://oeis.org/A156919}{A156919} 
& $\mathscr{A}(2,2,\frac32)$ \\ 
Lehmer's polynomials & 
\href{https://oeis.org/A185410}{A185410} 
& $\mathscr{A}(1,2,\frac12)$\\ \hline
\end{tabular}	
\end{center}

\begin{figure}[!h]
\begin{center}\small
\begin{tabular}{c c c}
\includegraphics[height=3.5cm]{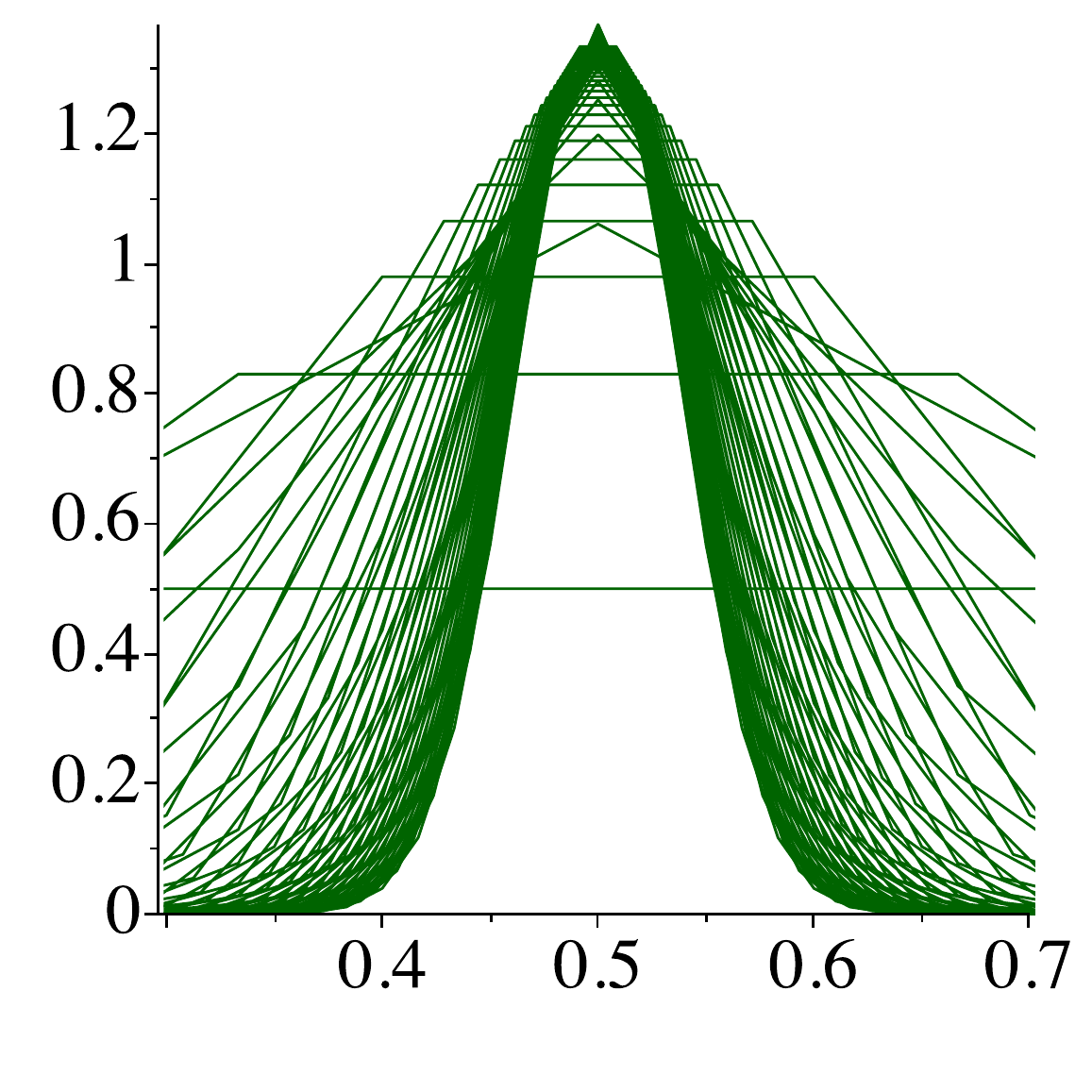} &
\includegraphics[height=3.5cm]{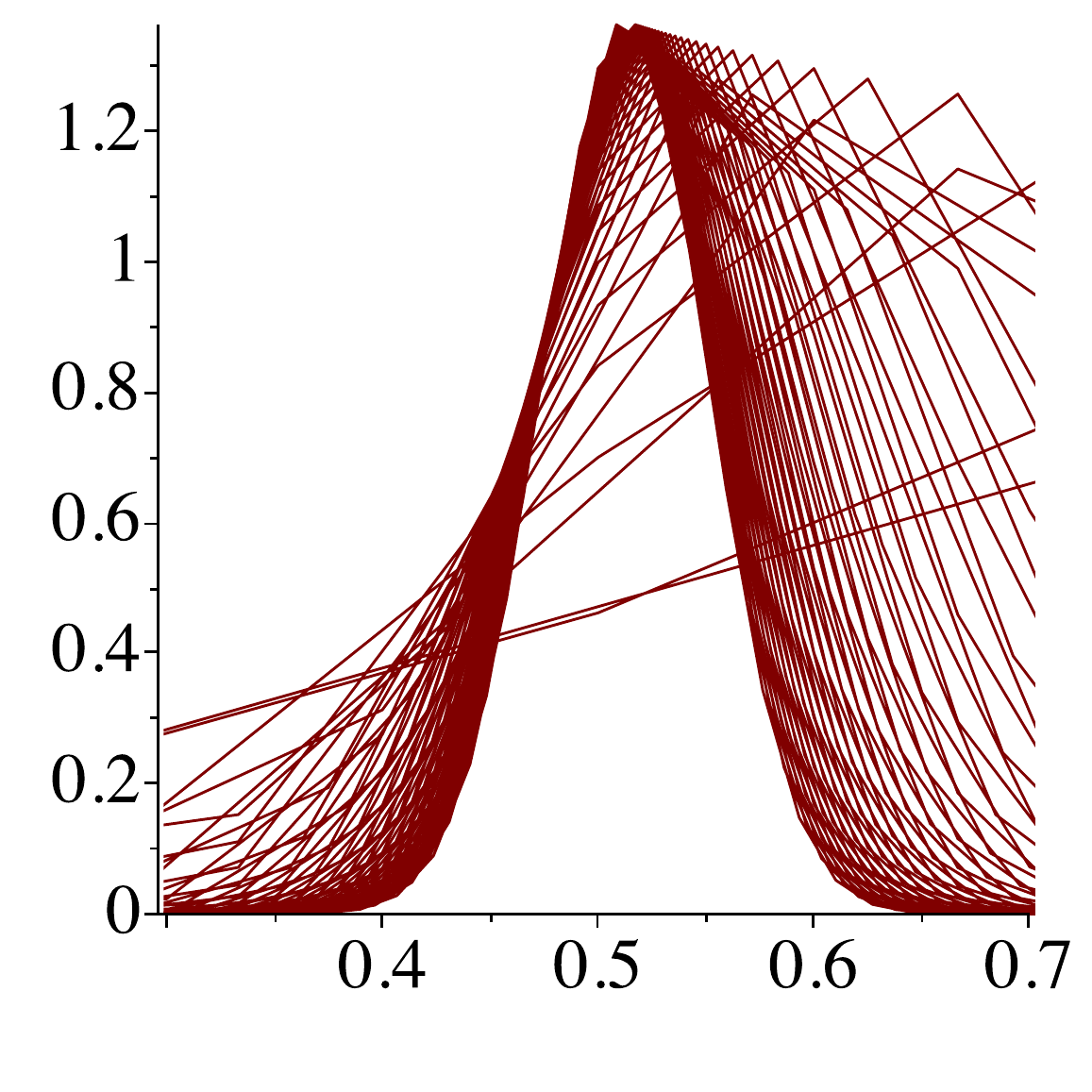} &
\includegraphics[height=3.5cm]{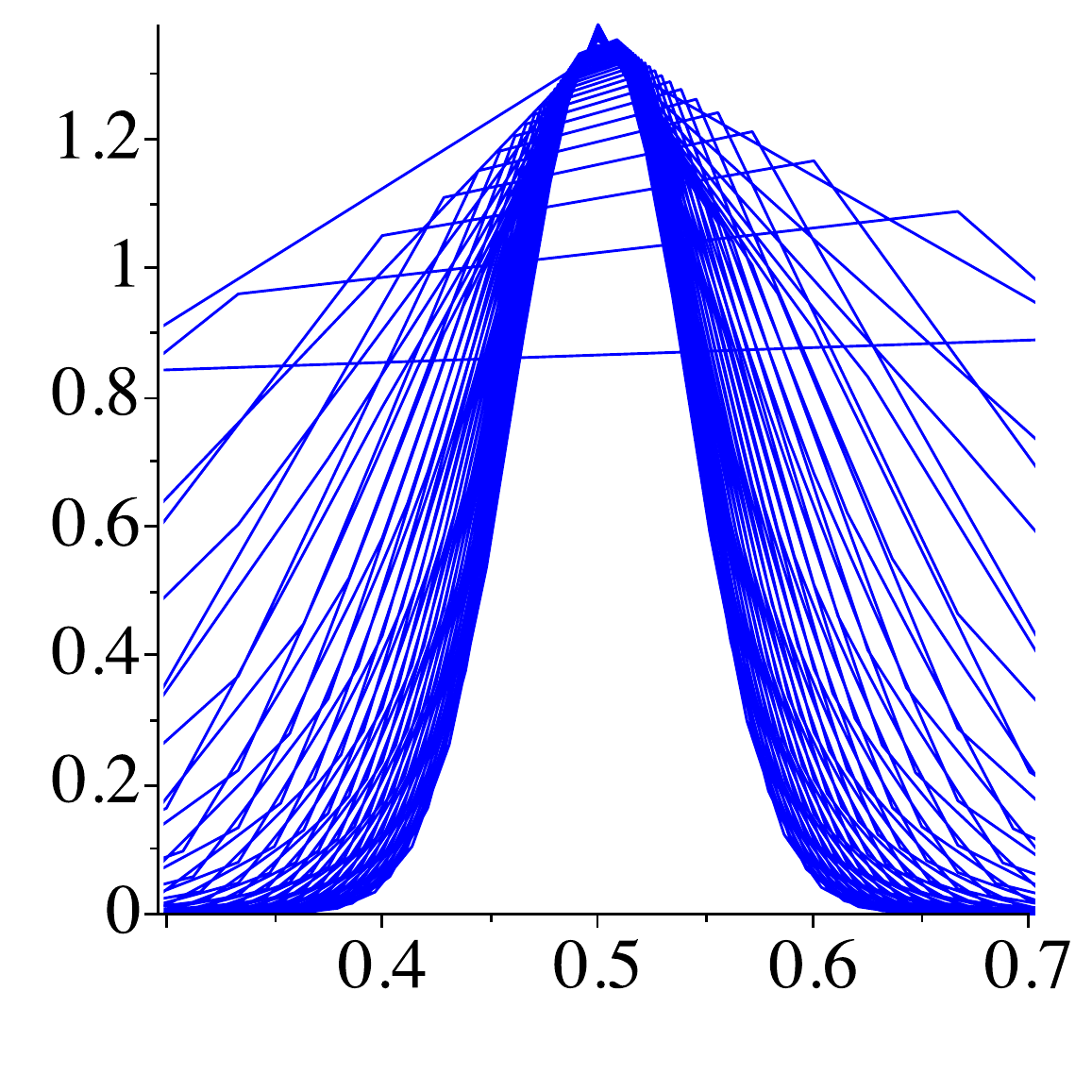} \\
Type $B$ Eulerian & Lehmer's $P_n$ \eqref{lehmer}
& Lehmer's $R_n$ \eqref{lehmer2} \\
\href{https://oeis.org/A060187}{A060187} $\mathscr{A}(1,2,1)$ 
& \href{https://oeis.org/A185410}{A185410} 
$\mathscr{A}(1,2,\frac12)$&  \\
$\lpa{\frac12n,\frac1{12}n+\frac1{12}}$ & 
$\lpa{\frac12n-\frac14,\frac1{12}n+\frac1{24}}$ &
$\lpa{\frac12n-\frac{11}{12},\frac1{12}n-\frac1{360}}$ 
\end{tabular}    
\end{center}
\caption{While we have the same CLT $\mathscr{N}(\frac12n,
\frac1{12}n;n^{-\frac12})$ for the three classes of polynomials,
their differences are reflected in the finer asymptotic
approximations to the mean and the variance, displayed in the last
row with the format $($mean, variance$)$; see \eqref{Apqr-mv}.}
\label{fig-2v2v}
\end{figure}

\paragraph{Stirling permutations of the second kind \cite{Ma2017}:
$\mathscr{A}(q,2,\frac q2)$} Ma and Yeh \cite{Ma2017} extended the
Stirling permutations of Gessel and Stanley \cite{Gessel1978} and
studied the so-called \emph{cycle ascent plateau}, leading to
polynomials of the type $\mathscr{A}(q,2,\frac q2)$. When $q=1$, we
get Lehmer's polynomial (\href{https://oeis.org/A185410}{A185410}),
and when $q=2$, we get Eulerian numbers (up to a factor of $2^n$).
The CLT $\mathscr{N}\lpa{\frac12n,\frac1{12}n;n^{-\frac12}}$ for the
coefficients (for any real $q>0$) follows from Theorem~\ref{thm-Apqr}.

\paragraph{Franssen's $\mathscr{A}(p,2,p)$ \cite{Franssens2006}}
The expansion 
\[
    \left(\frac{u-v}{ue^{-(u-v)z}-ve^{(u-v)z}}\right)^p
    = \sum_{n\ge0}R_n(u,v;p)\frac{z^n}{n!}
\] 
is studied in \cite{Franssens2006}. Let $P_n(v) := R_n(1,v;p)$. Then
$P_n\in \ET{2vn+p+(p-2)v,2v;1}$, which is of type
$\mathscr{A}(p,2,p)$. Note that when $p=1$ we get type $B$ Eulerian
numbers and when $p=2$, we get $2^n\eulerian{n+1}{k}$. For any real
$p>0$, we then obtain the asymptotic normality
$\mathscr{N}\lpa{\frac12n,\frac1{12}n; n^{-\frac12}}$ for the
coefficients of $P_n$.

\subsection{General $q>0$}
\paragraph{Savage and Viswanathan's $\mathscr{A}\lpa{1,q,\tfrac1q}$
\cite{Savage2012}} A class of polynomials called $1/k$-Eulerian is 
examined in \cite{Savage2012} (we changed their $k$ to $q$ for 
convenience) and is of type $P_n \in\ET{qvn+ 1-qv,qv;1}$.

In addition to Eulerian numbers when $q=1$, one gets Lehmer's
polynomials \eqref{lehmer} (or
\href{https://oeis.org/A185410}{A185410}) when $q=2$. By
Corollary~\ref{cor-eq}
\[
    X_n\lpa{1,q,\tfrac1q}
    \deq X_n\lpa{0,1,\tfrac1q}
    \deq X_n\lpa{1,q,\tfrac1q+1}
    \deq X_n\lpa{q,q,\tfrac1q+1},
\] 
for any $q>0$, which is a special case of \eqref{A01r} and 
\eqref{Arrr}. 

\paragraph{Strasser's $\mathscr{A}\lpa{1,q,\tfrac2q}$
\cite{Strasser2011}} A general framework studied in
\cite{Strasser2011} is of the form $P_n\in\ET{qvn+1-(q-1)v,qv;1}$,
where $q=1,2,\dots$. These polynomials are palindromic. Note that
when $q=0, 1$ and $2$, one gets binomial coefficients
\href{https://oeis.org/A007318}{A007318}, Eulerian numbers
\href{https://oeis.org/A008292}{A008292}, and MacMahon numbers
\href{https://oeis.org/A060187}{A060187}, respectively.
\begin{center}
\begin{tabular}{c c | c c | c c } \hline    
\href{https://oeis.org/A142458}{A142458} 
& $\mathscr{A}(1,3,\frac23)$ 
& \href{https://oeis.org/A142459}{A142459} 
& $\mathscr{A}(1,4,\frac12)$  
& \href{https://oeis.org/A142460}{A142460} 
& $\mathscr{A}(1,5,\frac25)$\\ 
  \href{https://oeis.org/A142461}{A142461} 
  & $\mathscr{A}(1,6,\frac13)$ 
& \href{https://oeis.org/A142462}{A142462} 
& $\mathscr{A}(1,7,\frac27)$
& \href{https://oeis.org/A167884}{A167884} 
& $\mathscr{A}(1,8,\frac14)$ \\ \hline
\end{tabular}     
\end{center}
On the other hand, the first few rows of $P_n(v)$ read $P_1(v) = 
1+v$, $P_2(v) = 1+2(1+q)v + v^2$ and
\begin{align*}
    P_3(v) &= 1+(3+6q+2q^2)v + (3+6q+2q^2)v^2 + v^3.
\end{align*}
Numerically, 
\vspace*{-0.4cm}
\begin{center}
\begin{tabular}{c|cccccccc}
$q$ & $1$ & $2$ & $3$ & $4$ & $5$ & $6$ & $7$ & $8$ \\ \hline
$3+6q+2q^2$ &$11$ & $23$ & $39$ & $59$ & $83$ & $111$ & $143$ 
& $179$\\
\end{tabular}    
\end{center}
We see that the CLT $\mathscr{N} \lpa{\frac12n,\frac1{12}n; 
n^{-\frac12}}$ remains the same for $q>0$ although these 
coefficients are more concentrated near the middle range 
for growing $q$. 

\paragraph{Brenti's $q$-Eulerian polynomials \cite{Brenti1994}}
A different $q$-analogue of Eulerian numbers considered in
\cite{Brenti1994} is of the form $P_n\in\ET{qvn+1-v,qv;1}$, which is
of type $\mathscr{A}(1,q,1)$; see also \cite{Steingrimsson1994}.
These polynomials also arise in the analysis of carries processes;
see \cite{Nakano2014}. The reciprocal polynomials are of type
$\mathscr{A}(q-1,q,1)$, which appeared on the webpage
\cite{Luschny2013}. In addition to Eulerian and MacMahon numbers for
$q=1$ and $q=2$, respectively, we also have
\begin{center}
\begin{tabular}{lll}\hline
	\href{https://oeis.org/A225117}{A225117} 
	& $\mathscr{A}(2,3,1)$ 
    & Reciprocal of $\mathscr{A}(1,3,1)$ \\
    \href{https://oeis.org/A225118}{A225118} 
	& $\mathscr{A}(3,4,1)$ 
    & Reciprocal of $\mathscr{A}(1,4,1)$ \\ 
    \href{https://oeis.org/A158782}{A158782} 
	& $\mathscr{A}(1,4,1)$ & $v\mapsto v^2$: 
	$\ET{4v^2n+1-v^2,2v(1+v);1}$\\ \hline
\end{tabular}	
\end{center}
The CLT and LLT when $q\ge1$ were derived in \cite{Chow2012} by the 
real-rootedness and Bender's approach \cite{Bender1973}, respectively.

\paragraph{Eulerian numbers associated with arithmetic progressions}
Eulerian numbers associated with the arithmetic progression $\{p, 
p + q, p + 2q,\dots\}$ are considered in Xiong et al.\ 
\cite{Xiong2013}, which corresponds to the polynomials 
$P_n\in\ET{qvn+(q-p)(1-v),qv(1-v);1}$; see also 
\cite{Mezo2016,Ramirez2018}. 

These polynomials are of type $\mathscr{A}(q-p,q,1)$, which have
nonnegative coefficients when $0\le p\le q$.

By Corollary~\ref{cor-eq}, $X_n(q-p,q,1)\deq X_n(p,q,1)$, and
polynomials of the latter type arise in the following extension of
Euler's original construction
\[
    P_n(v) := (1-v)^{n+1}\sum_{j\ge0}(p+qj)^nv^j
	\qquad(n\ge1),
\]
with $P_0(v)=1$ for a given pair $(p,q)$; see \cite{Eriksen2000,
Ramirez2018}. The polynomials associated with the type
$\mathscr{A}(p,q,1)$ were rediscovered in \cite{Samadi2004} in
digital filters and those with $\mathscr{A}(q-p,q,1)$ in
\cite{Pita-Ruiz-V.2017} in connection with sums of squares. In
particular, $(p,q)=(1,0)$ or $(1,1)$ gives Eulerian numbers and
$(p,q)=(1,2)$ the MacMahon numbers. Furthermore, two more sequences
were found in OEIS:
\begin{center}
\begin{tabular}{cc|cc}\hline
	\href{https://oeis.org/A178640}{A178640} 
	& $\mathscr{A}(5,8,1)=$ reciprocal of 
    $\mathscr{A}(3,8,1)$ &
	\href{https://oeis.org/A257625}{A257625} 
	& $\mathscr{A}(3,6,1)$ \\ \hline
\end{tabular}	
\end{center} 

A more general type is studied in Barry \cite{Barry2013}:
\[
    X_n(q(p+r)-p,q,p+r) \deq X_n(p,q,p+r);
\]
Theorem~\ref{thm-Apqr} applies when $p\ge0$, and $q,r>0$, and we get
always the same CLT $\mathscr{N}\lpa{\frac12n, \frac1{12}n;
n^{-\frac12}}$. See also \cite{Liu2015} for other properties such as
continued fraction expansions and $q$-log convexity.

Yet another type 
\[
    X_n(qr-r+1,q,r) \deq X_n(r-1,q,r),
\]
(referred to as the $r$-Eulerian-Fubini polynomials) was studied in 
\cite{Corcino2018}. The same CLT holds when $q>0$ and $r\ge1$. 

\paragraph{OEIS: $\mathscr{A}\lpa{p,q,\tfrac{2p}q}$}
Two dozens of OEIS sequences have the 
pattern 
\[
    [v^k]P_n(v) = \phi_k [v^k]P_{n-1}(v) 
	+ \phi_{n-k} [v^{k-1}]P_{n-1}(v)
	\qquad(1\le k\le n; n\ge1),
\]
with $P_0(v)=1$, where $\phi_k=p+qk$. Such polynomials $P_n$'s satisfy
$P_n\in\ET{qvn +p +(p-q)v,qv;1}$, which is of type 
$\mathscr{A}\lpa{p,q,\tfrac{2p}q}$. The sequences we found are listed 
below. 
\begin{center}
\begin{tabular}{cl|cl|cl}\hline
\href{https://oeis.org/A256890}{A256890} 
& $\mathscr{A}(2,1,4)$ 
& \href{https://oeis.org/A257180}{A257180} 
& $\mathscr{A}(3,1,6)$ 
& \href{https://oeis.org/A257606}{A257606} 
& $\mathscr{A}(4,1,8)$ \\
\href{https://oeis.org/A257607}{A257607} 
& $\mathscr{A}(5,1,10)$ 
& \href{https://oeis.org/A257608}{A257608} 
& $\mathscr{A}(1,9,\frac29)$ 
& \href{https://oeis.org/A257609}{A257609} 
& $\mathscr{A}(2,2,2)$ \\
\href{https://oeis.org/A257610}{A257610} 
& $\mathscr{A}(2,3,\frac43)$ 
& \href{https://oeis.org/A257611}{A257611} 
& $\mathscr{A}(3,2,3)$ 
& \href{https://oeis.org/A257612}{A257612} 
& $\mathscr{A}(2,4,1)$ \\
\href{https://oeis.org/A257613}{A257613} 
& $\mathscr{A}(4,2,4)$ 
& \href{https://oeis.org/A257614}{A257614} 
& $\mathscr{A}(2,5,\frac45)$ 
& \href{https://oeis.org/A257615}{A257615} 
& $\mathscr{A}(5,2,5)$ \\
\href{https://oeis.org/A257616}{A257616} 
& $\mathscr{A}(2,6,\frac23)$ 
& \href{https://oeis.org/A257617}{A257617} 
& $\mathscr{A}(2,7,\frac47)$ 
& \href{https://oeis.org/A257618}{A257618} 
& $\mathscr{A}(2,8,\frac12)$ \\
\href{https://oeis.org/A257619}{A257619} 
& $\mathscr{A}(2,9,\frac49)$ 
& \href{https://oeis.org/A257620}{A257620} 
& $\mathscr{A}(3,3,2)$ 
& \href{https://oeis.org/A257621}{A257621} 
& $\mathscr{A}(3,4,\frac32)$ \\
\href{https://oeis.org/A257622}{A257622} 
& $\mathscr{A}(4,3,\frac83)$ 
& \href{https://oeis.org/A257623}{A257623} 
& $\mathscr{A}(3,5,\frac65)$ 
& \href{https://oeis.org/A257624}{A257624} & 
$\mathscr{A}(5,3,\frac{10}3)$ \\ 
\href{https://oeis.org/A257625}{A257625} 
& $\mathscr{A}(3,6,1)$ 
& \href{https://oeis.org/A257626}{A257626} 
& $\mathscr{A}(6,3,4)$ 
& \href{https://oeis.org/A257627}{A257627} & 
$\mathscr{A}(3,7,\frac67)$ \\ \hline
\end{tabular}	
\end{center}
When $p=1$, one obtains Strasser's generalizations and more
OEIS sequences are listed above. Note that
both $(1,1,2)$ and $(2,2,2)$ lead to Eulerian numbers and $(1,2,1)$
to MacMahon numbers. All these types of polynomials produce the same
$\mathscr{N}\lpa{\frac12n,\frac1{12}n;n^{-\frac12}}$ limiting 
behavior.

\paragraph{A summarizing table for generic types} We summarize the 
above discussions in the following table, listing only generic types 
and their equivalent ones. 

\centering
\renewcommand{\arraystretch}{1.3}
\begin{longtable}{lll}
	\multicolumn{3}{c}{{}} \\
	\multicolumn{1}{c}{References} &
	\multicolumn{1}{c}{Type \& its equivalent types} \\ \hline		
LI Shanlan \cite{Li1867} 
& $\mathscr{A}(1,1,q+1); 
\mathscr{A}(q,1,q+1), 
\mathscr{A}(0,1,q), 
\mathscr{A}(q,1,q)$ \\
\makecell[l]{Riordan \cite{Riordan1958}\\
             Foata and Sch\"utzenberger \cite{Foata1970}}
& $\mathscr{A}(q,1,q); 
\mathscr{A}(0,1,q), \mathscr{A}(q,1,q+1), 
\mathscr{A}(1,1,q+1)$ \\ 
Brenti \cite{Brenti1994}, Luschny \cite{Luschny2013} 
& $\mathscr{A}(1,q,1); \mathscr{A}(q-1,1,1)$ \\
Dwyer \cite{Dwyer1940}, Harris \cite{Harris1994}
& $\mathscr{A}(q,1,1); \mathscr{A}(1-q,1,1)$ \\  \hline
Savage and Viswanathan \cite{Savage2012} 
& $\mathscr{A}(1,q,\frac1q);
\mathscr{A}(0,1,\frac1q), 
\mathscr{A}(1,q,\frac{q+1}q), 
\mathscr{A}(q,q,\frac{q+1}{q})$ \\
Strasser \cite{Strasser2011} 
& $\mathscr{A}(1,q,\frac2q)$  \\  
\makecell[l]{Morisita \cite{Morisita1971}\\
             Carlitz and Scoville \cite{Carlitz1974}\\
             Hitczenko and Janson \cite{Hitczenko2014}}
& $ \mathscr{A}(p,1,p+q); 
\mathscr{A}(q,1,p+q)$  \\ 
\makecell[l]{Xiong et al.\ \cite{Xiong2013}, 
	OEIS\\ Eriksen et al.\ \cite{Eriksen2000}}
& $ \mathscr{A}(p,q,1); 
\mathscr{A}(q-p,q,1)$ \\ \hline
Ma and Yeh \cite{Ma2017} & $\mathscr{A}\lpa{q,2,\frac q2};
\mathscr{A}\lpa{0,2,\frac q2}, 
\mathscr{A}\lpa{q, 2, \frac{q+2}2}, 
\mathscr{A}\lpa{2,2,\frac{q+2}2}$ \\
Franssens \cite{Franssens2006} &  $\mathscr{A}(q,2,q)$  \\ 
OEIS & $\mathscr{A}(p,q,\frac{2p}q)$ &   \\
Oden et al.\ \cite{Oden2006} & $\mathscr{A}(p-q,q,\frac {2p}q);
\mathscr{A}(p+q,q,\frac{2p}q)$ \\ 
Corcino et al.\ \cite{Corcino2018} &
$\mathscr{A}(pq-p+1,q,p); \mathscr{A}(p-1,q,p)$ \\
Barry \cite{Barry2013} & $\mathscr{A}(p,q,r); 
\mathscr{A}(qr-p,q,r)$  \\ \hline
\caption{A summary of generic types of $\mathscr{A}(p,q,r)$
and their equivalent ones.}
\end{longtable}	

\justifying

\subsection{Other extensions with the same CLT and their variants}
\label{ss-euler-clt}

We briefly mention some other examples not of the form 
$\mathscr{A}(p,q,r)$ but with the same CLT 
$\mathscr{N}\lpa{\frac12n,\frac1{12}n}$; more examples with the same 
CLT are discussed in Section~\ref{sec-extensions}.  

\subsubsection{The two examples in the Introduction}
The first example (see Figure~\ref{fig-complex-roots}) is of the form 
$\ET{vn+(1+v)^2,v;1}$ with $\alpha(v)=\beta(v)=v$ and
$\gamma(v)=(1+v)^2$. We can directly apply Theorem~\ref{thm-clt} and 
get the same CLT $\mathscr{N}\lpa{\frac12n,\frac1{12}n}$ for the 
distribution of the coefficients. The EGF 
\[
    e^{(1-v)z+v(1-e^{(1-v)z})}\left(\frac{1-v}
    {1-ve^{(1-v)z}}\right)^5
\]
can be derived by the procedures in Section~\ref{ss-pde}. 
Analytically, this is of the form $\mathscr{A}(1,1,5)$ times the  
entire function $e^{v(1-e^{(1-v)z})}$, and we get the optimal 
Berry-Esseen bound $n^{-\frac12}$ by applying Theorem~\ref{thm-saqp}. 

Similarly, the second example
\href{https://oeis.org/A244312}{A244312} \eqref{A244312} in the
Introduction leads to the same CLT
$\mathscr{N}\lpa{\frac12n,\frac1{12}n}$ by the method of moments
because it can be rewritten as $P_n\in\GT{1}{vn-1+(1-v)
\mathbf{1}_{n \text{ is odd}}, v;v}$, where again 
$\alpha(v)=\beta(v)=v$, and $\gamma(v)$ is less important in the
dominant terms of the asymptotic approximations to the moments. In 
particular, the mean and the variance are given respectively by 
\[
    \mathbb{E}(X_n)
    = \begin{cases}
        \frac{n^2}{2(n-1)}, & n\ge2 \text{ is even };\\
        \frac{n+1}2, & n\ge3 \text{ is odd }, 
    \end{cases}
    \quad \text{and}\quad
    \mathbb{V}(X_n)
    = \begin{cases}
        \frac{n(n^2-2n-2)}{12(n-1)^2}, 
        & n\ge4 \text{ is even};\\
        \frac{(n+1)(n-3)}{12(n-2)}, 
        & n\ge3 \text{ is odd}.
    \end{cases}
\]

The optimal Berry-Esseen bound is expected to be of order
$n^{-\frac12}$, but the analytic proof via Theorem~\ref{thm-saqp}
fails due to the lack of solution to the PDE \eqref{A244312-egf}
satisfied by the EGF of $P_n$. Note that it can be shown that
\[
    P_n(v) = (1-v)^n\sum_{j\ge0}
    j^{\tr{\frac12n}}(j+1)^{\cl{\frac12n}-1}v^{j+1}
    \qquad(n\ge1).
\]
From this expression, we can derive the optimal Berry-Esseen bound 
$n^{-\frac12}$; details will be given elsewhere. 

In such a context, we see particularly that the method of moments 
provides more robustness in the variation of $\gamma(v)$ in the 
recurrence \eqref{Pnv-gen} as long as the coefficients $[v^k]P_n(v)$ 
remain nonnegative, although the analytic approach is not limited to 
Eulerian type or nonnegativity of the coefficients. 

\subsubsection{$r$-Eulerian numbers again} 
\label{sss-r-E}

The following six OEIS sequences are all
generated by the same recurrence $P_n\in\GT{2}{vn+1,v}$, with
initial conditions $P_2(v)$ different from that ($1+4v+v^2$) of
Eulerian numbers:
\begin{center}
\begin{tabular}{cl|cl|cl}\hline
\href{https://oeis.org/A166340}{A166340} & $1+8v+v^2$ & 
\href{https://oeis.org/A166341}{A166341} & $1+10v+v^2$ &
\href{https://oeis.org/A166343}{A166343} & $1+12v+v^2$ \\ 
\href{https://oeis.org/A166344}{A166344} & $1+6v+v^2$ & 
\href{https://oeis.org/A166345}{A166345} & $1+2v+v^2$ & 
\href{https://oeis.org/A188587}{A188587} & $1+v+v^2$	\\ \hline
\end{tabular}	
\end{center}
See also the paper by Conger \cite{Conger2010} for the polynomials 
$\GT{r}{vn+1-2v,v;A_r(v)}$ for fixed $r=1,2,\dots$, where
$A_r(v)$ is Eulerian polynomial of order $r-1$. Since
Theorem~\ref{thm-clt} does not depend specially on the initial
conditions, we obtain the same CLT
$\mathscr{N}\lpa{\frac12n,\frac1{12}n}$ by a simple shift of the
recurrence $n\mapsto n-r$ and then by applying Theorem~\ref{thm-clt}.
The corresponding EGF can also be worked out, which leads to an
effective version of CLT by Theorem~\ref{thm-saqp}.

\subsubsection{Eulerian numbers of type $D$} 
\label{ss-type-d}

Brenti \cite{Brenti1994} (see also \cite{Chow2003}) shows that the
EGF of the Eulerian polynomials $P_n(v)$ of type $D$ is given by
\[
    F(z,v) 
	= \frac{(1-v)\lpa{e^{(1-v)z}-vze^{2(1-v)z}}}
    {1-ve^{2(1-v)z}}.
\]
By the decomposition ($P_n$ being palindromic)
\begin{align*}
    F(z,v) - (1-v)z=
    \frac{1-v}{1-ve^{2(1-v)z}}\lpa{e^{(1-v)z}-z}
    \in \mathscr{A}(1,2,1)-z\mathscr{A}(0,2,1),
\end{align*}
we see that, up to the term $(1-v)z$, type $D$ is a difference of
type $B$ and type $A$ Eulerian numbers; see \cite{Stembridge1994}.
Theorem~\ref{thm-clt} does not apply because these polynomials do not
have the pattern \eqref{Pnv-gen}. However, the coefficients do
satisfy the same CLT $\mathscr{N}\lpa{\frac12n,\frac1{12}n;
n^{-\frac12}}$ by applying Theorem~\ref{thm-saqp}.

\subsubsection{Exponential perturbation} 
Polynomials of the form 
\[
    P_n(v) = (2vn+ 1-v)P_{n-1}(v) +2v(1-v)P_{n-1}'(v)
	\mp v(1-v)^{n-1}\qquad(n\ge1),
\]
with $P_0(v)=0$ (for ``$+$'') and $P_0(v)=1$ (for ``$-$'') are studied
in \cite{Borowiec2016}, which correspond to
\href{https://oeis.org/A262226}{A262226} (``$-$'') and
\href{https://oeis.org/A262227}{A262227} (``$+$''), respectively. The
EGF equals
\[
    \frac{(1-v)e^{(1-v)z}}
    {2\lpa{1-ve^{2(1-v)z}}}\mp\frac{e^{(1-v)z}}2.
\]
While Theorem~\ref{thm-clt} does not apply, the method of proof
easily extends to this case because the extra ``exponential
perturbation" term does not contribute to the dominant asymptotics of 
all finite moments. We then get the same CLT 
$\mathscr{N}\lpa{\frac12n,\frac1{12}n}$ (as that for 
$\mathscr{A}(1,2,1)$). For both polynomials, Theorem~\ref{thm-saqp} 
applies. 

Another sequence \href{https://oeis.org/A180246}{A180246} corresponds 
essentially to $\mathscr{A}(2,1,1)$ (differing by the term $(-v)^n$). 
This is a concrete polynomial with $p>qr$ (see \eqref{Pnv-Eabc} and 
\eqref{Epqr}), and thus the coefficients are not all positive. More 
precisely, if $P_n$ is of type $\mathscr{A}(2,1,1)$, then $P_n$ is, 
up to minor exponential perturbation, of type $\mathscr{A}(0,1,1)$ 
(Eulerian numbers) because 
\[
    e^{2(1-v)z}\frac{1-v}{1-ve^{(1-v)z}}
    = \frac1{v^2}\left(\frac{1-v}{1-ve^{(1-v)z}}
    +(1-v)\lpa{1+ve^{(1-v)z}} \right).
\]
On the other hand, all coefficients $[v^k]P_n(v)$ are positive except 
the following three ones:
\begin{align*}
    [z^n]P_n(v) &= (-1)^n, 
    \quad [z^{n-1}]P_n(v) = (-1)^{n-1}(n+1),\\
    [z^{n-2}]P_n(v) &= (-1)^{n}\left(
    \binom{n+1}{2}+(-1)^n\right).
\end{align*}
Thus if we consider the random variables defined via the absolute
values of all coefficients, then we still obtain the same CLT
$\mathscr{N}\lpa{\frac12n, \frac1{12}n;n^{-\frac12}}$ because the
above possibly negative coefficients are asymptotically negligible.
The same argument applies to the more general type 
$\mathscr{A}(p,1,1)$, or (see \cite{Harris1994})
\[
    P_n(v) = \sum_{0\le k\le n}
    v^k\sum_{0\le j\le k}\binom{n+1}{j}(-1)^j
    (k+p-j)^n,
\]
where $p>1$. For, 
\[
    e^{p(1-v)z}\frac{1-v}{1-ve^{(1-v)z}}
    = v^{-p}\frac{1-v}{1-ve^{(1-v)z}}+O(1),
\]
uniformly for $z\sim -\frac{\log v}{1-v}$. Thus, up to a few possibly 
negative coefficients that are asymptotically negligible, the
polynomials are essentially Eulerian polynomials. 

\begin{center}
\begin{tabular}{lll}\hline
Type $D$ Eulerian 
& \href{https://oeis.org/A066094}{A066094} & 
$\mathscr{N}\lpa{\frac12n,\frac1{12}n;n^{-\frac12}}$ \\ 
$\sum_{0\le j\le k} \binom{n+1}{j}(-1)^j(k+2-j)^n$
& \href{https://oeis.org/A180246}{$|$A180246$|$} & 
$\mathscr{N}\lpa{\frac12n,\frac1{12}n;n^{-\frac12}}$ \\ 
Primary type $D$ Eulerian  
& \href{https://oeis.org/A262226}{A262226} & 
$\mathscr{N}\lpa{\frac12n,\frac1{12}n;n^{-\frac12}}$ \\ 
Complementary type $D$ Eulerian 
& \href{https://oeis.org/A262227}{A262227} & 
$\mathscr{N}\lpa{\frac12n,\frac1{12}n;n^{-\frac12}}$ \\ 
\hline
\end{tabular}	
\end{center}

\subsubsection{Eulerian polynomials multiplied by $1+v$}
\label{sss-conway}
Let $P_n(v) := (1+v)\sum_{0\le k<n}\eulerian{n}{k}v^k$. Such 
polynomials arose in the study of low-dimensional lattices (see 
\cite{Conway1997}), and satisfy the recurrence 
\[
	\mathscr{E}\left\langle\!\left\langle
    vn+\frac{1-v}{1+v}, v;1+v
    \right\rangle\!\right\rangle.
\]
These polynomials correspond to
\href{https://oeis.org/A008518}{A008518} and are specially
interesting because $\gamma(v)$ (in the notation of
Theorem~\ref{thm-clt}) is not a polynomial. The same limit law
$\mathscr{N}\lpa{\frac12n,\frac1{12}n}$ holds by an extension of
Theorem~\ref{thm-clt} (because $\gamma(v)=\frac{1-v}{1+v}$ is not
analytic in $|v|\le 1$). However, from the proof of
Theorem~\ref{thm-clt}, it is clear that the analyticity of
$\gamma(v)$ in $|v|<1$ and the finiteness of $\gamma^{(j)}(1)$ for
each $j\ge0$ are sufficient to guarantee the same CLT. In contrast, 
Theorem~\ref{thm-saqp} easily applies. 

\section{Applications II: $\alpha(v)\neq\beta(v)$
or quadratic $\alpha(v)$, $\beta(v)$}
\label{sec-quadratic}

We consider in this section other Eulerian-type polynomials for which
Theorem~\ref{thm-clt} applies. Exact solutions for the associated 
PDEs when $\alpha(v)\ne\beta(v)$ are still possible but they are 
often of a less explicit form (especially when compared with the
equal case~\eqref{Pnv-Eabc}). Yet our approaches still apply as far 
as the limit laws are concerned. 

We discuss a few such frameworks for which explicit EGFs are 
available before specializing to concrete examples. Note that in all 
cases we discuss below, Theorem~\ref{thm-clt} applies and we obtain a 
CLT easily. Following the same spirit of Section~\ref{sec-abv}, 
we use the special forms of EGFs for a more synthetic discussion of 
the examples as well as for establishing a stronger CLT with optimal 
rate by Theorem~\ref{thm-saqp}. 

\subsection{Polynomials with $(\alpha(v),\beta(v))
=(qv,v)\Longrightarrow\mathscr{N}\lpa{\frac q{q+1}\,n, 
    \frac{q^2}{(q+1)^2(q+2)}\,n}$}

\label{sec-barbero}

A class of higher-order Eulerian numbers is proposed in Barbero G.\
et al.\ \cite{Barbero2015} satisfying the recurrence
$P_n\in\ET{qvn+p+(r-p-q)v, v;1}$, where $q\ge1$ and $r\ge p\ge1$ are
integers. The EGF has the closed-form expression \cite{Barbero2014}
\begin{align}\label{barbero-egf}
    F(z,v) := \sum_{n\ge0}P_n(v)\frac{z^n}{n!}
    =\left(\frac{T_{q}\lpa{e^{(1-v)^q z}
    S_q(v)}}{v}\right)^p
    \left(\frac{1-v}{1-T_q\lpa{e^{(1-v)^q z}
    S_q(v)}}\right)^{r},
\end{align}
where $T_q(S_q(v))=S_q(T_q(v))=v$, $S_q$ is a one-parameter family of 
functions given by 
\[
    S_q(v) = ve^{L_q(v)},\quad
    \text{with}\quad
    L_q(v) = \sum_{1\le j<q}\binom{q-1}{j}
    \frac{(-v)^j}{j}.
\]
If we change $L_q(v)$ to 
\begin{align}\label{Lqv}
    L_q(v) := \int_0^v\frac{(1-t)^{q-1}-1}{t}\,\dd t,
\end{align}
then \eqref{barbero-egf} holds for real $p,q,r$. For convenience, we
write the framework \eqref{barbero-egf} as $F \in \mathscr{T}(p,q.r)$.

\begin{thm} \label{thm-barbero}
Assume $P_n\in\ET{qvn+p+(r-p-q)v, v;1}$. If 
\begin{align}\label{Fpqr-nn}
    q\ge1, r\ge p\ge0, \text{ and } r+p>0,
\end{align}
then the coefficients of $P_n$ satisfy the CLT
\begin{align}\label{barbero-clt}
    \mathscr{N}\llpa{\frac q{q+1}\,n, 
    \frac{q^2}{(q+1)^2(q+2)}\,n; n^{-\frac12}}. 
\end{align}
\end{thm}
\begin{proof}
By examining the corresponding recurrence for the coefficients, we 
see that if $q\ge1$ and $r\ge p\ge0$, then $[v^k]P_n(v)\ge0$; the 
additional condition $r+p>0$ guarantees positivity of $P_n(1)$. 
Thus under \eqref{Fpqr-nn}, Theorem~\ref{thm-clt} applies and we see 
that the coefficients of $P_n(v)$ satisfy the CLT \eqref{barbero-clt}
without rate. On the other hand, Theorem~\ref{thm-saqp} also applies 
by taking there $\kappa=r$ and 
\[
    \Psi(z,v) := \frac{1-T_q\lpa{e^{(1-v)^q z}S_q(v)}}{1-v}.
\]
The dominant singularity $\rho(v)$ is given by 
\[
    \rho(v) := \frac{\log S_q(1)-\log S_q(v)}{(1-v)^q}
    = \frac1{(1-v)^q}\int_v^1 t^{-1}(1-t)^{q-1}\dd t. 
\]
The mean and the variance constants can then be computed by the
relations $\rho'(1)=-\frac1{q+1}$ and $\rho''(1)=\frac2{q+2}$.
\end{proof}

In particular, 
\begin{center}
\begin{tabular}{cccc} 
$q=1$ & $q=2$ & $q=3$ & $q=4$ \\  \hline
$\mathscr{N}\lpa{\frac12n,\frac1{12}n}$ &
$\mathscr{N}\lpa{\frac23n,\frac1{9}n}$ &
$\mathscr{N}\lpa{\frac34n,\frac9{80}n}$ &
$\mathscr{N}\lpa{\frac45n,\frac8{75}n}$ \\ 
\end{tabular}	
\end{center}
Interestingly, as a function of $q$, the variance coefficient 
$\frac{q^2}{(q+1)^2(q+2)}$ first increases and then steadily 
decreases to $0$ as $q$ grows, the maximum occurring at 
$q=\frac{1+\sqrt{17}}2\approx 2.56$ with the value 
$\frac18\lpa{71-17\sqrt{17}}\approx 0.113$.

The reciprocal polynomial of $P_n$ satisfies the recurrence 
\[
    Q_n\in\ET{(q-1+v) n 
    +r+1-p-q-(1-p)v, v;1},
\]
whose coefficients follow the CLT $\mathscr{N}\lpa{\frac 1{q+1}\,n, 
\frac{q^2}{(q+1)^2(q+2)}\,n;n^{-\frac12}}$ under the same conditions 
$r\ge p\ge0$, $r+p>0$ and $q\ge1$.

\subsubsection{$q=\frac12\Longrightarrow 
\mathscr{N} \lpa{\frac13n,\frac2{45}n;n^{-\frac12}}$}
\label{sec-v2v}

David and Barton examined in their classical book \cite{David1962}
the number of increasing runs of length at least two
(\href{https://oeis.org/A008971}{A008971}), and the number of peaks
in permutations (\href{https://oeis.org/A008303}{A008303}), in
addition to Eulerian numbers. They derived the corresponding
recurrences:
\begin{small}
\begin{center}
\begin{tabular}{llll}\hline
\# $(|\!\uparrow\!\! \text{ runs}|\ge2)$ in permutations 
& \href{https://oeis.org/A008971}{A008971} 
& $\ET{vn+1-v,2v;1}$ 
&$\mathscr{T}\lpa{\frac12,\frac12,1}$ \\
\# peaks in permutations 
& \href{https://oeis.org/A008303}{A008303} 
& $\GT{1}{vn+2(1-v),2v;1}$ 
& $\mathscr{T}\lpa{1,\frac12,1}$\\ \hline
\end{tabular}	
\end{center}
\end{small}
The first few rows of both sequences are given in
Table~\ref{tab-v-2v}. To apply Theorem~\ref{thm-barbero} (which
starts the recurrence from $n=1$), we shift $n$ in both recurrences
by $1$, changing $\gamma(v)$ from ``$1-v$'' and ``$2(1-v)$'' to
``$1$'' and ``$2-v$'' respectively. Then the polynomials
$2^{-n}P_n(v)$ are of type $\mathscr{T}\lpa{\frac12, \frac12,1}$ and
$\mathscr{T}\lpa{1,\frac12,1}$, respectively. We thus obtain the same
CLT $\mathscr{N} \lpa{\frac13n,\frac2{45}n;n^{-\frac12}}$ for both
statistics by Theorem~\ref{thm-barbero}. In particular, about
two-thirds of runs have length $\ge2$; also note that the variance
constant $\frac2{45}$ is very small.

\begin{table}[!ht]
\begin{center}
\begin{minipage}{0.45\textwidth}
\begin{tabular}{c|ccccc}
    \multicolumn{6}{c}{\href{https://oeis.org/A008971}{A008971}}\\
	$n\backslash k$
    & $0$ & $1$ & $2$ & $3$ & $4$ \\ \hline
    $1$ & $1$ \\
    $2$ & $1$ & $1$  & & \\
    $3$ & $1$ & $5$ \\
    $4$ & $1$ & $18$ & $5$  \\
    $5$ & $1$ & $58$ & $61$ \\
    $6$ & $1$ & $179$ & $479$ & $61$ \\
    $7$ & $1$ & $543$ & $3111$ & $1385$\\
    $8$ & $1$ & $1636$ & $18270$ & $19028$ & $1385$\\ 
\end{tabular}
\end{minipage}\qquad\;
\begin{minipage}{0.45\textwidth}
\begin{tabular}{c|cccc}
    \multicolumn{5}{c}{\href{https://oeis.org/A008303}{A008303}}\\
	$n\backslash k$
    & $0$ & $1$ & $2$ & $3$  \\ \hline
    $1$ & $1$ \\
    $2$ & $2$ & & \\
    $3$ & $4$ & $2$ \\
    $4$ & $8$ & $16$  \\
    $5$ & $16$ & $88$ & $16$ \\
    $6$ & $32$ & $416$ & $272$ \\
    $7$ & $64$ & $1824$ & $2880$ & $272$\\
    $8$ & $128$ & $7680$ & $24576$ & $7936$\\
\end{tabular}
\end{minipage}
\vspace*{.2cm}
\caption{The first few rows of 
\href{https://oeis.org/A008971}{A008971} (left) and 
\href{https://oeis.org/A008303}{A008303} (right).}
\label{tab-v-2v}
\end{center}
\end{table}

Instead of using \eqref{barbero-egf}, the exact solutions for the 
bivariate EGFs have the simpler alternative forms
\begin{equation}\label{v2v}
    \begin{split}
    \href{https://oeis.org/A008971}{\text{A008971}}:\;&
    \frac{\sqrt{1-v}}{\sqrt{1-v}\cosh\lpa{\sqrt{1-v}\,z}
    -\sinh\lpa{\sqrt{1-v}\,z}}, \\
    \href{https://oeis.org/A008303}{A008303}: \; &
    1+\frac{v\sinh\lpa{\sqrt{1-v}\,z}}
    {\sqrt{1-v}\cosh\lpa{\sqrt{v-1}\,z}-\sinh\lpa{\sqrt{1-v}\,z}},
    \end{split}
\end{equation}
respectively, which can be derived directly by the approach of  
Section~\ref{ss-pde}; see \cite{Chow2014b,Entringer1969,Ma2013a,
Petersen2015,Warren1996}.

These numbers also appear in other different contexts \cite{Dale1988,
Hackl2018, Kermack1938, Mallows2008, Minai1993, Norton2013} (notably
\cite{Kermack1938}). See also \cite{Flajolet1997} for a connection to
binary search trees. D\'esir\'e Andr\'e \cite{Andre1895} seems the
first to give a detailed study of
\href{https://oeis.org/A008303}{A008303} (up to a proper shift) where
he examined the number of ascending or descending runs in cyclic
permutations. He derived not only the recurrence for the polynomials
and the first two moments of the distribution, but also solved the
corresponding PDE for the EGF. For more information (including
asymptotic normality), see \cite{David1962,Warren1996} and the
references therein.

\subsubsection{$q=1\Longrightarrow 
\mathscr{N} \lpa{\frac12n,\frac1{12}n;n^{-\frac12}}$} 
In this case, $L_1(z) = 0$, $S_1(v)=T_1(v)=v$, so that  
\[
    F(z,v) =e^{p(1-v)z}
    \left(\frac{1-v}{1-v e^{(1-v)z}}\right)^r,
\]
implying that $\mathscr{T}(p,1,r) = \mathscr{A}(p,1,r)$, 
which we already discussed in Section~\ref{sec-abv}. 

\subsubsection{$q=2\Longrightarrow 
\mathscr{N} \lpa{\frac23n,\frac19n;n^{-\frac12}}$} 
In this case, $S_2(z) = ze^{-z}$ and $T_2(z) = ze^{T_2(z)}
=\sum_{n\ge1}\frac{n^{n-1}}{n!}\, z^n$ is the Cayley tree function
(essentially the Lambert $W$-function; see \cite{Corless1996} and
\href{https://oeis.org/A000169}{A000169}), so that
\begin{align}\label{T-p2r}
    F(z,v) = \left(\frac{T_2
    \lpa{ve^{-v+(1-v)^2 z}}}{v}\right)^p
    \left(\frac{1-v}{1-T_2\lpa{ve^{-v+(1-v)^2 z}
    }}\right)^r.
\end{align}
The simple relations 
\begin{align}\label{T-class-eq}
    \partial_z \mathscr{T}(p,2,p) = p\mathscr{T}(p,2,p+2)
    \quad\text{and}\quad 
    \partial_z \mathscr{T}(0,2,p) = pv\mathscr{T}(1,2,p+2),
\end{align}
imply an equivalence relation for the underlying random variables 
in each case. 

In particular, $\mathscr{T}(0,2,1)$ gives the second order Eulerian
numbers (or Eulerian numbers of the second kind): $P_n\in
\ET{(2n-1)v,v;1}$.

Such polynomials arise in many different combinatorial and
computational contexts; see for example \cite{Carlitz1965,
Corless1996, Gautschi1959, Gessel1978, Graham1994, Janson2008, 
Lang2017, Petersen2015} and OEIS 
\href{https://oeis.org/A008517}{A008517} for 
more information. In addition to enumerating the number of ascents in 
Stirling permutations (see \cite{Bona2008,Gessel1978, Janson2008}), 
we mention here two other relations: as derivative polynomials \cite{Corless1996}
\[
    \mathbb{D}_x^{n+1}T_2(e^x)
    = \frac{P_n(-T_2(e^x))}{(1-T_2(e^x))^{2n+1}}
    \qquad(n\ge1),
\]
and as coefficients in an asymptotic expansion \cite{Carlitz1965}
\[
    \frac{n!}{(nv)^n}\left(e^{nv}
    -\sum_{0\le j\le n}\frac{(nv)^j}{j!}\right)
    = \sum_{0\le j<K}\frac{(-1)^jP_j(v)}
    {n^j(1-v)^{2j+1}}+ O\lpa{n^{-K}},
\]
for any $K=1,2,\dots$.

The CLT $\mathscr{N}\lpa{\frac 23n,\frac19n}$ seems first proved 
in \cite{Bergeron1992, Mahmoud1993} in the context of leaves in 
plane-oriented recursive trees, and later in \cite{Bona2008, 
Janson2008}, the approaches used including analytic, urn models and 
real-rootedness, respectively. 

The corresponding reciprocal polynomials $Q_n(v) := v^{n+1}P_n(\frac
1v)$ satisfy $Q_n\in\ET{(1+v)n-1-2v,v;1}$, which is 
\href{https://oeis.org/A163936}{A163936}. We
summarize these in the following table.

\begin{center}
\begin{tabular}{llll}\hline
Second order Eulerian ($1\le k\le n$) & 
\href{https://oeis.org/A008517}{A008517} & 
$\mathscr{T}(0,2,1)$ & $\mathscr{N}\lpa{\frac 23n, \frac19n; 
n^{-\frac12}}$\\ 
Reciprocal of \href{https://oeis.org/A008517}{A008517} & 
\href{https://oeis.org/A112007}{A112007} &&
$\mathscr{N}\lpa{\frac 13n, \frac19n;n^{-\frac12}}$\\ \hline
Second order Eulerian ($0\le k< n$) & 
\href{https://oeis.org/A201637}{A201637} & 
$\mathscr{T}(1,2,1)$ & $\mathscr{N}\lpa{\frac 23n, 
\frac19n;n^{-\frac12}}$\\ 
Reciprocal of \href{https://oeis.org/A201637}{A201637} & 
\href{https://oeis.org/A163936}{A163936} & &
$\mathscr{N}\lpa{\frac 13n, \frac19n;n^{-\frac12}}$\\ 
Essentially $=$ \href{https://oeis.org/A163969}{A163969} & 
\href{https://oeis.org/A288874}{A288874} & &
$\mathscr{N}\lpa{\frac 13n, \frac19n;n^{-\frac12}}$\\ 
\hline
\end{tabular}	
\end{center}

In addition to $\mathscr{T}(0,2,1)$ and $\mathscr{T}(1,2,1)$, the
polynomials defined on $\mathscr{T}(1,2,3)$ also correspond, by
\eqref{T-class-eq}, to the second-order Eulerian numbers, and 
appeared in \cite{Eu2014}, together with two other variants:
\[
    \mathscr{T}(0,2,2) \text{ with }P_0(v)=v,
    \quad\text{and}\quad
	\mathscr{T}(1,2,0).
\]
The first ($\mathscr{T}(0,2,2)$ and $\mathscr{T}(1,2,4)$ by
\eqref{T-class-eq}) leads, by Theorem~\ref{thm-barbero}, to the same
$\mathscr{N}\lpa{\frac 23n,\frac19n;n^{-\frac12}}$ as for the second
order Eulerian numbers because \eqref{Fpqr-nn} holds. The second type
($\mathscr{T}(1,2,0)$) contains negative coefficients but corresponds
essentially to the second order Eulerian numbers after dividing by 
$1-v$.

Another example with $q=2$ is sequence 
\href{https://oeis.org/A214406}{A214406}, which is the second
order Eulerian numbers of type $B$ and counts the Stirling
permutations \cite{Gessel1978,Janson2011} by ascents. The polynomials
can be generated by $P_n\in\ET{4vn +1-3v,2v;1}$ and its reciprocal
transform is $Q_n\in\ET{(2n-1)(1+v),2v;1}$. By considering
$2^{-n}P_n(v)$, we see that these numbers are of type
$\mathscr{T}(\frac12,2,1)$ and the coefficients follow a CLT with
optimal convergence rate.

The last example \href{https://oeis.org/A290595}{A290595} is of a 
different form: $P_n\in\ET{3(1+v)n-2-v,3v;1}$, whose reciprocal $Q_n$ 
satisfies $Q_n\in\ET{6vn+2-5v,3v;1}$ and is, up to the factor $3^n$, 
of type $\mathscr{T}\lpa{\frac23,2,1}$. Thus the EGF of $P_n$ is 
given by 
\[
    \left(vT_2\lpa{v^{-1}e^{-\frac1v(1-3(1-v)^2 z)}}
    \right)^{\frac23}\left(\frac{v-1}{v\lpa{1-
    T_2\lpa{v^{-1}e^{-\frac1v(1-3(1-v)^2 z)}}}}\right),
\]
and we obtain the same CLT $\mathscr{N}\lpa{\frac13n, 
\frac19n;n^{-\frac12}}$ for the distribution of $[v^k]P_n(v)$.

\begin{center}
\begin{tabular}{llll}\hline
Second order Eulerian type $B$ & 
\href{https://oeis.org/A214406}{A214406} & 
$\mathscr{T}(\frac12,2,1)$ 
&  $\mathscr{N}\lpa{\frac 23n, \frac19n;n^{-\frac12}}$\\ 
Reciprocal of \href{https://oeis.org/A214406}{A214406} & 
\href{https://oeis.org/A288875}{A288875} & &
$\mathscr{N}\lpa{\frac 13n, \frac19n;n^{-\frac12}}$\\ \hline
$\ET{6vn+2-5v,3v;1}$  &  & 
$\mathscr{T}(\frac23,2,1;3z)$ 
&  $\mathscr{N}\lpa{\frac 23n, \frac19n;n^{-\frac12}}$\\ 
Reciprocal of $\mathscr{T}\lpa{\frac23,2,1;3z}$ & 
\href{https://oeis.org/A290595}{A290595} & &
$\mathscr{N}\lpa{\frac 13n, \frac19n;n^{-\frac12}}$\\ \hline
\end{tabular}	
\end{center}

See also Section~\ref{ss-1-sv} for polynomials related to 
$\mathscr{T}\lpa{\frac1q,2,1}$. 

\subsubsection{$q=3\Longrightarrow 
\mathscr{N} \lpa{\frac34n,\frac9{80}n;n^{-\frac12}}$} 
We found only one OEIS example:
\begin{center}
\begin{tabular}{cccc}\hline
Third order Eulerian ($0\le k<n$) & \href{https://oeis.org/A219512}{A219512} & 
$\mathscr{T}(1,3,1)$ & 
$\mathscr{N}\lpa{\frac34n, \frac{9}{80}n;n^{-\frac12}}$\\ \hline
\end{tabular}	
\end{center}
or
\begin{align}\label{A219512}
    P_n\in\ET{3vn+1-3v,v;1}.
\end{align}

For the EGF, in addition to Barbero G.\ et al.'s 
solution~\eqref{barbero-egf}, an alternative form is as follows. 
Define $J(z,v)$ 
\begin{align*}
    J(z,v) := \int_0^z \frac{\dd t}{(1+t)(1+tv)^3}
    = \frac{\log\frac wv +2(v-w)
    -\tfrac12(v^2-w^2)}{(1-v)^3},
\end{align*}
where $w := \frac{v(1+z)}{1+vz}$. Then the EGF $F(z,v)-1$ is the 
compositional inverse of $J$, namely, it satisfies
\[
    F(J(z,v),v)-1 = z.
\]
This can be readily checked by \eqref{barbero-egf}. Indeed, for any 
polynomials of type $\mathscr{T}(1,q,1)$ with $q>0$, we have 
$F(J(z,v),v)-1=z$, where 
\begin{align*}
    J(z,v) &:= \int_0^z \frac{\dd t}{(1+t)(1+tz)^q}
    =\frac1{(1-v)^q}
    \left(\log\frac{1+z}{1+vz} 
	+L_q\left(\frac{v(1+z)}{1+vz}\right)
	-L_q(v)\right),
\end{align*}
with $L_q$ defined in \eqref{Lqv}. 

Note that the random variables associated with the coefficients of 
$\mathscr{T}(1,3,1)$ are equivalent to those of $\mathscr{T}(1,3,4)$
by a simple shift $n\mapsto n+1$ in \eqref{A219512}. We obtain the 
same CLT $\mathscr{N}\lpa{\frac34n, \frac{9}{80}n; n^{-\frac12}}$. 

\subsubsection{$q>1\Longrightarrow\mathscr{N}\lpa{\frac q{q+1}\,n, 
    \frac{q^2}{(q+1)^2(q+2)}\,n}$} 
These higher order Eulerian numbers are discussed 
in \cite{Barbero2014, Barbero2015}; see also Section~\ref{ss-polya} 
on P\'olya urn models. We list the CLTs for $q=4,\dots,7$; note that 
our results are not limited to integer $q$.

\vspace*{-.3cm}
\begin{center}
\begin{tabular}{cc|cc}
Type & CLT & Type & CLT \\ \hline
$\ET{4vn+1-4v,v;1}$ & $\mathscr{N}\lpa{\frac45n, \frac{8}{75}n;n^{-\frac12}}$ &
$\ET{5vn+1-5v,v;1}$ & $\mathscr{N}\lpa{\frac56n, \frac{25}{252}n;n^{-\frac12}}$ \\
$\ET{6vn+1-6v,v;1}$ & $\mathscr{N}\lpa{\frac67n, \frac{9}{98}n;n^{-\frac12}}$ &
$\ET{7vn+1-7v,v;1}$ & $\mathscr{N}\lpa{\frac67n, \frac{49}{576}n;n^{-\frac12}}$ \\ \hline
\end{tabular}    
\end{center}

\subsection{Polynomials with $(\alpha(v),\beta(v))
=(p+qv,v)$ \\ $\Longrightarrow$ $\mathscr{N}\lpa{\frac q{p+q+1}\,n, 
\frac{q(p+1)(p+q)}{(p+q+1)^2(p+q+2)}\,n}$}\label{ss-ru}
Rz\c adkowski and Urli\'nska \cite{Rzadkowski2019} study the 
recurrence 
\begin{align}\label{Rz-Ur}
    P_n\in\ET{(p+qv)n +1-p-q v,v;1},
\end{align}
where $p, q$ are not necessarily integers. When $p=0$, we obtain
higher order Eulerian numbers $\mathscr{T}(1,q,1)$; in particular,
$(p,q)=(0,1)$ gives Eulerian numbers, and $(p,q)=(0,m)$ the $m$th 
order Eulerian numbers. If $[v^k]P_n(v)\ge0$ for $n,k\ge0$, then we 
obtain the CLT
\begin{align}\label{Rza-Url}
    \mathscr{N}(\mu n, \sigma^2n),
    \text{ where }
    \mu := \frac{p}{p+q+1} \text{ and }
	\sigma^2 := 
    \frac{q(p+1)(p+q)}{(p+q+1)^2(p+q+2)},
\end{align}
provided that the variance coefficient $\sigma^2>0$. Note that 
for fixed $q$ and increasing $p$, the mean coefficient $\mu $ 
increases to unity and the variance coefficient $\sigma^2$ 
first increases and then decreases to zero, while for fixed $p$ and
increasing $q$, $\mu $ decreases steadily and $\sigma^2$ 
undergoes a similar unimodal pattern as in the case of fixed $q$ and 
increasing $p$.

By \eqref{Fzv-ru}, we can also apply Theorem~\ref{thm-saqp} by taking 
(assuming $p+q>0$)
\[
    \Psi(z,v) = \frac{1-T\lpa{S(v)
    +\frac{(1-v)^{p+q}z}{v^p}}}{1-v},\quad
    \text{where}\quad 
    S(v) = \int_v^1 t^{-p-1}(1-t)^{p+q-1}\dd t,
\]
and $T(S(v))=v$. With the notations of Theorem~\ref{thm-saqp}, since 
\[
    \rho(v) 
    = \frac{v^p}{(1-v)^{p+q}}
    \int_v^1 t^{-p-1}(1-t)^{p+q-1}\dd t,
\]
we obtain $\rho'(1)=-\frac{q}{(p+q)(p+q+1)}$ and
$\rho''(1)=\frac{2q(q+1)}{(p+q)(p+q+1)(p+q+2)}$. We then deduce an
optimal rate $n^{-\frac12}$ in the CLT \eqref{Rza-Url}.

If $(p,q)=(-1,1)$, then $P_n(v) = v^{n-1}$. Another simple example
for which $\sigma^2$ equals zero is $(p,q)= (-\frac12,\frac12)$
and in this case
\[
    P_{2n}(v)=\frac{v^{n-1}+v^n}2,\quad
	\text{and}\quad P_{2n-1}(v) = v^n,
\]
which does not lead to a CLT. 

Yet another example discussed in \cite{Rzadkowski2019} is
$(p,q)=(-\frac12,1)$ (which seems connected to \href{https://oeis.org/A160468}{A160468} in some way).
We then obtain $\mathscr{N}\lpa{\frac23n, \frac{2}{45}n}$ for the
distributions of the coefficients. The EGF can be solved to be of the
form
\[
    F(z,v) = \frac{1-v}v\cdot
    \frac{1+\sin\lpa{\sqrt{v(1-v)}\,z+\arcsin(2v-1)}}
    {1-\sin\lpa{\sqrt{v(1-v)}\,z+\arcsin(2v-1)}}.
\]
To apply Theorem~\ref{thm-saqp}, we use the notation of \eqref{F-ABC} 
and take (due to a double zero)
\[
    \Psi(z,v) 
    = \sqrt{\frac{1-\sin\lpa{\sqrt{v(1-v)}z
	+\arcsin(2v-1)}}{1-v}},
\]
so that 
\[
    \rho(v) = \frac{2\arccos\sqrt{v}}{\sqrt{v(1-v)}}
    = \frac{\pi-2\arcsin\sqrt{v}}{\sqrt{v(1-v)}}.
\]
Thus Theorem~\ref{thm-saqp} applies with $\rho'(1)=-\frac43$ and 
$\rho''(1) = \frac{32}{15}$, and we obtain the CLT with rate 
$\mathscr{N}\lpa{\frac23n,\frac{2}{45}n;n^{-\frac12}}$.

\paragraph{A CLT example with $\beta(1)<0$}
An example reducible to the form $(\alpha(v),\beta(v))
=(p+qv,v)$ but slightly different from \eqref{Rz-Ur} is Warren's  
model of two-coin trials studied in \cite{Warren1999}, leading to the 
recurrence 
\[
    P_n\in\GT{1}{(1-\theta_2+\theta_2v)(n-1), 
    -(\theta_1-\theta_2)v; 1-\theta_2+\theta_2v},
\]
where $0<\theta_1\ne\theta_2<1$. Since $[v^k]P_n(v)\ge0$ for all
pairs $(\theta_1,\theta_2)$ by the original construction (or by
examining the recurrence satisfied by the coefficients), we can apply
Theorem~\ref{thm-clt} and obtain the CLT
\[
    \mathscr{N}\left(\frac{\theta_2}{1-\theta_1+\theta_2}\,n,
    \frac{(1-\theta_1)\theta_2}{(1-2\theta_1+2\theta_2)
    (1-\theta_1+\theta_2)^2}\,n \right),
\]
provided that $0<\theta_1<\theta_2+\frac12$ (so that 
$1-2\theta_1+2\theta_2>0$). This example is interesting 
because if $\theta_2<\theta_1<\theta_2+\frac12$, then, putting in the 
form of \eqref{Pnv-gen}, we see that the factor  
\[
    \beta(v) = -(\theta_1-\theta_2)v
\]
becomes negative at $v=1$, and this is one of the \emph{few examples
in this paper with negative $\beta(1)$ and the coefficients of
$P_n(v)$ still following a CLT}. See Section~\ref{ss-ck} and 
\cite{Warren1999} for other models of a similar nature. By solving 
the corresponding PDE (with $F(z,v)=(1-\theta_2+\theta_2v)z + O(z^2)$ 
as $z\to0$), we obtain the EGF
\begin{align*}
    F(z,v) &= \frac1{\theta_2-\theta_1}
    \log\frac{1-v}{1-T\lpa{S(v)+v^{-\frac{1-\theta_2}
    {\theta_2-\theta_1}}(1-v)^{\frac1{\theta_2-\theta_1}}
    z}}\\ &\qquad +\frac{1-\theta_2}{\theta_2-\theta_1}
    \log\frac{T\lpa{S(v)+v^{-\frac{1-\theta_2}
    {\theta_2-\theta_1}}(1-v)^{\frac1{\theta_2-\theta_1}}
    z}}{v}, 
\end{align*}
where $T(S(v))=v$ and 
\[
    S(v) := \frac1{\theta_2-\theta_1}\int v^{-\frac{1-\theta_2}
    {\theta_2-\theta_1}-1}(1-v)^{\frac1{\theta_2-\theta_1}-1}
	\dd v.
\]
%

Another extension studied in \cite{Charalambides2002} has the form 
$P_n\in\ET{n+h_n(v-1),v;1}$ for some given sequence $h_n$. In the 
case when $h_n=p+qn$, we obtain the CLT 
\[
    \mathscr{N}\llpa{\frac{q}{2}\,n,
    \frac{q(2-q)}{12}\,n},
\]	
by Theorem~\ref{thm-clt} when the coefficients are nonnegative and
$0<q<2$.

\subsection{Polynomials with $(\alpha(v),\beta(v))=
\lpa{\frac12(1+v),\frac12(3+v)}\Longrightarrow\mathscr{N}
\lpa{\frac16n,\frac{23}{180}n}$} \label{sec-1plusv}

The sequence \href{https://oeis.org/A162976}{A162976} counts the 
number of permutations of $n$ elements having exactly $k$ double 
and initial descents; the generating polynomials $P_n$ satisfy the 
recurrence $P_n\in\GT{1}{\tfrac12(1+v)n, \tfrac12(3+v); 1}$. 
This recurrence can be verified by the EGF 
\begin{align}\label{A162976}
    F(z,v) = 1-\frac{2}{1+v-\sqrt{(1-v)(3+v)}\,
    \cot\lpa{\frac12z\sqrt{(1-v)(3+v)}}},
\end{align}
obtained by using the expression in Goulden and Jackson's book 
\cite[p.\ 195, Ex.\ 3.3.46]{Goulden1983} after a direct 
simplification; see also Zhuang \cite{Zhuang2016}. The CLT 
$\mathscr{N}\lpa{\frac16n,\frac{23}{180}n}$ for the coefficients of 
$P_n$ follows easily from Theorem~\ref{thm-clt}. 
Theorem~\ref{thm-saqp} also applies with the dominant singularity at
\begin{align}\label{dd-rho}
    \rho(v) = \frac{2\arccos\frac{1+v}2}
	{\sqrt{(1-v)(3+v)}}.
\end{align}

Two other recurrences arise from a study of similar permutation 
statistics in \cite{Zhuang2016}:
\begin{align} \label{A162975}
    P_n(v) &= \frac{(1+v)n\pm (1-v)}{2}\,P_{n-1}(v)
    +\frac{(3+v)(1-v)}2\,P_{n-1}'(v)
    \pm \frac{(1-v)(n-1)}2\,P_{n-2}(v),
\end{align}
for $n\ge2$ with $P_0(v)=1$. These recurrences follow from the EGFs
($w := \sqrt{(1-v)(3+v)}$) 
\begin{align} \label{A162975-egfs}
    \frac{w\, e^{\pm\frac12(1-v)z}}
    {w\cos\lpa{\frac12zw}-(1+v)\sin\lpa{\frac12zw}},
\end{align}
derived in \cite{Zhuang2016}; see also \cite{Elizalde2003}.  
Taking both plus signs on the right-hand side of \eqref{A162975} 
together with $P_1(v)=1$ gives the sequence 
\href{https://oeis.org/A162975}{A162975} (enumerating double
ascents); the other recurrence with both minus signs together with
$P_1(v)=v$ gives the sequence
\href{https://oeis.org/A097898}{A097898} (enumerating left-right
double ascents or unit-length runs); see \cite{Gessel1977,Zhuang2016}
for more information. Theorem~\ref{thm-clt} does not apply directly
but the same method of moments do and we get the same CLT
$\mathscr{N} \lpa{\frac16n,\frac{23}{180}n}$. The main reason that
the method of moments works for \eqref{A162975} is that the last term
is asymptotically negligible after the normalization $\bar{P}_n(v) :=
\frac{P_n(v)}{P_n(1)}=\frac{P_n(v)}{n!}$:
\[
    \bar P_n(v) = \frac{(1+v)n\pm (1-v)}{2n}\,\bar P_{n-1}(v)
    +\frac{(3+v)(1-v)}{2n}\,\bar P_{n-1}'(v)
    \pm \frac{1-v}{2n}\,\bar P_{n-2}(v).
\]
Alternatively, one applies the analytic method to the EGFs
\eqref{A162975-egfs} (with the same $\rho(v)$ as \eqref{dd-rho})
and obtains additionally an optimal convergence
rate in the CLT $\mathscr{N}\lpa{\frac16n,\frac{23}{180}n; 
n^{-\frac12}}$.

\begin{center}
\begin{tabular}{lllll}
\multicolumn{1}{c}{OEIS} &
\multicolumn{1}{c}{coeff.\ $P_{n-1}(v)$} &
\multicolumn{1}{c}{ceoff.\ $P_{n-1}'(v)$} &
\multicolumn{1}{c}{coeff.\ $P_{n-2}(v)$} &
\multicolumn{1}{c}{$(\mu_n,\sigma_n^2)$} \\ \hline	%
\href{https://oeis.org/A162976}{A162976} & $\frac{(1+v)n}2$ 
& $\frac{(3+v)(1-v)}2$ & $0$ 
& $\lpa{\frac16n+\frac16,\frac{23}{180}n+\frac{23}{180}}$\\ 
\href{https://oeis.org/A162975}{A162975} & $\frac{(1+v)n+1-v}2$ 
& $\frac{(3+v)(1-v)}2$ & $\frac{(n-1)(1-v)}2$ 
& $\lpa{\frac16n-\frac13,\frac{23}{180}n-\frac{37}{180}}$\\ 
\href{https://oeis.org/A097898}{A097898} & $\frac{(1+v)n-1+v}2$ 
& $\frac{(3+v)(1-v)}2$ & $-\frac{(n-1)(1-v)}2$  
& $\lpa{\frac16n+\frac23,\frac{23}{180}n+\frac{83}{180}}$\\ \hline
\end{tabular}	
\end{center}
We also show in this table the differences in the lower order terms 
of the asymptotic mean and asymptotic variance. 

\subsection{Polynomials with quadratic $\alpha(v)$}

We consider in this subsection recurrences of the form 
\eqref{Pnv-gen} where $\alpha(v)$ is a quadratic polynomial. 

\subsubsection{$(\alpha(v),\beta(v))
=(v^2, v(1+v))\Longrightarrow \mathscr{N}\lpa{\tfrac23n,
\tfrac{8}{45}n}$}\label{sss-v2}

Most of the examples we found involving quadratic $\alpha(v)$ have
the form (after a shift of $n$ or a change of scales)
$P_n\in\ET{v^2n+q-p+pv-v^2,v(1+v);1}$. For such a pattern, since the
degree of $P_n$ is $n$, it proves simpler to look at its reciprocal
$Q_n(v) = v^nP_n(\frac1v)$, which then has the simpler generic form
$Q_n\in\ET{vn+p+(q-p-1)v,1+v;1}$. If $q\ge p>0$, then 
$[v^k]P_n(v)\ge0$ and $P_n(1)>0$, and we obtain, by 
Theorem~\ref{thm-clt}, the CLTs $\mathscr{N}\lpa{\tfrac13n,
\tfrac{8}{45}n}$ and $\mathscr{N}\lpa{\tfrac23n,
\tfrac{8}{45}n}$
for the coefficients of $Q_n$ and of $P_n$, respectively.  

We now show how to enhance the CLTs by computing the corresponding 
EGFs. In general, assume $Q_n\in\ET{vn+p+(q-p-1)v,1+v}$. Let $G(z,v)$ 
be the EGF of $Q_n(v)$. Then $G$ satisfies the PDE
\[
    (1-vz)\partial_z G - (1-v^2)\partial_v G 
    = (p+(q-p)v)G,
\]
with $G(0,v)=Q_0(v)$. The solution, by the method of characteristics 
described in Section~\ref{ss-pde}, is given by ($u := \sqrt{1-v^2}$ 
and $w=\arcsin(v)$)
\begin{equation}\label{Qpq}
\begin{split} 
    G(z,v) &= Q_0\lpa{\sin(uz+w)}
    \left(\frac{1+\sin(uz+w)}{1+v}\right)^p
    \left(\frac{u}{\cos(uz+w)}\right)^q.
\end{split}    
\end{equation}
Write this class of functions as $\mathscr{Q}(p,q)$. Then 
\begin{align}\label{d-Qpq}
    \partial_z \mathscr{Q}(q,q)
    = q \mathscr{Q}(q,q+1) \quad \text{when}\quad 
    Q_0(v)=1. 
\end{align}

With \eqref{Qpq} available, we can apply Theorem~\ref{thm-saqp} when 
$q\ge p>0$ with $\rho(v) = \frac{\arccos(v)}{\sqrt{1-v^2}}$, and the 
local expansion 
\[
    -\log\rho(e^s) = \tfrac13s+\tfrac4{45}s^2+\tfrac8{2835}s^3
    -\tfrac{44}{14175}s^4+\cdots,
\]
giving the CLT with optimal rate $\mathscr{N}\lpa{\frac13n,
\frac8{45}n;n^{-\frac12}}$.

\paragraph{Liagre's $\mathscr{Q}(2,3)$ and $\mathscr{Q}(1,3)$}
Jean-Baptiste Liagre \cite{Liagre1855} studied (motivated by a
statistical problem) as early as 1855 the combinatorial and
statistical properties of the number of turning points (peaks and
valleys) in permutations, and as far as we were aware, his paper
\cite{Liagre1855} is the first publication on permutation
statistics leading to an Eulerian recurrence, and contains the two 
recurrences
\begin{align}\label{A008970}
    \begin{cases}
        P_n\in\GT{2}{v^2n+1+2v-3v^2,v(1+v);1},\\
        P_n\in\GT{3}{v^2n+1+v-3v^2,v(1+v);1}.
    \end{cases}
\end{align}
The former (\href{https://oeis.org/A008970}{A008970}) counts the
number of turning points in permutations of $n$ elements divided by
two, while the latter (not in OEIS) that in cyclic permutations
divided by two.

We can apply Theorem~\ref{thm-clt} by a direct shift of the two
recurrences (so both has the initial conditions $P_0(v)=1$), and
obtain the same CLT $\mathscr{N}\lpa{\frac23n,\frac8{45}n}$. The CLT
for \href{https://oeis.org/A008970}{A008970} can be obtained by the
general theorem of Wolfowitz in \cite{Wolfowitz1944} although, quite
unexpectedly, it was first stated (without proof) by Bienaym\'e as
early as 1874 in a very short note \cite{Bienayme1874} (with a total
of 13 lines); see also Netto's book \cite[pp.\ 105--116]{Netto1901}.
Bienaym\'e's result is described as ``far ahead of its time'' in 
Heyde and Seneta's book \cite{Heyde1977}. 
For more historical accounts, see \cite{Barton1965, Heyde1977, 
Warren1996}. The normalized versions (with $P_0(v)=1$) are given as 
follows. 
\begin{center}
\begin{tabular}{llll} \hline
$\frac12\#$($n$-perms. with $k$ turning points) 
& \href{https://oeis.org/A008970}{A008970} & 
$\ET{v^2n+1+2v-v^2,v(1+v);1}$ \\
$\frac12\#$($n$-cyclic perms. with $k$ turning points)
&  & $\ET{v^2n+1+v,v(1+v);1}$ \\ \hline
\end{tabular}	
\end{center}

The reciprocal polynomials $Q_n(v) := v^{n-2}P_{n-2}\lpa{\frac1v}$
and $Q_n(v) := v^{n-2}P_{n-3}\lpa{\frac1v}$ are of type 
$\mathscr{Q}(2,3)$ and $\mathscr{Q}(1,3)$, respectively, with the 
initial condition $Q_0(v)=1$ and $Q_0(v)=v$, respectively. 
By \eqref{Qpq}, we have the EGFs of $P_n$ and $Q_n$, respectively
($u := \sqrt{1-v^2}$ and $w=\arcsin(v)$):
\begin{align*}
    \begin{cases}
    \left(\dfrac{1+\sin(uz+w)}{1+v}\right)^2
    \left(\dfrac{u}{\cos(uz+w)}\right)^3,\\
    \sin(uz+w)\dfrac{(1+\sin(uz+w)}
    {1+v}\left(\dfrac{u}{\cos(uz+w)}\right)^3.
    \end{cases}
\end{align*}
Note that in the first case, an alternative form for the EGF was 
derived by Morley \cite{Morley1897} in 1897 
\[
    \sum_{n\ge1}\frac{Q_{n+1}(v)}{n!}\, z^n
    = \frac{1-v}{(1+v)\lpa{1-\sin\lpa{uz+w}}}
    -\frac1{1+v},
\]
which can be obtained by a direct integration of $\mathscr{Q}(2,3)$. 
These EGFs are then suitable for applying Theorem~\ref{thm-saqp}, and 
an optimal Berry-Esseen bound is thus implied in the corresponding 
CLTs for the coefficients. 

\paragraph{Alternating runs in permutations: $\mathscr{Q}(2,2)$}
By \eqref{d-Qpq}, we see that the total number of turning points or
alternating runs (which is twice \href{https://oeis.org/A008970}{A008970}) in all permutations of $n$
elements (not half of them) is of type $\mathscr{Q}(2,2)$. This
corresponds to sequence \href{https://oeis.org/A059427}{A059427}. For more details and information,
see David and Barton's book \cite[pp.\ 158--161]{David1962}, the 
review paper \cite{Barton1965} and  
\cite{Andre1884,Bienayme1875,Bona2004}. The normalized version (with
$P_0(v)=1$) is

\medskip

\centering
\begin{tabular}{ccccc} \hline
alternating runs in perms. 
& \href{https://oeis.org/A059427}{A059427} & $\ET{v^2n+2v-v^2,v(1+v);1}$ 
& $\mathscr{N}\lpa{\frac13n,
\frac8{45}n;n^{-\frac12}}$\\ \hline
\end{tabular}	
\justifying

\medskip

\noindent
This sequence of polynomials has a larger literature than Liagre's 
statistics. In particular, finding closed-form expressions for 
$[v^k]P_n(v)$ has been the subject of many papers; see for example 
\cite{Ma2012,Ma2013a} and the references therein. 

\paragraph{Alternating runs in up signed permutations: 
$\mathscr{Q}\lpa{\frac32,2}$}
Extending further the alternating runs to signed permutations, 
Chow and Ma \cite{Chow2014a} studied the recurrence
\begin{align}\label{chow-ma-2014}
    P_n\in\GT{1}{2v^2n-1+3v-2v^2,2v(1+v);v}.
\end{align}
They also derived the closed form expression for the EGF of $P_n$: 
\[
    \frac1{1+v}+\frac{v\sqrt{1-v}}
    {(1+v)\sqrt{\cosh\lpa{2z\sqrt{1-v^2}}-v-
    \sqrt{1-v^2}\sinh\lpa{2z\sqrt{1-v^2}}}}.
\]
The reciprocal transformation $Q_n=v^nP_n\lpa{\frac1v}$ satisfies
\[
    Q_n\in\GT{1}{2vn+3(1-v),2(1+v);1}.
\]
This is of type $\mathscr{Q}\lpa{\frac32,2}$ after normalizing
$Q_n(v)$ by $2^n$. Thus the same CLT 
$\mathscr{N}\lpa{\frac23n,\frac{8}{45}n;n^{-\frac12}}$ holds for the 
distribution of the number of alternating runs in signed permutations.

\paragraph{Derivative polynomials: $\mathscr{Q}(0,2)$}
Another sequence \href{https://oeis.org/A198895}{A198895}, which 
corresponds to the derivative polynomials of $\tan v + \sec v$, 
satisfies the recurrence
\[
	P_n\in\GT{1}{v^2n+1-v^2,v(1+v);1+v}.
\]
One gets the CLT$\lpa{\frac 23n,\frac8{45}n;n^{-\frac12}}$ for the 
coefficients $[v^k]P_n(v)$, a result (without rate) also proved in 
\cite{Ma2012a} by the real-rootedness approach. Its reciprocal 
polynomial satisfies the simpler form
\[
	Q_n\in\ET{vn,1+v;1+v}.
\]
This is of type $\mathscr{Q}(0,2)$. 

\paragraph{Up-down runs in permutations: $\mathscr{Q}(1,2)$}
A very similar sequence is \href{https://oeis.org/A186370}{A186370} 
(number of permutations of $n$ elements having $k$ up-down runs):
\[
	P_n\in\GT{1}{v^2n+v-v^2,v(1+v);v}.
\]
One gets the same CLT $\mathscr{N}\lpa{\frac23n,
\frac8{45}n;n^{-\frac12}}$. Its reciprocal polynomial satisfies the 
simpler form 
\[
	Q_n\in\GT{1}{vn+1-v,1+v;1},
\]
which is of type $\mathscr{Q}(1,2)$. Interestingly, $Q_1(v)=v$
generates the same sequence of polynomials for $n\ge2$.

\subsubsection{$(\alpha(v),\beta(v))=
\lpa{\frac12(1+v^2),\frac12(1+v^2)}\Longrightarrow 
\mathscr{N}\lpa{\frac12n,\frac5{12}n}$}
\label{sss-chebikin}

The generating polynomials for the numbers of alternating descents 
($\pi(i)\gtrless\pi(i+1)$ depending on the parity of $i$) or for the 
number of $3$-descents (either of the patterns $132$, $213$ or 
$321$) satisfy (see \cite{Chebikin2008,Ma2015})
\[
    P_n\in\GT{1}{\tfrac12(1+v^2)n+v(1-v),\tfrac12(1+v^2);1}.
\]
They are palindromic and correspond to  
\href{https://oeis.org/A145876}{A145876}.

This leads, by Theorem~\ref{thm-clt}, to the CLT
$\mathscr{N}\lpa{\frac12n, \frac5{12}n}$ for the coefficients. For
the optimal convergence rate $n^{-\frac12}$, we can use the EGF 
derived in \cite{Chebikin2008} (see also \cite{Zhuang2016})
\begin{align}\label{egf-chebikin}
    \frac{1+\sin((1-v)z)-\cos((1-v)z)}
    {\cos((1-v)z)-v-v\sin((1-v)z)},
\end{align}
and then apply Theorem~\ref{thm-saqp} with 
\[
    \rho(v) = \frac{\text{arccos}\lpa{\frac{2v}{1+v^2}}}{v-1}.
\]
A very interesting property of $P_n(v)$ is that all roots lie on the
left half unit circle, namely, $v=e^{i\theta}$ with
$\frac12\pi\le\theta \le \frac32\pi$; see \cite{Ma2015} for more
information and Figure~\ref{fig-chebikin} for an illustration. Such a
root-unitary property implies an alternative proof of the CLT via the
fourth moment theorem of \cite{Hwang2015}: \emph{the fourth centered
and normalized moment tends to three iff the coefficients are
asymptotically normally distributed.} This is in contrast to proving
the unboundedness of the variance when all roots are real; also
without the root-unitary property Theorem~\ref{thm-clt} requires the
moments of all orders.
\begin{figure}[!ht]
\centering    
\includegraphics[width=2.2cm]{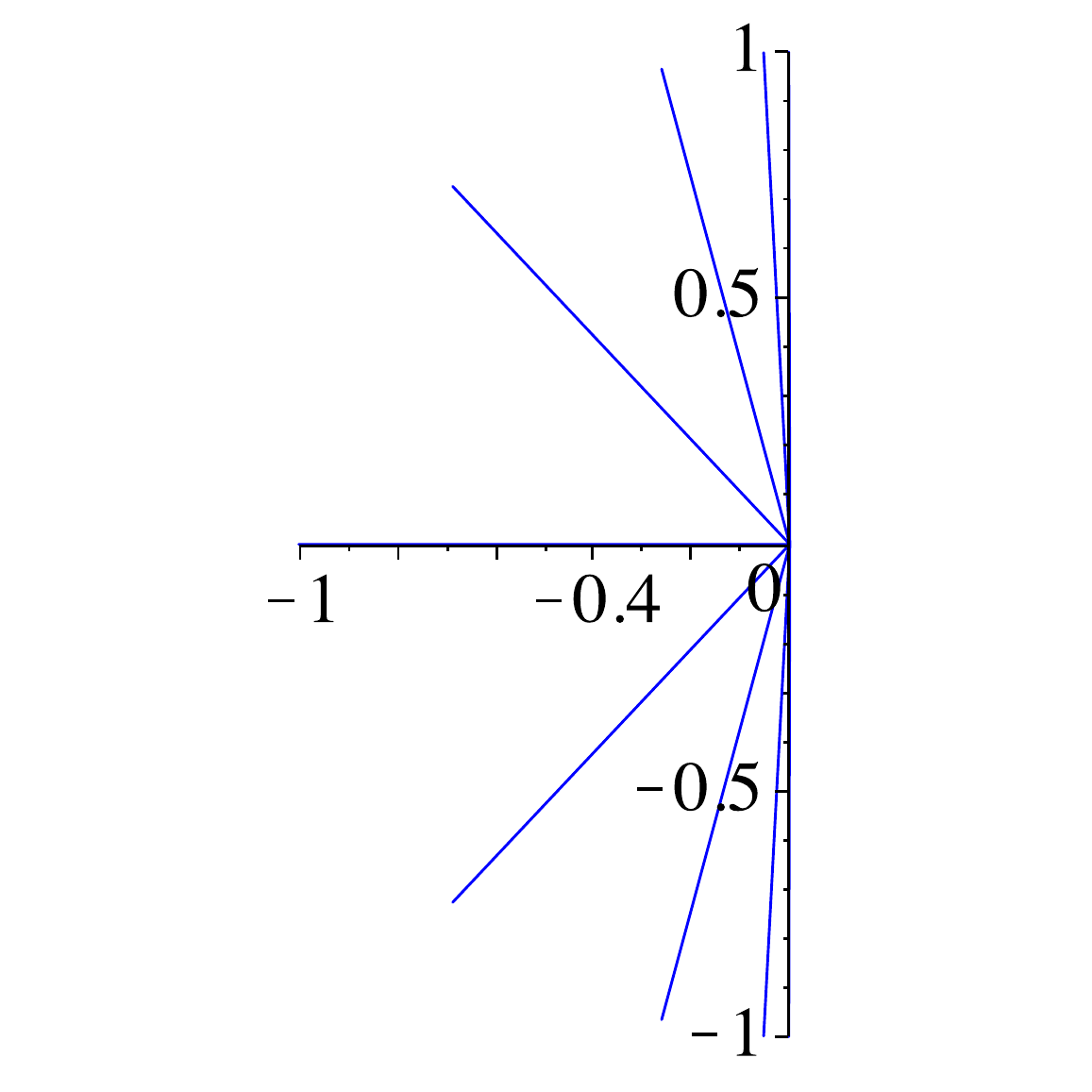}\;
\includegraphics[width=2.2cm]{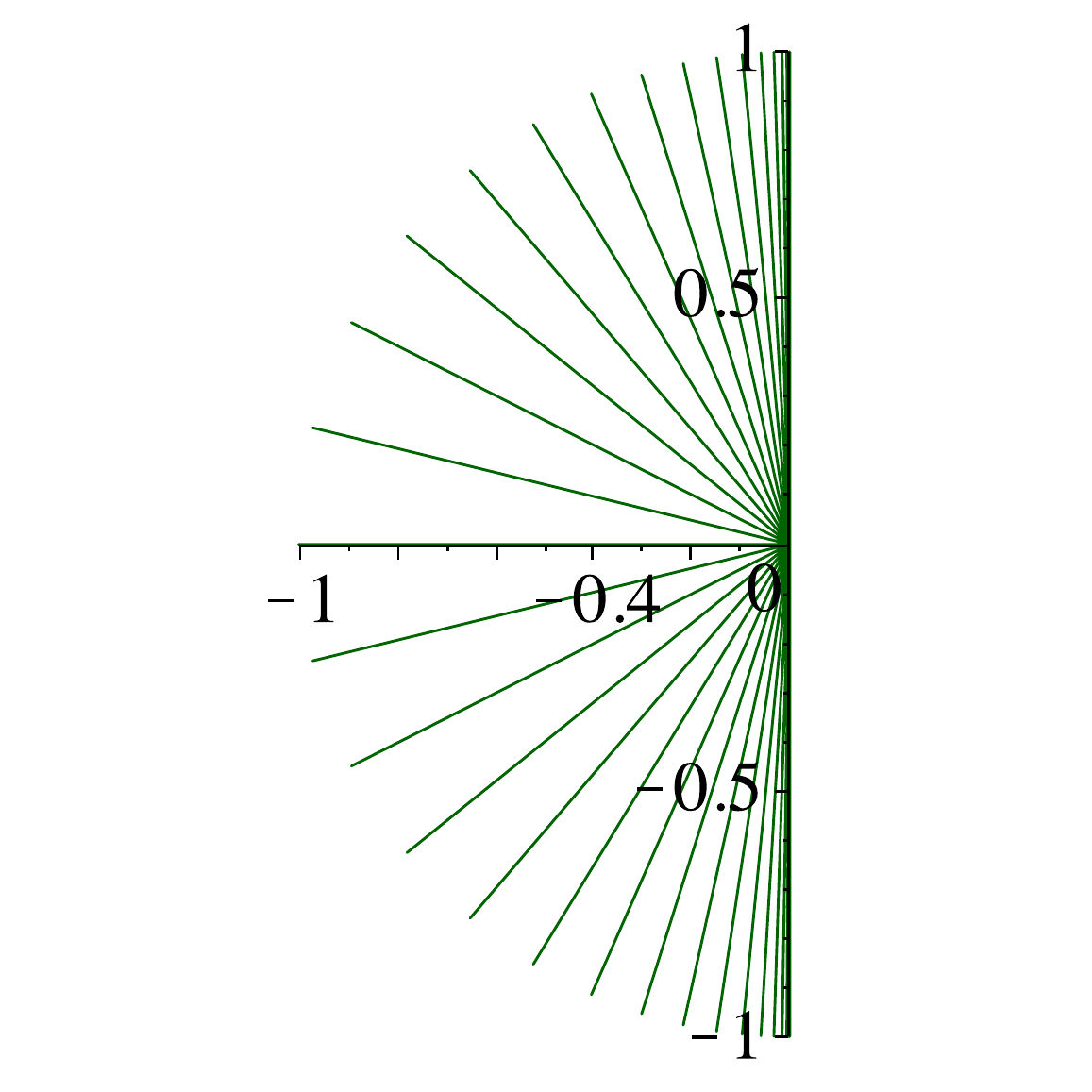}\;
\includegraphics[width=2.2cm]{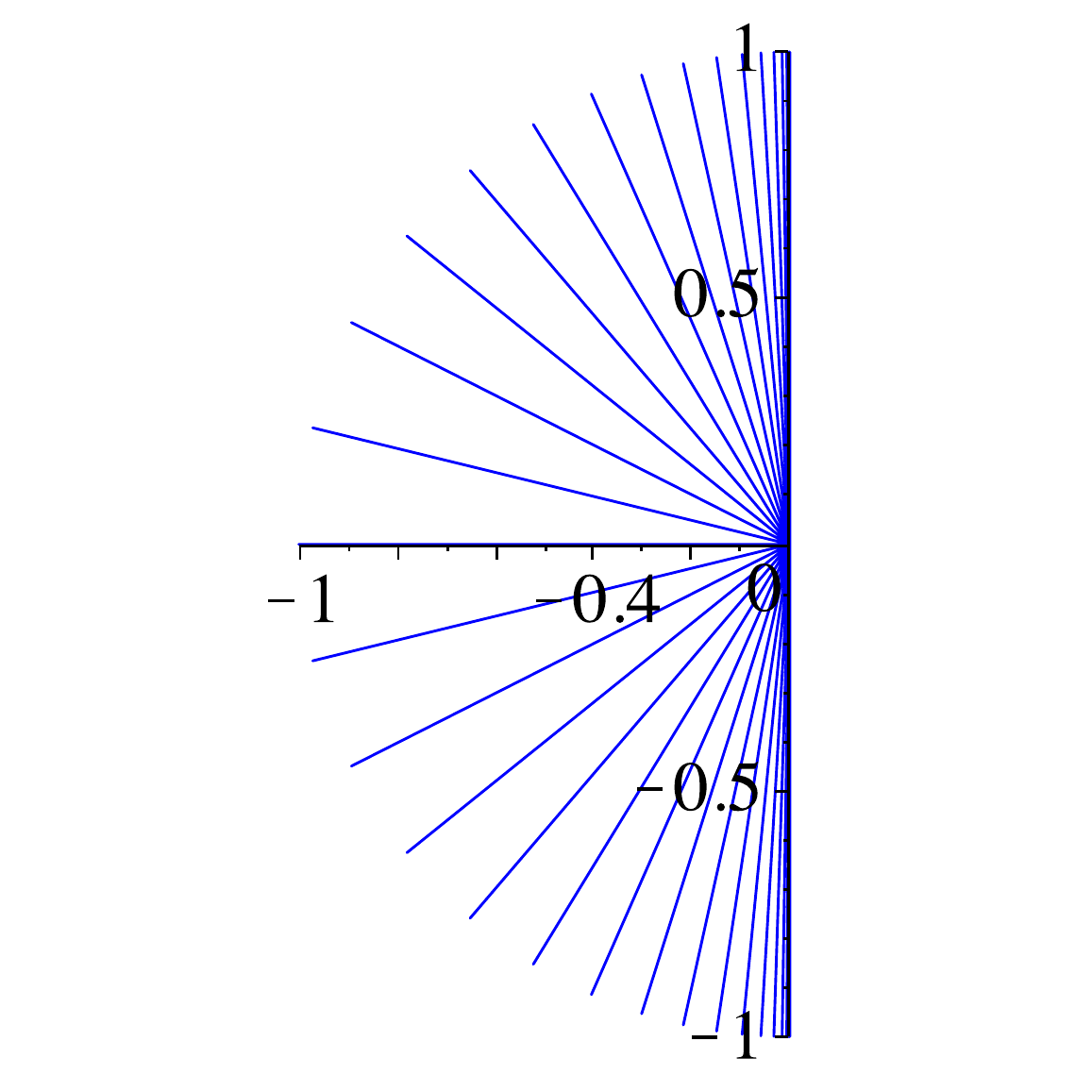}\;
\includegraphics[width=2.2cm]{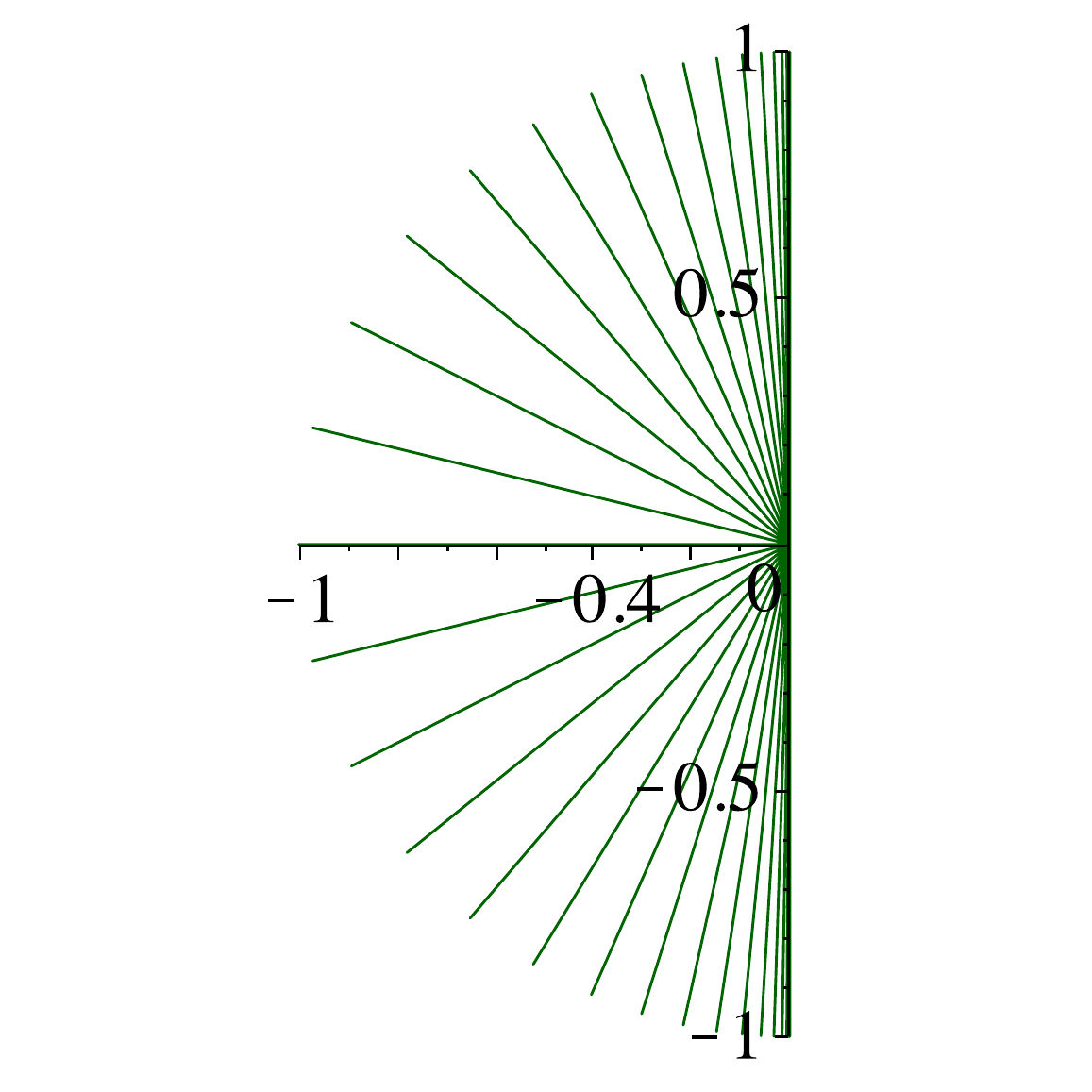}\;
\includegraphics[width=2.2cm]{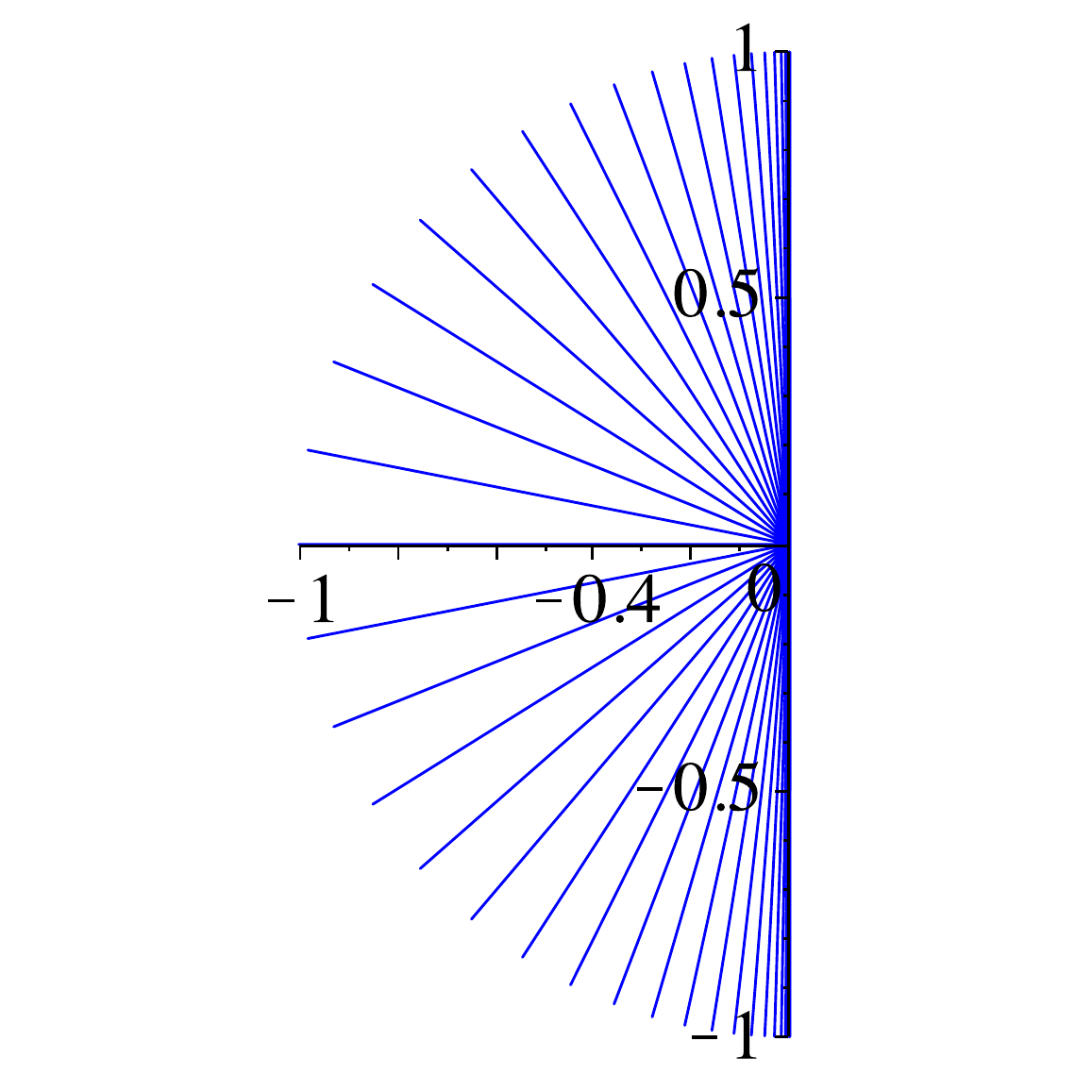}\;
\includegraphics[width=2.2cm]{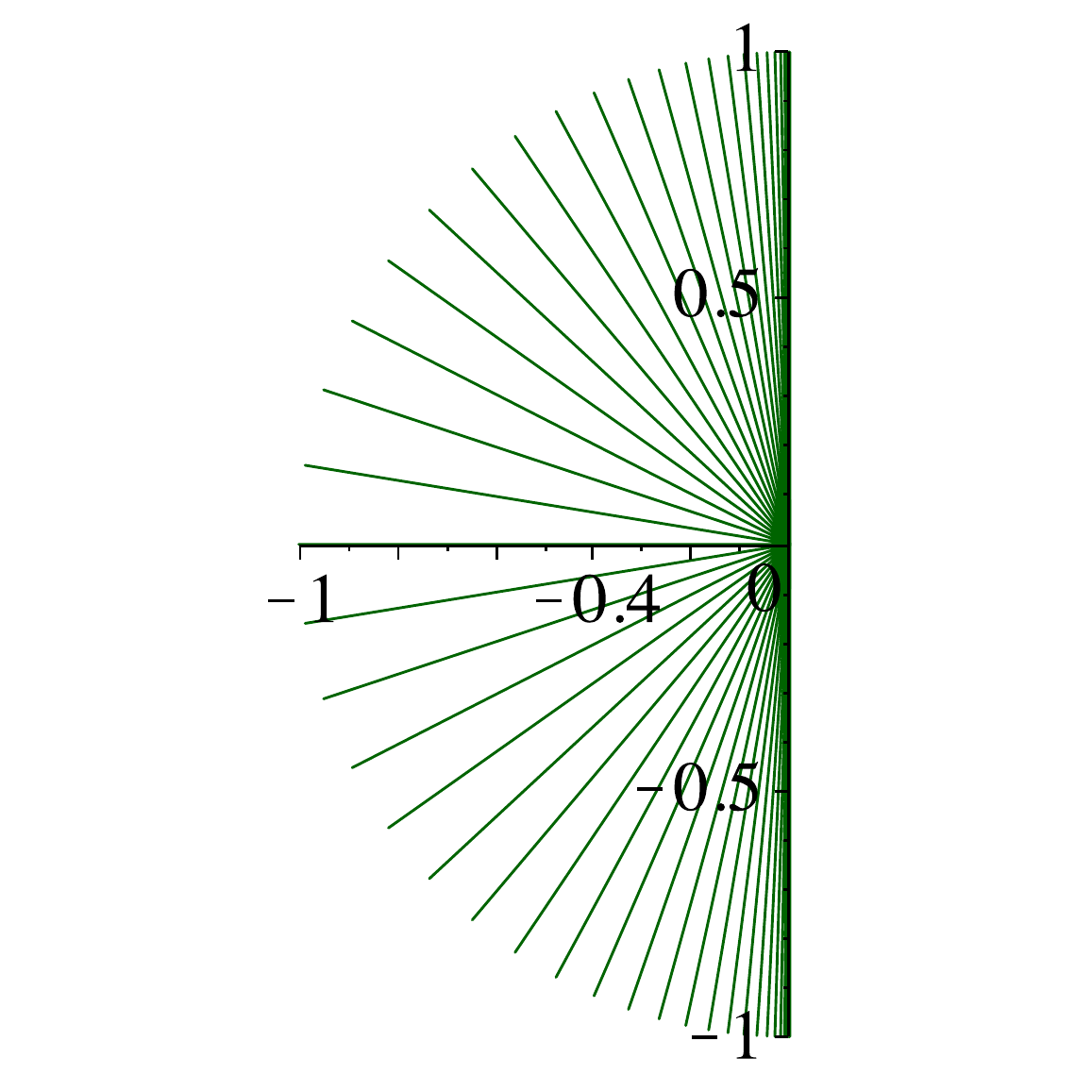}
\medskip
\caption{Distribution of the zeros of the 
\href{https://oeis.org/A145876}{A145876} polynomials 
$P_n(v)$ for $n=10,20,\dots,60$.}\label{fig-chebikin}
\end{figure}

\subsubsection{$(\alpha(v),\beta(v))
=(v(1+v), v(1+v))\Longrightarrow \mathscr{N}\lpa{\frac34n,
\frac{7}{48}n}$}\label{sss-v1pv}

In the context of tree-like tableaux, the generating polynomial for 
the number of symmetric tree-like tableaux of size $2n+1$ with $k$ 
diagonal cells satisfies the recurrence \cite{Aval2013a}
\begin{align}\label{aval-rr}
    P_n\in\ET{v(1+v)n,v(1+v);v}.
\end{align}
We obtain, by Theorem~\ref{thm-clt}, the CLT
$\mathscr{N}\lpa{\frac34n,\frac{7}{48}n}$ for the coefficients. This
CLT was proved in \cite{Hitczenko2018} by the real-rootedness
approach. The reciprocal polynomial $Q_n(v) =
v^{n+1}P_n\lpa{\frac1v}$ satisfies the simpler recurrence
$Q_n\in\ET{(1+v)n,1+v;1}$, where the right-hand side differs from
that of $P_n$ only by a factor $v$. By the techniques of 
Section~\ref{ss-pde}, the EGF has the exact form
\begin{align}\label{egf-v1pv}
    F(z,v) = \frac{v(1-v)}{(1+v)e^{z(v-1)}-2v}
    =e^{(1-v)z}\frac{v(1-v)}{1+v-2ve^{(1-v)z}},
\end{align}
which can then be used to prove an optimal Berry-Esseen bound 
$\mathscr{N}\lpa{\frac34n,\frac{7}{48}n;n^{-\frac12}}$ by 
Theorem~\ref{thm-saqp} with $\rho(v) = 
\frac1{1-v}\log\frac{1+v}{2v}$. 

See also \cite{Aval2013} for another recurrence of the same type
$P_n\in\ET{v(1+v)n+1+v-v^2,v(1+v)}$ whose reciprocal is of type
$\ET{(1+v)n+1,1+v}$. We have the same CLT for the coefficients.

\subsubsection{$(\alpha(v),\beta(v))
=(2v^2, v(1+v))\Longrightarrow \mathscr{N}\lpa{n,
\frac13n}$}\label{sss-2v2}
The $n$th order $\theta$-derivative $\theta := v\mathbb{D}_v$ of
$\sqrt{\frac{1+v}{1-v}}$ leads to the sequence of polynomials
\cite{Ma2016}
\begin{align}\label{A256978}
    P_n\in\ET{v(2vn+ 1-2v),v(1+v);1};
\end{align}
these polynomials are palindromic and correspond to 
\href{https://oeis.org/A256978}{A256978}. The degree of $P_n$
is $2n-1$, and the CLT $\mathscr{N}\lpa{n, \frac13n}$ follows from 
Theorem~\ref{thm-clt}. Furthermore, since the EGF of $P_n$ satisfies 
\cite{Ma2016}
\begin{align}\label{egf-2v2}
    \sqrt{\frac{(1-v)\lpa{1+ve^{(1-v^2)z}}}
    {(1+v)\lpa{1-ve^{(1-v^2)z}}}},
\end{align}
we obtain additionally the stronger CLT $\mathscr{N}\lpa{n,
\frac13n;n^{-\frac12}}$ by Theorem~\ref{thm-saqp} with $\rho(v) =
-\frac{\log v}{1-v^2}$. 

More generally, the same CLT holds for the $\theta$-derivative
polynomials of $\lpa{\frac{1+v}{1-v}}^q$ (with $q>0$) satisfying
$P_n\in\ET{2v((n-1)v+ q),v(1+v);1}$. Note that the usual derivative
polynomial of $\sqrt{\frac{1+v}{1-v}}$ leads to polynomials of the
type $P_n\in\ET{2vn+ 1-2v,1+v;1}$ with a different CLT; see
Section~\ref{sec-2v-1plusv}.

Another example of the form $P_n\in\ET{2v^2n+1+v,v(1+v);1+v}$
appeared in \cite{Carlitz1973a}, which enumerates the rises (or
falls) in permutations of $2n$ elements satisfying $2n+1-\pi(j)
=\pi(2n+1-j)$; see \cite{Adin2001,Ma2013b} for a shifted version of
the form $P_n\in\ET{2v^2n+1+v-2v^2,v(1+v);1}$ (enumerating the
flag-descent statistic in signed permutations). The CLT
$\mathscr{N}\lpa{n, \frac13n}$ for the coefficients of both
polynomials holds by Theorem~\ref{thm-clt}. Note that the latter
$P_n$ (from \cite{Adin2001}) corresponds to
\href{https://oeis.org/A101842}{A101842} and can be computed by
\[
    P_n(v) = (1+v)^n \sum_{0\le k<n}\eulerian{n}{k}v^k,
\]
implying that the EGF is given by 
\[
    e^{(1-v^2)z}\frac{1-v}{1-ve^{(1-v^2)z}}.
\]
Then Theorem~\ref{thm-saqp} applies with $\rho(v) = \frac{-\log v}
{1-v^2}$ and an optimal convergence rate $n^{-\frac12}$ in the CLT 
is guaranteed; see Figure~\ref{tab-v1pv} for the histograms and finer 
expressions of the mean and the variance. 

More generally, all polynomials $P_n$ of the form $(1+v)^nR_n(v)$, 
where $R_n(v)$ is of type $\mathscr{A}(p,q,r)$ with $p,q,r\ge0$ and 
$qr\ge p\ge0$, are Eulerian with $(\alpha(v),\beta(v)) = 
(2qv^2,qv(1+v))$, which leads to the same CLT $\mathscr{N}\lpa{n, 
\frac13n;n^{-\frac12}}$. An OEIS instance of this type is 
\href{https://oeis.org/A165891}{A165891}, which corresponds to 
$\ET{2v^2n+1+2v-v^2,v(1+v);1}$ and is related to 
\href{https://oeis.org/A101842}{A101842} by a factor of $1+v$;
see Figure~\ref{tab-v1pv}.
\begin{figure}[!h]
\centering
\begin{tabular}{c|ccc}
OEIS & \href{https://oeis.org/A256978}{A256978} & 
\href{https://oeis.org/A101842}{A101842} & 
\href{https://oeis.org/A165891}{A165891} \\ \hline
$a_n(v)$ &
$2v^2n+v-2v^2$ & 
$2v^2n+1+v-2v^2$ & 
$2v^2n+1+2v-v^2$ \\
&\includegraphics[width=3cm]{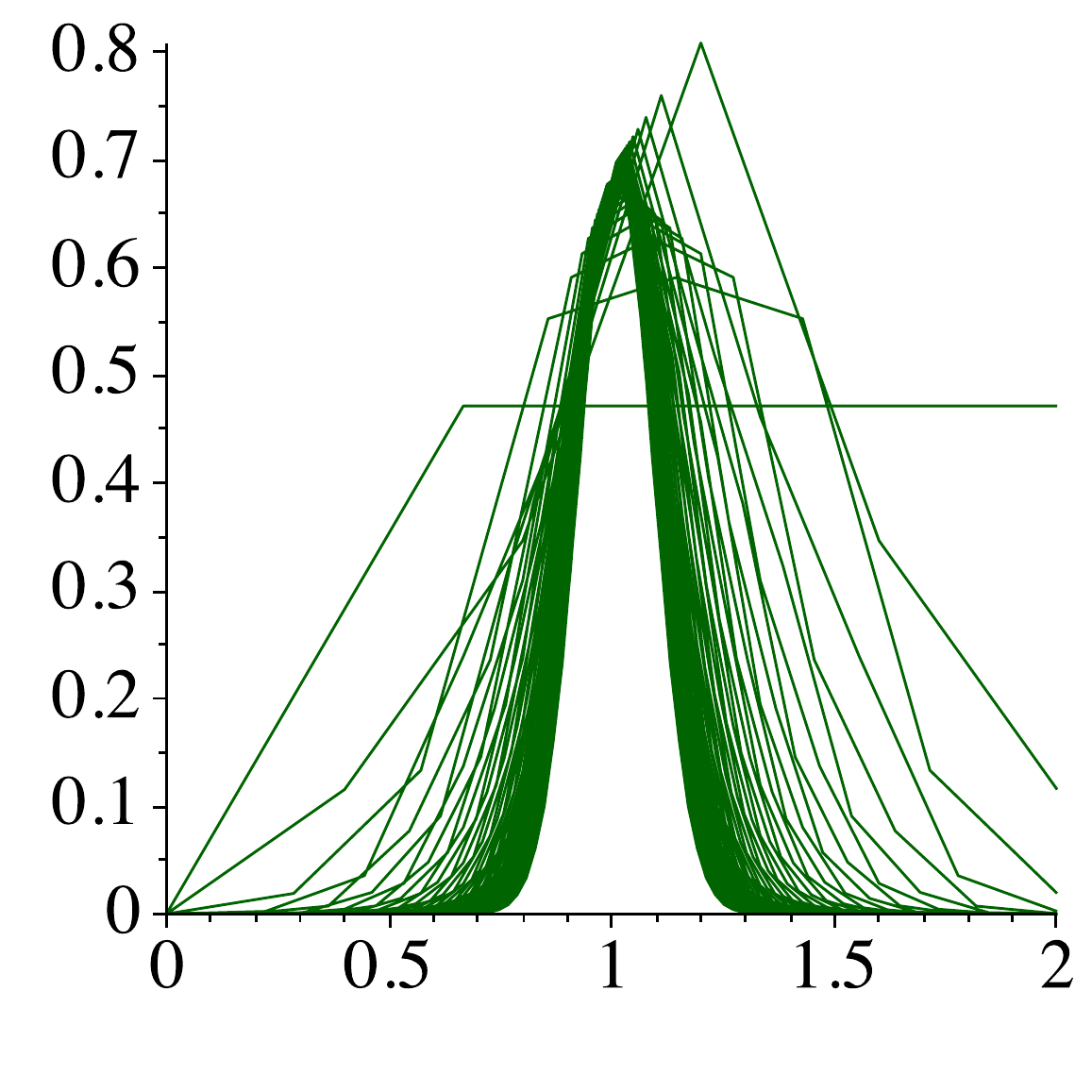} &
\includegraphics[width=3cm]{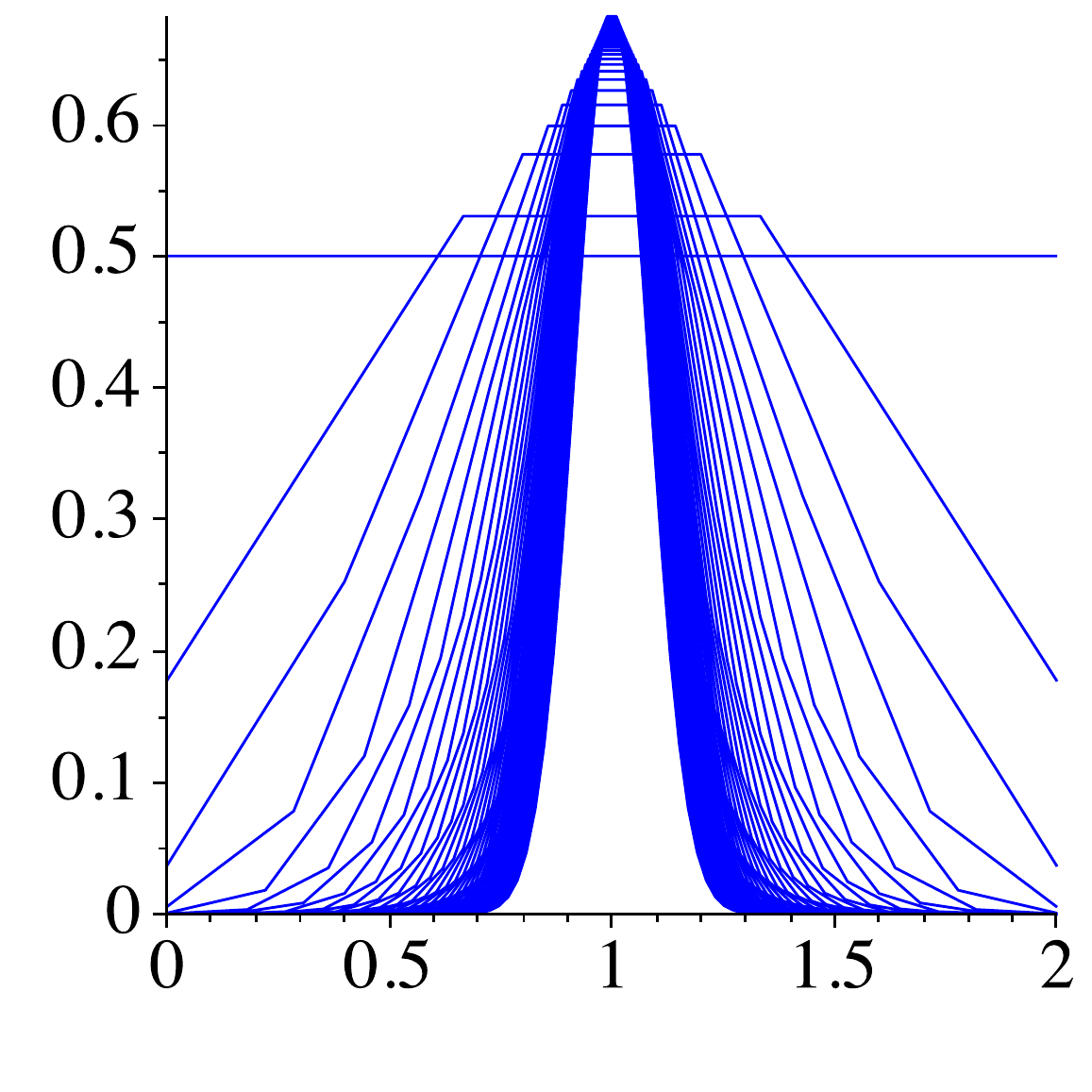} &
\includegraphics[width=3cm]{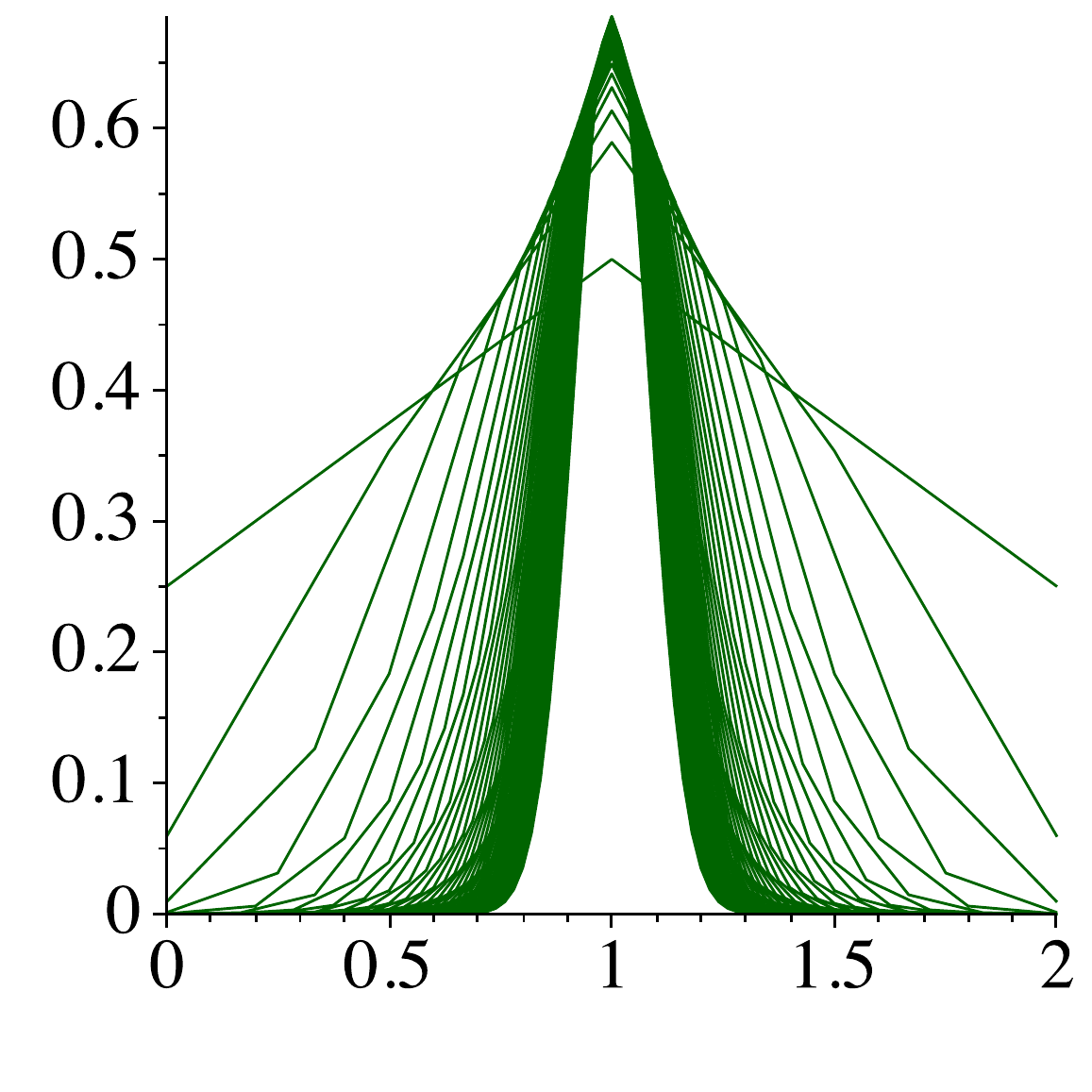} \\
$(\mathbb{E}(X_n),\mathbb{V}(X_n)))$ &
$\lpa{n,\frac{n(n-1)(4n-5)}{3(2n-1)(2n-3)}}$
& $\lpa{n-\frac12,\frac13n+\frac1{12}}$ &
$\lpa{n,\frac13n+\frac16}$ \\ \hline
\end{tabular}
\vspace*{.2cm}
\caption{The histograms of the three OEIS polynomials of the format
$\ET{a_n(v), v(1+v);1}$ for $n=2,\dots,50$. Their coefficients all
satisfy the same CLT $\mathscr{N}(n, \frac13n; n^{-\frac12})$ and 
their differences in the exact mean and the exact variance are shown 
in the last row.}\label{tab-v1pv}
\end{figure}

\subsection{Polynomials with an extra normalizing factor}

We discuss in this subsection polynomials of the form 
\begin{align}\label{enPnv}
    R_n\in\EET{\frac{\alpha(v)n+\gamma(v)}{e_n}}
    {\frac{\beta(v)}{e_n}},  
\end{align}
where $e_n$ is a nonzero normalizing factor such as $n$. If we
consider $P_n(v) := R_n(v)\prod_{1\le j\le n}e_j$, then $P_n$
satisfies $P_n\in\ET{\alpha(v)n+\gamma(v),\beta(v)}$, 
which falls into our framework \eqref{Pnv-gen}. 

\subsubsection{$(\alpha(v),\beta(v)) =(2qv,q(1+v)) \Longrightarrow 
\mathscr{N}\lpa{\frac12n,\frac14n}$}
\label{sec-2v-1plusv}

Examples in this category are often periodic in the sense that 
$[v^k]P_n(v) = 0$, say when $n-k$ is odd or even. In particular, if 
$P_n(v)$ is of the form $P_n \in\ET{(pn+r)v, q(1+v)}$,
then $P_n(v)$ is periodic. For example, the derivative 
polynomials of arcsine function 
(\href{https://oeis.org/A161119}{A161119}):
\[
    P_n(v) := (1-v^2)^{n+\frac12}
    \mathbb{D}_v^{n+1}
    \arcsin(v)\qquad(n\ge0)
\] 
satisfies $P_n\in\ET{(2n-1)v,1+v;1}$, and a CLT of the form 
$\mathscr{N}\lpa{\frac12n,\frac14n}$ holds for the coefficients. Also 
we have the EGF
\[
    F(z,v) = \lpa{(1-vz)^2-z^2}^{-\frac12},
\]
yielding an optimal rate $\mathscr{N}\lpa{\frac12n, 
\frac14n;n^{-\frac12}}$ by Theorem~\ref{thm-saqp} with $\rho(v) = 
\frac1{1+v}$, as well as the expression
\begin{align}\label{A161119}
    P_n(v) &= \sum_{0\le k\le \tr{\frac12n}}
    \frac{n!^2}{k!^2(n-2k)!4^k}\,v^{n-2k}.
\end{align}
Thus $[v^k]P_n(v)=0$ if $n-k$ is odd. The reciprocal polynomial 
corresponds to \href{https://oeis.org/A161121}{A161121}.

On the other hand, the polynomials $P_n(v) := \sum_{0\le k\le n}
(2-(-1)^{n-k})\binom{n}{k}v^k$ satisfies the recurrence
$P_n\in\GT{1}{2v,\frac{1+v}{n-1};3+v}$; see
\href{https://oeis.org/A162315}{A162315}. We then get the CLT
$\mathscr{N}\lpa{\frac12n, \frac14n}$. Note that we get binomial
coefficients (Pascal's triangle
\href{https://oeis.org/A007318}{A007318}) if $P_1(v) = 1+v$.
\begin{figure}[!ht]
\begin{center}
\includegraphics[height=3cm]{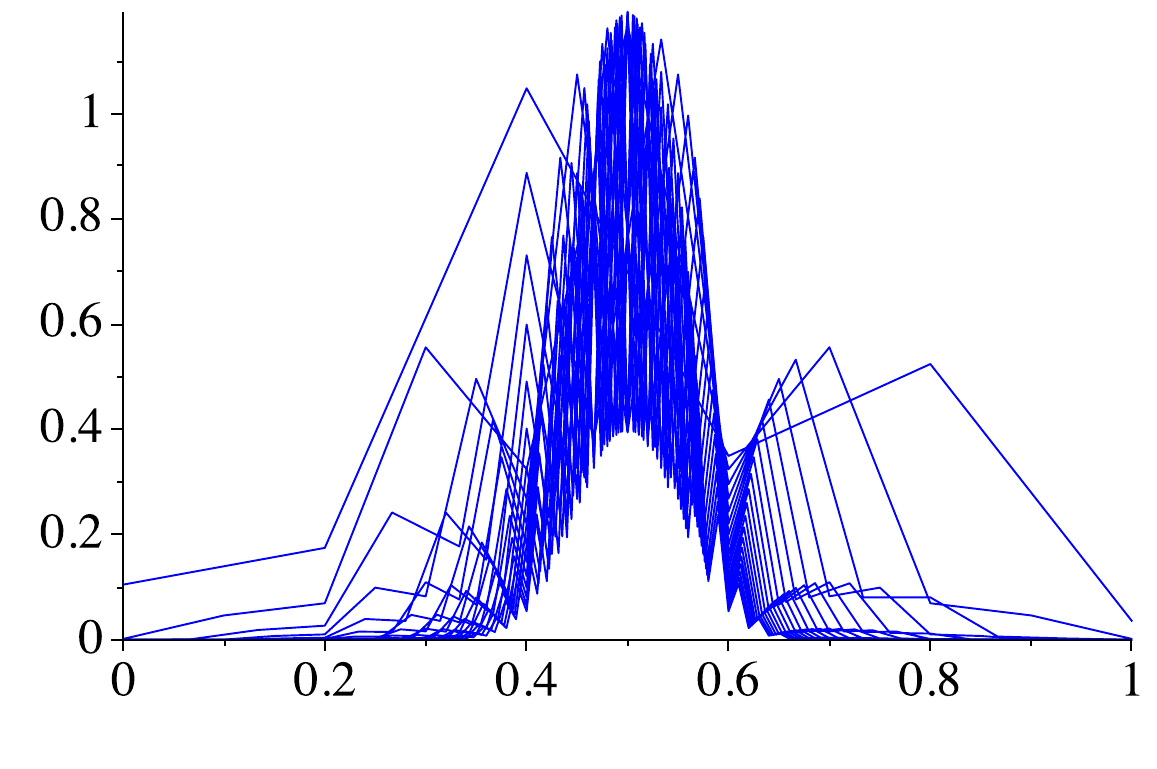}\,
\includegraphics[height=3cm]{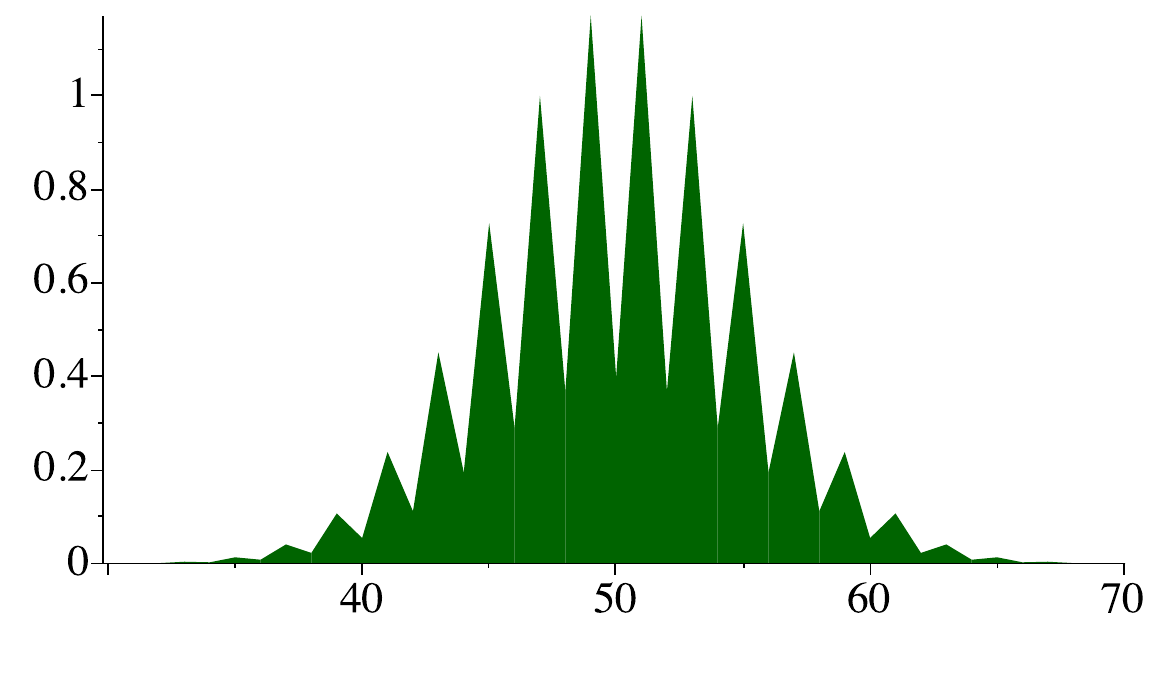}
\end{center}
\caption{The distributions of the coefficients of 
\href{https://oeis.org/A162315}{A162315}: for $n=5,10,\dots, 100$ 
(left) and $n=100$ (right). We see that they are highly oscillating 
in nature.}	\label{fig-A162315}
\end{figure}
Also note specially that despite the oscillating nature of the 
coefficients (see for example Figure~\ref{fig-A162315}), we still 
have a CLT, which is a global property, not a local one. 

The reciprocal polynomials $Q_n(v) = v^nP_n\lpa{\frac1v}$ satisfy
$Q_n\in\ET{1+v^2,\frac{v(1+v)}{n-1}}$ with the initial condition 
$Q_1(v) = 1+3v$; see \href{https://oeis.org/A124846}{A124846}. The 
coefficients of $Q_n$ yield the CLT 
$\mathscr{N}\lpa{\frac12n,\frac14n}$.

These two OEIS sequences, together with a few others leading to the 
same CLT $\mathscr{N}\lpa{\frac12n,\frac14n}$, are summarized in the 
following table. In all cases, it is possible to derive an optimal 
Berry-Esseen bound but we omit the details because these examples are 
comparatively simpler (put together here mainly to show the modeling 
diversity of the Eulerian recurrences). 

\begin{center}
\begin{tabular}{llll}\\
\multicolumn{1}{c}{OEIS} &
\multicolumn{1}{c}{$e_n$} &
\multicolumn{1}{c}{Type} &
\multicolumn{1}{c}{$[v^k]P_n(v)$} \\ \hline
\href{https://oeis.org/A161119}{A161119} 
& $1$ & $\ET{2vn-v,1+v;1}$ & \eqref{A161119} \\
\href{https://oeis.org/A161121}{A161121} & $1$ 
& $\ET{(1+v^2)n-v^2,v(1+v);1}$ & \eqref{A161119} \\ 
\href{https://oeis.org/A162315}{A162315} & $n-1$ 
& $\GT{1}{2vn-2v,1+v;3+v}$ & $(2-(-1)^{n-k})\binom{n}{k}$ \\ 
\href{https://oeis.org/A007318}{A007318} & $n-1$ 
& $\GT{1}{2vn-2v,1+v;1+v}$ & $\binom{n}{k}$ \\ 
\href{https://oeis.org/A124846}{A124846} & $n-1$ 
& $\GT{1}{(1+v^2)n-1-v^2,v(1+v);1+3v}$  
& $(2-(-1)^k)\binom{n}{k}$ \\ 
\href{https://oeis.org/A121448}{A121448} & $n+2$ 
& $\ET{4vn+2v,2(1+v);1}$ 
& $\frac{2^k}{n+1}\binom{n+1}{k}\binom{n+1-k}{\frac{n-k}2}$\\ 
\href{https://oeis.org/A143358}{A143358} & $n+1$ 
& $\ET{4vn+2,2(1+v);1}$ &
$2^k\binom{n}{k}\binom{n-k}{\tr{\frac12(n-k)}}$\\ \hline
\end{tabular}
\end{center}	
\justifying

\medskip

In particular, the sequence \href{https://oeis.org/A121448}{A121448} 
is also periodic because $\binom{n+1-k}{\frac{n-k}2}=0$ when $n-k$ is 
odd. 

On the other hand, the $n$th order derivative of 
$\sqrt{\frac{1+v}{1-v}}$ leads to the polynomials satisfying the 
recurrence $P_n\in\ET{2vn+ 1-2v,1+v;1}$; compare \eqref{A256978}. 
The EGF is given by 
\begin{align}\label{rho-2qv}
    (1-(1+v)z)^{-\frac32}
    (1+(1-v)z)^{-\frac12},
\end{align}
from which we deduce the CLT $\mathscr{N}\lpa{\frac12n, 
\frac14n;n^{-\frac12}}$ by Theorem~\ref{thm-saqp} with 
$\rho(v)=\frac1{1+v}$. 

\subsubsection{$(\alpha(v),\beta(v))
=(2(1+v),3+v) \Longrightarrow \mathscr{N}\lpa{\frac14n,
\frac3{16}n}$}\label{sss-23}

The sequence \href{https://oeis.org/A091867}{A091867}, which
enumerates the number of Dyck paths of semi-length $n$ having $k$
peaks at odd height, has its generating polynomial satisfying the
recurrence
\[
    P_n\in\EET{\frac{2((1+v)n-1)}{n+1}}
    {\frac{3+v}{n+1};1}.
\]

A closed-form expression is known (see \href{https://oeis.org/A091867}{A091867})
\begin{align}\label{A091867}
    [v^{n-k}]P_n(v) = \frac1{k+1}\binom{n}{k}
    \sum_{0\le j\le k}(-1)^j
    \binom{k+1}{j}\binom{2k-2j}{k-j}. 
\end{align}
Due to the presence of the factor $(-1)^j$, the asymptotics of this 
expression is less transparent; however, we get the CLT 
$\mathscr{N}\lpa{\frac14n,\frac3{16}n}$ by Theorem~\ref{thm-clt} 
using the expression of $(\alpha(v),\beta(v))$. The corresponding 
reciprocal polynomials 
\href{https://oeis.org/A124926}{A124926} satisfy
\[
    Q_n \in\EET{\frac{(1+3v^2)n+1-3v^2}{n+1}}
    {\frac{v(1+3v)}{n+1};1}.
\]
On the other hand, since the ordinary generating function (OGF) of 
$P_{n-1}$ satisfies 
\begin{align}\label{rho-21pv}
    \frac12-\frac12\sqrt{\frac{1-(3+v)z}{1+(1-v)z}},
\end{align}
an optimal Berry-Esseen bound also follows from 
Theorem~\ref{thm-saqp} with $\rho(v) = \frac1{3+v}$. Furthermore, by 
this OGF we have for $n\ge1$
\[
    P_{n-1}(v) 
    = \frac1n[w^{n-1}]\left(1+ vw + \frac{w^2}{1-w}\right)^n.
\]
From this and Lagrange inversion formula \cite{Stanley2012}, we 
derive the expression (without alternating terms; cf.\ 
\eqref{A091867}) 
\[
    [v^{n-k}]P_n(v) = 
    \frac1{n+1}\binom{n+1}{k+1}
    \sum_{0\le j\le \tr{\frac12k}}
    \binom{k+1}{j} \binom{k-1-j}{j-1}.
\]
Although non-alternating, the asymptotics of the right-hand side 
still remains obscure. 
%
%

These sequences and a few others of the same type are listed as 
follows. 
\vspace*{-.6cm}
\begin{small}
\begin{center}
\begin{tabular}{llll}\\
\multicolumn{1}{c}{OEIS} &
\multicolumn{1}{c}{$e_n$} &
\multicolumn{1}{c}{Type} &
\multicolumn{1}{c}{CLT}  \\ \hline	
\href{https://oeis.org/A091867}{A091867} & $n+1$ 
& $\ET{(2v+2)n-2,3+v;1}$ & 
$\mathscr{N}\lpa{\frac14n,\frac3{16}n;n^{-\frac12}}$ \\ 
\href{https://oeis.org/A124926}{A124926} & $n+1$ 
& $\ET{(1+3v^2)n+1-3v^2,v(1+3v);1}$  
& $\mathscr{N}\lpa{\frac 34n,\frac3{16}n;n^{-\frac12}}$ \\ \hline
\href{https://oeis.org/A171128}{A171128} & $n$ 
& $\ET{(2v+2)n-1-v,3+v;1}$ & 
$\mathscr{N}\lpa{\frac14n,\frac3{16}n;n^{-\frac12}}$ \\ 
\href{https://oeis.org/A135091}{A135091} & $n$ 
& $\ET{(1+3v^2)n+v(1-3v),v(1+3v);1}$  
& $\mathscr{N}\lpa{\frac 34n,\frac3{16}n;n^{-\frac12}}$ \\ \hline
\href{https://oeis.org/A091869}{A091869} & $n+1$ 
& $\GT{1}{(2v+2)n-1-v,3+v;1}$ & 
$\mathscr{N}\lpa{\frac14n,\frac3{16}n;n^{-\frac12}}$ \\ 
\href{https://oeis.org/A091187}{A091187} & $n+1$ 
& $\GT{1}{(1+3v^2)n+1+3v-6v^2,v(1+3v);1}$  
& $\mathscr{N}\lpa{\frac 34n,\frac3{16}n;n^{-\frac12}}$ \\ \hline
\href{https://oeis.org/A171651}{A171651} & $n+1$ 
& $\ET{(2v+2)n+2,3+v;1}$ & 
$\mathscr{N}\lpa{\frac14n,\frac3{16}n;n^{-\frac12}}$ \\ \hline
\end{tabular}
\end{center}
\end{small}
Here the first six are grouped in reciprocal pairs. Each of these
has a closed-form expression for their OGFs (as well as a summation 
formula similar to \eqref{A091867}); we list below only their OGFs.

\begin{small}
\begin{center}
\renewcommand{\arraystretch}{2.2}
\begin{tabular}{ll}\hline
\href{https://oeis.org/A171128}{A171128} & $\dfrac1{\sqrt{(1-(1-v)z)(1-(3+v)z)}}$\\
\href{https://oeis.org/A091869}{A091869} & $\dfrac{1-(1+v)z-\sqrt{(1+(1-v)z)(1-(3+v)z)}}{2z}$ \\
\href{https://oeis.org/A171651}{A171651} & $\dfrac{1-(3+v)z+\sqrt{(1+(1-v)z)(1-(3+v)z)}}{2(1-(3+v)z)}$
\\ \hline
\end{tabular}    
\end{center}
\end{small}

\subsubsection{$(\alpha(v),\beta(v))
=(q(1+3v),2qv) \Longrightarrow \mathscr{N}\lpa{\frac12n,
\frac18n}$} \label{sec-3v2v}

The generating polynomials of Narayana numbers (enumerating peaks in 
Dyck paths; see \cite{Sulanke1999} and \href{https://oeis.org/A090181}{A090181})
\[
    P_n(v) := \sum_{1\le k\le n}\frac1k\binom{n}{k-1}
	\binom{n-1}{k-1}v^k\qquad(n\ge1),
\]
also satisfy
\begin{align}\label{1-plus-3v}
    (n+1)P_n(v) = ((1+3v)n-1-v)P_{n-1}(v)
    +2v(1-v)P_{n-1}'(v)\qquad(n\ge1)
\end{align}
in addition to the usual three-term recurrence
\[
    (n+1)P_n(v) = (2n-1)(1+v)P_{n-1}(v)
    -(n-2)(1-v)^2P_{n-2}(v).
\]
These polynomials are palindromic and the CLT
$\mathscr{N}\lpa{\frac12n,\frac18n}$ for $[v^k]P_n(v)$ follows easily
from Theorem~\ref{thm-clt}. An essentially identical sequence
\href{https://oeis.org/A001263}{A001263} corresponds to
$v^{-1}P_n(v)$. The OGF of $P_n$ satisfies
\begin{align}\label{ogf-narayana}
    f(z,v) := \sum_{n\ge0}P_n(v) z^n  
    = \frac{1-(1+v)z-\sqrt{1-2(1+v)z+(1-v)^2z^2}}{2z},
\end{align}
from which we get an additional convergence rate $n^{-\frac12}$ by
Theorem~\ref{thm-saqp} with $\rho(v) = (1+\sqrt{v})^{-2}$. These and
a few others satisfying $P_n\in\ET{\frac{(1+3v)n+\gamma(v)}{e_n},
\frac{2v}{e_n}}$, leading to the same CLT $\mathscr{N}\lpa{\frac12n,
\frac18n; n^{-\frac12}}$, are collected in the following table.
\vspace*{-.5cm}
\begin{center}
\begin{tabular}{llll}\\
	\multicolumn{1}{c}{OEIS} &
	\multicolumn{1}{c}{$e_n$} &
	\multicolumn{1}{c}{Type} &
	\multicolumn{1}{c}{$[v^k]P_n(v)$} \\ \hline
\href{https://oeis.org/A086645}{A086645} & $n-1$ 
& $\GT{1}{(1+3v)(n-1),2v;1+v}$ 
& $\binom{2n}{2k}$\\
\href{https://oeis.org/A103328}{A103328} & $n-1$ 
& $\GT{1}{(1+3v)n-4v,2v;2}$ 
& $\binom{2n}{2k+1}$ \\ 
\href{https://oeis.org/A091044}{A091044} & $n$ 
& $\ET{(1+3v)n+1-v,2v;1}$ 
& $\frac12\binom{2n}{2k+1}$ \\ 
\href{https://oeis.org/A001263}{A001263} & $n+1$ 
& $\ET{(1+3v)n-1-v,2v;1}$ 
& $\frac1{k}\binom{n}{k-1}\binom{n-1}{k-1}$\\ 
\href{https://oeis.org/A090181}{A090181} & $n+1$ 
& $\ET{(1+3v)n-1-v,2v;1}$
& $\frac1k\binom{n}{k-1}\binom{n-1}{k-1}$ \\ 
\href{https://oeis.org/A131198}{A131198} & $n+1$ 
& $\ET{(1+3v)n+1-3v,2v;1}$
& $\frac1{n-k}\binom{n}{k+1}\binom{n-1}{k}$ \\ 
\href{https://oeis.org/A118963}{A118963} & $n$ 
& $\GT{1}{(1+3v)n+1-3v,2v;2}$
& $\frac{n+1}{n}\binom{n}{k}\binom{n}{k+1}$ \\ 
\href{https://oeis.org/A008459}{A008459} & $n$ 
& $\ET{(1+3v)n-2v,2v;1}$
& $\binom{n}{k}^2$ \\ \hline
\end{tabular}	
\end{center}
In particular, we see that the coefficients $\binom{n}{k}^2$ follow
asymptotically a CLT $\mathscr{N}\lpa{\frac12n,\frac18n}$, the
variance being smaller than that of $\binom{n}{k}$; more generally,
$\binom{n}{k}^\alpha$ follows asymptotically the CLT 
$\mathscr{N}\lpa{\frac12n,\frac 1{4\alpha}n}$ for large $n$ when 
$\alpha>0$; see Figure~\ref{fig-binom}.

\begin{figure}[!ht]
\begin{center}
\includegraphics[height=3cm]{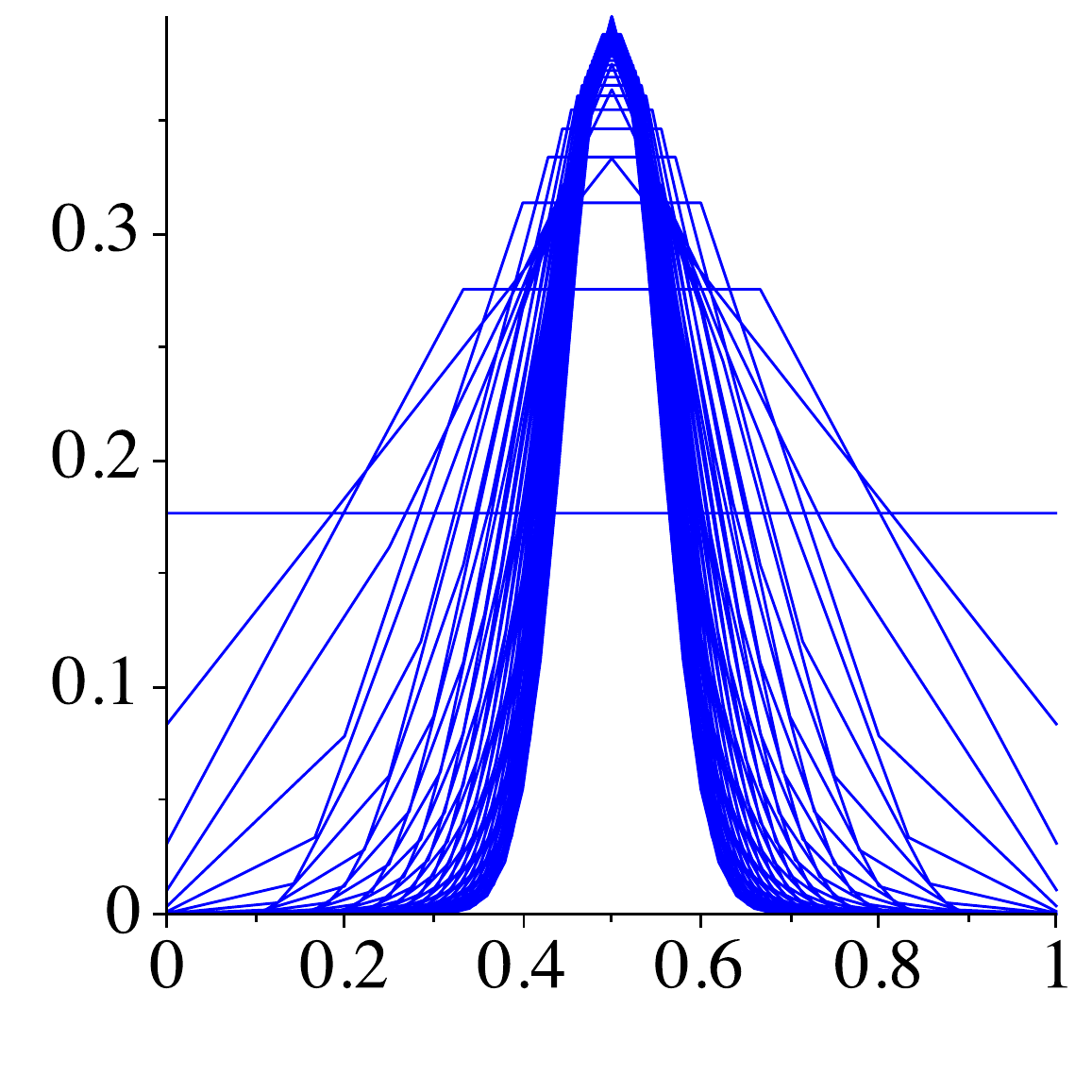}\;
\includegraphics[height=3cm]{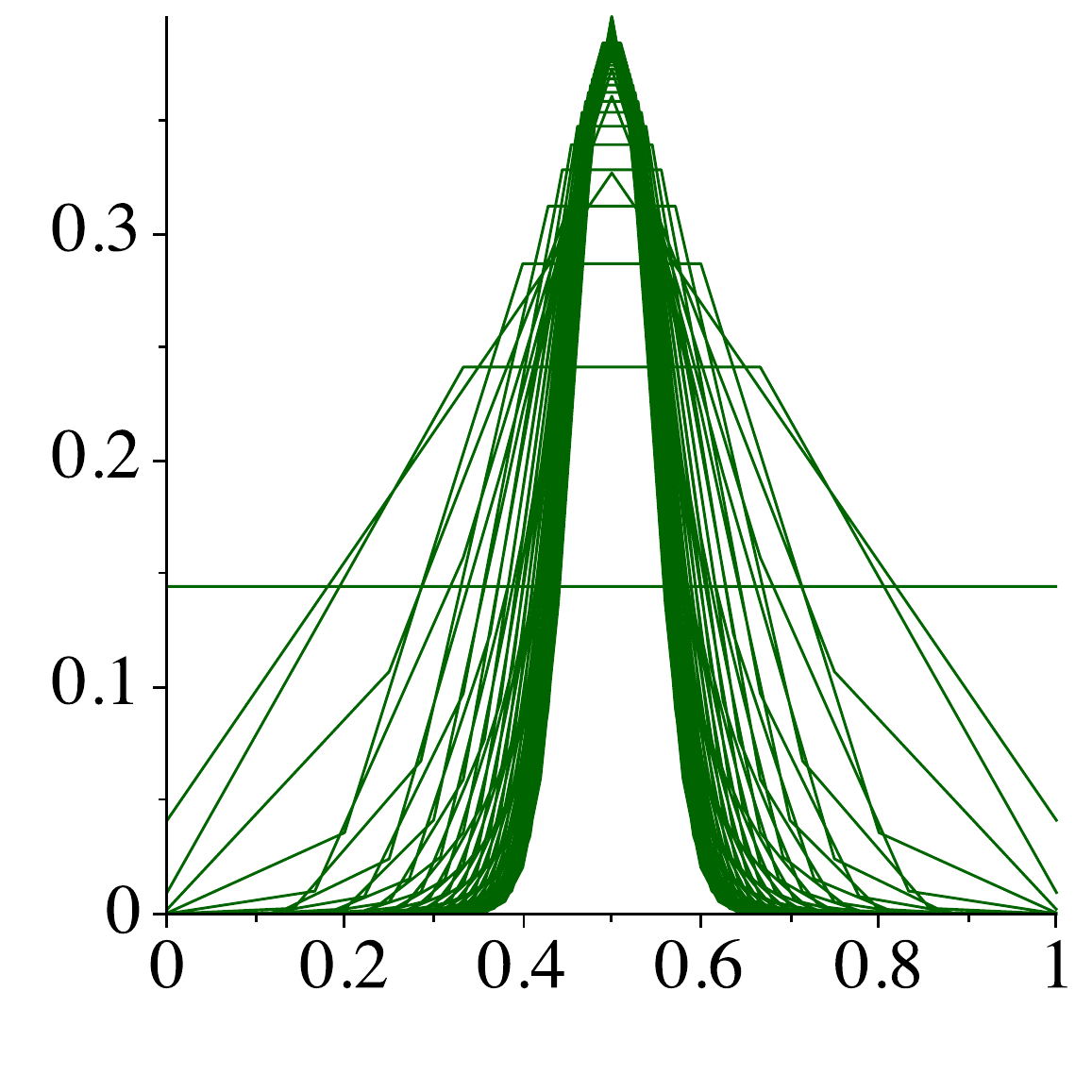}\;
\includegraphics[height=3cm]{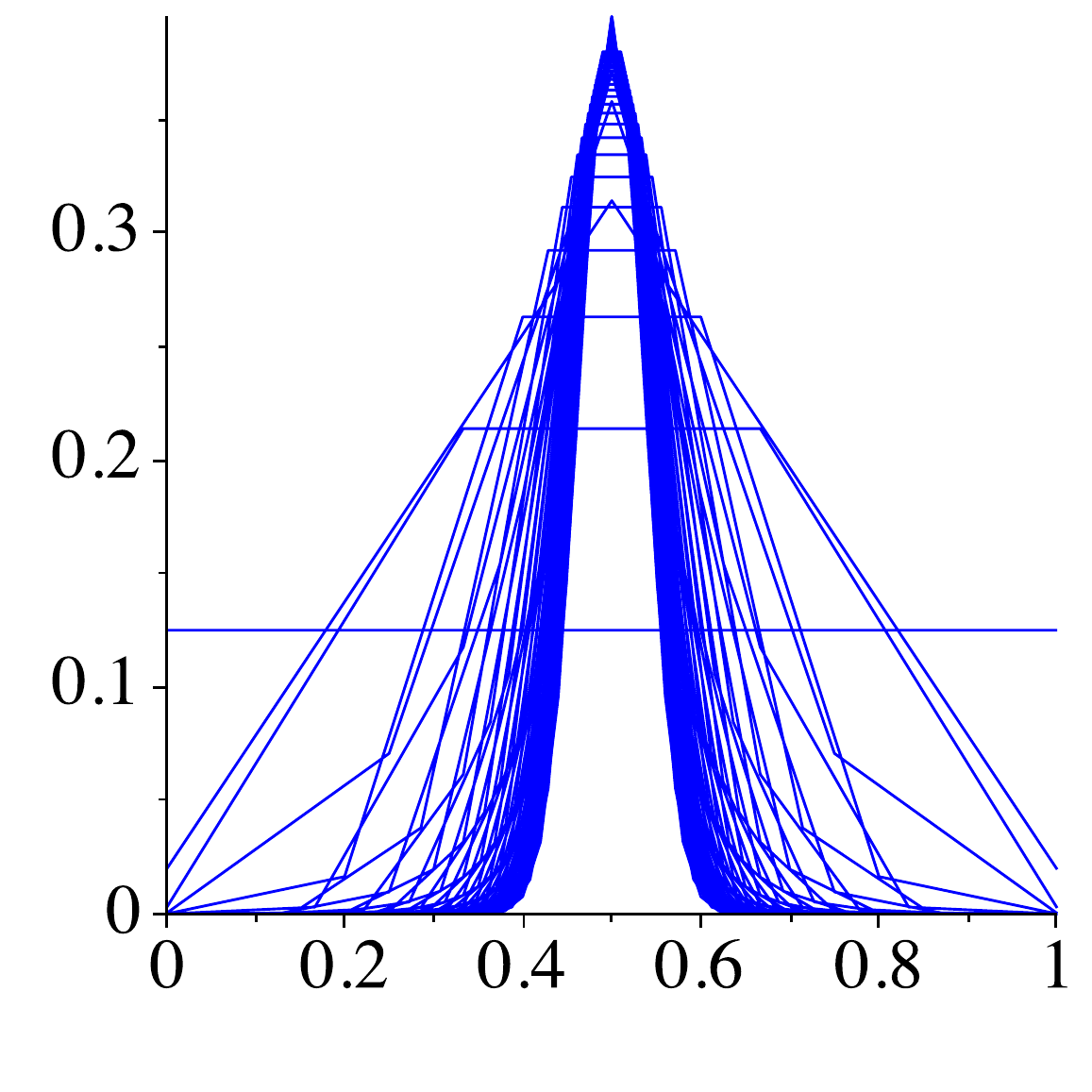}\;
\includegraphics[height=3cm]{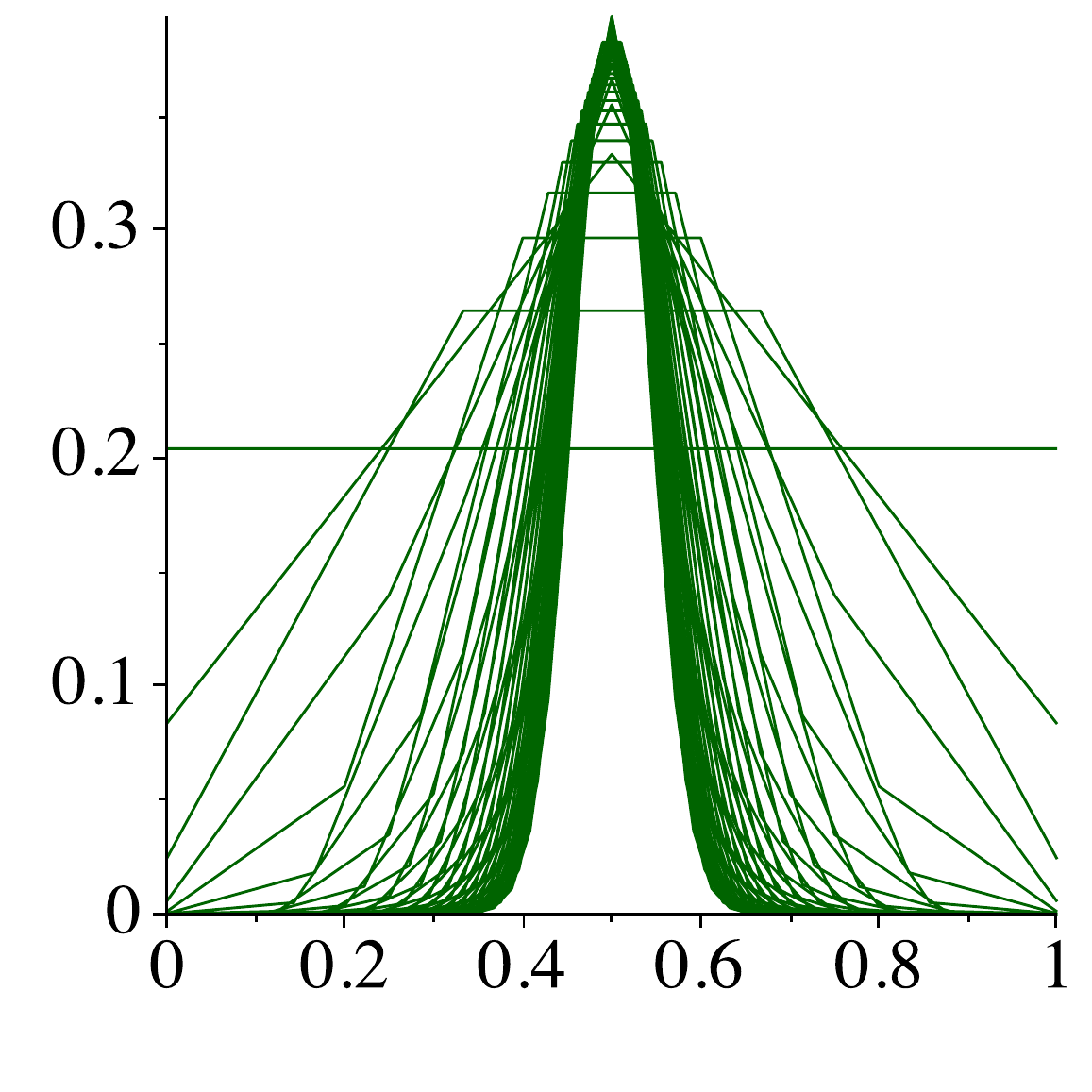}
\end{center}
\caption{Normalized histograms of $\binom{n}{k}^\alpha$ for
$n=1,\dots,50$ and $\alpha=2,3,4$ (the first three), respectively,
and the Eulerian distribution (rightmost). The variance for the
second and the fourth are both asymptotic to $\frac1{12}n$. Note that
$\binom{n}{k}^3$ correspond to
\href{https://oeis.org/A181543}{A181543} and $\binom{n}{k}^4$ to
\href{https://oeis.org/A202750}{A202750}.}
\label{fig-binom}
\end{figure}

While the generating polynomials of $\binom{2n}{2k}$ satisfy 
\eqref{1-plus-3v}, those of $\binom{2n+1}{2k}$ and 
$\binom{2n+1}{2k+1}$ satisfy the following recurrences
\vspace*{-.5cm}
\begin{center}
\begin{tabular}{ccc}\\ \hline
\href{https://oeis.org/A091042}{A091042} & $\binom{2n+1}{2k}$ & 
$\ET{\frac{2(1+3v)n-1-3v}{2n-1},\frac{4v}{2n-1};1}$   \\
\href{https://oeis.org/A103327}{A103327} & $\binom{2n+1}{2k+1}$ &
$\ET{\frac{2(1+3v)n+1-5v}{2n-1},\frac{4v}{2n-1};1}$ \\ \hline
\end{tabular}    
\end{center}
The two sequences form a reciprocal pair. 

\subsubsection{$(\alpha(v),\beta(v))
=(5+3v,2(1+v)) \Longrightarrow \mathscr{N}\lpa{\frac14n,
\frac5{32}n}$}\label{sss-5p3v}

The polynomials (\href{https://oeis.org/A114608}{A114608}, 
enumerating the number of peaks in bicolored Dyck paths)
\[
    P_n(v) = \frac1n\sum_{0\le k\le n}v^k
    \sum_{0\le j\le n-k}\binom{n}{j+1}\binom{n-k}{j}
    2^j,
\]
satisfy $P_n\in\ET{\frac{(5+3v)n-3-v}{n+1}, \frac{2(1+v)}{n+1};
1}$. The CLT $\mathscr{N}\lpa{\frac14n,\frac5{32}n}$ then follows
from Theorem~\ref{thm-clt}, and an effective version with 
$n^{-\frac12}$ convergence rate follows from Theorem~\ref{thm-saqp} 
using the OGF
\begin{align}\label{rho-5p3v}
    \frac{1+(1-v)z-\sqrt{1-2(3+v)z+(1-v)^2z^2}}{4z},
\end{align}
with $\rho(v) = \lpa{\sqrt{2}+\sqrt{1+v}}^{-2}$. 

\subsubsection{$(\alpha(v),\beta(v))
=\lpa{\frac13(7+2v),\frac13(5+4v)} \Longrightarrow 
\mathscr{N}\lpa{\frac19n, \frac2{27}n}$}\label{sss-7p2v}

The generating polynomial (\href{https://oeis.org/A181371}{A181371}) 
of the pattern occurrences of ``$01$'' in ternary words satisfies 
$P_n\in\ET{\frac{(7+2v)n+2(1-v)}{3n},\frac{5+4v}{3n};1}$. This 
follows from the OGF 
\begin{align}\label{rho-7p2v}
    \sum_{n\ge0} P_n(v) z^n 
	= \frac1{1-3z+(1-v)z^2}.
\end{align}
From this we deduce the CLT $\mathscr{N}\lpa{\frac19n, \frac2{27}n
;n^{-\frac12}}$ for the coefficients $[v^k]P_n(v)$ by 
Theorem~\ref{thm-saqp} with $\rho(v)=\frac2{3+\sqrt{5+4v}}$. 

\subsubsection{$(\alpha(v),\beta(v))
=(1+3v^2,v(1+v)) \Longrightarrow \mathscr{N}\lpa{n,
\frac12n}$}\label{sss-1p3v2}

The sequence \href{https://oeis.org/A088459}{A088459} enumerates
peaks in symmetric Dyck paths and the corresponding polynomials
satisfy $\ET{\frac{(1+3v^2)n+1+v}{n+1}, \frac{v(1+v)}{n+1};1+v}$. One
then gets the CLT $\mathscr{N}\lpa{n,\frac12n}$ by
Theorem~\ref{thm-clt}. This and a few other polynomials from OEIS are
listed as follows. 
\vspace*{-.5cm}
\begin{center}
\begin{tabular}{lll}\\ \hline
\href{https://oeis.org/A088459}{A088459}	
& Peaks in symmetric Dyck paths
& $\ET{\frac{(1+3v^2)n+1+v}{n+1}, \frac{v(1+v)}{n+1};1+v}$ \\
\href{https://oeis.org/A059064}{A059064}	
& Card-matching numbers	
& $\ET{\frac{(1+3v^2)n-2v^2}{n}, \frac{v(1+v)}{n};1}$ \\
\href{https://oeis.org/A059065}{A059065}	
& Card-matching numbers 
& $\ET{(1+3v^2)n^2-2v^2n, v(1+v)n;1}$\\
\href{https://oeis.org/A152659}{A152659}	
& Turns in lattice paths
& $\ET{\frac{(1+3v^2)n+1+2v-v^2}{n+1}, \frac{v(1+v)}{n+1};2}$ \\
\href{https://oeis.org/A247644}{A247644}	
& Even rows of \href{https://oeis.org/A088855}{A088855}
& $\ET{\frac{(1+3v^2)n+1+2v-v^2}{n+1},\frac{v(1+v)}{n+1};1}$\\
\hline
\end{tabular}
\end{center}

A convergence rate in the CLT can be obtained by solving the
corresponding PDEs and then by applying Theorem~\ref{thm-saqp}. For
example, the OGF for \href{https://oeis.org/A152659}{A152659} is
given by
\begin{align}\label{rho-1p3v2}
	\frac{2}{z(v-f(z,v^2))}-\frac2{vz},
\end{align}
where $f$ is the generating function \eqref{ogf-narayana} of Narayana 
numbers. Thus $\rho(v) = (1+v)^{-2}$ and the CLT 
$\mathscr{N}\lpa{n,\frac12n;n^{-\frac12}}$ is implied. 

\subsection{Polynomials with $(\alpha(v),\beta(v))
=(-1+(q+1)v,qv)\Longrightarrow\mathscr{N}\lpa{\frac{q+1}{2q}\,n, 
\frac{q^2-1}{12q^2}\,n}$}\label{ss-ck}

A generalization of Morisita's model \eqref{morisita} proposed 
by Charalambides and Koutras in \cite{Charalambides1993} is of the 
form 
\[
    P_n\in\EET{\frac{(-1+(q+1)v)n+1+p+(qr-p-q-1)v}{n}}
    {\frac{qv}{n};1}.
\]
The OGF $f(z,v) := \sum_{n\ge0}P_n(v)z^n$ is given by 
\begin{align}\label{egf-ck}
    \lpa{1+(1-v)z}^{p}
    \left(\frac{1-v}{1-v(1+(1-v)z)^q}\right)^{r}.
\end{align}
We write this class as $f\in\mathscr{M}(p,q,r)$ or $f\in
\mathscr{M}(p,q,r;z)$. The type $\mathscr{M}(p,q,1)$ was studied in
\cite{Charalambides1982}, and the type $\mathscr{M}\lpa{\frac
pq,\frac1q,1;qz}$ in \cite{Hsu1999} in connection with degenerate
Stirling numbers. It is interesting to compare these forms with those
(\eqref{Pnv-Eabc} and \eqref{Epqr}) for $\mathscr{A}(p,q,r)$ where
the factor ``$e^{(1-v)z}$'' there is ``mimicked" by ``$1+(1-v)z$''
here. If $[v^k]P_n(v)\ge0$ or $(-1)^n[v^k]P_n(v)\ge0$ and $|q|>1$,
then we obtain the CLT $\mathscr{N}\lpa{\frac{q+1}{2q}n,
\frac{q^2-1}{12q^2}n}$ for the coefficients by Theorem~\ref{thm-clt}
and $\mathscr{N}\lpa{\frac{q+1}{2q}n,
\frac{q^2-1}{12q^2}n;n^{-\frac12}}$ by Theorem~\ref{thm-saqp} with
$\rho(v) =-\frac{1-v^{-\frac1q}} {1-v}$.

The reciprocal polynomial $Q_n(v) := v^nP_n\lpa{\frac1v}$ satisfies
\[
    Q_n\in\EET{\frac{(1+(q-1)v)n +qr-1-p+(1+p-q)v}{n}}
    {\frac{qv}{n};1}.
\]
This gives the pair $(\alpha(v), \beta(v)) = (1+(q-1)v, qv)$, and
then the CLT $\mathscr{N} \lpa{\frac{q-1}{2q}\,n,
\frac{q^2-1}{12q^2}\,n;n^{-\frac12}}$.
If $f\in\mathscr{M}(p,q,r;z)$, then the reciprocal polynomial is of
type $\mathscr{M}(p-qr,-q,r;-z)$.

\paragraph{Runs in words: $\mathscr{M}(0,q,1;z)$ or
$\mathscr{M}(-q,-q,1;-z)$} This class of polynomials appeared in
Carlitz's study \cite{Carlitz1978b, Carlitz1979} of ``degenerate"
Eulerian numbers (which corresponds to $\mathscr{M}(0,q,1;\frac
zq)$), as well as that of rises in sequences (with repetitions)
\cite{Carlitz1966}, and was later referred to as the Carlitz numbers
in \cite[\S 14.3]{Charalambides2002a}. Such numbers also enumerate
increasing runs in $q$-ary words and have the closed-form expression
\[
    P_n(v) = \sum_{0\le k\le n} v^k
    \sum_{0\le j\le k} (-1)^{k-j}
    \binom{n+1}{k-j}\binom{qj}{n};
\]
see also \cite{Dais2001} for the occurrence of these numbers in
algebraic geometry. Note that when $q=2$, one gets the simpler
expression $\binom{n+1} {2n-2k+1}$ for $[v^k]P_n(v)$. We obtain the
CLT $\mathscr{N}\lpa{\frac{q+1}{2q}\,n, 
\frac{q^2-1}{12q^2}\,n;n^{-\frac12}}$ when
$q>1$ is an integer. When $q=1$, we get the OGF $\frac1{1-vz}$, and
the limit law is degenerate. The cases $q=2,3,4$ appear in OEIS:
\vspace*{-.5cm}
\begin{center}
\begin{tabular}{llll}\\
	\multicolumn{1}{c}{Description} &
	\multicolumn{1}{c}{OEIS} &
	\multicolumn{1}{c}{Type} &
	\multicolumn{1}{c}{CLT} \\ \hline
$\uparrow$ runs in binary words &
\href{https://oeis.org/A119900}{A119900} & $\mathscr{M}(0,2,1;z)$ & 
$\mathscr{N}\lpa{\frac34n,\frac1{16}n;n^{-\frac12}}$ \\
\href{https://oeis.org/A119900}{A119900} without zeros & \href{https://oeis.org/A109447}{A109447} &
& $\mathscr{N}\lpa{\frac14n,\frac1{16}n;n^{-\frac12}}$ \\
Reciprocal of \href{https://oeis.org/A119900}{A119900} & \href{https://oeis.org/A202064}{A202064} &
$\mathscr{M}(-2,-2,1;-z)$ & 
$\mathscr{N}\lpa{\frac14n,\frac1{16}n;n^{-\frac12}}$ \\
\href{https://oeis.org/A202064}{A202064} without zeros & \href{https://oeis.org/A034867}{A034867} &
& $\mathscr{N}\lpa{\frac14n,\frac1{16}n;n^{-\frac12}}$ \\ \hline
$\uparrow$ runs in ternary words &
\href{https://oeis.org/A120987}{A120987} & $\mathscr{M}(0,3,1;z)$ 
& $\mathscr{N}\lpa{\frac23n,\frac2{27}n;n^{-\frac12}}$ \\
Reciprocal of \href{https://oeis.org/A120987}{A120987} &
\href{https://oeis.org/A120906}{A120906} & $\mathscr{M}(-3,-3,1;-z)$ 
& $\mathscr{N}\lpa{\frac13n,\frac2{27}n;n^{-\frac12}}$ \\ \hline
$\uparrow$ runs in quaternary words &
\href{https://oeis.org/A265644}{A265644} & $\mathscr{M}(0,4,1;z)$ 
& $\mathscr{N}\lpa{\frac58n,\frac5{64}n;n^{-\frac12}}$ \\ \hline
\end{tabular}
\end{center}

\paragraph{Patterns in words: $\mathscr{M}(1,2,1)$} Similar to the 
numbers \href{https://oeis.org/A119900}{A119900} above, we also have 
the following variants for the sequence $\binom{n}{2k}$. 
\vspace*{-.5cm}
\begin{center}
\begin{tabular}{llll}\\
	\multicolumn{1}{c}{Description} &
	\multicolumn{1}{c}{OEIS} &
	\multicolumn{1}{c}{Type} &
	\multicolumn{1}{c}{CLT} \\ \hline
$\binom{n}{2n-2k}$ &
\href{https://oeis.org/A098158}{A098158} 
& $1+vz\mathscr{M}(1,2,1;z)$ & 
$\mathscr{N}\lpa{\frac34n,\frac1{16}n;n^{-\frac12}}$ \\
$2\binom{n}{2k}$ &
\href{https://oeis.org/A119462}{A119462} 
& $2\mathscr{M}(-1,-2,1;-z)$ & 
$\mathscr{N}\lpa{\frac34n,\frac1{16}n;n^{-\frac12}}$ \\
shifted version of \href{https://oeis.org/A098158}{A098158} & 
\href{https://oeis.org/A098157}{A098157} 
& $\mathscr{M}(1,2,1;z)$ & 
$\mathscr{N}\lpa{\frac34n,\frac1{16}n;n^{-\frac12}}$ \\ 
\href{https://oeis.org/A098158}{A098158} without zeros 
& \href{https://oeis.org/A109446}{A109446} &  & 
$\mathscr{N}\lpa{\frac34n,\frac1{16}n;n^{-\frac12}}$\\ 
Reciprocal of \href{https://oeis.org/A098158}{A098158} 
& \href{https://oeis.org/A202023}{A202023} & 
$\mathscr{M}(-1,-2,1;-z)$ 
& $\mathscr{N}\lpa{\frac14n,\frac1{16}n;n^{-\frac12}}$\\ 
\href{https://oeis.org/A202023}{A202023} without zeros 
& \href{https://oeis.org/A034839}{A034839} & &
$\mathscr{N}\lpa{\frac14n,\frac1{16}n;n^{-\frac12}}$ \\ \hline
\end{tabular}
\end{center}

\paragraph{Binomial extension of Eulerian numbers: 
$\mathscr{M}(p,q,1)$} This class was studied in 
\cite{Charalambides1982,Koutras1994}, where occurrences and 
applications are mentioned. 
\vspace*{-.5cm}
\begin{center}
\begin{tabular}{llll}\\
	\multicolumn{1}{c}{Description} &
	\multicolumn{1}{c}{OEIS} &
	\multicolumn{1}{c}{Type} &
	\multicolumn{1}{c}{CLT} \\ \hline
$(1-v)^{n+1}\sum_{j\ge0}\binom{3j+n}{n}v^j$ &
\href{https://oeis.org/A178618}{A178618} 
& $\mathscr{M}(-1,-3,1;-z)$ & 
$\mathscr{N}\lpa{\frac13n,\frac2{27}n;n^{-\frac12}}$ \\
$(1-v)^{n+1}\sum_{j\ge0}\binom{4j+n}{n}v^j$ &
\href{https://oeis.org/A178619}{A178619} 
& $\mathscr{M}(-1,-4,1;-z)$ & 
$\mathscr{N}\lpa{\frac38n,\frac5{64}n;n^{-\frac12}}$ \\ \hline
\end{tabular}
\end{center}
In general, the polynomials 
$(1-v)^{n+1}\sum_{j\ge0}\binom{qj+n}{n}v^j$ are of type 
$\mathscr{M}(-1,-q,1;-z)$ for any real $q$, and one obtains the CLT 
$\mathscr{N}\lpa{\frac{q+1}{2q}\,n, 
\frac{q^2-1}{12q^2}\,n;n^{-\frac12}}$ when $q\ge2$ is an integer. 

\paragraph{Degenerate limit law: $\mathscr{M}(2,1,3)$} Consider
\href{https://oeis.org/A106246}{A106246} for which $a_{n,k}=
\binom{n}{k}\binom{2}{n-k}$. Then $P_n\in\ET{\frac{(2v-1)n+3-v}{n},
\frac{v}{n};1}$. This is of type $\mathscr{M}(2,1,3)$. Of course, the
random variable $X_n$ is degenerate or follows in the limit the Dirac
distribution. The reciprocal polynomials $Q_n(v):=
v^nP_n\lpa{\frac1v}$ satisfies $Q_n\in\ET{n+2v,v;1}$. This is of the
type of problems we will examine in the next three sections.

Finally, for $\mathscr{M}(p,1,r)$, the GF becomes 
\[
    \frac{(1+(1-v)z)^p}{(1-vz)^r},
\]
which has nonnegative coefficients when $0\le p\le r$. 

See Section~\ref{ss-1-sv} for a sequence of polynomials closely 
related to $\mathscr{M}(0,2,\frac32)$.

\section{Non-normal limit laws} \label{sec-nnll}

We now work out the method of moments for the recurrence
\eqref{Pnv-gen} when the limit laws are not normal. It turns out all
examples we found are of the simpler form
\begin{align}\label{Pn-nn}
    P_n\in\EET{\frac{\alpha n + \gamma+\gamma'(v-1)}{e_n}}
    {\frac{\beta+\beta'(v-1)}{e_n};c_0+c_1(v-1)},
\end{align}
which are polynomials in $v$ of degree at most $n+1$, where $\alpha,
\beta,\beta'\gamma,\gamma'$ are constants (often integers) and
$\{e_j\}_{j\ge1}$ is a positive sequence. For this framework, if we
apply naively Theorem~\ref{thm-clt} (after normalizing by 
$\prod_{1\le j\le n}e_j$), then we see that $\mu=\sigma^2=0$ (since 
$\alpha(v)=\alpha$ is a constant); thus Theorem~\ref{thm-clt} fails 
but we will see that the same method of proof still applies.

It is also possible to apply the complex-analytic approach to all
cases we discuss here and quantify the convergence rates and even the
asymptotic densities, but we omit this approach here for brevity and
for the following reasons: first, the EGFs or OGFs of $P_n$ under
\eqref{Pn-nn} are comparatively simpler than those in the case of
normal limit laws and the application of singularity analysis is
straightforward; second, the method of moments does not rely on the
availability of more tractable EGFs or OGFs and is completely 
elementary and to some extent more general, although the limit 
results are generally weaker and less easy to be further strengthened.

\subsection{Recurrence for the factorial moments}

Throughout this section, let $P_n$ be defined by \eqref{Pn-nn}. 
Assume that  
\begin{align}\label{nn-cond1}
    [v^k]P_n(v)\ge0 \text{ for all }  k,n\ge0 
    \text{ and }
    P_0(1)=c_0>0,\alpha>0, \alpha+\gamma>0,
\end{align} 
which then implies, by the relation
\[
    P_n(1) = P_0(1)\prod_{1\le j\le n}
    \frac{\alpha j+\gamma}{e_j}, 
\]
that $P_n(1)>0$ for $n\ge1$. Since the coefficients are nonnegative
and $P_n(1)>0$, we define the random variables $X_n$ as in
\eqref{Xnk-general}. In particular, $P_0(0)=c_0-c_1\ge0$, implying 
that $\frac{c_1}{c_0}\in[0,1]$.

For convenience, introduce, \emph{throughout this section}, the 
notations
\begin{align}\label{tau}
    \tau_1 := -\frac{\beta}\alpha,\quad
	\tau_2 := \frac{\gamma}\alpha,\quad\text{and}\quad 
    \tau_3 := -\frac{\gamma'}{\beta'}. 
\end{align}
Here $\tau_3$ is defined when $\beta'\ne0$, and  by \eqref{nn-cond1}, 
$1+\tau_2>0$. 

To compute the factorial moments of $X_n$, we
rewrite \eqref{Pn-nn} as
\[
    \overbar{P}_n(t) := \frac{P_n(1+t)}{P_n(1)} 
    = \frac{\alpha n + \gamma+\gamma't}
	{\alpha n+\gamma}\,\overbar{P}_{n-1}(t)
	- \frac{t(\beta+\beta't)}{\alpha n+\gamma}\,\overbar{P}_{n-1}'(t),
\]
with $\overbar{P}_0(t) = 1+\frac{c_1}{c_0}t$. 
\begin{lmm} Let $\overbar{P}_{n,m} := \overbar{P}_n^{(m)}(0)$ denote
the $m$-th factorial moment of $X_n$. Then for $n,m\ge1$
\begin{align}\label{Qnm-fm}
    \overbar{P}_{n,m} 
    = \left(1+\frac{m\tau_1}{n+\tau_2}\right)
    \overbar{P}_{n-1,m} + \frac{m(\gamma'-(m-1)\beta')}
	{\alpha (n+\tau_2)}\,\overbar{P}_{n-1,m-1},
\end{align}
with the initial conditions $\overbar{P}_{n,0}=1$, $\overbar{P}_{0,1}
=\frac{c_1}{c_0}$, and $\overbar{P}_{0,m}=0$ for $m\ge2$.
\end{lmm}

\paragraph{Asymptotics of the mean}
By solving \eqref{Qnm-fm} for $m=1$, we obtain the following exact
expression for the mean $\overbar{P}_{n,1}$.
\begin{lmm} Let $n_0\ge0$ be the largest $n$ for which 
$n+\tau_1+\tau_2=0$; let $n_0=0$ if no such $n$ exists. 
Then the expected value $\mathbb{E}(X_n)=\overbar{P}_{n,1}$ of 
$X_n$ satisfies for $n>n_0$
\begin{align}\label{Qn1}
    \mathbb{E}(X_n) 
	= \frac{\gamma'}{\beta}
	+\left(\mathbb{E}(X_{n_0}) 
    -\frac{\gamma'}{\beta}\right)
    \frac{\Gamma\lpa{n+\tau_1+\tau_2+1} 
    \Gamma\lpa{n_0+\tau_2+1}}
    {\Gamma\lpa{n+\tau_2+1} 
    \Gamma\lpa{n_0+\tau_1+\tau_2+1}}.
\end{align}
\end{lmm}
It turns out that the sign of $\tau_1$ is crucial in determining the 
type of the limit law being discrete or continuous in almost all 
cases we discuss. 
\begin{cor}\label{cor-tau}
If $\beta>0$ (or $\tau_1<0$), then 
\[
    \mathbb{E}(X_n)  
    = \frac{\gamma'}{\beta} + O\lpa{n^{\tau_1}};
\]
if $\beta<0$ (or $\tau_1>0$), then 
\[
    \mathbb{E}(X_n)  
	= \left(\frac{c_1}{c_0} 
    -\frac{\gamma'}{\beta}\right)
    \frac{\Gamma\lpa{\tau_2+1}}
    {\Gamma\lpa{\tau_1+\tau_2+1}}\, n^{\tau_1}
    +O\left(1+ n^{\tau_1-1}\right).
\]
\end{cor}
\begin{proof}
When $\tau_1>0$, we can take $n_0=0$ because of the condition 
$1+\tau_1>0$ (or $\alpha+\gamma>0$) in \eqref{nn-cond1}. Then the 
approximations in both cases follow directly from \eqref{Qn1} and 
$\mathbb{E}(X_0) = \frac{c_1}{c_0}$. 
\end{proof}

The discussion of the special case when $\beta=0$ is simpler and 
deferred to Section~\ref{sec-beta-is-0}. 

Note specially that in the first case of positive $\beta$ the
dominant term is independent of the initial values $c_0$ and $c_1$,
and so are all moments, as well as the limit law, as we will see
later, in contrast to the negative $\beta$ case in which all moments 
asymptotics and the limit law depend critically on the initial 
values.

\paragraph{Dependence of the parameters}
From Corollary~\ref{cor-tau} and the nonnegativity of the 
coefficients $[v^k]P_n(v)$ (and the mean), we obtain the following 
relations. 
\begin{cor} \label{cor:beta-gamma}
    If $\beta>0$, then $\gamma'\ge 0$; if $\beta<0$, then 
    $\gamma'\ge\frac{c_1}{c_0}\beta$. 
\end{cor}

More relations among the variables can be derived.
\begin{lmm}\label{lmm-tau3}
Assume that the relations \eqref{nn-cond1} hold. If 
$\beta'>0$, then $\gamma'=\ell\beta'$ for some positive integer 
$\ell$; if $\beta'<0$, then $\gamma'\ge\beta'$ (or $\tau_3\ge-1$).
\end{lmm}
\begin{proof}
Consider first $\beta'>0$. By the expression
\[
    [v^{n+1}]P_n(v) 
    = c_1\prod_{1\le j\le n}\frac{\gamma'-j\beta'}{e_j}
    \qquad(n\ge1),
\]
and the nonnegativity of $[v^k]P_n(v)$ for all $k$, we deduce that 
$\gamma'=\ell\beta'$ for some positive integer $\ell$. Similarly, if 
$\beta'<0$, then by induction $\gamma'\ge \beta'$.
\end{proof}
The situation when $\gamma'=\beta'$ (or $\tau_3=-1$) leads to a 
Bernoulli limit law; see Theorem~\ref{thm-dll} below. 

\paragraph{Solution to the recurrence}
We prove in what follows that the factorial moments in the first case 
($\tau_1<0$) are all bounded, leading to a discrete limit law, and 
that those in the second case ($\tau_1>0$) all behave like powers of 
the mean, yielding mostly a continuous limit law. 

For higher moments, we consider the following recurrence, which is 
Lemma~\ref{lmm-xnyn} but specially formatted in the current setting. 
\begin{lmm} Let $n_0\ge0$ be the largest $n$ for which
$n+m\tau_1+\tau_2=0$; let $n_0=0$ if no such $n$ exists. Then the
solution to the recurrence
\begin{align}\label{xn-yn-nn}
    x_n = \left(1+\frac{m\tau_1}{n+\tau_2}\right)x_{n-1}
    + \frac{y_n}{\alpha(n+\tau_2)}\qquad(n\ge n_0+1;m\ge0),
\end{align}
with $x_{n_0}\ne0$ is given by
\begin{equation}\label{xn-mm}
\begin{split}    
    x_n &= x_{n_0} \frac{\Gamma\lpa{n+m\tau_1+\tau_2+1}
    \Gamma(n_0+\tau_2+1)}{\Gamma\lpa{n+\tau_2+1}
    \Gamma(n_0+m\tau_1+\tau_2+1)}\\
    &\qquad +\frac{\Gamma\lpa{n+m\tau_1+\tau_2+1}}
    {\alpha\Gamma\lpa{n+\tau_2+1}}
    \sum_{n_0< k\le n}\frac{y_k\Gamma\lpa{k+\tau_2}}
    {\Gamma\lpa{k+m\tau_1+\tau_2+1}}.
\end{split}
\end{equation}
\end{lmm}
Starting with the recurrence \eqref{Qnm-fm} and the mean, we can
derive asymptotic approximations to $\overbar{P}_{n,m}$ successively
by induction for $m\ge2$, and then conclude the limit laws by the
method of moments. Unlike normal limit laws, there is no need to
center the random variables, which makes the calculations simpler;
however, the expressions for the limiting moments are generally more
involved (than those in the normal cases).

\subsection{EGF and PDE}

The recurrence \eqref{Pn-nn} (for $n\ge1$ with with $P_0(v)=c_0+c_1(v-1)$) leads to the PDE satisfied by the 
EGF of $P_n$
\[
    (1-\alpha z)F_z' -(\beta-\beta'(1-v))(1-v)F_v' 
	- (\alpha+\gamma-\gamma'(1-v))F = 0,
\]
where $F(z,v) := \sum_{n\ge0}\frac{P_n(v)}{n!}\, z^n$. The solution
can be derived by the standard procedure described in
Section~\ref{ss-pde}.

\begin{prop} Assume $\alpha>0$ and $\beta\ne0$. 
The EGF of $P_n$ (satisfying \eqref{Pn-nn}) is given as follows.  
\begin{itemize}
    \item If $\beta'=0$, then 
    \begin{align} \label{poisson-egf}
    	F(z,v) = (1-\alpha z)^{-\frac{\alpha+\gamma}{\alpha}}
    	e^{-\frac{\gamma'}
    	{\beta}(1-v)(1-(1-\alpha z)^{\frac\beta\alpha})}
    	\left(c_0-c_1(1-v)(1-\alpha z)^{\frac\beta\alpha} \right).
    \end{align}
    \item If $\beta'\ne0$, then 
    \begin{align} \label{nn-egf}
        F(z,v) = \frac{
    	c_0(\beta-\beta'(1-v))+(c_0\beta'-c_1\beta)(1-v)
    	(1-\alpha z)^{\frac{\beta}{\alpha}}}
    	{\beta(1-\alpha z)^{\frac{\alpha+\gamma}{\alpha}}
    	\Lpa{\frac{\beta-\beta'(1-v)
    	+\beta'(1-v)(1-\alpha z)^{\frac{\beta}{\alpha}}}
    	{\beta}}^{1-\frac{\gamma'}{\beta'}}}.
    \end{align}
\end{itemize}    
\end{prop}
Note that \eqref{poisson-egf} also follows from \eqref{nn-egf} by
taking the limit as $\beta'\to0$. Also if $\beta'=0$, then $\beta>0$
because otherwise the coefficients are not all nonnegative. By
varying the seven parameters, the simple solution \eqref{nn-egf} is
capable of generating many different non-normal limit laws, as we 
will examine in the next two sections but instead by an elementary 
approach.

\subsection{Discrete limit laws}

We consider in this subsection the case when the limit law is 
discrete, which arises mostly when $\beta>0$, beginning with the 
following asymptotic transfer.  

\begin{lmm} Assume that $x_n$ satisfies \eqref{xn-yn-nn} with 
$m\ge1$ and $\tau_1<0$. Then 
\begin{align}\label{xn-yn-at1}
    y_n \sim K \quad\text{implies that}\quad
    x_n \sim \frac{K}{m\beta}. 
\end{align}
\end{lmm}
\begin{proof}
By \eqref{xn-mm} using the asymptotic approximation  
\eqref{gamma-ratio} to the ratio of Gamma functions.
\end{proof}

Recall that $\tau_3=-\frac{\gamma'}{\beta'}$; see \eqref{tau}.  
\begin{prop} \label{prop-mfm} Assume $\tau_1<0$ (or $\beta>0$). Then 
the $m$-th factorial moment of $X_n$ satisfies
\begin{align}\label{Qnm-bdd}
	\mathbb{E}(X_n^{\underline{m}})
	\sim K_m := \begin{cases} \displaystyle
        \llpa{\frac{\gamma'}{\beta}}^m, 
		& \text{if }\beta'=0\\ \displaystyle
        \frac{\Gamma(m+\tau_3)}{\Gamma(\tau_3)}
        \left(-\frac{\beta'}{\beta}\right)^m, &
        \text{if }\beta'<0\\
		\displaystyle
        \frac{\ell!}{(\ell-m)!}
        \left(\frac{\beta'}{\beta}\right)^m, &
        \text{if }\beta'>0, \gamma'=\ell \beta',
    \end{cases}
\end{align}
for $m\ge0$, where $x^{\underline{m}} := \prod_{0\le j<m}(x-j)$.
\end{prop}
\begin{proof}
By the recurrence \eqref{Qnm-fm}, the asymptotic transfer 
\eqref{xn-yn-at1} and induction. 
\end{proof}
By Corollary~\ref{cor:beta-gamma}, since $\beta>0$, we have 
$\gamma'>0$ in all cases of $\beta'$.

\begin{thm}[$\beta>0\Longrightarrow$ discrete limit laws]
\label{thm-dll} Let $P_n(v)$ be defined by the recurrence
\eqref{Pn-nn}. Assume that (i) $[v^k]P_n(v)\ge0$ for $k,n\ge0$, (ii)
$P_n(1)>0$ for $n\ge0$, and (iii) $\beta>0$. Define $X_n$ by
$\mathbb{E}(v^{X_n}) := \frac{P_n(v)}{P_n(1)}$. Then
\begin{itemize}
	\item if $\beta'=0$, then $X_n$ follows asymptotically a Poisson 
	distribution with parameter $\frac{\gamma'}{\beta}$; 
	\item if $\beta'<0$, then $X_n$ follows asymptotically a negative 
	binomial distribution with parameters $\tau_3$ 
    and $-\frac{\beta'}{\beta-\beta'}$; 
	\item if $\beta'>0$, $\beta'<\beta$ and $\gamma'=\ell \beta'$ 
	for $\ell=1,2,\dots$, then $X_n$ is the sum of $\ell$ independent 
    and identically distributed Bernoulli random variables with 
    parameter $\frac{\beta'}{\beta}$ (or binomial with parameters 
    $\ell$ and $\frac{\beta'}{\beta}$). 
\end{itemize}
\end{thm}

\begin{proof}
If $\beta'=0$, then by Proposition~\ref{prop-mfm}, we see that the
probability generating function of the limit law equals
$e^{\frac{\gamma'}{\beta}(v-1)}$, which is nothing but that of a
Poisson random variable with mean $\frac{\gamma'}\beta$.

Now if $\beta'<0$, then $\tau_3>0$ (since $\gamma'>0$), and the 
variance is asymptotic to $\frac{(\beta-\beta')\gamma'}{\beta^2}$ 
in this case. By \eqref{Qnm-bdd}, we deduce that the probability 
generating function of the limit law equals
\[
    \bigl(1+\tfrac{\beta'}{\beta}(v-1)
    \bigr)^{-\tau_3},
\]
so we get a negative binomial with parameters $\tau_3$ and 
$-\frac{\beta'}{\beta-\beta'}\in(0,1)$. 

Finally, if $\beta'>0, \beta\ne\beta'$ and $\gamma'=\ell \beta'$, 
then we obtain the probability generating function
$\lpa{1+\frac{\beta'}{\beta}(v-1)}^\ell$, which is the sum of $\ell$
Bernoulli random variables with mean $\frac{\beta'}{\beta}$.
\end{proof}

\subsection{Continuous limit laws}

The case when $\tau_1>0$ ($\beta<0$) is phenomenally more interesting
as the underlying random variables have generally a wider range of
variations. We may without loss of generality assume that $\beta'<0$
because otherwise the coefficients $[v^k]P_n(v)$ are not all
nonnegative. From Lemma~\ref{lmm-tau3}, we see that $\tau_3>-1$ (the
equality being already covered by Theorem~\ref{thm-dll}). We derive
first the asymptotics of the factorial moments.
\begin{prop} \label{prop-Km}
Assume $\tau_1>0$ ($\beta<0$). Then the $m$th moment of $X_n$ is 
asymptotic to 
\begin{align} \label{mm-fm}
    \mathbb{E}\left(\frac{X_n}
    {\frac{\beta'}{\beta}\,n^{\tau_1}}\right)^m
    \sim \mathbb{E}\left(\frac{X_n}
    {\frac{\beta'}{\beta}\,n^{\tau_1}}\right)^{\underline{m}}
    \sim K_m \qquad(m\ge0),
\end{align}
where 
\begin{align}\label{Km}
    K_m =  \frac{\Gamma\lpa{m+\tau_3}
    \Gamma(\tau_2+1)}{\Gamma\lpa{\tau_3+1}
    \Gamma\lpa{m\tau_1+\tau_2+1}}
    \left(\frac{c_1\beta}{c_0\beta'}\,m
    +\tau_3\right)\qquad(m\ge0).
\end{align}
\end{prop}
\begin{proof}
We prove the second estimate of \eqref{mm-fm} by induction. Assume
that the $m$th factorial moment $\overbar{P}_{n,m}$ (see 
\eqref{Qnm-fm}) satisfies 
\[
    \overbar{P}_{n,m}
    \sim K_m \left(\frac{\beta'}{\beta}\,n^{\tau_1}\right)^m
    \qquad(m\ge0),
\]
where $K_0=1$ and, by Corollary~\ref{cor-tau},
\[
    K_1 = 
    \left(\frac{c_1\beta}{c_0\beta'}+\tau_3\right)
    \frac{\Gamma\lpa{\tau_2+1}}
    {\Gamma\lpa{\tau_1+\tau_2+1}}.
\]
So we assume now $m\ge2$. Since $\tau_1>0$, we can take $n_0=0$ in 
\eqref{xn-mm} with $x_0=0$ (using $\overbar{P}_{0,m}=0$ for 
$m\ge2$), and have 
\begin{align*}
    \overbar{P}_{n,m} &= \frac{m(\gamma'-(m-1)\beta')}
    {\alpha}\cdot\frac{\Gamma\lpa{n+m\tau_1+\tau_2+1}}
    {\Gamma\lpa{n+\tau_2+1}}
    \sum_{0\le k< n}\frac{\Gamma\lpa{k+\tau_2+1}
    \overbar{P}_{k,m-1}}
    {\Gamma\lpa{k+m\tau_1+\tau_2+2}},
\end{align*}
so that by induction for $m\ge2$
\begin{align*}
    K_m\left(\frac{\beta'}{\beta}\right)^m 
    &= \frac{m!(-\beta')^{m}\Gamma(m+\tau_3)}
    {\gamma'\alpha^{m-1}\Gamma(\tau_3)} \\
    &\qquad \times\sum_{0\le k_1<\cdots<k_{m-1}<\infty}
    \frac{\Gamma\lpa{k_1+\tau_2+1}\overbar{P}_{k_1,1}
    \prod_{2\le j<m}\Gamma\lpa{k_j+j\tau_1+\tau_2+1}}
    {\prod_{2\le j\le m}\Gamma\lpa{k_{j-1}+j\tau_1+\tau_2+2}}\\
    &= \frac{m!(-\beta')^{m}\Gamma(m+\tau_3)}
    {\beta\alpha^{m-1}\Gamma(\tau_3)}\,S_m^{[1]}\\
    &\qquad + \frac{m!(-\beta')^{m}
	\Gamma(m+\tau_3)\Gamma(\tau_2+1)}
    {\gamma'\alpha^{m-1}\Gamma(\tau_3)
	\Gamma(\tau_1+\tau_2+1)}\left(\frac{c_1}{c_0}
	-\frac{\gamma'}{\beta}\right)S_m^{[2]},
\end{align*}
where the ratio $\frac{\Gamma(m+\tau_3)}{\Gamma(\tau_3)}$ is 
interpreted as zero when $\tau_3=0$, and, by \eqref{Qn1},
\begin{align*}
    S_m^{[1]} &= \sum_{0\le k_1<\cdots<k_{m-1}<\infty}
    \frac{\Gamma\lpa{k_1+\tau_2+1}}
	{\Gamma\lpa{k_1+2\tau_1+\tau_2+2}}
    \prod_{2\le j<m}\frac{\Gamma\lpa{k_j+j\tau_1+\tau_2+1}}
    {\Gamma\lpa{k_j+(j+1)\tau_1+\tau_2+2}},\\
    S_m^{[2]} &= \sum_{0\le k_1<\cdots<k_{m-1}<\infty}
    \prod_{1\le j<m}\frac{\Gamma\lpa{k_j+j\tau_1+\tau_2+1}}
    {\Gamma\lpa{k_j+(j+1)\tau_1+\tau_2+2}}.
\end{align*}
By induction, we prove the following identities 
\begin{align*}
    S_m^{[1]} &= \frac{\Gamma\lpa{\tau_2+1}}
    {m(m-2)!\tau_1^{m-1}\Gamma\lpa{m\tau_1+\tau_2+1}}\\
    S_m^{[2]} &= \frac{\Gamma\lpa{\tau_1+\tau_2+1}}
    {(m-1)!\tau_1^{m-1}\Gamma\lpa{m\tau_1+\tau_2+1}}.
\end{align*}
Consider first $S_m^{[1]}$. We have
\begin{align*}
	S_{m+1}^{[1]} &= \sum_{0\le k_1<\cdots <k_{m-1}<\infty}
    \frac{\Gamma\lpa{k_1+\tau_2+1}}
	{\Gamma\lpa{k_1+2\tau_1+\tau_2+2}}\left(
    \prod_{2\le j<m}\frac{\Gamma\lpa{k_j+j\tau_1+\tau_2+1}}
    {\Gamma\lpa{k_j+(j+1)\tau_1+\tau_2+2}}\right)
    \cdot\Sigma_m,
\end{align*}
where
\[    
    \Sigma_m := \sum_{k_{m-1}<k_m<\infty}
	\frac{\Gamma\lpa{k_m+m\tau_1+\tau_2+1}}
	{\Gamma\lpa{k_m+(m+1)\tau_1+\tau_2+2}}
    =\frac{\Gamma(k_{m-1}+m\tau_1+\tau_2+2)}
    {\tau_1\Gamma(k_{m-1}+(m+1)\tau_1+\tau_2+2)}.
\]
It follows that    
\begin{align*}
    S_{m+1}^{[1]}
    &= \sum_{0\le k_1<\cdots <k_{m-1}<\infty}
    \frac{\Gamma\lpa{k_1+\tau_2+1}}
	{\Gamma\lpa{k_1+2\tau_1+\tau_2+2}}\left(
    \prod_{2\le j\le m-2}\frac{\Gamma\lpa{k_j+j\tau_1+\tau_2+1}}
    {\Gamma\lpa{k_j+(j+1)\tau_1+\tau_2+2}}\right)\\
    &\qquad \times 
    \frac{\Gamma\lpa{k_{m-1}+(m-1)\tau_1+\tau_2+1}}
    {\tau_1\Gamma(k_{m-1}+(m+1)\tau_1)+\tau_2+2},
\end{align*}
which is a summation of a similar type. By iterating the same 
simplification, we see that 
\begin{align*}
    S_{m+1}^{[1]} &= \frac1{(m-1)!\tau_1^{m-1}}
	\sum_{k_1\ge0}\frac{\Gamma\lpa{k_1+\tau_2+1}}
	{\Gamma\lpa{k_1+(m+1)\tau_1+\tau_2+2}}\\
	&= \frac{\Gamma(\tau_2+1)}{(m+1)(m-1)!\tau_1^m
	\Gamma((m+1)\tau_1+\tau_2+1)}.
\end{align*}
The proof of $S_m^{[2]}$ is similar. This proves the second estimate 
of \eqref{mm-fm}. Finally, since (the curly braces denoting the 
Stirling numbers of the second kind)
\[
    \mathbb{E}{X_n^m}
    = \sum_{0\le j\le m}\Stirling{m}{j}
    \mathbb{E}\lpa{X_n^{\underline{j}}}
    \sim \mathbb{E}\lpa{X_n^{\underline{m}}},
\]
the first estimate of \eqref{mm-fm} then follows from the second one.
This proves the Proposition. 
\end{proof}

Alternatively, the generating function \eqref{nn-egf} provides at 
least two different proofs of Proposition~\ref{prop-Km}: either by 
computing the asymptotics of the $m$th factorial moment 
$m![z^nt^m]F(z,1+t)$ for each $m$ or by working on the characteristic 
function $[z^n]F\lpa{z,e^{i\theta}}$, details being omitted here. 

Once \eqref{Km} is available, we can specify the limit law according 
to the given values of the parameters. Indeed, the form \eqref{Km} 
leads generally to the mixture of two distributions of generalized 
Mittag-Leffler type.  

Recall that the Mittag-Leffler function represents one of the
extensions of $e^s$ as well as a good bridge between $e^s$ and
$\frac1{1-s}$:
\[
    E_{p,q}(s) 
	:= \sum_{j\ge0}\frac{s^j}{\Gamma\lpa{p j+q}}
	\qquad(p>0, q\in\mathbb{C}).
\]
[The extension from $q=1$ in Mittag-Leffler's original definition was
due to A. Wiman.] The Mittag-Leffler distribution can be defined
either with $E_{p,q}$ as the distribution function (properly
parametrized) or with $E_{p,q}$ as the moment generating function
(properly normalized). We use the latter, namely, $Y$ follows a
Mittag-Leffler distribution if $\mathbb{E}\lpa{e^{Ys}} = \Gamma(q)
E_{p,q}(s)$.

\begin{defi} A random variable $Y$ is said to follow a generalized 
Mittag-Leffler (GML) distribution, written conveniently as $Y\sim 
\text{GML}(p,q,r)$, if $\mathbb{E}(e^{Ys}) =E_{p,q,r}(s)$, where 
\begin{align}\label{E-pqr-s}
   E_{p,q,r}(s) 
   := \frac{\Gamma(q)}{\Gamma(r)}
   \sum_{j\ge0}\frac{\Gamma(j+r)s^j}{j!\Gamma\lpa{p j+q}}
\qquad(p,r\ge0, q\in\mathbb{C})
\end{align} 
represents the (normalized) three-parameter Mittag-Leffler function 
(a special case of the Fox-Wright function and also known as the 
Prabhakar function; see \cite{Gorenflo2014}).
\end{defi}
A few special cases include
\begin{itemize}
    \item $r=0$: $Y$ is degenerate;
    \item $p=0$, $r>0$: $Y$ is Gamma distributed;
    \item $r=1$, $p,q>0$: $Y$ is a Mittag-Leffer distribution;
    \item $p=1$: $\text{GML}(p,q,r)\sim \text{Beta}(r,q-r)$.
\end{itemize}


\begin{thm}\label{thm-nncll} If $\tau_1\le 1$, $\tau_2\ge\tau_3\ge0$ 
and $0\le\frac{c_1\beta}{c_0\beta'}\le1$, then the limit law of 
$X_n/\lpa{\frac{\beta'}{\beta}\,n^{\tau_1}}$ is a mixture of two 
generalized Mittag-Leffler distributions:
\begin{align}\label{2gmls}
    \frac{X_n}{\frac{\beta'}{\beta}\,n^{\tau_1}}
    \to \frac{c_1\beta}{c_0\beta'}\,
    \mathrm{GML}\lpa{\tau_1,\tau_2+1,\tau_3+1}
    +\left(1-\frac{c_1\beta}{c_0\beta'}\right)
    \mathrm{GML}\lpa{\tau_1,\tau_2+1,\tau_3}.
\end{align}
When $\tau_1=1$, this leads to a Beta mixture; when $\tau_1<1$, the 
limit law has the density 
\begin{align}\label{nn-fx}
    f(x) := \frac{\Gamma(\tau_2+1) }{\Gamma(\tau_3+1)}
	\,x^{\tau_3-1}
    \sum_{\ell\ge0}\frac{\tau_3
    -\frac{c_1\beta}{c_0\beta'}(\ell+\tau_3)}
	{\ell!\Gamma(1-(\ell+\tau_3)\tau_1+\tau_2)}\,(-x)^\ell,
\end{align}
where $\frac1{\Gamma(-s)}$ is interpreted as zero if $s=0,1,\dots$.
\end{thm}
Note specially that \eqref{nn-fx} is independent of the condition 
$\frac{c_1\beta}{c_0\beta'}\le 1$. 

In terms of the Wright generalized Bessel function \cite{Wright1940}
\[
    W_{p,q}(z) := \sum_{\ell\ge0}
    \frac{z^\ell}{\ell!\Gamma(p\ell+q)}
    \qquad(p>-1;p\in\mathbb{C}),
\]
we have 
\begin{align*}
    f(x) &= \left(1-\frac{c_1\beta}{c_0\beta'}\right)
    \frac{\Gamma(\tau_2+1) }{\Gamma(\tau_3)}
	\,x^{\tau_3-1}W_{-\tau_1,1+\tau_2-\tau_1\tau_3}(-x)\\
    &\qquad +\frac{c_1\beta}{c_0\beta'}\cdot
    \frac{\Gamma(\tau_2+1) }{\Gamma(\tau_3+1)}\,
    x^{\tau_3}W_{-\tau_1,1+\tau_2-\tau_1-\tau_1\tau_3}(-x).
\end{align*}
If the limit law exists and is not degenerate then $\tau_3=0$ iff  
$c_1>0$. 
\begin{proof}
Decomposing \eqref{Km} into two parts
\begin{align*}
    K_m &= \frac{c_1\beta}{c_0\beta'}\cdot 
    \frac{\Gamma(m+\tau_3+1)\Gamma(\tau_2+1)}
    {\Gamma\lpa{\tau_3+1}
    \Gamma\lpa{m\tau_1+\tau_2+1}}
    + \left(1-\frac{c_1\beta}{c_0\beta'}\right)
    \frac{\Gamma(m+\tau_3)\Gamma(\tau_2+1)}
    {\Gamma\lpa{\tau_3}
    \Gamma\lpa{m\tau_1+\tau_2+1}},
\end{align*}
where the second term is interpreted as zero if $\tau_3=0$. This
decomposition shows that the limit law of $X_n$ is the mixture of
two distributions whose moment sequences are of the form (if
$\frac{c_1\beta}{c_0\beta'}\le 1$)
\begin{align}\label{g-beta}
    \frac{\Gamma(m+r)\Gamma(q)}
    {\Gamma(r)\Gamma(m\tau_1+q)},
\end{align} 
which is GML$(\tau_1,q,r)$ distributed. Thus \eqref{2gmls} follows. 
This moment sequence determines uniquely the distribution because the
corresponding moment generating function is analytic at $s=0$.

On the other hand, observe that the $m$th moment of a Beta$(r,q)$ 
distribution is given by
\[
    \frac{\Gamma(m+r)\Gamma(r+q)}
    {\Gamma(r)\Gamma(m+r+q)}\qquad(m\ge0);
\]
thus in the special case when $\tau_1=1$, \eqref{2gmls} leads to 
the mixture of two beta distributions:
\begin{align}\label{2gmls-b}
    \frac{X_n}{\frac{\beta'}{\beta}\,n}
    \to \frac{c_1\beta}{c_0\beta'}\,
    \mathrm{Beta}\lpa{1+\tau_3,\tau_2-\tau_3}
    +\left(1-\frac{c_1\beta}{c_0\beta'}\right)
    \mathrm{Beta}\lpa{\tau_3,1+\tau_2-\tau_3}.
\end{align}

Now regarding the moment sequence \eqref{g-beta} as the Mellin 
transform of some density function, say $f_0$, we then have, by 
inverse Mellin transform,
\begin{align*}
    f_0(x) &= \frac1{2\pi i}\int_{(c)}
    \frac{\Gamma(s+r)\Gamma(q)}
    {\Gamma(r)\Gamma(\tau_1s+q)}\,x^{-s-1} \dd s,
\end{align*}
when $r,q>0$. Neglecting the possible cancelation from the poles of 
the factor $\Gamma\lpa{s\tau_1+q}$, we compute the residues of the
integrand at $s=-r-\ell$, $\ell=0,1,\dots$, giving rise to the 
absolutely convergent series expression ($\tau_1<1$)
\begin{align*}
    f_0(x) 
    &=\frac{\Gamma(q) }{\Gamma(r)}\,x^{r-1}
    \sum_{\ell\ge0}\frac{(-x)^{\ell}}{\ell!
    \Gamma(-\tau_1(\ell+r)+q)},
\end{align*}
where $\frac1{\Gamma(-s)}$ is interpreted as zero if $s=0,1,\dots$.
This proves \eqref{nn-fx}. 
\end{proof}

The series representation \eqref{nn-fx} is in most cases useful for  
deriving more explicit expressions, but becomes less transparent if 
one is interested in large $x$ asymptotics. We can derive an integral 
representation by Euler's reflection formula for Gamma function as 
follows. We begin with 
\[
    \frac{1}{\Gamma(-\tau_1(\ell+r)+q)}
    =-\frac1\pi\,\Gamma(1+\tau_1(\ell+r)-q)
    \sin(\pi(\tau_1(\ell+r)-q));
\]
which, together with the integral representation of Gamma function, 
yields 
\begin{align*}
    f_0(x) &=\frac{\Gamma(q) }{\pi\Gamma(r)}\,x^{r-1}
    \sum_{\ell\ge0}\frac{(-x)^{\ell}}{\ell!}\,
    \Gamma(\tau_1(\ell+r)+1-q)\sin(\pi(\tau_1(\ell+r)-q))\\
    &= \frac{\Gamma(q) }{\pi\Gamma(r)}\,x^{r-1}
    \int_0^\infty e^{-t} t^{\tau_1 r-q} \left(\sum_{\ell\ge0}
    \frac{(-x t^{\tau_1})^{\ell}}{\ell!}
    \sin(\pi(\tau_1(\ell+r)-q))\right)\dd t\\
    &= \frac{\Gamma(q) }{\pi\Gamma(r)}\,x^{r-1}
    \int_0^\infty e^{-t-x t^{\tau_1}\cos(\tau_1\pi)} 
	t^{\tau_1 r-q} 
	\sin\lpa{(\tau_1 r-q)\pi - x t^{\tau_1}
	\cos(\tau_1\pi)}\dd t,
\end{align*}
whenever the integral is convergent, which is the case if 
$\tau_1r-q>-1$. By saddle-point method, one can then derive more 
precise asymptotic expansions for large $x$; we omit the details. 

\section{Applications III: non-normal discrete limit laws}

We now discuss concrete polynomials (satisfying the Eulerian
recurrence \eqref{Pn-nn}) whose coefficients follow asymptotically a
discrete limit law .

\subsection{Poisson limit laws: $\beta'=0$}
Examples of this category have the general pattern
\(
    P_n\in\ET{\frac{\alpha n+\gamma+\gamma'(v-1)}{e_n},
    \frac{\beta}{e_n}},
\)
with $\beta$ a positive constant, for some nonzero sequence $e_n$. 

\paragraph{$\beta=\gamma'=1\Longrightarrow$ Poisson$(1)$}
The generating polynomial of the number of permutations of $n$ 
elements with $k$ fixed points (or rencontres numbers 
\href{https://oeis.org/A008290}{A008290}) has the EGF 
$\frac{e^{(v-1)z}}{1-z}$, and satisfies the recurrence
$P_n\in\ET{n-1+v,1;1}$.

By Theorem~\ref{thm-dll}, the coefficients converge to Poisson$(1)$.
This and a weighted version, together with its reciprocal are listed
below; they all follow asymptotically the same Poisson distribution;
see also \cite[p.\ 117]{Diaconis1988}.

\vspace*{-.3cm}
\begin{center}
\begin{tabular}{lclll}\\
	\multicolumn{1}{c}{OEIS} &
	\multicolumn{1}{c}{$e_n$} & 
	\multicolumn{1}{c}{Type} &
	\multicolumn{1}{c}{EGF} &
	\multicolumn{1}{c}{Notes} \\ \hline	
	\href{https://oeis.org/A008290}{A008290} 
    & $1$ & $\ET{n-1+v,1;1}$ & $\frac{e^{(v-1)z}}{1-z}$ 
	& Rencontres \#s\\ 
	\href{https://oeis.org/A180188}{A180188} 
    & $\frac{n}{n+1}$ & $\ET{n-1+v,1;1}$
	& $\frac{1-(1-v)z(1-z)}{(1-z)^2}\,e^{(v-1)z}$	
	& $\begin{array}{l}
	\text{\!\!Circular successions}\\
	\text{\!\!(Multiple of \href{https://oeis.org/A008290}{A008290})}
	\end{array}$ \\ 
	\href{https://oeis.org/A098825}{A098825} & $1$ &
    $\ET{(n-1)v^2+1,v^2;1}$ & $\frac{e^{(1-v)z}}{1-vz}$
	& Reciprocal of \href{https://oeis.org/A008290}{A008290} 
    \\ \hline
\end{tabular}	    
\end{center}
\justifying

\medskip

Similarly, the number of $r$-successions ($\pi(i)=i+r$) in 
permutations has the generating polynomials satisfying (see 
\cite{Liese2010,Rakotondrajao2007}) $P_n\in\GT{r}{n-1+v,1;r!}$.

\begin{center}
\begin{tabular}{c|c|c}	\hline
$r=1$ \href{https://oeis.org/A123513}{A123513} & $r=2$ \href{https://oeis.org/A264027}{A264027} & $r=3$ \href{https://oeis.org/A264028}{A264028} \\ \hline
\end{tabular}	
\end{center} 
By considering $R_n(v) := P_{n+r}(v)$, we then get
$P_n\in\ET{n+r-1+v,1;r!}$. The corresponding EGF is
$\frac{r!e^{(v-1)z}} {(1-z)^{r+1}}$. All lead to Poisson$(1)$ limit 
law. Note that we also have
\begin{center}
\begin{tabular}{cccc}\hline
	\href{https://oeis.org/A010027}{A010027} 
    & $\ET{(n-1)v^2+1+v,v^2;1}$ 
	& $\frac{e^{(1-v)z}}{(1-vz)^2}$ 
	& Reciprocal of \href{https://oeis.org/A123513}{A123513} 
    \\ \hline
\end{tabular}	
\end{center}

Finally, the sequence
\href{https://oeis.org/A193639}{A193639} can be defined recursively 
by 
\(
    P_n\in\ET{2n(2n-2+v),2n;1}.
\)
Normalize this sequence by considering $R_n(v) :=
\frac{P_n(v)}{2^nn!}$, which then satisfies $R_n\in\ET{2n-2+v,1;1}$.
This is identical to \href{https://oeis.org/A079267}{A079267}. The
same Poisson$(1)$ limit law holds for the distribution of the
coefficients.
\begin{center}
\begin{tabular}{llll}\hline
	\href{https://oeis.org/A079267}{A079267} 
    & $\ET{2n-2+v,1;1}$ &  
	$\frac{e^{(v-1)(1-\sqrt{1-2z})}}{\sqrt{1-2z}}$
	& Short-pair matchings\\
	\href{https://oeis.org/A193639}{A193639} 
    & $\ET{2n(2n-2+v),2n;1}$ & 
	& Consecutive rencontres \\ \hline
\end{tabular}	
\end{center}
Note that in all these cases, we can derive more precise asymptotic 
approximations to the distributions, either by the EGF using  
analytic means or by the explicit expression of the coefficients 
using elementary arguments. We leave this to the interested readers. 

\paragraph{$\beta=2, \gamma'=1\Longrightarrow$ Poisson$(\frac12)$}
\href{https://oeis.org/A055140}{A055140} enumerates the number of
matchings of $2n$ people with partners such that exactly $k$ couples
are left together; the generating polynomials satisfy
\(
    P_n\in\ET{2n-2+v,2;1}.
\)
By Theorem~\ref{thm-dll}, the distribution tends to
Poisson$(\frac12)$. This sequence shares a common property with
\href{https://oeis.org/A008290}{A008290}: $[v^n]P_n(v)=1$ but
$[v^{n-1}]P_n(v)=0$.

A sequence leading to the same limit Poisson$(\frac12)$ distribution
is \href{https://oeis.org/A155517}{A155517}, which is defined on
$P_n(v)= \text{\href{https://oeis.org/A055140}{A055140}}_n(v)$ by
$\tr{\frac12n}!2^{\tr{\frac12n}}P_{\cl{\frac12n}}(v)$.

\begin{center}
\begin{tabular}{llll}\hline
	\href{https://oeis.org/A055140}{A055140} 
    & $\ET{2n-2+v,2;1}$ &  
	$\frac{e^{(v-1)z}}{\sqrt{1-2z}}$  & Partner-matchings\\
    \href{https://oeis.org/A155517}{A155517} 
    & & & $\tr{\frac12n}!2^{\tr{\frac12n}}
    \text{\href{https://oeis.org/A055140}
    {A055140}}_{\cl{\frac12n}}(v)$ \\ \hline
\end{tabular}	
\end{center}

\subsection{Geometric and negative-binomial limit laws: $\beta'<0$}
The examples of this category now have the general pattern
\begin{align}\label{Pnv-ab}
    P_n\in\EET{\frac{\alpha n+\gamma+\gamma'(v-1)}{e_n}}
    {\frac{\beta+\beta'(v-1)}{e_n}},
\end{align}
with $\beta>0, \beta'<0$, $\tau_3$ a positive integer and 
$-\frac{\beta'}{\beta-\beta'}>0$. 

Consider \href{https://oeis.org/A158815}{A158815}, counting the 
number of nonnegative paths consisting of up-steps and down-steps of 
length $2n$ with $k$ low peaks (a low peak has its peak vertex at 
height 1). Then 
\(
    P_n\in\ET{\frac{4n-3+v}{n},\frac{3-v}{n};1},
\)
which follows from the OGF 
\[
    \frac{2}{\sqrt{1-4z}\lpa{3-\sqrt{1-4z}-
    v\lpa{1-\sqrt{1-4z}}}}. 
\]
By Theorem~\ref{thm-dll}, $\tau_3=1$ and 
$-\frac{\beta'}{\beta-\beta'} =\frac13$; thus we obtain the geometric 
limit law:
\[
    \mathbb{P}(X_n=k) \to 2 \cdot 3^{-k-1}
    \qquad(k=0,1,\dots). 
\]
The reciprocal polynomials $Q_n(v) := v^nP_n\lpa{\frac1v}$ satisfy
$Q_n\in\ET{(1+3v^2)n +v(1-3v),-v(1-3v);1}$.

Similarly, the sequence \href{https://oeis.org/A065600}{A065600},
counting the number of hills in Dyck paths, can be generated by
$P_n\in\ET{\frac{4n-4+2v}{n+1}, \frac{3-v}{n+1};1}$. Since
$\tau_3= 2$ and $-\frac{\beta'}{\beta-\beta'} =\frac13$, we obtain, 
by Theorem~\ref{thm-dll}, a negative binomial limit law with 
parameters $2$ and $\frac13$:
\[
    \mathbb{P}(X_n=k) \to 4(k+1)\cdot3^{-k-2} \qquad(k=0,1,\dots).
\]

Finally, the sequence \href{https://oeis.org/A202483}{A202483} defined by 
\[
    a_{n,k} := [z^n]\left(\frac{1-(1-9z)^{\frac13}}
    {4-(1-9z)^{\frac13}}\right)^k,
\]
satisfies the recurrence $P_n\in\ET{\frac{9n-5+2v}{n+1},
\frac{4-v}{n+1};1}$. We obtain a negative binomial limit law with
parameters $\tau_3 = 2$ and $-\frac{\beta'}{\beta-\beta'}=\frac14$:
\[
    \mathbb{P}(X_n=k) \to 9(k+1)\cdot4^{-k-2} \qquad(k=0,1,\dots).
\]
These examples are summarized as follows. 
\begin{center}
\begin{tabular}{llll}\hline
	\href{https://oeis.org/A158815}{A158815} 
    & $\ET{\frac{4n-3+v}{n},\frac{3-v}{n};1}$ &  
	\text{Geometric}$(\frac23)$  & Low peaks in paths\\
    \href{https://oeis.org/A065600}{A065600} 
    & $\ET{\frac{4n-4+2v}{n+1}, \frac{3-v}{n+1};1}$ &  
    \text{Negative-Binomial}$(2,\frac13)$ & Hills in Dyck paths \\ 
    \href{https://oeis.org/A202483}{A202483} 
    & $\ET{\frac{9n-5+2v}{n+1},\frac{4-v}{n+1};1}$ & 
    \text{Negative-Binomial}$(2,\frac14)$ 
    & $[z^n]\left(\frac{1-(1-9z)^{\frac13}}
    {4-(1-9z)^{\frac13}}\right)^k$\\
    \hline
\end{tabular}	
\end{center}

\subsection{A Bernoulli limit law}
All examples we examined so far with discrete limit laws have
$\beta>0$. We now consider a different example 
\href{https://oeis.org/A103451}{A103451} with $\beta<0$ and 
\[
    P_n(v) = 1+v^{n+1}\qquad(n\ge0).
\] 
The limit law is obviously Bernoulli$\lpa{\frac12}$. Such 
polynomials satisfy the recurrence 
\begin{align}\label{Pn-Bernoulli}
    P_n\in\EET{1}{-\frac{v}{n};1+v}.
\end{align}
We see that in this case $\beta<0$ but the limit law is discrete
(also following from \eqref{Km}). 

\section{Applications IV: non-normal continuous limit laws}
\label{sec-nncll}

Polynomials satisfying \eqref{Pn-nn} with $\beta<0$ whose
coefficients tends to some continuous limit law are examined in this
section. In all cases we consider, since the variance tends to
infinity and the limit law is not normal, we deduce that the roots of 
the polynomials are not all real.

\subsection{Beta limit laws and their mixtures 
($\frac{\beta}{\alpha}=-1$)}

A large number of polynomials whose coefficients converge to Beta 
limit laws have the same pattern 
\begin{align}\label{beta-rr}
    P_n \in\EET{\frac{\alpha n + pv+q}{e_n}}
    {-\frac{\alpha v}{e_n};h_0+h_1v},
\end{align}
where $h_0, h_1\ge0$ and $h_0+h_1>0$. By \eqref{nn-egf}, we see that 
the EGF of $P_n$ is given by 
\[
    F(z,v) = \frac{h_0(1-\alpha vz)+h_1v(1-\alpha z)}
    {(1-\alpha z)^{\frac q\alpha+1}
	(1-\alpha vz)^{\frac p\alpha+1}},
\]
which shows that the recurrence \eqref{beta-rr} is indeed simpler 
than most others treated in this paper. Thus the discussions of the 
examples in this category will be brief. 

Since we assume that $\alpha>0$, it can be checked that 
\[
    [v^n]P_n(v)\ge 0 \text{ for }n,k\ge0
	\quad\text{iff}\quad p,q\ge0,
\]
in contrast to the more general form \eqref{Pn-nn} for which general 
conditions for the nonnegativity of the coefficients remain less 
clear. 

The following beta limit law is a special case of 
Theorem~\ref{thm-nncll}. 
\begin{cor} \label{thm:beta} Assume that $P_n(v)$ satisfies the
recurrence \eqref{beta-rr}. If $p,q>0$, then the coefficients of
$P_n(v)$
follows asymptotically a mixture of two Beta distributions:
\begin{align}\label{beta-mixture}
    \frac{h_1}{h_0+h_1}\,
    \mathrm{Beta}\llpa{\frac p\alpha+1,
	\frac q\alpha}
    +\frac{h_0}{h_0+h_1}\,
    \mathrm{Beta}\llpa{\frac p\alpha,
	\frac q\alpha+1}.
\end{align}
\end{cor}
\begin{proof}
Since $\beta=\beta'=-\alpha$, we have $\tau_1=1$, $\tau_2
=\frac{p+q}\alpha$ and $\tau_3=\frac{q}{\alpha}$, so that 
\eqref{beta-mixture} follows from \eqref{2gmls-b}.
\end{proof}

In particular, the mean is asymptotically linear and the variance asymptotically quadratic with the leading constants given by 
\begin{align*}
    \frac{\mathbb{E}(X_n)}{n}
	&\sim K_1 = \frac{ph_0+(p+\alpha)h_1}{(h_0+h_1)
	(p+q+\alpha)}, \\
    \frac{\mathbb{V}(X_n)}{n^2}
	&\sim K_2-K_1^2 = \alpha\frac{p(q+\alpha)h_0^2
	+2(p+\alpha)(q+\alpha)h_0h_1
	+q(p+\alpha)h_1^2}{(h_0+h_1)^2
	(p+q+\alpha)^2(p+q+2\alpha)},
\end{align*}
respectively.

\subsubsection{Uniform (Beta$(1,1)$) limit laws} \label{sec:ull}
Uniform distribution is a special case of Beta distributions:
$\text{Beta}(1,1)$. A very simple example in OEIS with this
distribution is \href{https://oeis.org/A123110}{A123110} (shifted by
$1$), which can be generated by \eqref{beta-rr} with $P_n\in
\ET{\frac{n+1}{n},-\frac{v}{n};v}$. Then $P_n(v) = v+\cdots +
v^{n+1}$ for $n\ge0$, and one obviously has a Uniform$[0,1]$ limit
law for the coefficients with mean and variance asymptotic to $\frac
n2$ and $\frac{n^2}{12}$, respectively. This and other examples are
listed as follows.
\begin{center}
\begin{tabular}{llll}
	\multicolumn{1}{c}{OEIS} &
	\multicolumn{1}{c}{Type} &
	\multicolumn{1}{c}{Limit law} \\ \hline 
\href{https://oeis.org/A000012}{A000012} 
& $\ET{\frac{n+v}{n},-\frac{v}{n};1}$ & Uniform$[0,1]$ \\
\href{https://oeis.org/A123110}{A123110} 
& $\ET{\frac{n+1}{n},-\frac{v}{n};v}$  & Uniform$[0,1]$ \\
\href{https://oeis.org/A279891}{A279891} 
& $\ET{\frac{n+1+v}{n},-\frac{v}{n};2+2v}$ & Uniform$[0,1]$ \\ 
\hline
\end{tabular}    
\end{center}
Note that all roots of these polynomials lie on the unit circle. 
Also if we change the initial condition of 
\href{https://oeis.org/A123110}{A123110} to $P_0(v)=1$
(instead of $v$), then $P_n(v)\equiv 1$ for all $n\ge0$. This shows 
the high sensitivity of the limit law on initial conditions. 

\subsubsection{Arcsine (Beta$\lpa{\frac12,\frac12}$) law} Arcsin law
is another special case of Beta distribution:
$\text{Beta}(\frac12,\frac12)$. A classical example in this category
is Chung-Feller's arcsine law \cite{Chung1949}. First, the number
of simple random walks (up or down with the same probability) of
length $2n$ with $2k$ steps above zero is given by
$\binom{2k}{k}\binom{2n-2k}{n-k}$ (alternatively, paths of length
$2n$ with the last return to zero at $2k$ has the same distribution),
which is \href{https://oeis.org/A067804}{A067804}. Then, the
corresponding generating polynomials are of type
$P_n\in\ET{\frac{4n-2+2v}{n}, -\frac{4v}{n};1}$. We obtain, by
Corollary~\ref{thm:beta}, the arcsine limit law for the coefficients.

Another essentially identical sequence leading to the same law is 
\href{https://oeis.org/A059366}{A059366}. 

\vspace*{-.5cm}
\begin{center}
\begin{tabular}{llllll} \\
	\multicolumn{1}{c}{OEIS} &
	\multicolumn{1}{c}{Type} &
	\multicolumn{1}{c}{$[v^k]P_n(v)$} & 
	\multicolumn{1}{c}{Limit law} &
	\multicolumn{1}{c}{Limit density} \\ \hline 	
\href{https://oeis.org/A059366}{A059366} 
& $\ET{2n-1+v,-2v;1}$ & 
$\frac{n!}{2^n}\binom{2k}{k}\binom{2(n-k)}{n-k}$ & 
arcsine & 
$\frac1{\pi\sqrt{x(1-x)}}$ \\
\href{https://oeis.org/A067804}{A067804} 
& $\ET{\frac{4n-2+2v}{n},-\frac{4v}{n};1}$ & 
$\binom{2k}{k}\binom{2(n-k)}{n-k}$ & 
arcsine & 
$\frac1{\pi\sqrt{x(1-x)}}$ \\ \hline
\end{tabular}	
\end{center}

By the connection to Legendre polynomials, all roots of $P_n(v)$ lie 
on the unit circle; see also \cite{Hwang2015}. 

\subsubsection{\text{Beta}$(q,q)$ with $q>1$} 
Consider the expansion (\href{https://oeis.org/A120406}{A120406})
\[
    \frac{1-2(1+v)z-\!\sqrt{(1-4z)(1-4vz)}}
	{2(1-v)^2z^2}
	= \sum_{n\ge0}P_n(v)z^n.
\]
Then $P_n(v)\in\ET{\frac{4n+2+6v}{n+2},-\frac{4v}{n+2};1}$. We 
obtain a Beta$\lpa{\frac32,\frac32}$ (semi-elliptic) limit law for the
coefficients.

Another example is \href{https://oeis.org/A091441}{A091441}, which
counts the number of permutations of two types of objects so that
each cycle contains at least one object of each type. Shifting by one
(so as to start the recurrence from $n=1$) leads to the polynomial of
type $\ET{n+1+2v,-v;1}$. We then obtain the limit law Beta$(2,2)$ 
(parabolic) for the coefficients.

\begin{small}
\begin{center}
\begin{tabular}{lllll}\\
	\multicolumn{1}{c}{OEIS} &
	\multicolumn{1}{c}{Type} &
	\multicolumn{1}{c}{$[v^k]P_n(v)$} & 
	\multicolumn{1}{c}{Limit law} &
	\multicolumn{1}{c}{Limit density} \\ \hline 	 
\href{https://oeis.org/A120406}{A120406} &  
$\ET{\frac{4n+2+6v}{n+2},-\frac{4v}{n+2};1}$ & 
$\frac{2\binom{n}{k}^2\binom{2n+2}{n}}
{\binom{2n+2}{2k+1}}$ & Beta$\lpa{\frac32,\frac32}$ &
$\frac{8\sqrt{x(1-x)}}\pi$ \\
\href{https://oeis.org/A091441}{A091441} & $\ET{n+1+2v,-v;1}$ & 
$n!(k+1)(n+1-k)$ & Beta$(2,2)$ &
$\frac16x(1-x)$ \\ 
\href{https://oeis.org/A003991}{A003991} & 
$\ET{\frac{n+1+2v}{n},-\frac{v}{n};1}$ & 
$(k+1)(n+1-k)$ & Beta$(2,2)$ &
$\frac16x(1-x)$ \\  \hline
\end{tabular}	
\end{center}
\end{small}

\subsubsection{Beta$(p,q)$ with $p\ne q$}
A generic example is the negative hypergeometric distribution, first 
introduced by Condorcet in 1785 (see \cite[Ch.\ 6, Sec.\ 
2.2]{Johnson1992}) and defined by 
\[
    \mathbb{P}(X_n=k)= [v^k]P_n(v)
    = \frac{\binom{p+k-1}{k}
    \binom{q+n-k-1}{n-k}}{\binom{p+q+n-1}{n}}
    \qquad(n\ge0;p,q>0).
\]
Then $P_n(v)$ is of type $\ET{\frac{n+q-1+pv}{p+q+n-1},
-\frac{v}{p+q+n-1};1}$, and the limit law of $X_n$ is, by 
Corollary~\ref{thm:beta}, Beta$(p,q)$. See also \cite{Janardan1988} 
where this distribution arises in a ``social attraction model''.
For clarity, we separate the factor $e_n$ (see \eqref{beta-rr}) in 
the following table. 

\centering
\begin{longtable}{cclccc} \\
	\multicolumn{1}{c}{OEIS} &
	\multicolumn{1}{c}{$e_n$} &
	\multicolumn{1}{c}{Type} &
	\multicolumn{1}{c}{Limit law} &
	\multicolumn{1}{c}{Limit density} \\ \hline 	
\href{https://oeis.org/A162608}{A162608} & $1$ & $\ET{n+2v,-v;1}$  
& Beta$(2,1)$ & $2x$ \\
\href{https://oeis.org/A002260}{A002260} & $n$ & $\ET{n+2v,-v;1}$  
& Beta$(2,1)$ & $2x$ \\
\href{https://oeis.org/A051683}{A051683} & $\frac{n}{n+1}$ & $\ET{n+2v,-v;1}$  
& Beta$(2,1)$ & $2x$\\
\href{https://oeis.org/A002262}{A002262} & $n$ & $\ET{n+1+v,-v;v}$  
& Beta$(2,1)$ & $2x$ \\ \hline
\href{https://oeis.org/A138770}{A138770} & $1$ & $\ET{n+1+v,-v;2}$  
& Beta$(1,2)$ & $2(1-x)$\\ 
\href{https://oeis.org/A004736}{A004736} & $n$ & $\ET{n+1+v,-v;1}$  
& Beta$(1,2)$ & $2(1-x)$ \\
\href{https://oeis.org/A212012}{A212012} & $n$ & $\ET{n+1+v,-v;2}$  
& Beta$(1,2)$ & $2(1-x)$ \\ 
\href{https://oeis.org/A202363}{A202363} & $\frac{n}{n+2}$ & $\ET{n+1+v,-v;1}$  
& Beta$(1,2)$ & $2(1-x)$\\ \hline
\href{https://oeis.org/A122774}{A122774} & $1$ & $\ET{2n-1+2v,-2v;1}$ 
& Beta$(1,\frac12)$ & $\frac1{2\sqrt{1-x}}$\\
\href{https://oeis.org/A104633}{A104633} & $n$ & $\ET{n+2+2v,-v;1}$ &
Beta$(2,3)$ & $12x(1-x)^2$ \\
\href{https://oeis.org/A127779}{A127779} & $n$ & $\ET{n+1+3v,-v;1}$ &
Beta$(3,2)$ & $12x^2(1-x)$ \\
\href{https://oeis.org/A033820}{A033820} & $n+1$ & $\ET{4n-2+6v,-4v;1}$ & 
Beta$(\frac32,\frac12)$ 
& $\frac{2\sqrt{x}}{\pi\sqrt{1-x}}$ \\ \hline
\end{longtable}    
\justifying

\medskip

\noindent Here (\href{https://oeis.org/A127779}{A127779}, \href{https://oeis.org/A104633}{A104633}) are a reciprocal pair. In 
particular, \href{https://oeis.org/A033820}{A033820} is connected to the enumeration of paths avoiding 
the line $x=y$; see \cite{Gessel1992, Shur2003}.

\paragraph{More OEIS sequences with Beta$(2,1)$ limit law} Three 
simple sequences of polynomials are also Eulerian although they are 
not of the form \eqref{beta-rr}. We list them here for completeness. 
\begin{center}\small
\begin{tabular}{cll}
\multicolumn{1}{c}{OEIS} &
\multicolumn{1}{c}{$P_n(v)$} &
\multicolumn{1}{c}{Type} \\ \hline 
\href{https://oeis.org/A071797}{A071797}
& $\sum\limits_{1\le j\le 2n}(j+1)v^j$
& $\ET{\frac{2(2n^2-(1-2v+2v^2)n+v^2)}
    {2n(2n-1)},-\frac{2v(1+v)n-v^2}{2n(2n-1)};1}$ \\ 
\href{https://oeis.org/A074294}{A074294}
& $\sum\limits_{0\le j\le 2n+1}(j+1)v^j$	
& $\ET{\frac{2n^2+(1+2v+2v^2)n+v(1+2v)}{2n(2n+1)},
    -\frac{2v(1+v)n-v(1+2v)}{2n(2n+1)};1+2v}$ \\
\href{https://oeis.org/A293497}{A293497}
& $\sum\limits_{0\le j\le 2n}(j+1)v^j$
& $\ET{\frac{4n^2+2v(2+v)n+v(1+v)}{2n(2n-1)},
    -\frac{2v(1+v)n-v^2}{2n(2n-1)};v}$	\\ \hline
\end{tabular}	
\end{center}

Without a priori information on the exact forms of the polynomials,
we can still apply the method of moments (with more complicated
calculations) and get the limit law, although the corresponding PDEs
seem more difficult to solve. A simple reason these recurrences lead
to non-normal limit laws is that the dependence on $v$ in each of the
multiplicative factors is only at the lower order terms such as
$O(n^{-1})$ and smaller ones.

\subsubsection{Beta mixtures}
For simplicity, we abbreviate the Beta$(p,q)$ distribution by 
$B_{p,q}$ in the following table.

\centering
\begin{tabular}{ lllll } 
	\multicolumn{1}{c}{OEIS} &
	\multicolumn{1}{c}{$e_n$} &
	\multicolumn{1}{c}{Type} &
	\multicolumn{1}{c}{Limit law} &
	\multicolumn{1}{c}{Limit density} \\ \hline
\makecell{\href{https://oeis.org/A051162}{A051162}\\
\href{https://oeis.org/A134478}{A134478}} & $n$ 
& $\ET{n+1+v,-v;1+2v}$ 
& $\frac23B_{2,1}+\frac13B_{1,2}$ & $\frac23(1+x)$ \\
\href{https://oeis.org/A294317}{A294317} & $n$ 
& $\ET{n+1+v,-v;2+v}$ 
& $\frac13B_{2,1}+\frac23B_{1,2}$ & $\frac16(2-x)$\\
\href{https://oeis.org/A087401}{A087401} & $n$ 
& $\ET{n+2+v,-v;v+v^2}$ 
& $\frac12B_{3,1}+\frac12B_{2,2}$ & $\frac32x(2-x)$ \\
\href{https://oeis.org/A141418}{A141418} & $n$ 
& $\ET{n+1+2v,-v;1+v}$ 
& $\frac12B_{3,1}+\frac12B_{2,2}$ & $\frac32x(2-x)$ \\
\href{https://oeis.org/A193891}{A193891} & $n$ 
& $\ET{n+1+3v,-v;1+2v}$ 
& $\frac23B_{4,1}+\frac13B_{3,2}$ & $\frac43x^2(3-x)$ \\
\href{https://oeis.org/A193892}{A193892} & $n$ 
& $\ET{n+3+v,-v;2+v}$ 
& $\frac13B_{2,3}+\frac23B_{2,4}$ & $\frac43(1-x)^2(2+x)$ \\
\href{https://oeis.org/A193895}{A193895} & $n$ 
& $\ET{n+2+2v,-v;2+v}$ 
& $\frac13B_{3,2}+\frac23B_{2,3}$ & $4x(1-x)(2-x)$ \\
\href{https://oeis.org/A193896}{A193896} & $n$ 
& $\ET{n+2+2v,-v;1+2v}$ 
& $\frac23B_{3,2}+\frac13B_{2,3}$ & $4x(1-x^2)$ \\ \hline
\end{tabular}
\justifying
\medskip

\noindent
Note that (\href{https://oeis.org/A051162}{A051162},
\href{https://oeis.org/A294317}{A294317}), 
(\href{https://oeis.org/A193891}{A193891},
\href{https://oeis.org/A193892}{A193892}) and
(\href{https://oeis.org/A193895}{A193895},
\href{https://oeis.org/A193896}{A193896}) are reciprocal pairs.
See Figure~\ref{fig-beta} for the histograms of some polynomials
leading to Beta limit laws.

\begin{figure}[!ht]
\begin{center}
\includegraphics[height=3cm]{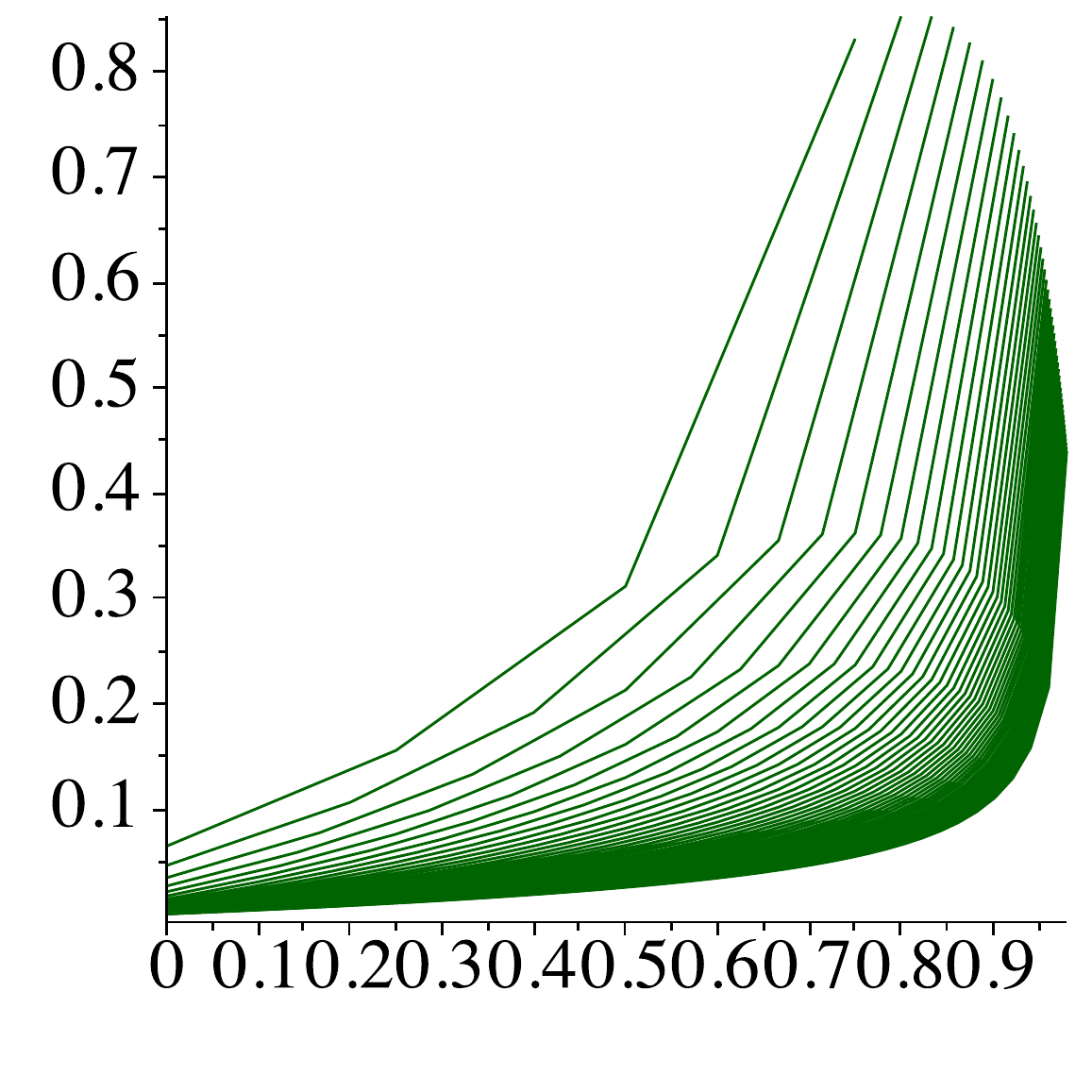}\,
\includegraphics[height=3cm]{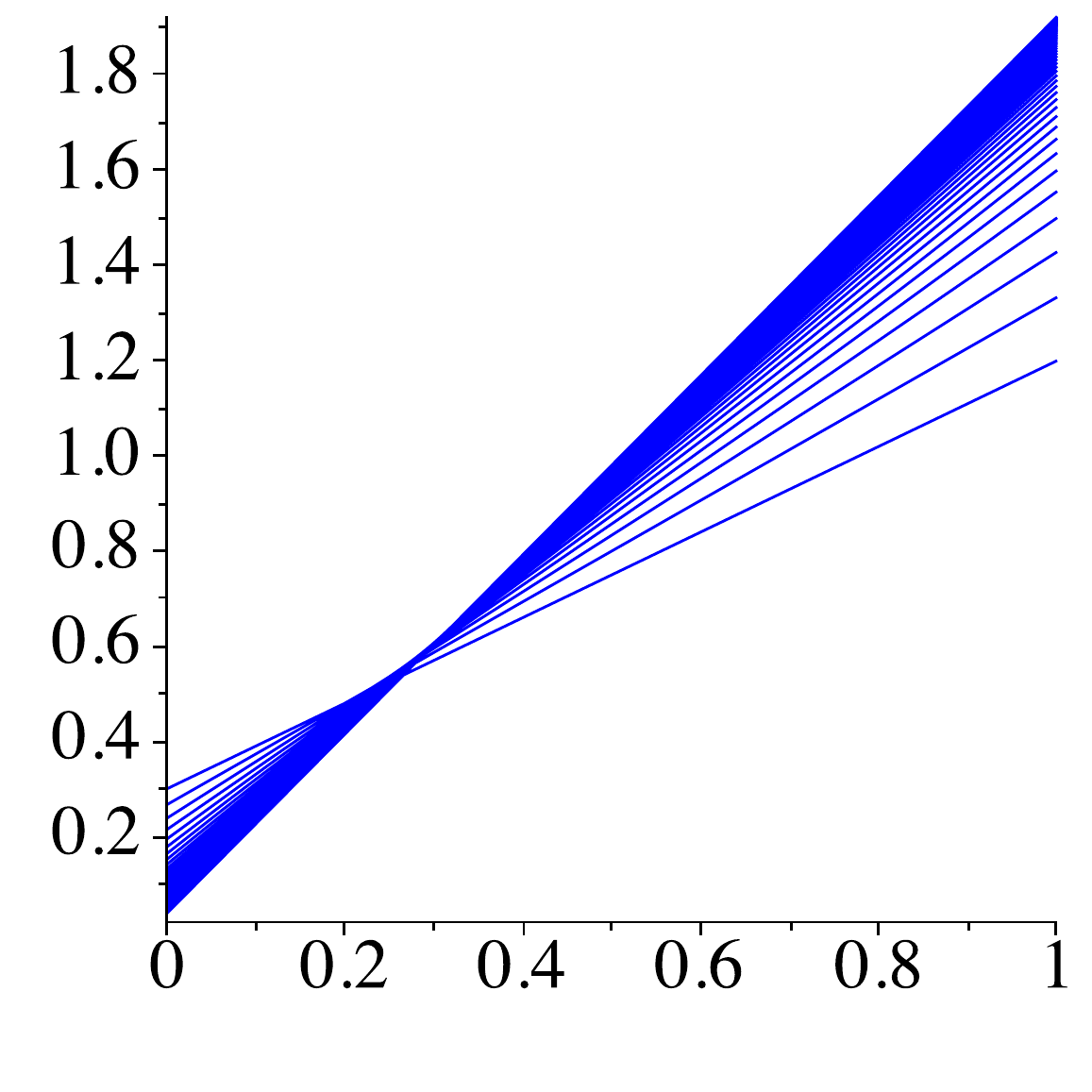}\,
\includegraphics[height=3cm]{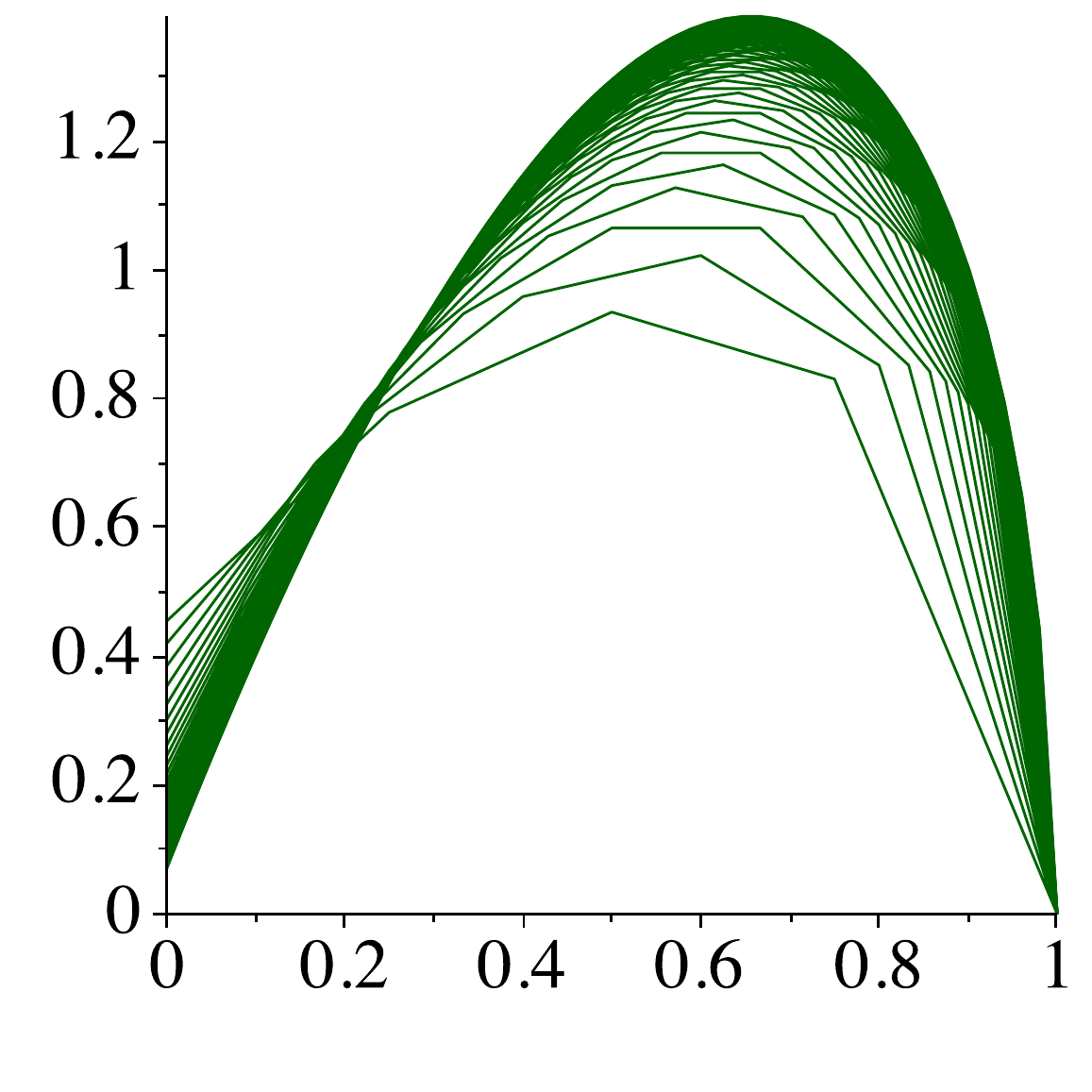}\,
\includegraphics[height=3cm]{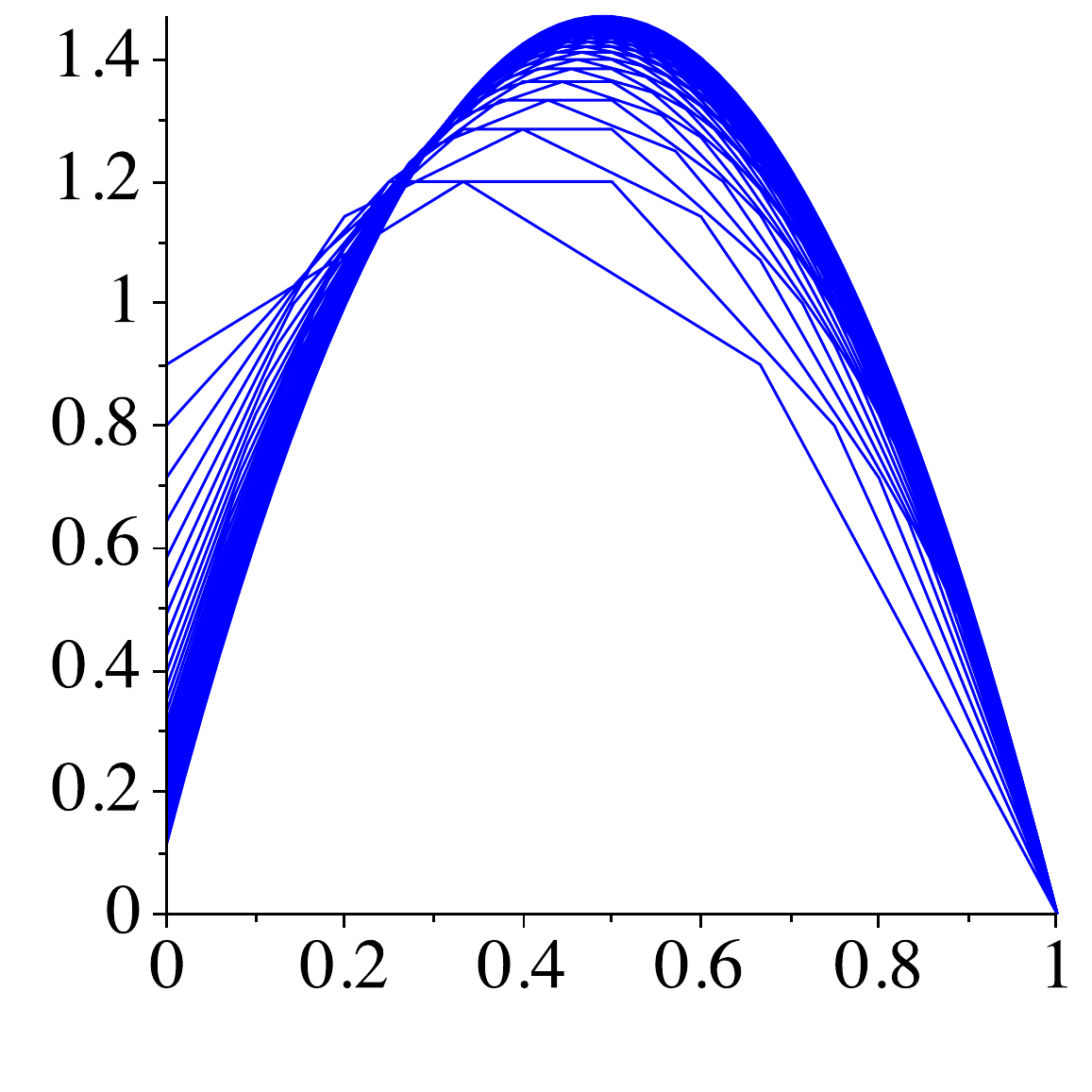}\,
\includegraphics[height=3cm]{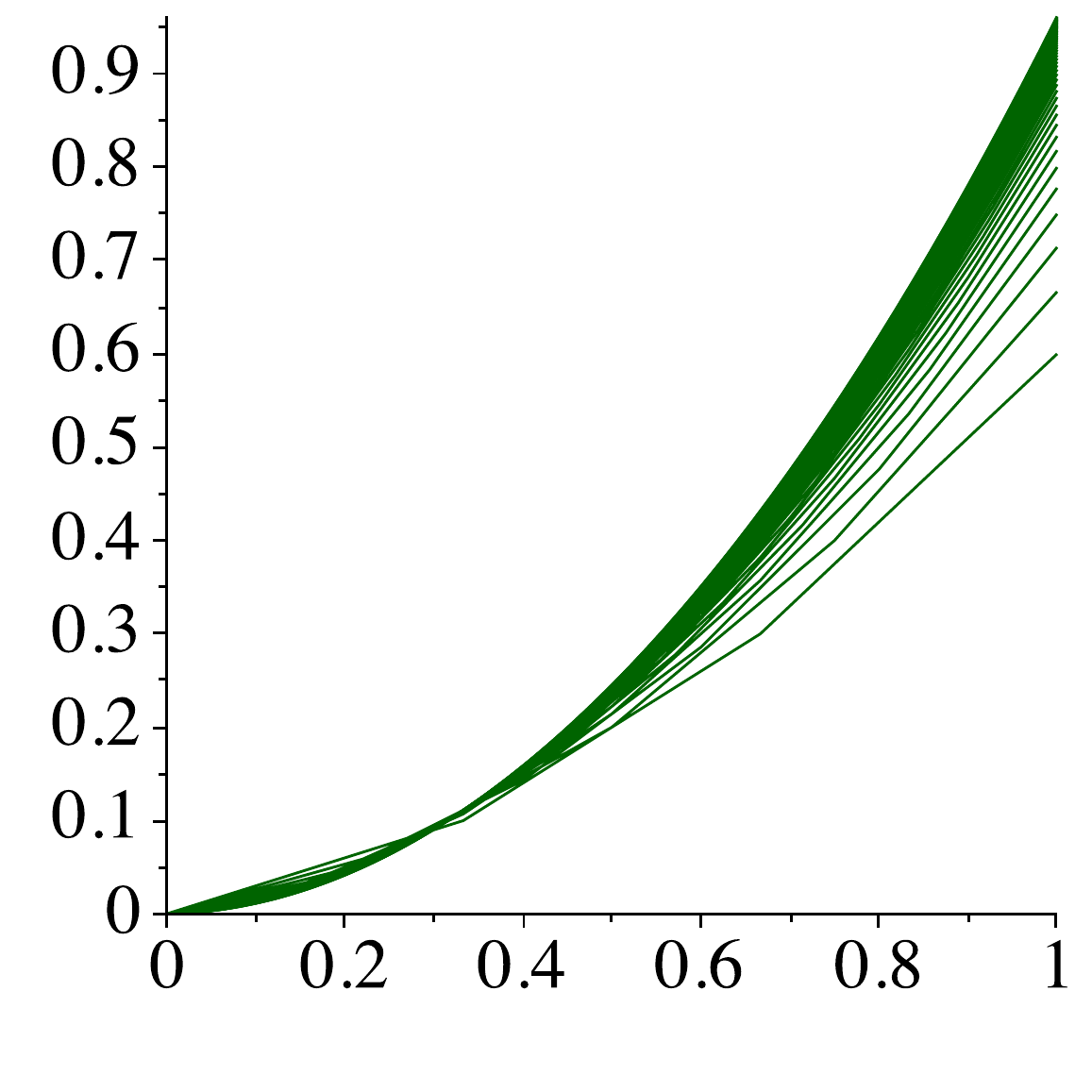}
\end{center}
\caption{Distributions of the coefficients of polynomials of type 
$\ET{n+pv+q,-v;1}$ for $n=3,\dots, 50$: (from left to right) 
$(p,q)=(2,-0.5)$, $(2,0)$, $(2,0.5)$, $(2,1)$, $(3,1)$.}
\label{fig-beta}	
\end{figure}

\subsection{Uniform limit laws again}
\label{ss-ull}
We saw two occurrences of uniform limit law in the above table (being
a special case of beta distribution):
\href{https://oeis.org/A279891}{A279891} and
\href{https://oeis.org/A123310}{A123310}. Other less trivial examples
are the following.
\begin{center}
\begin{tabular}{llll} \\
\multicolumn{1}{c}{OEIS} &
\multicolumn{1}{c}{Type} &
\multicolumn{1}{c}{OGF} & 
\multicolumn{1}{c}{Limit law}  \\ \hline 	
\href{https://oeis.org/A104709}{A104709} & $\ET{2n+1+v,-(1+v);1}$
& $\frac1{(1-2z)(1-(1+v)z)}$ & Uniform$[0,\frac12]$ \\
\href{https://oeis.org/A193851}{A193851} & $\ET{3n+1+2v,-(1+2v);1}$
& $\frac1{(1-3z)(1-(1+2v)z)}$ & Uniform$[0,\frac23]$ \\
\href{https://oeis.org/A193861}{A193861} & $\ET{3n+2+v,-(2+v);1}$
& $\frac1{(1-3z)(1-(2+v)z)}$ & Uniform$[0,\frac13]$\\ \hline
\end{tabular}    
\end{center}

\medskip
Their reciprocal polynomials are of the same form \eqref{Pnv-gen} but 
with quadratic $\alpha(v)$. See Figure~\ref{fig-unif} for a graphical 
rendering. 

\medskip

\begin{center}
\begin{tabular}{cccccc} 
	\multicolumn{1}{c}{OEIS} &
	\multicolumn{1}{c}{Type} &
	\multicolumn{1}{c}{Limit law} &
	\multicolumn{1}{c}{Recip. of} \\ \hline 	
\href{https://oeis.org/A054143}{A054143} & $\ET{(1+2v-v^2)n+v(1+v),-v(1+v);1}$
& Uniform$[\frac12,1]$ &\href{https://oeis.org/A104709}{A104709} \\
\href{https://oeis.org/A193850}{A193850} & $\ET{(2+2v-v^2)n+v(2+v),-v(2+v);1}$
& Uniform$[\frac23,1]$ & \href{https://oeis.org/A193851}{A193851}\\
\href{https://oeis.org/A193860}{A193860} & $\ET{(1+4v-2v^2)n+v(1+2v),-v(1+2v);1}$
& Uniform$[\frac13,1]$ & \href{https://oeis.org/A193861}{A193861} \\ \hline
\end{tabular}
\end{center}

\begin{figure}[!ht]
\begin{center}
\begin{tabular}{c c c}
\includegraphics[height=3.2cm]{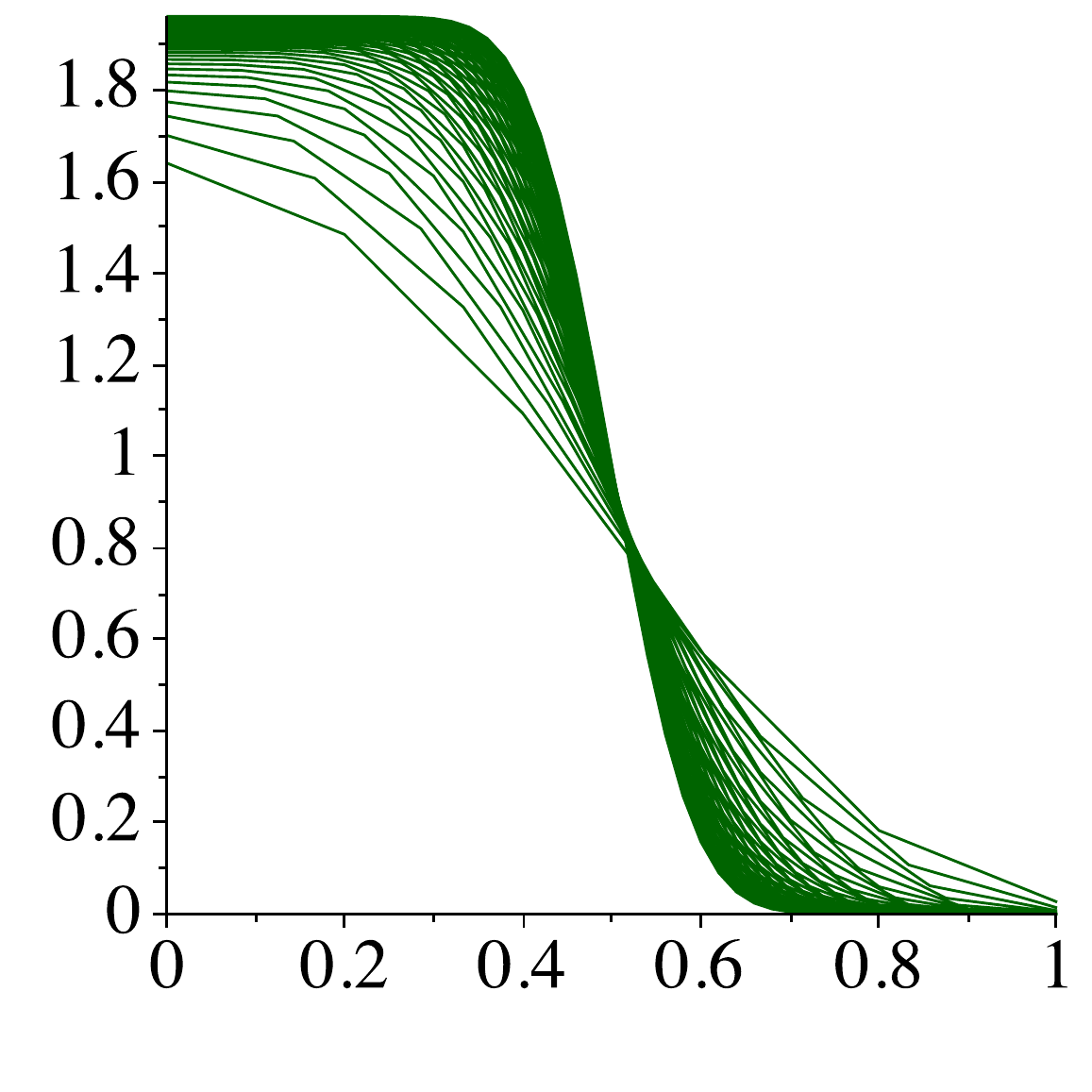} &
\includegraphics[height=3.2cm]{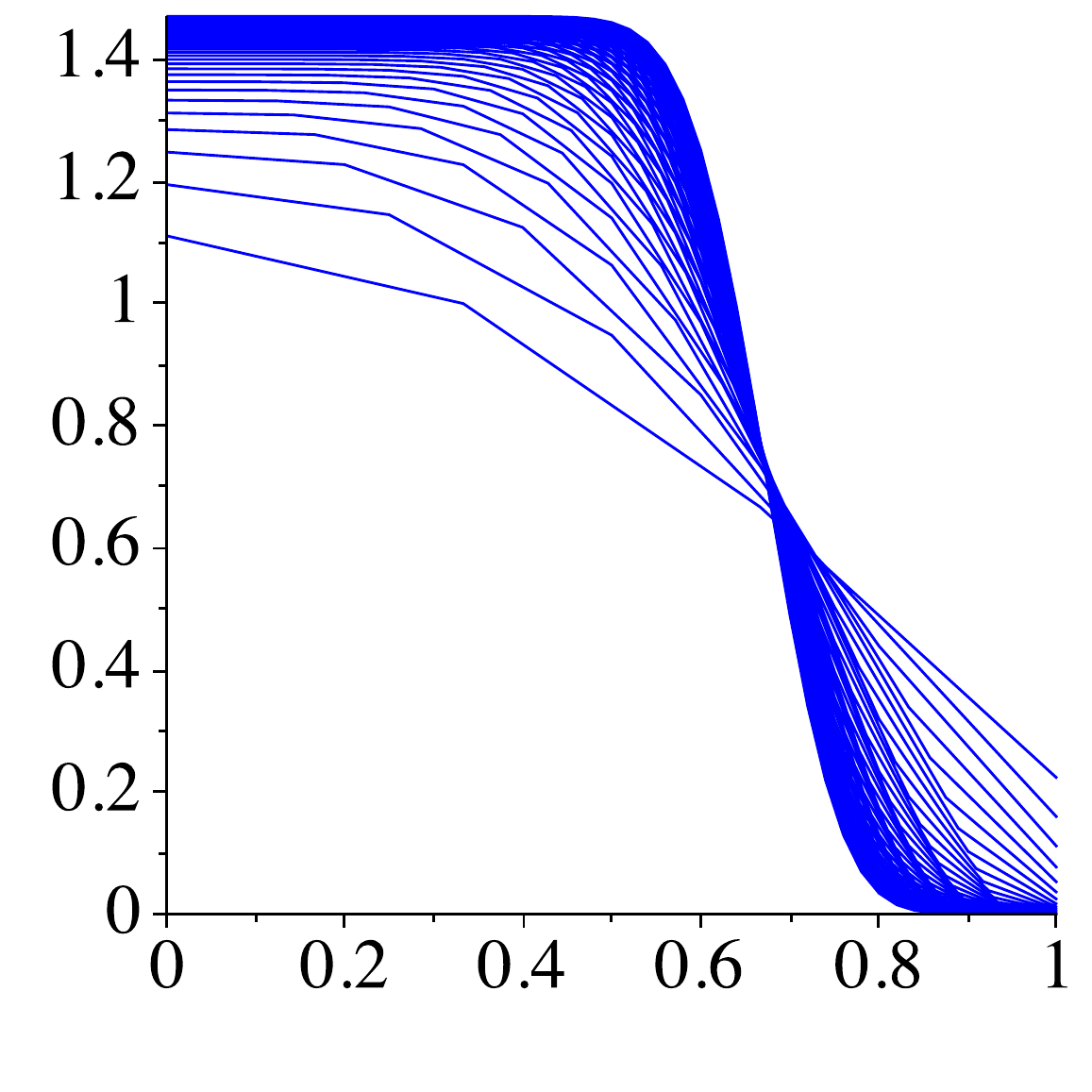} &
\includegraphics[height=3.2cm]{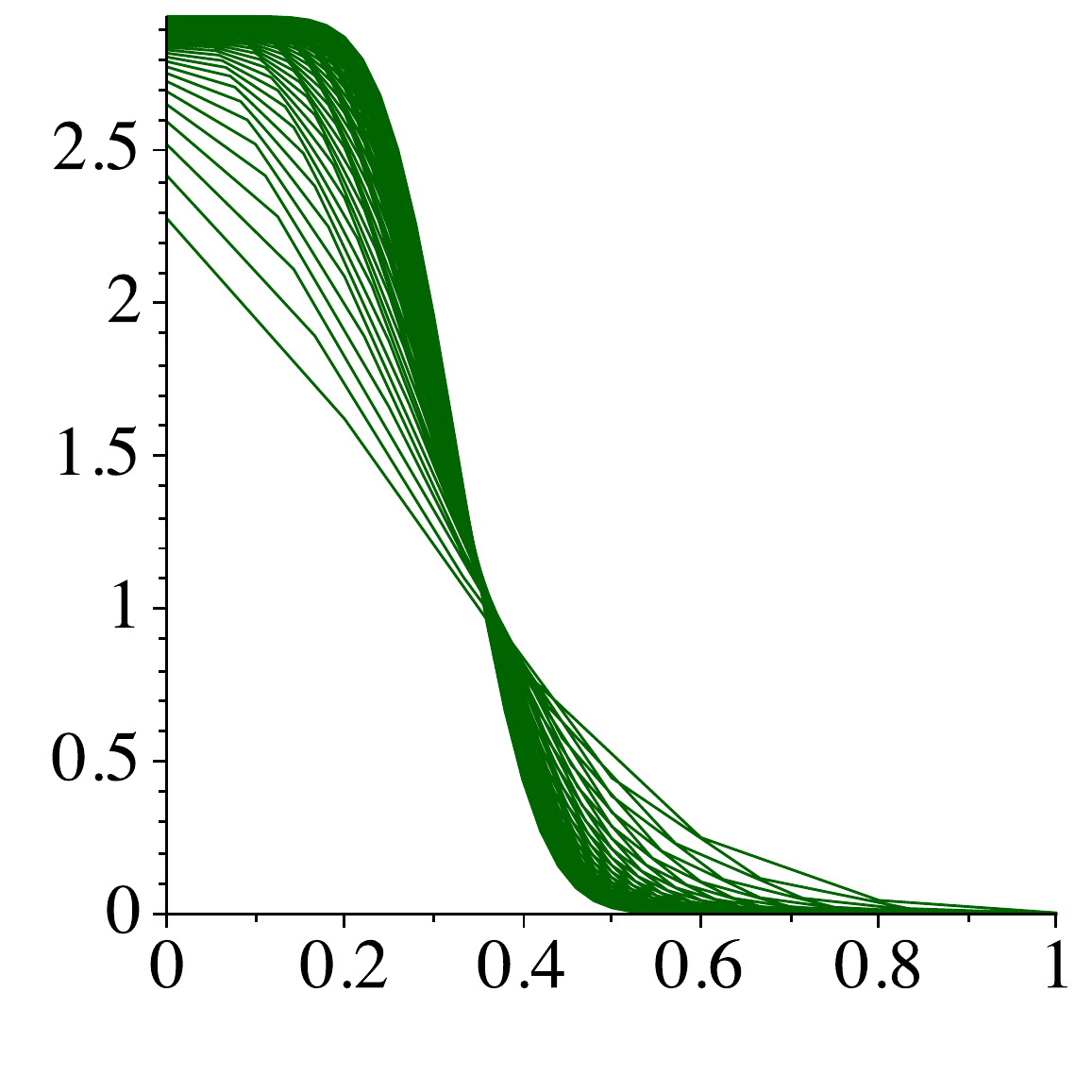} \\
\href{https://oeis.org/A104709}{A104709} &
\href{https://oeis.org/A193851}{A193851} &
\href{https://oeis.org/A193861}{A193861}
\end{tabular}
\end{center}
\caption{The histograms corresponding to A104709, A193851 and 
A193861.}\label{fig-unif}
\end{figure}

\medskip

Other Eulerian recurrences not of the form \eqref{beta-rr} but with a 
uniform limit law include
\begin{small}
\begin{center}
\begin{tabular}{llll} \hline
\href{https://oeis.org/A118175}{A118175} 
& $P_n(v)=\sum_{1\le j\le 2n}v^j$ 
& $P_n\in\ET{\frac{v(2-v)n-v(1-2v)}{n},-\frac{v^2}{n};1}$ 
& Uniform$[0,1]$\\
\href{https://oeis.org/A071028}{A071028} 
& $P_n(v)=\sum_{0\le j\le n}v^{2j}$
& $P_n\in\ET{\frac{n+v^2}{n},-\frac{v(1+v)}{2n};1}$ 
& Uniform$[0,1]$\\ 
\hline
\end{tabular}    
\end{center}
\end{small}

Note that a very similar-looking EGF $\frac1{(1-z)(1-(1+2v)z)}$, 
which is \href{https://oeis.org/A193862}{A193862} (reciprocal of 
\href{https://oeis.org/A115068}{A115068}, enumerating elements in 
Coxeter group with certain descent sets), leads to the CLT 
$\mathscr{N}\lpa{\frac23n,\frac29n}$, although both sequences do not 
satisfy the recurrence \eqref{Pnv-gen}. This follows from a direct 
calculation.

\subsection{Rayleigh and half-normal limit laws 
($\frac\beta\alpha=-\frac12$)}

We consider here $\tau_1=\frac12$ for which many different
limit laws are possible. For example, the polynomials with
\[
    P_n\in\ET{2n-2+v,-v;1+v},
\]
contain only nonnegative coefficients, and follow a limit law with
the density $\frac18x^3 e^{-\frac14x^2}$. This is proved directly 
from \eqref{nn-fx}. Similarly, the polynomials 
$P_n\in\ET{2n+bv,-v;1}$ leads to the limit law with the density
\[
    \frac{b2^{-b}}{\Gamma\lpa{\frac12(b+1)}}
    \,x^{b-1}\int_x^\infty e^{-\frac14t^2}\dd t
    \qquad(b>0;x>0).
\]

Instead of describing all possible limit laws for which we have
few applications, we address the following question, based on the
examples we collected: \emph{under which conditions will the limit
law of the coefficients be either Rayleigh or half-normal (two of the
most common non-normal laws in lattice paths, random trees, random
mappings, etc.)}? For more instances and techniques for these two
laws, see \cite{Drmota1997,Wallner2016} and the references therein.
It turns out that these are very special laws from our framework and
very strong restrictions are needed. We give a complete
characterization of this question.

Recall that the Rayleigh and half-normal distributions with scale 
$\sigma>0$ (which corresponds to the mode of the distribution) have 
the densities
\[
    \frac{x}{\sigma^2}\,e^{-\frac{x^2}{2\sigma^2}}
	\quad\text{and}\quad
	\frac{\sqrt{2}}{\sqrt{\pi}\,\sigma}\, 
	e^{-\frac{x^2}{2\sigma^2}}\qquad(x\ge0),
\]
respectively. While the Taylor expansion of the former contains only 
odd powers, that of the latter contains only even powers. The 
corresponding $m$th moments have the forms 
\begin{align}\label{ray-half-mm}
    \sqrt{\pi}\frac{\Gamma(m+1)}
	{\Gamma\lpa{\frac m2+\frac12}}\,
    \left(\frac{\sigma}{\sqrt{2}}\right)^m
	\quad\text{and}\quad
	\frac{\Gamma(m+1)}{\Gamma\lpa{\frac m2+1}}
	\,\left(\frac{\sigma}{\sqrt{2}}\right)^m,
\end{align}
respectively. 

\subsubsection{Characterizations of Rayleigh and half-normal limit 
laws}

To describe our characterization of the two special limit laws, 
we define the function
\[
    \mathscr{F}_\alpha(p,q,\rho;z) 
	:= (1-\alpha z)^{-p}
	\lpa{\lpa{1-\rho(1-v)}
	\sqrt{1-\alpha z}+\rho(1-v)}^{-q},
\]
which equals $\alpha^q$ times $F$ in \eqref{nn-egf} when
$\beta=-\frac12\alpha$, $\beta'=-\frac12\rho$,
$\gamma=\alpha\lpa{p-1+\frac12q}$, $\gamma' = \frac12 q\rho$, $c_0=1$
and $c_1=0$. For our uses, we need the following conditions for the
nonnegativity of the coefficients $[v^kz^n]\mathscr{F}_\alpha$.
\begin{lmm} \label{lmm-nonneg}
Let $P_n(v) := [z^n]\mathscr{F}_\alpha(p,q,\rho;z)$. 
Assume $\alpha>0$. (i) If $p=0$ and $q>0$, then $[v^k]P_n(v)\ge 0$ 
for all $n,k\ge0$ iff $0\le \rho\le 1$; and (ii) if $p\ge\frac12$ and 
$0<q\le 2$, then  $[v^k]P_n(v)\ge 0$ for all $n,k\ge0$ iff $0\le 
\rho\le \frac32$.
\end{lmm}
\begin{proof}
Assume without loss of generality $\alpha=1$. Consider first the case 
when $p=0$ and $q>0$:
\begin{align*}
    P_n(v)
    = [z^n]\lpa{\lpa{1+\rho(v-1)}
	\sqrt{1-z}-\rho(v-1)}^{-q}
    =: [z^n](1-g(z))^{-q},
\end{align*}
where $g(z) := \tilde{\rho}(v)(1-\sqrt{1-z})$ with 
$\tilde{\rho}(v):= 1+\rho(v-1)$. Then 
\[
    g = z \,
    \frac{\tilde{\rho}(v)^2}{2\tilde{\rho}(v)-g}.
\]
By Lagrange inversion formula \cite{Stanley1999}
\begin{align}
    P_n(v) &= [z^n](1-g(z))^{-q} 
    = \frac qn[t^{n-1}]\frac1{(1-t)^{q+1}}
    \left(\frac{\tilde{\rho}(v)^2}
    {2\tilde{\rho}(v)-t}\right)^n\nonumber \\
    &= \frac{q}{n}
    \sum_{1\le j\le n}\binom{2n-1-j}{n-1}
	\binom{q+j-1}{q}
    \frac{\tilde{\rho}(v)^{j} }{2^{2n-j}}
	\label{F01}.
\end{align}
Then 
\begin{align}\label{coeff-F01}
    [v^k]P_n(v) = \frac{q}{n}\,\rho^k
    \sum_{k\le j\le n}\binom{2n-1-j}{n-1}
	\binom{q+j-1}{q}
    \binom{j}{k}\frac{(1-\rho)^{j-k}}{2^{2n-j}}.
\end{align}
If $0\le \rho\le 1$, then all coefficients are nonnegative and we 
obtain $[v^k]P_n(v)\ge0$. On the other hand, since $P_1(v) = 
\frac12q((1-\rho)+\rho v)$, we see that if $[v^k]P_n(v)\ge0$ for 
$k,n\ge0$, then $\rho\in[0,1]$. This proves the necessity. 

For the second case $p\ge\frac12$ and $0<q\le 2$, writing $\rho=1+t$ 
and $Z := 1-\sqrt{1-z}$, we have 
\begin{align*}
    [v^k]P_n(v) 
    &= [z^n](1-z)^{-p}[v^k](1+tZ-(1+t)vZ)^{-q}\\
    &= \binom{q+k-1}{k}(1+t)^k
    [z^n]Z^k(1-Z)^{-2p}(1+tZ)^{-q-k}.
\end{align*}
By using the relation $Z(2-Z)=z$, applying Lagrange inversion 
formula and then changing the variables $Z\mapsto 2w$, we obtain 
\begin{align*}
    &[z^n]Z^k(1-Z)^{-2p}(1+tZ)^{-q-k}\\
    &\qquad= 2^{k-2n}[w^{n-k}](1-2w)^{1-2p}(1+2tw)^{-q-k}
    (1-w)^{-n-1}\\
    &\qquad= 2^{k-2n}[w^{n-k}](1-2w)^{1-2p}
    \left((1+2tw)(1-w)\right)^{-q-k}
    (1-w)^{-n-1+q+k}.
\end{align*}
Since $p\ge\frac12$, we see that $[w^j](1-w)^{1-2p}\ge0$ for all 
$j\ge0$; on the other hand, since $0<q\le 2$, we have $n+1-q-k\ge0$ 
for $0\le k\le n-1$, implying that $[w^j](1-w)^{-n-1+q+k}\ge0$ for 
$j\ge0$ and $0\le k\le n-1$; also $[v^n]P_n(v) = 
\binom{q+n-1}{n}(1+t)^{-q}2^{-n}$ is always nonnegative. 
Furthermore, for $0\le k\le n$, if $1-2t\ge0$, then 
\[
    [w^j]\left((1+2tw)(1-w)\right)^{-q-k}
    \ge0 \text{ for }j\ge0.
\]
For the necessity, we observe first that $[v]P_1(v) = \frac12q\rho<0$ 
if $\rho<0$; also 
\[
    [v^{n-1}]P_n(v) = \binom{q+n-2}{n-1}(1+t)^{n-1}
    2^{-n-1}((1-2t)n+O(1)),
\]
which becomes negative if $t>\frac12$ or $\rho>\frac32$ for large 
enough $n$. This implies the necessity of $0\le \rho\le\frac32$. 
\end{proof}

\begin{thm} \label{thm:RH}
Assume that $P_n(v)$ satisfies the recurrence \eqref{Pn-nn} with 
$\tau_1=\frac12$ and $\beta'<0$. Let $\sigma := 
-\frac{2\sqrt{2}\beta'}{\alpha}$. Then the coefficients of
the polynomials $\mathbb{E}(v^{X_n}) := \frac{P_n(v)}{P_n(1)}$ are
asymptotically Rayleigh distributed
\[
    \frac{X_n}{\sigma\sqrt{n}}\cid X,
\]
where $X$ has the density $xe^{-\frac12x^2}$ for $x\ge0$ iff 
the EGF $F$ of $P_n$ has one of the following five forms:
$F\in\{\mathcal{R}_1,\dots,\mathcal{R}_5\}$,
where 
\begin{align*}
	\mathcal{R}_1(z) 
	&:= (c_0+c_1(v-1)) 
	\mathscr{F}_\alpha\lpa{0,1,\tfrac{c_1}{c_0};z}, \\
    \mathcal{R}_2(z) 
	&:= c_0\mathscr{F}_\alpha\lpa{0,1,
	-\tfrac{2\beta'}{\alpha};z} \quad\text{with}\;
	-\tfrac12\alpha\le \beta'<0, \\
	\mathcal{R}_3(z) 
	&:=(c_0+c_1(v-1)) \mathscr{F}_\alpha 
	\lpa{\tfrac12,2,\tfrac{c_1}{c_0};z}, \\
	\mathcal{R}_4(z) &:= c_0 \mathscr{F}_\alpha
	\lpa{\tfrac12,2,-\tfrac{2\beta'}{\alpha};z} 
	 \quad\text{with}\;
	-\tfrac34\alpha\le \beta'<0, \\	
    \mathcal{R}_5(z) &:= c_0 \mathscr{F}_\alpha
	\lpa{\tfrac32,2,\tfrac{3c_1}{2c_0};z} + c_1(v-1) 
	\mathscr{F}_\alpha\lpa{\tfrac32,3,\tfrac{3c_1}{2c_0};z}.
\end{align*}

On the other hand, the sequence of random variables $\{X_n\}$ is  
asymptotically half-normally distributed  
\[
    \frac{X_n}{\sigma\sqrt{n}}\cid Y,
\]
where $Y$ has the density $\sqrt{\frac{2}{\pi}}\,e^{-\frac12x^2}$ for
$x\ge0$ iff the EGF of $P_n$ has one of the following three
forms: $F\in\{\mathcal{H}_1,\mathcal{H}_2, \mathcal{H}_3\}$, where
\begin{align*}
	\mathcal{H}_1(z) &:= (c_0+c_1(v-1))\mathscr{F}_\alpha
	\lpa{\tfrac12,1,\tfrac{c_1}{c_0};z}, \\
	\mathcal{H}_2(z) &:= c_0
	\mathscr{F}_\alpha\lpa{\tfrac12,1,-\tfrac{2\beta'}{\alpha};z} 
	\quad\text{with}\;
	-\tfrac34\alpha\le \beta'<0, \\
	\mathcal{H}_3(z) &:= (c_0+c_1(v-1)) \mathscr{F}_\alpha
	\lpa{\tfrac32,2,\tfrac{c_1}{c_0};z}.
\end{align*}
\end{thm}
We see that in either case the seven parameters in \eqref{Pn-nn} are
now reduced to only three (including $\alpha$) as far as the two
limit laws are concerned. Also the coefficients of $\mathcal{R}_1,
\mathcal{R}_3, \mathcal{R}_5, \mathcal{H}_1, \mathcal{H}_3$ are
always nonnegative since $0\le\frac{c_1}{c_0}\le 1$, but for
$\mathcal{R}_2, \mathcal{R}_4, \mathcal{H}_2$ one needs further
restrictions on $\beta'$ using Lemma~\ref{lmm-nonneg}.
\begin{proof}
Consider first the Rayleigh limit law. Since $\tau_1=\frac12$, we 
have, by Proposition~\ref{prop-Km},
\begin{align}\label{Kmt}
    \mathbb{E}\left(\frac{X_n}{\sigma \sqrt{n}}\right)^m
    \sim \tilde{K}_m,\;\text{where} \;
    \tilde{K}_m 
    =  \frac{\Gamma\lpa{m+\tau_3}\Gamma(\tau_2+1)
    (\rho_1 m +\tau_3)}
    {\Gamma\lpa{\tau_3+1}
    \Gamma\lpa{\frac m2+1+\tau_2}2^{\frac m2}}.
\end{align}
Here $\sigma := -\frac{2\sqrt{2}\beta'}\alpha$ and $\rho_1 :=
-\frac{c_1\alpha}{2c_0\beta'}$. By equating $\tilde{K}_m$ to the
moments \eqref{ray-half-mm} of the Rayleigh distribution, we are led
to the identity for all $m\ge0$
\begin{align}\label{Kmt2}
    \tilde{K}_m 
    =  \frac{\Gamma\lpa{m+\tau_3}
    \Gamma(\tau_2+1)(\rho_1 m +\tau_3)}
    {\Gamma\lpa{\tau_3+1}\Gamma\lpa{\frac m2+1+\tau_2}2^{\frac m2}}
	= \sqrt{\pi}\frac{\Gamma(m+1)}
	{\Gamma\lpa{\frac m2+\frac12}2^{\frac m2}}.
\end{align}
If $c_1=0$, then $\rho_1=0$ and the above identity becomes
\[
    \tilde{K}_m 
    =  \frac{\tau_3\Gamma\lpa{m+\tau_3}\Gamma(\tau_2+1)}
    {\Gamma\lpa{\tau_3+1}\Gamma\lpa{\frac m2+1+\tau_2}
    2^{\frac m2}}
	= \sqrt{\pi}\frac{\Gamma(m+1)}
	{\Gamma\lpa{\frac m2+\frac12}2^{\frac m2}}.
\]
Since this holds for $m\ge0$ (including $m\to\infty$), we see that, 
by Stirling's formula,
\[
    \frac{\tilde{K}_m}{\lpa{\frac me}^{\frac m2}}
	= \frac{\tau_3\Gamma(\tau_2+1)}{\Gamma(\tau_3+1)}
	\,m^{\tau_3-\tau_2-1}2^{\tau_2+\frac12}(1+o(1)),
\]
for large $m$, while for the Rayleigh moments
\[
	\frac{\sqrt{\pi}\Gamma(m+1)}
	{\Gamma\lpa{\frac m2+\frac12} 
	\lpa{\frac2e m}^{\frac m2}}
	=\sqrt{\pi m}(1+o(1)).
\]
It follows that $\tau_3=\frac32+\tau_2$. Substituting this into
\eqref{Kmt2} with $m=2$, and then solving for $\tau_2$, we obtain two
solutions: $\tau_2=\pm\frac12$. If $\tau_2=-\frac12$, then
$\tau_3=1$, and $P_n$ has the pattern
\[
    P_n\in\ET{\alpha n-\tfrac12\alpha -\beta'(v-1),
    -\tfrac12\alpha+\beta'(v-1);c_0},
\]
which implies that the EGF equals $\mathcal{R}_2$ by \eqref{nn-egf}. 

On the other hand, if $\tau_2=\frac12$, then $\tau_3=2$, and $P_n$ 
satisfies
\[
    P_n\in\ET{\alpha n+\tfrac12\alpha -2\beta'(v-1),
    -\tfrac12\alpha+\beta'(v-1);c_0},
\]
so that $F=\mathcal{R}_4$. Note that $\mathcal{R}_2$ and 
$\mathcal{R}_4$ are connected by a differentiation:
\begin{align}\label{R2-R4}
	\partial_z\mathscr{F}_\alpha(0,1,\rho;z)
	= \tfrac12\alpha(1+\rho(v-1))
	\mathscr{F}_\alpha\lpa{\tfrac12,2,\rho;z}.
\end{align}

Assume now $c_1>0$. Then $\rho_1>0$ and 
\[
    \frac{\tilde{K}_m}{\lpa{\frac me}^{\frac m2}}
	=\rho_1 \frac{\Gamma(\tau_2+1)}{\Gamma(\tau_3+1)}
	\,m^{\tau_3-\tau_2}2^{\tau_2+\frac12}(1+o(1)),
\]
for large $m$, implying that $\tau_3=\frac12+\tau_2$. Substituting 
this into \eqref{Kmt2} with $m=2,4$ and then solving for $\rho_1$ and 
$\tau_2$, we get three feasible solutions: 
\[
    (\rho_1,\tau_2) = \bigl\{\lpa{1,-\tfrac12},
	\lpa{1,\tfrac12}, \lpa{\tfrac23,\tfrac32}\bigr\},
\]
leading to the three patterns
\begin{align*}
    P_n \in \begin{cases}
		\ET{\alpha n -\tfrac12\alpha,-\tfrac12\alpha
	    -\tfrac{c_1\alpha}{2c_0}(v-1);c_0+c_1(v-1)},\\
		\ET{\alpha n+ \tfrac12\alpha+ \tfrac{c_1\alpha}{2c_0}
        (v-1), 
		-\tfrac12\alpha-\tfrac{c_1\alpha}{2c_0}(v-1);c_0+c_1(v-1)},\\
		\ET{\alpha n+\tfrac32\alpha+\tfrac{3c_1\alpha}{2c_0}(v-1),
		-\tfrac12\alpha-\tfrac{3c_1\alpha}{4c_0}(v-1);c_0+c_1(v-1)},
    \end{cases}
\end{align*}
respectively in sequential order. These correspond to
$\mathcal{R}_1$, $\mathcal{R}_3$ and $\mathcal{R}_5$, respectively.
Note that $\sigma = \frac{\sqrt{2}c_1}{c_0}$ in the cases of
$\mathcal{R}_1$ and $\mathcal{R}_3$, and $\sigma =
\frac{3c_1}{\sqrt{2}c_0}$ in the other case, so that $\sigma$ equals
$\sqrt{2}$ times the third parameter of the function
$\mathscr{F}_\alpha$ in all cases $\mathcal{R}_1,\dots,
\mathcal{R}_5$. Also $\mathcal{R}_1$, $\mathcal{R}_3$ and
$\mathcal{R}_5$ are essentially connected by successive derivatives
(up to change of parameters and multiplicative factors) by the
relations \eqref{R2-R4} and
\begin{align*}
	\partial_z\mathscr{F}_\alpha
	\lpa{\tfrac12,2,\rho;z}
	= \tfrac32\alpha
	\mathscr{F}_\alpha\lpa{\tfrac32,2,\rho;z}
	+\alpha \rho(v-1)\mathscr{F}_\alpha\lpa{\tfrac32,3,\rho;z}.
\end{align*}

The proof for half-normal limit law is similar, starting from the 
asymptotic estimate
\[
	\frac{\Gamma(m+1)}
	{\Gamma\lpa{\frac m2+1} 
	\lpa{\frac2e m}^{\frac m2}}
	=\sqrt{2}(1+o(1)),
\]
implying either $\rho_1=0, \tau_3=1+\tau_2$ or $\rho_1>0,
\tau_3=\tau_2$. By the same arguments used above, we then obtain
$\tau_2=0$ in the former case, and $(\rho_1,\tau_2)= (1,0)$ or
$\lpa{\frac12,1}$ in the latter case, yielding the three patterns
\begin{align*}
    P_n
    \in \begin{cases}
		\ET{\alpha n -\beta'(v-1),
		-\tfrac12\alpha+\beta'(v-1);c_0},\\
		\ET{\alpha n,
		-\tfrac12\alpha-\tfrac{c_1}{2c_0}\alpha(v-1);
        c_0+c_1(v-1)},\\
		\ET{\alpha n+ \alpha+ \tfrac{c_1}{c_0}\alpha(v-1),
		-\tfrac12\alpha-\tfrac{c_1}{c_0}\alpha(v-1);c_0+c_1(v-1)},
    \end{cases}
\end{align*}
corresponding to $\mathcal{H}_2$, $\mathcal{H}_1$, and 
$\mathcal{H}_3$, respectively. 
\end{proof}

Another interesting property of $X_n$ is that the 
difference polynomials $\Delta_n(v) := P_n(v)-P_{n-1}(v)$ have 
only positive coefficients and the same limit law as that for 
$P_n(v)$. 
\begin{cor} \label{cor:diff-ll}
Assume that $P_n(v)$ is as in Theorem~\ref{thm:RH},  
$[v^k]\Delta_n(v)\ge0$ and $\mathbb{E}\lpa{v^{Z_n}} := 
\frac{\Delta_n(v)}{\Delta_n(1)}$. Then $X_n$ and $Z_n$ follow the 
same limit laws. 
\end{cor}
The result holds in more general settings but we content ourselves 
with the current formulation due to limited applications. 
\begin{proof}
By Proposition~\ref{prop-Km}, we see that 
\[
    P_n^{(m)}(1) \sim P_n(1) K_m n^{m\tau_1}
    \qquad(m\ge0),
\]
and the corollary follows from the relation $P_n(1)=(\alpha n+\gamma)
P_{n-1}(1)$. 
\end{proof}

\subsubsection{Examples. I. Rayleigh laws}

Consider the Catalan triangle \href{https://oeis.org/A039598}{A039598}:
\[
    P_n(v) 
    := \sum_{0\le k\le n}
	\frac{2(k+1)}{n+k+2}\binom{2n+1}{n-k}v^k,
\]
which has a large number of combinatorial interpretations such as 
the number of leaves at level $k+1$ in ordered trees with $n+1$ 
edges. This sequence of polynomials satisfies the recurrence 
\begin{align}\label{A039598}
    P_n\in\EET{\frac{4n+2v}{n+1}}{-\frac{1+v}{n+1};1}.
\end{align}
The EGF of $(n+1)!P_n(v)$ is of type $\mathcal{R}_4$ (with $c_1=0$ 
and $\frac\gamma\alpha=\frac12$) and equals
$\mathscr{F}_4\lpa{\frac12,2,\frac12;z}$, which, by an
integration, gives
\[
    \sum_{n\ge0}P_n(v)z^{n+1}
    = \frac{1-\sqrt{1-4z}}{(1+v)\sqrt{1-4z}+1-v}.
\]
By Theorem~\ref{thm:RH}, we see that the limit law of the 
coefficients is Rayleigh with $\sigma=\frac12$, which also follows 
from the closed-form expression; see Figure~\ref{fig:RH}. Stronger 
asymptotic approximations and local limit theorems can also be 
derived.

This sequence has many minor variants that do not change the Rayleigh 
limit distribution of the coefficients; for example (the case 
\href{https://oeis.org/A122919}{A122919} following from Corollary~\ref{cor:diff-ll}): 
\begin{center}
\begin{tabular}{ccll} \hline
\href{https://oeis.org/A039598}{A039598} & $:=$ & $P_n(v)$ & Rayleigh$\lpa{\frac1{\sqrt{2}}}$ \\
\href{https://oeis.org/A039599}{A039599} & $=$ & $\frac{vP_n(v)+P_{n-1}(0)}{1+v}=:R_n(v)$ 
& Rayleigh$\lpa{\frac1{\sqrt{2}}}$ \\
\href{https://oeis.org/A050166}{A050166} & $=$ & Reciprocal of $P_n(v)$
& $\text{Rayleigh}\lpa{\frac1{\sqrt{2}}}$\\
\href{https://oeis.org/A122919}{A122919} & $=$ & $P_n(v)-P_{n-1}(v)$& Rayleigh$\lpa{\frac1{\sqrt{2}}}$\\
\href{https://oeis.org/A128899}{A128899} & $=$ & $vP_{n-1}(v)$& Rayleigh$\lpa{\frac1{\sqrt{2}}}$\\
\href{https://oeis.org/A118920}{A118920} & $=$ & $2P_n(v)$ & Rayleigh$\lpa{\frac1{\sqrt{2}}}$\\ 
\href{https://oeis.org/A053121}{A053121} & $=$ & $\left\{\begin{array}{ll}
    vP_{\tr{\frac12n}}(v^2), & n \text{ odd}\\
    R_{\frac12n}(v^2), & n \text{ even}
\end{array}\right.$ & Rayleigh$\lpa{\frac1{\sqrt{2}}}$ \\ 
\hline
\end{tabular}    
\end{center}

Some other OEIS sequences leading to Rayleigh limit laws are listed 
in the following table (using the format \eqref{Pnv-ab}). 

\begin{center}
\renewcommand{\arraystretch}{1.3}
\begin{tabular}{lllll} 
\multicolumn{1}{c}{OEIS} &
\multicolumn{1}{c}{Type} &
\multicolumn{1}{c}{$[v^k]P_n(v)$} &
\multicolumn{1}{c}{Limit law} \\ \hline	
\href{https://oeis.org/A039599}{A039599} &  
$\ET{\frac{4n+v-3}{n},-\frac{1+v}{n};1}$ &
$\frac{2k+1}{n+k+1}\binom{2n}{n-k}$ 
& Rayleigh$\lpa{\frac1{\sqrt{2}}}$\\
\href{https://oeis.org/A102625}{A102625} & $\ET{2n+v,-v;v}$ &
$\frac{k(2n-k+1)!}{(n-k+1)!2^{n-k+1}}$ 
& Rayleigh$\lpa{\sqrt{2}}$\\ 
\href{https://oeis.org/A108747}{A108747} 
& $\ET{\frac{4n+2v}{n+1},-\frac{2v}{n+1};2v}$ &
$\frac{k2^k}{2n+2-k}\binom{2n+2-k}{n+1}$
& Rayleigh$\lpa{\sqrt{2}}$\\ \hline
\end{tabular} 
\end{center}

Their reciprocal polynomials also follow the same Rayleigh limit laws.
\begin{small}
\begin{center}
\begin{tabular}{lllll}
\multicolumn{1}{c}{Recip. of} &    
\multicolumn{1}{c}{OEIS} &
\multicolumn{1}{c}{Type} &
\multicolumn{1}{c}{Limit law} \\ \hline	 
\href{https://oeis.org/A039599}{A039599} & 
\href{https://oeis.org/A050165}{A050165} & 
$\ET{\frac{(1+4v-v^2)n-3v+v^2}{n},-\frac{v(1+v)}{n};1}$ &
Rayleigh$\lpa{\frac1{\sqrt{2}}}$\\
\href{https://oeis.org/A102625}{A102625} & 
\href{https://oeis.org/A193561}{A193561} 
& $\ET{(1+v)n+1,-v;1}$ &
Rayleigh$\lpa{\sqrt{2}}$\\
\href{https://oeis.org/A039598}{A039598} & 
\href{https://oeis.org/A050166}{A050166} & 
$\ET{\frac{(1+4v-v^2)n+1+v^2}{n+1},-\frac{v(1+v)}{n+1};1}$ &
Rayleigh$\lpa{\frac1{\sqrt{2}}}$\\
\hline
\end{tabular} 
\end{center}
\end{small}

Among these OEIS sequences, \href{https://oeis.org/A102625}{A102625}
was one of our motivating examples of non-normal limit laws (see
Figure~\ref{fig:RH}), and has many combinatorial interpretations such
as the root degree of plane-oriented recursive trees and the waiting
time in a memory game; see \cite{Acan2016, Bergeron1992, Ma2013b} and
OEIS \href{https://oeis.org/A102625}{A102625} page for more
information.
 
Yet another occurrence of \href{https://oeis.org/A102625}{A102625}
and Rayleigh limit law is as follows. Consider the Catalan triangle
\href{https://oeis.org/A009766}{A009766} (or ballot numbers):
\[
    R_n(v) := \sum_{0\le k\le n}
	\frac{n-k+1}{n+1}\binom{n+k}{k}v^k.
\]
Then $R_n(v)$ satisfies the recurrence
\[
    (n+1)R_n(v) = ((1+2v)n+1)R_{n-1}(v)
	-v(1-2v)R_{n-1}'(v)\qquad(n\ge1),
\]
with $R_0(v)=1$. The distribution of the coefficients is negative 
binomial with parameters $2$ and $\frac12$. Also they are related to 
$P_n(v)$ of \href{https://oeis.org/A102625}{A102625} by $P_{n+1}(v) = v^{n+2}R_n\lpa{\frac1{2v}}$. 

These sequences are rather simple in nature as they all have a neat
closed-form expression for the coefficients. Less trivial examples 
can be generated by using \eqref{F01} with $\rho\in\lpa{0,\frac12}$, 
say. 

\subsubsection{Examples. II. Half-normal laws} 
\label{sec-half-normal}

Consider sequence \href{https://oeis.org/A193229}{A193229}:
\[
    P_n(v) 
	= \sum_{0\le k\le n}\frac{(2n-k)!}{(n-k)!2^{n-k}}\,v^k;
\]
see \cite{Ma2013b} for a characterization via grammars. Then $P_n$
satisfies $\ET{2n-1+v,-v;1}$, which is of type $\mathcal{H}_2$, and
we get a half-normal limit law for the coefficients; see
Figure~\ref{fig:RH}. Note that a conjecture mentioned on the OEIS
webpage for \href{https://oeis.org/A193229}{A193229} can be easily proved, stating that $[v^k]P_n(v)$
is equal to the $(k+1)$st term in the top row of $M^n$, where
$M=(m_{i,j})$ with $m_{i,j}=i$ for $1\le j\le i+1$ and $i=0$ for
$j\ge i+2$.

An essentially identical sequence connected to Banach's matchbox 
problem is \href{https://oeis.org/A164705}{A164705}, which can be 
generated by $P_n\in \ET{4,-\frac{2v}{n};\tfrac12v}$ and has the 
closed-form expression $\binom{2n-k}{n}2^{k-1}$. The EGF is then of 
type $\mathcal{H}_1$, and we get the same half-normal limit law.

Interestingly, the sequence \href{https://oeis.org/A001497}{A001497},
which corresponds to Bessel polynomials, differs from
\href{https://oeis.org/A193229}{A193229} by a factor of $k!$, namely,
the EGF equals
\[
    \frac{e^{v(1-\sqrt{1-2z})}}{\sqrt{1-2z}},
\]
whose coefficients lead to a Poisson$(1)$ limit law. 

Another instance is \href{https://oeis.org/A111418}{A111418} (right-hand side of odd-numbered rows of 
Pascal's triangle): $[v^k]P_n(v)=\binom{2n+1}{n-k}$, and $P_n$ 
satisfies 
\(
    P_n\in\ET{\frac{4n-1+v}{n},-\frac{1+v}{n};1},
\)
again of type $\mathcal{H}_1$, so that the coefficients lead to a
half-normal limit law; see Figure~\ref{fig:RH}. The reciprocal
polynomial of $P_n$ corresponds to sequence
\href{https://oeis.org/A122366}{A122366}, which satisfies
\(
    Q_n\in\ET{\frac{(1+4v-v^2)n-v(1-v)}{n},-\frac{v(1+v)}{n};1};
\)
compare with the normal examples in Section~\ref{sec-3v2v}. A signed
version of \href{https://oeis.org/A111418}{A111418} is
\href{https://oeis.org/A113187}{A113187}:
\(
    R_n\in\ET{-\frac{4n-1+v}{n},-\frac{1+v}{n};1}.
\)
We have $(-1)^nR_n(-v)=P_n(v)$, and we get the same half-normal limit 
law for the absolute values of the coefficients.

These examples are summarized in the following table.
\begin{center}
\begin{tabular}{lllll}
\multicolumn{1}{c}{OEIS} &
\multicolumn{1}{c}{Type} &
\multicolumn{1}{c}{$[v^k]P_n(v)$} &
\multicolumn{1}{c}{Limit law} \\ \hline		
\href{https://oeis.org/A193229}{A193229} & $\ET{2n+v-1,-v;1}$ 
& $\frac{(2n-k)!}{(n-k)!2^{n-k}}$ & Half-Normal$\lpa{\sqrt{2}}$\\
\href{https://oeis.org/A164705}{A164705} &  
$\ET{4,-\frac{2v}{n};\frac12v}$ 
& $\binom{2n-k}{n}2^{k-1}$ & Half-Normal$\lpa{\sqrt{2}}$\\
\!\!$\begin{array}{l}
    \text{\href{https://oeis.org/A111418}{A111418}}\\
    |\text{\href{https://oeis.org/A113187}{A113187}}|
\end{array}$  & $\ET{\frac{4n+v-1}{n},-\frac{1+v}{n};1}$
& $\binom{2n+1}{n-k}$ & Half-Normal$\lpa{\frac1{\sqrt{2}}}$ \\ 
\href{https://oeis.org/A122366}{A122366} & 
$\ET{\frac{(1+4v-v^2)n-v(1-v)}{n},-\frac{v(1+v)}{n};1}$
& $\binom{2n+1}{k}$ & Half-Normal$\lpa{\frac1{\sqrt{2}}}$ \\
\hline    
\end{tabular}
\end{center}

\begin{figure}[!ht]
\begin{center}
\begin{tabular}{c c c c}
\includegraphics[width=3cm]{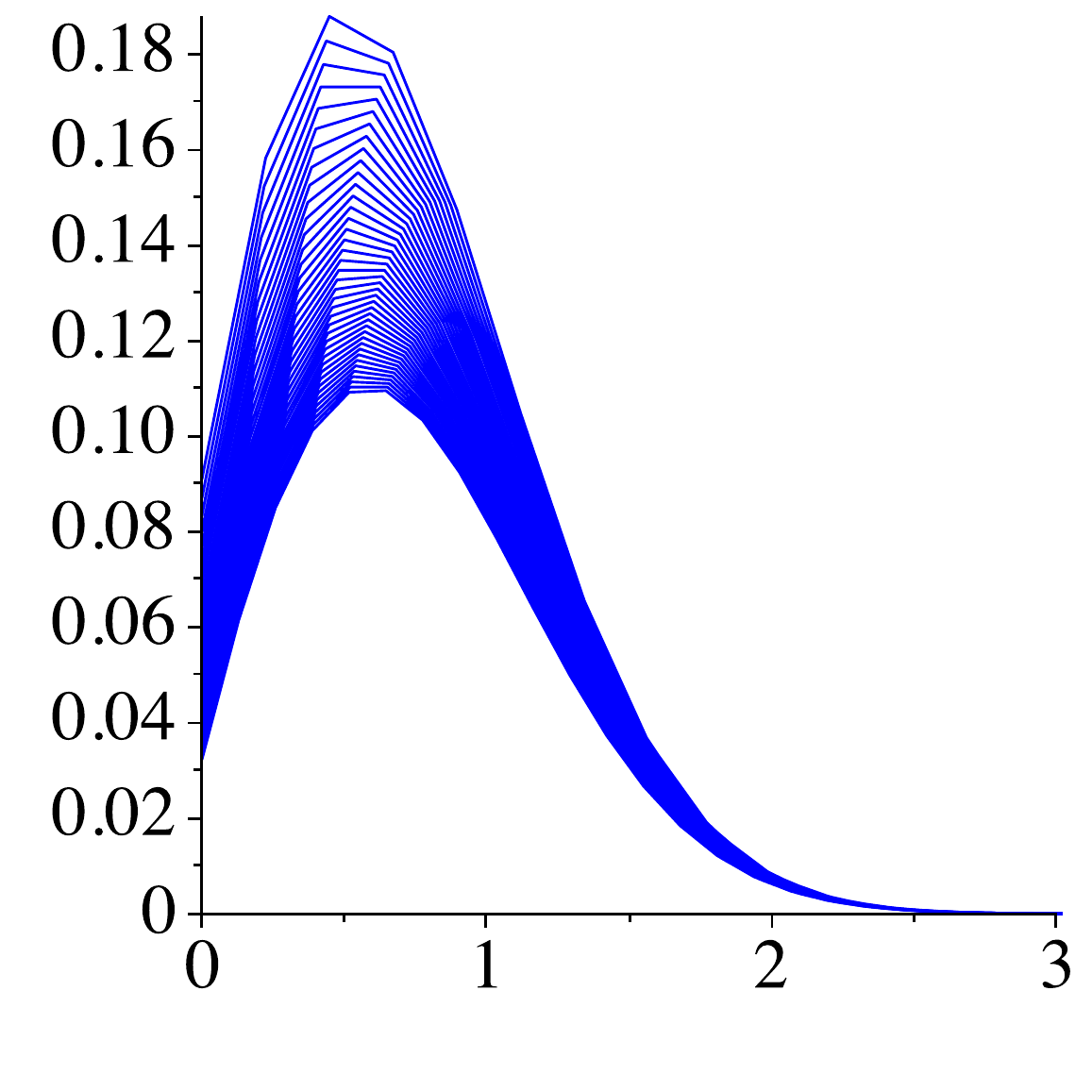} &
\includegraphics[width=3cm]{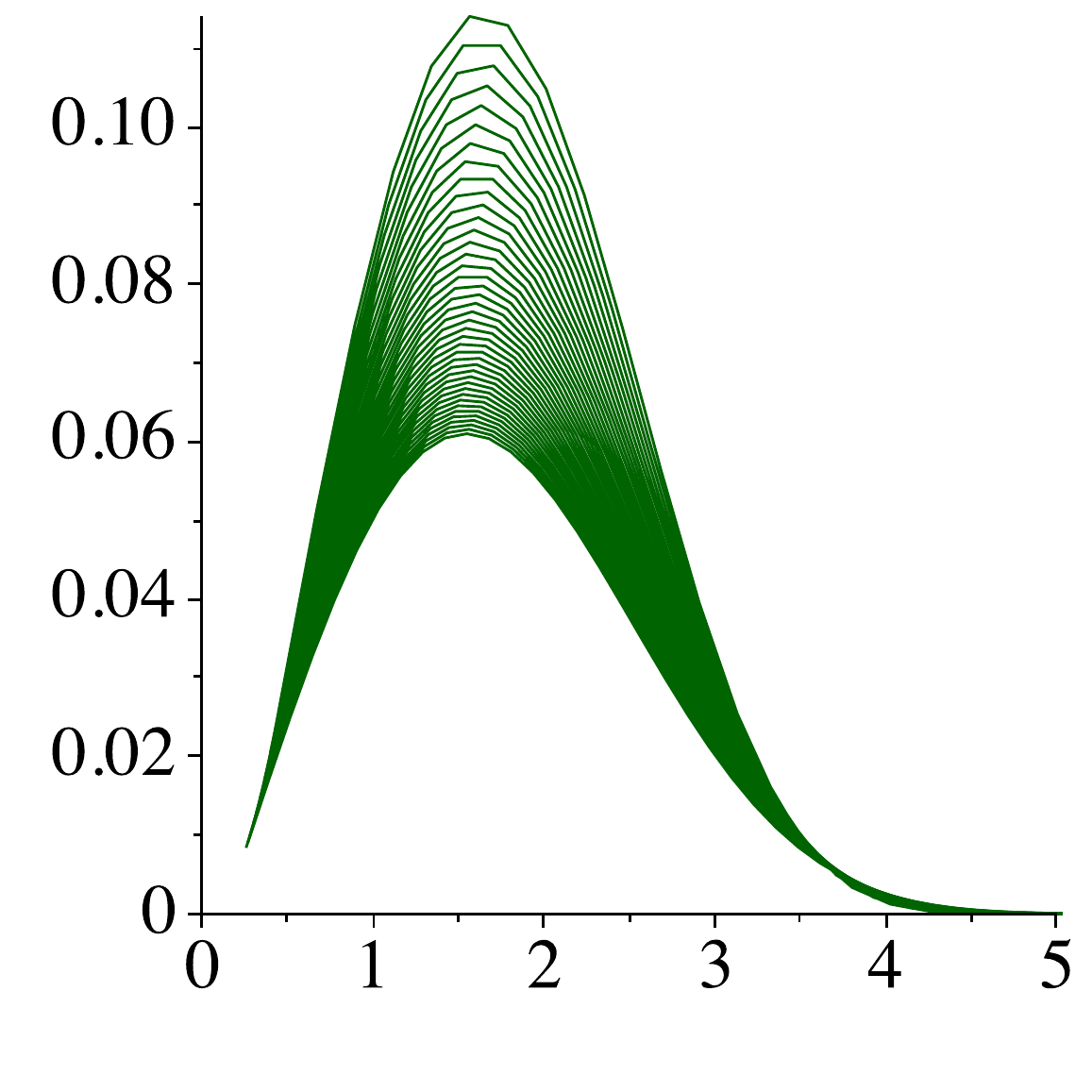} &
\includegraphics[width=3cm]{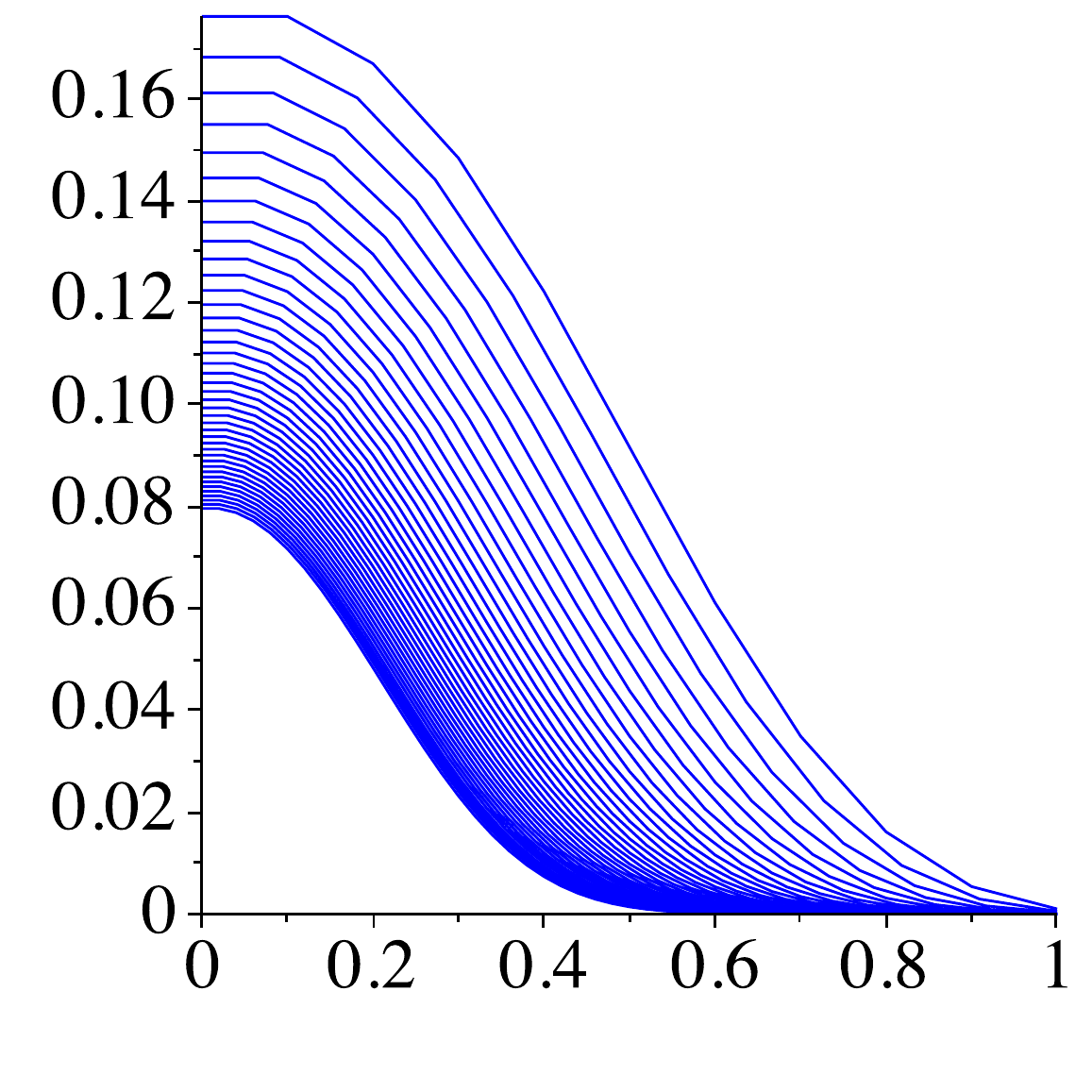} &
\includegraphics[width=3cm]{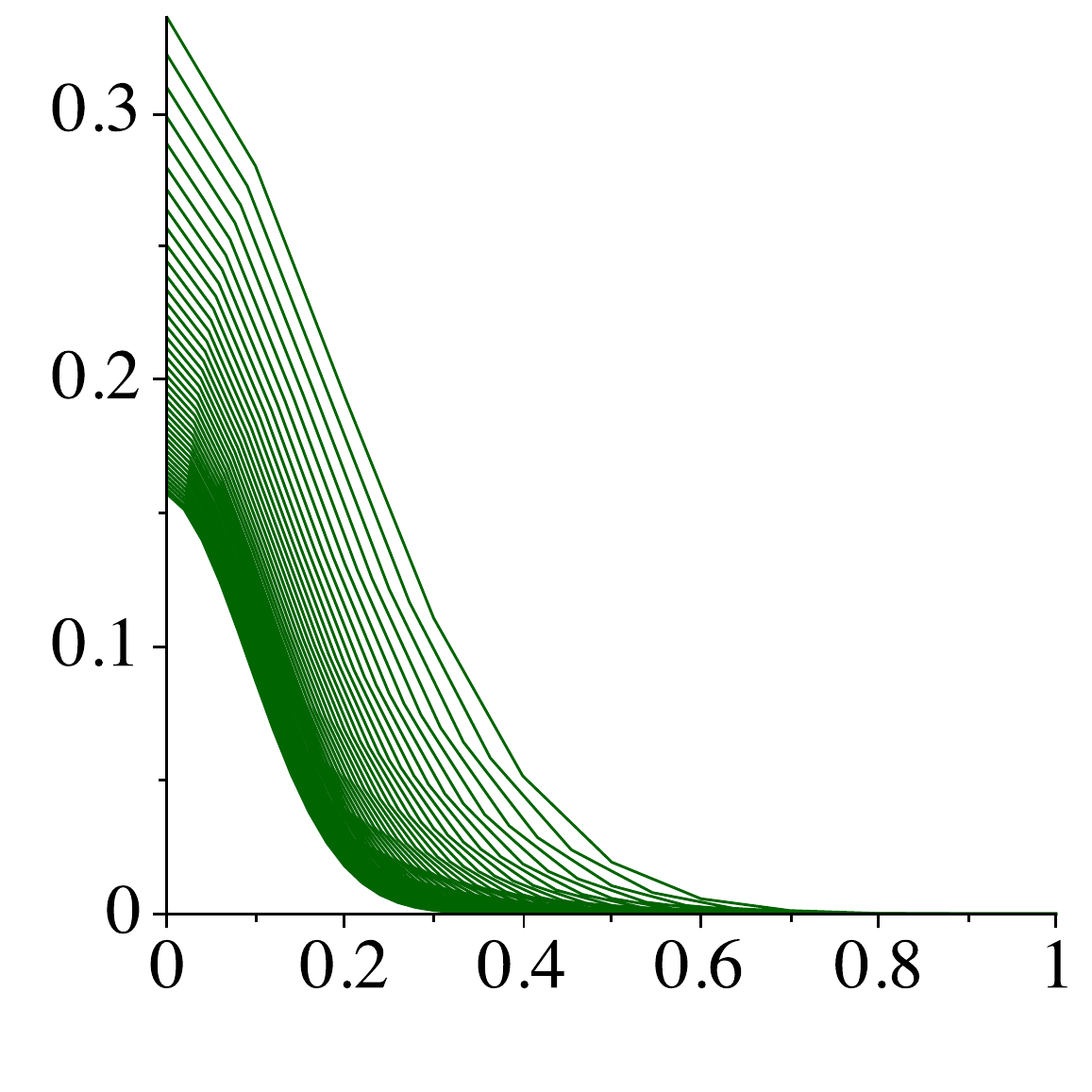} \\ 
\href{https://oeis.org/A039598}{A039598} & \href{https://oeis.org/A102625}{A102625} & \href{https://oeis.org/A193229}{A193229} & \href{https://oeis.org/A111418}{A111418} \\
\end{tabular}
\end{center}
\vspace*{-.4cm}
\caption{Rayleigh and half-normal limit laws: the two left  
histograms for $n=20,\dots,60$ and plotted against $\sqrt{n}$; the 
two right histograms for $n=10,\dots,50$ and plotted against $n$.}
\label{fig:RH}
\end{figure}

\subsection{Other limit laws}

We discuss other limit laws based on the recurrence \eqref{Pn-nn} in 
this subsection. 

\subsubsection{Mittag-Leffler limit laws}

Consider \href{https://oeis.org/A202550}{A202550}, which is defined 
by (with a shift of index)
\[
    [v^{k}]P_n(v) := [z^{n+1}]
    \left(\frac{1-(1-8z)^{\frac14}}
    {1+(1-8z)^{\frac14}}\right)^{k+1}
	\qquad(0\le k\le n).
\]
Then $P_n(v)$ satisfies the recurrence
\begin{align}\label{A202550}
    P_n\in\EET{\frac{8n+2v}{n+1}}{-\frac{1+v}{n+1};1}.
\end{align}
By Proposition~\ref{prop-Km}, we see that the $m$th moment of $X_n$
is asymptotic to
\[
    \frac{\Gamma(\tfrac14)\Gamma(m+1)}
	{2^m\Gamma\lpa{\frac m4+\frac14}}\qquad(m\ge0),
\]
and thus the limit law of the coefficients is a Mittag-Leffler 
distribution (with the moment generating function \eqref{E-pqr-s} 
with $r=1$) with $p=q=\frac14$. 

\begin{center}
\begin{tabular}{ccc}\hline
\href{https://oeis.org/A202550}{A202550} &  $P_n\in\ET{\frac{8n+2v}{n+1},-\frac{1+v}{n+1};1}$ &
Mittag-Leffler limit law\\ \hline
\end{tabular}    
\end{center}

In general, replacing $8$ by $\alpha\ge2$ in \eqref{A202550} 
guarantees $[v^k]P_n(v)\ge0$ and leads to the moment sequence
\[
    \frac{\Gamma\lpa{\frac2\alpha}\Gamma(m+1)}
	{2^m\Gamma\lpa{\frac2\alpha(m+1)}}\qquad(m\ge0),
\]
which yields a Mittag-Leffler distribution when $\alpha>2$. 
Interestingly, the case $\alpha=2$ gives the binomial coefficients 
(\href{https://oeis.org/A007318}{A007318}), 
namely, $P_n(v) = (1+v)^n$, and we get a CLT $\mathscr{N}\lpa{\frac 
12n, \frac14n}$ instead of a Mittag-Leffler distribution.

Another example leading to a Mittag-Leffler limit law is to extend
the recurrence for \href{https://oeis.org/A102625}{A102625} by considering $P_n\in\ET{\alpha
n-1,-v;v}$, for $\alpha\ge 2$. We then deduce, again by
Proposition~\ref{prop-Km}, that the limit law is a Mittag-Leffler
distribution:
\[
    \frac{X_n}{n^{\frac1q}}
	\stackrel{d}{\longrightarrow} X_q,
	\text{ where\;  }
    \mathbb{E}\lpa{e^{X_qs}}=\sum_{m\ge0}
	\frac{\Gamma\lpa{1-\frac1q}}
	{\Gamma\lpa{1+\frac{m-1}q}}\, s^m.
\]

Finally, the limit law for the coefficients of the polynomials
$P_n\in\ET{\alpha n,-(1+v);1+v}$ with $\alpha>2$ is also a 
Mittag-Leffler. 

\subsubsection{A mixture of discrete and continuous laws} 

\noindent
\begin{minipage}{0.7\textwidth}
\quad\; An example of a similar pattern to \eqref{Pn-Bernoulli} but
with a completely different behavior is
\href{https://oeis.org/A139524}{A139524}:
$P_n\in\ET{2,-\frac{1+v}{n};4+2v}$. A closed-form expression of $P_n$
is
\[
    P_n(v) = 2^{n+1} + 2(1+v)^{n+1} \qquad(n\ge0).
\]
The limit law is a mixture of Dirac (at zero) and a normal: 
$\mathbb{P}(X_n=0)\to\frac13$ and 

\end{minipage}
\begin{minipage}{0.3\textwidth}
\centering 
\includegraphics[height=4cm]{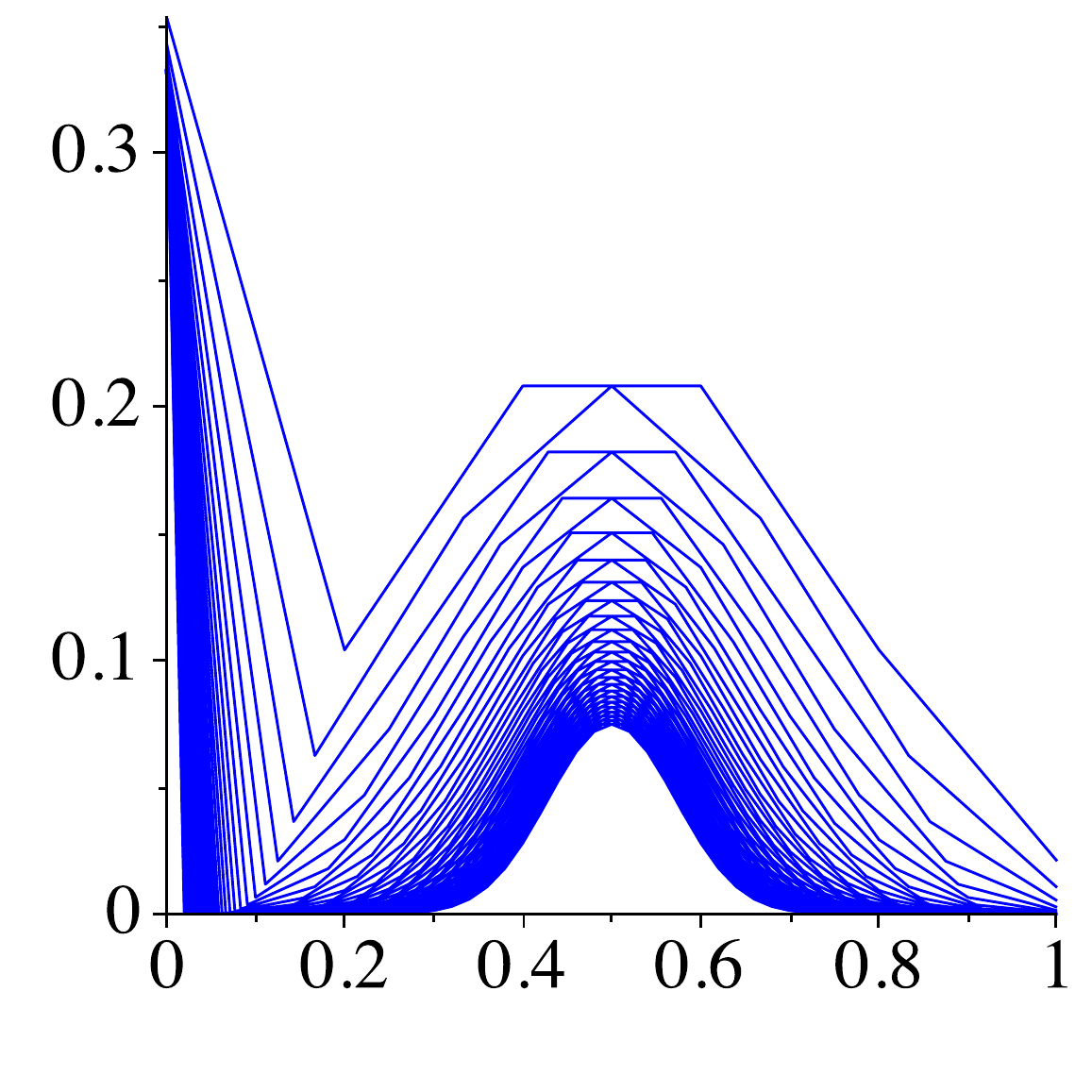}
\end{minipage}

\medskip

\[
    \mathbb{P}\left(X_n = \Bigl\lfloor
    \frac n2+\frac {\sqrt{n}}4x\Bigr\rfloor\right)
    = \frac23\cdot\frac{2e^{-\frac12 x^2}}{\sqrt{2\pi n}}
    \left(1+O\left(\frac{|x|+|x|^3}{\sqrt{n}}\right)\right),
\]
uniformly for $x=o(n^{\frac16})$.

Another similar example is $P_n\in\ET{\frac{n+1}{n},-\frac{v}{n}; 
1+v}$. Then 
\[
    P_n(v) = n+1+v+\cdots + v^{n+1}
    \qquad(n\ge0),
\]
and one gets a mixture of Dirac and uniform as the limit law. This 
sequence of polynomials corresponds to the signless version of  
\href{https://oeis.org/A167407}{A167407}. A similar variant is 
\href{https://oeis.org/A130296}{A130296} ($P_n(v) = nv+v^2+\cdots
+v^n$ for $n\ge1$), but it satisfies a rather messy recurrence 
involving $P_{n-1}'(v)$ and $P_{n-1}''(v)$ and is not Eulerian;
its reciprocal is \href{https://oeis.org/A051340}{A051340}.

\section{Extensions}
\label{sec-extensions}

In view of the richness and diversity of Eulerian recurrences, many
extensions have been made; here we briefly discuss some of them and
examine the extent to which the tools used in this paper applies as
far as the limit distribution of the coefficients is concerned. For
simplicity, we content ourselves with concrete examples rather than
the formulation of general theorems. Some extensions and 
generalizations will be elaborated elsewhere.

Throughout this section, we denote the Eulerian polynomials by 
$A_n(v) := \sum_{0\le k<n}\eulerian{n}{k}v^k$. 

\subsection{Non-homogeneous recurrence}
Eulerian recurrences containing an additional non-homogeneous term of 
the form 
\[
    P_n(v) = (\alpha(v)n+\gamma(v))P_{n-1}(v)
    +\beta(v) (1-v)P_{n-1}'(v) + T_n(v)\qquad(n\ge1),
\]
with $P_0(v)$ and $T_n(v)$ given, already appeared in our discussions 
of Lehmer's polynomials \eqref{lehmer2} and in 
Section~\ref{ss-type-d} on type $D$ Eulerian numbers.

We discuss here two more examples beginning with 
\href{https://oeis.org/A065826}{A065826}, which enumerates the 
descents in permutations starting with an ascent:
\[
    P_n(v) = (vn-1)P_{n-1}(v)+v(1-v)P_{n-1}'(v)+vA_{n}(v) 
    \qquad(n\ge2),
\]
with $P_1(v)=v$. It is easy to see that 
\[
    P_n(v) = \sum_{1\le k\le n}k\eulerian{n}{k-1}v^{k}
    \qquad(n\ge1),
\]
so that the EGF is given by 
\[
    v\frac{\partial}
    {\partial z}\frac{e^{(1-v)z}-1-(1-v)z}
	{(1-v)(1-ve^{(1-v)z})}. 
\]
This implies an optimal CLT $\mathscr{N}\lpa{\frac12n,\frac1{12}n;
n^{-\frac12}}$ by Theorem~\ref{thm-saqp}. 

The reciprocal polynomial $Q_n$ of $P_n$, satisfying the recurrence
\[
    Q_n(v) = (vn-2v)Q_{n-1}(v)+v(1-v)Q_{n-1}'(v)+vA_{n-1}(v) 
    \qquad(n\ge3),
\]
with $Q_2(v)=v$, appeared in a context of decoding schemes 
\cite{Sharon2007}. 

On the other hand, the derivative $R_n(v)$ ($=$
\href{https://oeis.org/A142706}{A142706}) of $A_n(v)$ also satisfies
a similar recurrence
\[
    R_n(v) 
    = (vn+2-3v)R_{n-1}(v)+v(1-v)R_{n-1}'(v)+(n-1)A_{n-1}(v) 
    \qquad(n\ge1),
\]
with $R_0(v)=0$. The same CLT $\mathscr{N}\lpa{\frac12n,\frac1{12}n;
n^{-\frac12}}$ for the coefficients hold. 
%

\begin{center}
\begin{tabular}{lll}\hline
$v(vA_n)'$ & 
\href{https://oeis.org/A065826}{A065826} & 
$\mathscr{N}\lpa{\frac12n,\frac1{12}n;n^{-\frac12}}$ \\
$A_n'(v)$ & \href{https://oeis.org/A142706}{A142706} & 
$\mathscr{N}\lpa{\frac12n,\frac1{12}n;n^{-\frac12}}$ \\  \hline
\end{tabular}	
\end{center}

Another recurrence appears in \cite{Conway1988} (in the context of 
Voronoi cells of lattices):
\begin{align*}
    a_{n,k} &= ka_{n-1,k}+(n-k+1)a_{n-1,k-1}
    +k^3\eulerian{n-1}{k-1}+(n-k+1)^3\eulerian{n-1}{k-2}.
\end{align*}
\begin{minipage}{0.7\textwidth}
If $P_n(v):=\sum_k a_{n+1,k}v^k$, then (not in OEIS)
\begin{align*}
    P_n(v) &= (vn+v)P_{n-1}(v)+v(1-v)P_{n-1}'(v)
    \\ &\quad
	+v(n+1)^3A_{n}(v) 
    -v(3vn(n+1)-1+v)A_{n-1}'(v)
    \\ &\quad
	+3v^2(vn+1)A_{n-1}''(v)
    +v^3(1-v)A_{n-1}'''(v),
\end{align*}
for $n\ge1$ with $P_0(v)=v$ (we shift $n$ by one).
By a direct use of our method of moments, we can prove the CLT
$\mathscr{N}\lpa{\frac12n,\frac1{12}n}$.
\end{minipage}\;\;
\begin{minipage}{0.25\textwidth}
\centering 
\includegraphics[height=3.5cm]{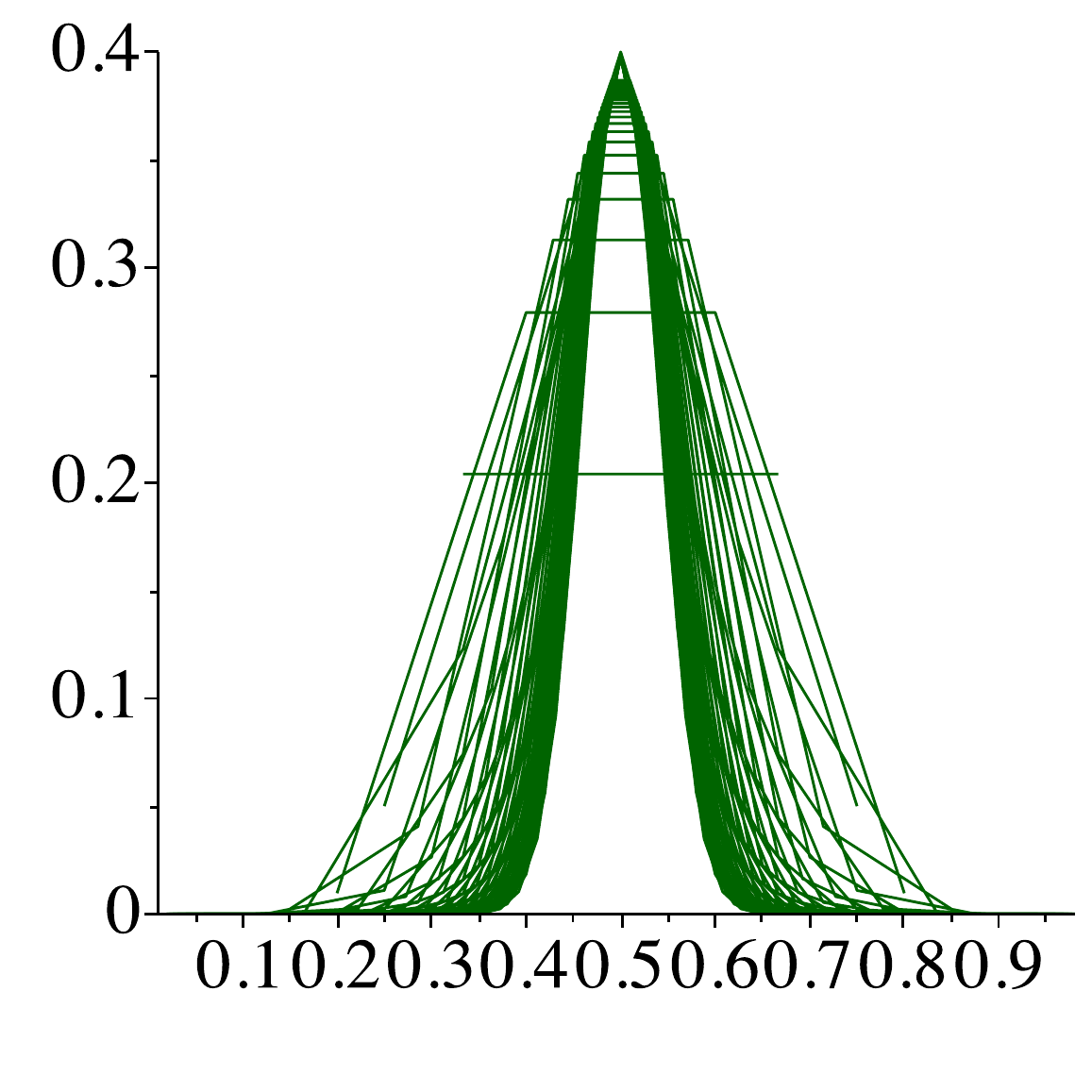}	
\end{minipage}

\medskip

\subsection{Eulerian recurrences involving $P_{n-2}(v)$}
Similar to the previous subsection, the framework 
\begin{align}\label{Pnv-abc}
    P_n(v) = a_n(v) P_{n-1}(v) + b_n(v)(1-v)P_{n-1}'(v)
    +c_n(v) P_{n-2}(v),
\end{align}
is also manageable by the approaches we use in this paper. We already
saw two examples in Section~\ref{sec-1plusv}. We consider more
examples here.

\paragraph{Fibonacci-Eulerian polynomials}
An example of the above type appeared in \cite{Carlitz1978}:
\[
    P_n(v) = vnP_{n-1}(v) + v(1-v)P_{n-1}'(v) +(1-v)^2
    P_{n-2}(v)\qquad(n\ge2),
\]
with $P_0(v)=1$ and $P_1(v)=v$. The polynomial $P_n(v)$ is closely
connected to Fibonacci polynomials $F_n(v) = vF_{n-1}(v) +F_{n-2}(v)$
for $n\ge2$ with $F_0(v)=1$ and $F_1(v)=v$ by the relations
\[
    \sum_{k\ge0}F_n(k)v^k = \frac{P_n(v)}{(1-v)^{n+1}}. 
\]
Note that $P_2(v) = 1-v+2v^2$ (the only polynomial with negative
coefficients). This ($P_n$) corresponds to 
\href{https://oeis.org/A259708}{A259708}. A CLT
$\mathscr{N}\lpa{\frac12n,\frac1{12}n}$ holds for the coefficients by 
the method of moments. In terms of Eulerian polynomials, we have 
(redefining $A_0(v) := v^{-1}$)
\[
    P_n(v) = v\sum_{0\le j\le \tr{\frac12n}}
    \binom{n-j}{j}(v-1)^{2j}A_{n-2j}(v)
    \qquad(n\ge1);
\]
see \cite{Carlitz1978}. This can alternatively be derived by solving 
the PDE of second order satisfied by the EGF using Riemann's method. 
Note that this expression of $P_n$ is itself an asymptotic expansion 
for large $n$ and finite $v$; in particular, 
\[
    P_n(v) = vA_n(v)\left(1+O\lpa{n^{-1}}\right),
\]
uniformly for bounded $v$, and the CLT 
$\mathscr{N}\lpa{\frac12n,\frac1{12}n; n^{-\frac12}}$ then follows. 

On the other hand, the Fibonacci polynomials $F_n(v)$ correspond to
\href{https://oeis.org/A168561}{A168561} (integer compositions into
odd parts); see also the Chebyshev polynomials (with signs)
\href{https://oeis.org/A049310}{A049310} and 
\href{https://oeis.org/A053119}{A053119}. Since the OGF of $F_n(v)$ 
is given by $(1-vz-z^2)^{-1}$, we deduce the CLT 
$\mathscr{N}\lpa{\frac1{\sqrt{5}}n,
\frac{4}{5\sqrt{5}}n;n^{-\frac12}}$ for the coefficients of $F_n(v)$ 
by Theorem~\ref{thm-saqp} with $\rho(v) = 
\frac12\lpa{\sqrt{4+v^2}-v}$. Note that $F_n$ also satisfies the
recurrence
\[
    2nF_n(v) = (n+1)vF_{n-1}(v)
    +(4+v^2)F_{n-1}'(v)\qquad(n\ge1).
\]
The sequence of polynomials corresponding to
\href{https://oeis.org/A102426}{A102426} satisfies the same
recurrence as $F_n$ but with different initial conditions; see also 
\href{https://oeis.org/A098925}{A098925}, 
\href{https://oeis.org/A169803}{A169803}, 
\href{https://oeis.org/A011973}{A011973}, and 
\href{https://oeis.org/A092865}{A092865}.

\begin{center}
\begin{tabular}{lll}\hline
\href{https://oeis.org/A168561}{A168561}
& $\frac1{1-vz-z^2}$
& $\mathscr{N}\lpa{\frac1{\sqrt{5}}n,
\frac{4}{5\sqrt{5}}n;n^{-\frac12}}$ \\ 
\href{https://oeis.org/A049310}{A049310}
& $\frac1{1-vz+z^2}$ 
& signed version of \href{https://oeis.org/A168561}{A168561} \\
\href{https://oeis.org/A053119}{A053119}
& $\frac1{1-z+v^2z^2}$ 
& reciprocal of \href{https://oeis.org/A049310}{A049310} \\
\href{https://oeis.org/A098925}{A098925}
& $\frac1{1-vz-vz^2}$ 
& $\mathscr{N}\lpa{\lpa{\frac12+\frac{\sqrt{5}}{10}}n,
\frac{\sqrt{5}}{25} n;n^{-\frac12}}$\\
\href{https://oeis.org/A092865}{A092865}
& $\frac1{1+vz+vz^2}$ 
& signed version of \href{https://oeis.org/A098925}{A098925}\\
\href{https://oeis.org/A011973}{A011973}
& $\frac1{1-z-vz^2}$ 
& $\mathscr{N}\lpa{\lpa{\frac12-\frac{\sqrt{5}}{10}}n,
\frac{\sqrt{5}}{25} n;n^{-\frac12}}$\\
\href{https://oeis.org/A169803}{A169803}
& $\frac{1+vz}{1-z-vz^2}$ 
& $\mathscr{N}\lpa{\lpa{\frac12-\frac{\sqrt{5}}{10}}n,
\frac{\sqrt{5}}{25} n;n^{-\frac12}}$\\
\href{https://oeis.org/A102426}{A102426}
& $\frac{z(1+z-z^2)}{(1-z^2)^2-vz^2}$ 
& $\mathscr{N}\lpa{\frac{\sqrt{5}}{10}n,
\frac{\sqrt{5}}{25} n;n^{-\frac12}}$\\
\hline
\end{tabular}	
\end{center}

\paragraph{Derangement polynomials}
The derangement polynomials in permutations represent another example
of \eqref{Pnv-abc}. They enumerate for example the number of
$n$-derangements with $k$ exceedances, and can be defined by (see
\cite{Brenti1990})
\begin{align}\label{de-poly}
    P_n(v) = \sum_{0\le k\le n}\binom{n}{k}(-1)^{n-k}
    A_k(v)\qquad(n\ge0),
\end{align}
which is sequence \href{https://oeis.org/A046739}{A046739} and  
\href{https://oeis.org/A271697}{A271697} (see also 
\href{https://oeis.org/A168423}{A168423} for a signed 
version) and satisfies the recurrence
\[
    P_n(v) = (n-1)vP_{n-1}(v) +v(1-v)P_{n-1}'(v)
    +(n-1)vP_{n-2}(v)\qquad(n\ge2),
\]
with $P_0(v)=1$ and $P_1(v)=0$. 

A CLT of the form $\mathscr{N}\lpa{\frac12n,\frac{25}{12}n}$ for the 
coefficients was given in \cite{Clark2002} but the variance 
coefficient $\frac{25}{12}$ there should be corrected to 
$\frac1{12}$. See also \cite{Chen2009} for the same CLT 
$\mathscr{N}\lpa{\frac12n,\frac1{12}n}$ for a type $B$ analogue with 
the recurrence 
\[
    R_n(v) = (2n-1)vR_{n-1}(v) +2v(1-v)R_{n-1}'(v)
    +2(n-1)vR_{n-2}(v)\qquad(n\ge2),
\]
with $R_0(v)=1$ and $R_1(v)=v$. Both proofs rely on the 
real-rootedness of the polynomials. 

In both cases, while it is possible to apply the method of moments, 
it is simpler to apply Theorem~\ref{thm-saqp} to the EGFs
\[
    e^{-vz}\frac{1-v}{1-ve^{(1-v)z}},
    \quad\text{and}\quad
    e^{-vz}\frac{1-v}{1-ve^{2(1-v)z}},
\]
respectively, yielding the stronger result 
$\mathscr{N}\lpa{\frac12n,\frac1{12}n; n^{-\frac12}}$.

\paragraph{Binomial-Eulerian and Eulerian-binomial polynomials}
The analytic approach based on EGF has an advantage that it applies 
easily to other variants whose EGFs are available in manageable forms 
such as sequence \href{https://oeis.org/A046802}{A046802}, the 
binomial-Eulerian polynomials (see \cite{Postnikov2008, 
Shareshian2017}):
\[
    F(z,v) = e^{z}\frac{1-v}{1-ve^{(1-v)z}}.
\]
This corresponds essentially to dropping the powers of $-1$ in 
\eqref{de-poly}:
\[
    P_n(v) = n![z^n]F(z,v) 
    = 1 + v\sum_{1\le k\le n}\binom{n}{k}A_k(v)
    = 1 + v\sum_{1\le k\le n}\binom{n}{k}
    \sum_{0\le j\le k}\eulerian{k}{j}v^j.
\]

Furthermore, exchanging the role of binomial and Eulerian numbers 
in the last double sum and dropping $1$ and the multiplicative factor 
$v$ yield the Eulerian-binomial polynomials 
\begin{align}\label{A090582}
    P_n(v) = \sum_{0\le k\le n}\eulerian{n}{k}
    \sum_{0\le j\le k}\binom{k}{j}v^j
\end{align}
whose EGF is $\frac{v}{1+v-e^{vz}}$.
This gives sequence \href{https://oeis.org/A090582}{A090582} and 
$P_n$ satisfies a different type of recurrence
\[
    P_n(v) = ((1+v)n-v)P_{n-1}(v) 
	-v(1+v)P_{n-1}'(v)\qquad(n\ge2),
\]
with $P_1(v)=1$. While the binomial-Eulerian polynomials lead to a CLT
$\mathscr{N}\lpa{\frac12n,\frac1{12}n;n^{-\frac12}}$, the 
Eulerian-binomial ones lead to the CLT 
\[
    \mathscr{N}\left(\frac{2\log 2 -1}{2\log 2}\,n,
    \frac{1-\log 2}{4(\log 2)^2}\,n;n^{-\frac12}\right),
\]
by Theorem~\ref{thm-saqp} with $\rho(v) = \frac{\log(1+v)}v$. 
Replacing $\eulerian{n}{k}$ by $\eulerian{n}{k-1}$ in \eqref{A090582} 
yields \href{https://oeis.org/A130850}{A130850}, and the same 
CLT holds.

\begin{center}
\begin{tabular}{lll}\hline
Fibonacci-Eulerian polynomials &
\href{https://oeis.org/A259708}{A259708} 
& $\mathscr{N}\lpa{\frac12n,\frac1{12}n;n^{-\frac12}}$ \\ 
Derangement polynomial & \makecell[l]{
\href{https://oeis.org/A046739}{A046739} \\
\href{https://oeis.org/A271697}{A271697}} & 
$\mathscr{N}\lpa{\frac12n,\frac1{12}n;n^{-\frac12}}$ \\ 
Binomial-Eulerian polynomial & 
\href{https://oeis.org/A046802}{A046802} 
& $\mathscr{N}\lpa{\frac12n,\frac1{12}n;n^{-\frac12}}$ \\ 
Eulerian-binomial polynomial 
& \href{https://oeis.org/A090582}{A090582} &
$\mathscr{N}\left(\frac{2\log 2 -1}{2\log 2}\,n,
\frac{1-\log 2}{4(\log 2)^2}\,n;n^{-\frac12}\right)$ \\ 
A simple variant of \href{https://oeis.org/A090582}{A090582} 
& \href{https://oeis.org/A130850}{A130850} &
$\mathscr{N}\left(\frac{2\log 2 -1}{2\log 2}\,n,
\frac{1-\log 2}{4(\log 2)^2}\,n;n^{-\frac12}\right)$ \\ \hline
\end{tabular}	
\end{center}

\subsection{Systems of Eulerian recurrences}
The following system of recurrences 
\[
    \left\{
        \begin{split}
            P_n(v) &= (n-1)vQ_{n-1}(v)
            +v(1-v)Q_{n-1}'(v)+vP_{n-1}(v);\\
            Q_n(v) &= (n-1)vP_{n-1}(v)
            +v(1-v)P_{n-1}'(v)+vQ_{n-1}(v),
        \end{split}
    \right.
\]
with $P_0(v)=0$ and $Q_0(v)=1$ appeared in \cite{Mantaci1993} and
enumerates the number of times $\pi(i)\le i$ in permutations
factorizable into odd and even number of transpositions,
respectively; see also \cite{Tanimoto2006}. Since $P_n(v)+Q_n(v)$
equals the Eulerian polynomials, we then consider $P_n-Q_n$ for which
a direct resolution of the corresponding PDE gives the solution ($F$
for $P_n$ and $G$ for $Q_n$)
\[
\left\{\begin{split}
	F(z,v) &= \frac12\cdot\frac{2v-1-ve^{-(1-v)z}}{1-v}
	+\frac12\cdot\frac{1-v}{1-ve^{(1-v)z}}, \\
	G(z,v) &= -\frac12\cdot \frac{1-e^{-(1-v)z}}{1-v}
	+\frac12\cdot\frac{1-v}{1-ve^{(1-v)z}}.
\end{split}\right.
\]
Observe that the first terms on the right-hand side are both
asymptotically negligible. Thus the coefficients follow
asymptotically the same CLT $\mathscr{N}\lpa{\frac12n, \frac1{12}n;
n^{-\frac12}}$.

Another example of a similar type appeared in \cite{Tanimoto2006} of 
the form 
\[
    \begin{cases}
    &P_n(v) = \begin{cases}
        vnP_{n-1}(v)+v(1-v)P_{n-1}'(v)& \\
        \qquad+(vn+1-v)Q_{n-1}(v)+v(1-v)Q_{n-1}'(v),
        &\text{if }n \text{ is even};\\
        (vn+1-v)P_{n-1}(v)+v(1-v)P_{n-1}'(v),
        &\text{if }n \text{ is odd};
    \end{cases}\\
    &Q_n(v) = \begin{cases}
        vnQ_{n-1}(v)+v(1-v)Q_{n-1}'(v)& \\
        \qquad+(vn+1-v)P_{n-1}(v)+v(1-v)P_{n-1}'(v),
        &\text{if }n \text{ is even},\\
        (vn+1-v)Q_{n-1}(v)+v(1-v)Q_{n-1}'(v),
        &\text{if }n \text{ is odd};
    \end{cases}
    \end{cases}
\]
with the initial conditions $P_n(v)=Q_n(v)=0$ for $n<2$, $P_2(v)=v$ 
and $Q_2(v)=1$. The coefficients of $P_n(v)$ and those of $Q_n(v)$ 
correspond to \href{https://oeis.org/A128612}{A128612} and 
\href{https://oeis.org/A128613}{A128613}, respectively, and they 
enumerate ascents in permutations of $n$ elements with an even and 
odd number of inversions, respectively. It is straightforward to 
check that 
\[
    P_n(v) = \frac{A_n(v) +(v-1)^{\tr{\frac12n}}
    A_{\cl{\frac12n}}(v)}{2}\quad\text{and}\quad
    Q_n(v) = \frac{A_n(v) -(v-1)^{\tr{\frac12n}}
    A_{\cl{\frac12n}}(v)}{2}.
\]
Following the same ideas of the method of moments, the terms
$(v-1)^{\tr{\frac12n}} A_{\cl{\frac12n}}(v)$ are asymptotically
negligible because they involve higher order derivatives at $v=1$,
and we get the same $\mathscr{N}\lpa{\frac12n,\frac1{12}n}$ for the
coefficients of both $P_n$ and $Q_n$.

See also \cite{Chow2014a} for the system of recurrences
\[
	\left\{\begin{split}
	    P_n(v) &= (2vn+1-2v)P_{n-1}(v)
	    +4v(1-v)P_{n-1}'(v)+vQ_{n-1}(v);\\
	    Q_n(v) &= (2vn+3-4v)Q_{n-1}(v)
	    +4v(1-v)Q_{n-1}'(v)+P_{n-1}(v),
	\end{split}\right.
\]
with $P_1(v)=Q_1(v)=1$, which is closely connected to
\eqref{chow-ma-2014}. Closed-form expressions for the EGFs of both
recurrences were derived in \cite{Chow2014a}, and from there we can
prove the CLT $\mathscr{N}\lpa{\frac13n, \frac2{45}n;n^{-\frac12}}$
for both recurrences.

\subsection{Recurrences depending on parity}
An example of this type is \href{https://oeis.org/A231777}{A231777}, 
which is more involved than \eqref{A244312} in the Introduction and 
enumerates the number of ascents from odd to even numbers:
\begin{align}\label{A231777}
    P_n(v) = 
    \begin{cases}
        \frac12(1+v)nP_{n-1}(v)
        +v(1-v)P_{n-1}'(v), 
        & \text{if } n \text{ is even};\\
        nP_{n-1}(v)
        +(1-v)P_{n-1}'(v), & \text{if } n \text{ is odd},
    \end{cases}
\end{align}
for $n\ge1$ with $P_0(v)=1$. These relations can be proved as
follows. When $n$ is even, the number of odd-to-even ascents remains
unchanged if $n$ is inserted (into a permutation of $n-1$ elements)
after an even number or between odd-to-even ascents (say $k$ of them)
or in front of all elements; there is a total of $\frac12n+k$ 
of them. Inserting into the remaining $\frac12n-k$ positions adds 
an additional odd-to-even ascent. We then obtain 
\[
    [v^k]P_n(v) = \lpa{\tfrac12n-k}[v^{k-1}]P_{n-1}(v) 
    +\lpa{\tfrac12n+k}[v^k]P_{n-1}(v).
\]
This proves the even case in \eqref{A231777}. The proof for the odd 
case is similar. 

From the previous analysis, the recurrence in the odd case appears
``less normal-like''; compare \eqref{Pn-nn}. However, we can still 
prove the CLT $\mathscr{N}\lpa{\frac18n,\frac{11}{192}n}$ for the 
coefficients of $P_n(v)$, the mean and the variance being equal to 
\[
    \mathbb{E}(X_n) 
    = \begin{cases}
        \frac{n+2}8, &\text{ if }n\text{ is even};\\
        \frac{n^2-1}{8n},&\text{ if }n\text{ is odd},
    \end{cases}
    \quad\text{and}\quad
    \mathbb{V}(X_n)
    = \begin{cases}
        \frac{(n+2)(11n-10)}{192(n-1)},
        &\text{ if }n\text{ is even};\\
        \frac{(n+1)(11n^2-3)}{192n^2},
        &\text{ if }n\text{ is odd}.
    \end{cases}    
\]

A related example is \href{https://oeis.org/A232187}{A232187}, which enumerates descents from odd to 
even numbers in parity alternating permutations:
\begin{align}\label{A232187}
    P_n(v) =\begin{cases}
        nP_{n-1}(v) + (1-v)P_{n-1}'(v), &
        \text{ if }n \text{ is even}; \\
        \ltr{\tfrac n2}!A_{\cl{\frac12n}}(v), &
        \text{ if }n \text{ is odd},
    \end{cases}
\end{align}
with $P_0(v)=1$. To prove these recurrences, we begin with $n$ even.
Insert $n$ at the end of a parity alternating permutation of $n-1$
elements with $k$ odd-to-even descents, which is started and ended
with an odd element. Rotate this permutation cyclically with an
arbitrary shift. Then $k$ such rotations decrease the number of
odd-to-even descents by $1$, while the other $n-k$ ones do not change
the odd-to-even descents count. We thus obtain the recurrence relation
\[
    [v^k]P_n(v)
    = (n-k)[v^k]P_{n-1}(v)+(k+1)[v^{k+1}]P_{n-1}(v),
\]
which proves the first recurrence in \eqref{A232187}. On the other
hand, when $n$ is odd, we construct $\lfloor\frac{n}2\rfloor!$ parity
alternating permutations of size $n$ with $k$ odd-to-even descents
from permutations $(\sigma_1, \sigma_2, \ldots,
\sigma_{\lceil\frac{n}2\rceil})$ of $\lceil\frac{n}2\rceil$ elements
with $k$ exceedances. For any $i=1,2,\ldots,\lfloor\frac{n}2\rfloor$,
construct the blocks $(2\sigma_{i+1}-1, 2i)$. Concatenate these
blocks arbitrarily (there being a total of $\lfloor\frac{n}2\rfloor!$
ways to permutes these blocks), and then append an element
$2\sigma_{\lceil\frac{n}2\rceil}-1$ to the tail, yielding parity
alternating permutations with the required property. Since this
construction is reversible, this proves \eqref{A232187} in the odd
case.

Let $\bar{A}_n(v) := \frac{A_n(v)}{n!}$. Then 
\[
    \frac{P_n(v)}{P_n(1)}
    = \bar{A}_{\cl{\frac12n}}(v)
    +\begin{cases}
        \frac1n(1-v)\bar{A}_{\frac12n}'(v),&
        \text{ if } n\text{ is even};\\
        0, &\text{ if } n\text{ is odd}. 
    \end{cases}
\]
The term $\bar{A}_{\cl{\frac12n}}(v)$ being asymptotically dominant, 
we then deduce the CLT $\mathscr{N}\lpa{\frac14n,\frac1{24}n}$ with 
the mean and the variance given by 
\[
    \mathbb{E}(X_n) 
    = \begin{cases}
        \frac{(n-1)(n-2)}{4n},& \text{ if }n\text{ is even};\\
        \frac{n-1}4,& \text{ if }n\text{ is odd},
    \end{cases}
    \quad\text{and}\quad
    \mathbb{V}(X_n)
    = \begin{cases}
        \frac{(n-2)(n+6)(2n+1)}{48n^2};
        &\text{ if }n\text{ is even},\\
        \frac{n+3}{24};
        &\text{ if }n\ge3\text{ is odd}.
    \end{cases}    
\]

The last example is \href{https://oeis.org/A136718}{A136718}, defined as (properly shifted)
\begin{align*}
    P_n(v) = 
    \begin{cases}
        nP_{n-1}(v) +(1-v)P_{n-1}'(v), & 
        \text{if }n\equiv \{1,2\} \bmod 3;\\
        (vn+1-v)P_{n-1}(v) +v(1-v)P_{n-1}'(v), & 
        \text{if }n\equiv 0 \bmod 3,
    \end{cases}
\end{align*}
with $P_0(v)=1$. The CLT $\mathscr{N}\lpa{\frac16n,\frac1{36}n}$ can 
be established by the method of moments. 

\subsection{$1-v \mapsto 1-sv$}
\label{ss-1-sv}
There exist dozens of examples satisfying a recurrence similar to
\eqref{Pnv-gen} but with ``$1-v$'' replaced by ``$1-sv$'' for some
constant $s>0$. We content ourselves with a brief discussion of some
examples that can be dealt with by simple modifications of our
approach.

\subsubsection{From $\mathscr{N}\lpa{\frac12n,\frac1{12}n}$ to
$\mathscr{N}\lpa{\lpa{2-\frac1{\log2}}n,
\lpa{\frac1{(\log 2)^2}-2}n}$}
Consider \href{https://oeis.org/A156920}{A156920}, which corresponds 
to the recurrence
\[
    P_n(v) = (2vn+1-v) P_{n-1}(v) 
    +v(1-2v)P_{n-1}'(v)\qquad(n\ge1),
\]
with $P_0(v)=1$. This is not of the form \eqref{Pnv-Eabc}, but is so 
after a simple change of variables $R_n(v):=P_n\lpa{\frac12v}$: 
\[
    R_n(v) = \lpa{vn+1-\tfrac12v} R_{n-1}(v) 
    +v(1-v)R_{n-1}'(v)\qquad(n\ge1),
\]
which is then of type $\mathscr{A}(1,1,\frac32)$ in the notation of 
Section~\ref{sec-abv}. By changing back $v\mapsto 2v$, we then obtain 
the EGF for \href{https://oeis.org/A156920}{A156920}
\[
    \sum_{n\ge0}\frac{P_n(v)}{n!}\, z^n
	= e^{(1-2v)z}\left(\frac{1-2v}{1-2ve^{(1-2v)z}}\right)^{\frac32}. 
\]
Note that $\partial_z \mathscr{A}(0,1,\frac12) = 
v\mathscr{A}(1,1,\frac32)$, and the former with $v\mapsto 2v$ 
corresponds to sequence \href{https://oeis.org/A211399}{A211399} 
whose reciprocal is sequence 
\href{https://oeis.org/A102365}{A102365}. 

Although the coefficients of $R_n(v)$ follows the same CLT 
$\mathscr{N}\lpa{\frac12n,\frac1{12}n;n^{-\frac12}}$ as in 
Section~\ref{sec-abv}, those of $P_n$ follow a CLT with 
(see Figure~\ref{fig:1msv})
\[
    \mathbb{E}(X_n)\sim \left(2-\frac1{\log 2}\right)n
	\quad\text{and}\quad
	\mathbb{V}(X_n)\sim \left(\frac1{(\log 2)^2}-2\right)n,
\]
by applying Theorem~\ref{thm-saqp} with $\rho(v) = 
\frac{\log(2v)}{2v-1}$. Numerically, both $2-\frac1{\log 2}\approx 
0.557$ and $\frac1{(\log 2)^2}-2\approx 0.0813$ are close to 
$\frac12$ and $\frac1{12}$, respectively. 

\begin{center}
\begin{tabular}{lll}\hline
\href{https://oeis.org/A156920}{A156920} & $\mathscr{A}(1,1,\frac32; v\mapsto 2v)$ &
$\mathscr{N}\lpa{\lpa{2-\frac1{\log 2}}n, 
\lpa{\frac1{(\log 2)^2}-2}n;n^{-\frac12}}$ \\ 
\href{https://oeis.org/A211399}{A211399} & $\mathscr{A}(0,1,\frac12; v\mapsto 2v)$ &
$\mathscr{N}\lpa{\lpa{2-\frac1{\log 2}}n, 
\lpa{\frac1{(\log 2)^2}-2}n;n^{-\frac12}}$ \\ 
\href{https://oeis.org/A102365}{A102365} & reciprocal of \href{https://oeis.org/A211399}{A211399} &
$\mathscr{N}\lpa{\lpa{\frac1{\log 2}-1}n, 
\lpa{\frac1{(\log 2)^2}-2}n;n^{-\frac12}}$ \\ \hline
\end{tabular}	
\end{center}

More generally, consider the recurrence ($s\in\mathbb{R}^+$)
\[
    P_n(v) = (qsvn+p+s(qr-p-q)v) P_{n-1}(v) 
    +qv(1-sv)P_{n-1}'(v)\qquad(n\ge1),
\]
with $P_0(v)=1$. Then $R_n(v):=P_n\lpa{\frac vs}$ satisfies 
\[
    R_n(v) = \lpa{q vn+p+(qr-p-q)v} R_{n-1}(v) 
    +qv(1-v)R_{n-1}'(v)\qquad(n\ge1),
\]
which is then of type $\mathscr{A}(p, q,r)$. We then deduce that the 
EGF of $P_n$ is given by 
\[
    e^{p(1-sv)z}\left(
	\frac{1-sv}{1-sve^{q(1-sv)z}}
	\right)^r.
\]
It follows, by Theorem~\ref{thm-saqp} with $\rho(v) 
=\frac{-\log(sv)}{q(1-sv)}$, that the CLT 
\begin{align}\label{mv-clt}
    \mathscr{N}\left(
	\left(\frac s{s-1}-\frac1{\log s} \right)n,
	\left(\frac1{\log^2 s}-\frac s{(s-1)^2}\right)n
	;n^{-\frac12} \right)
\end{align}
holds as long as $p\ge0, q,r>0$ and $qr\ge p$. Note that the two 
coefficients (of the mean and the variance) are positive for $s>0$ 
and equal to $\lpa{\frac12,\frac1{12}}$ when $s=1$. 

Some other examples are listed as follows. 
\begin{center}
\begin{tabular}{llll}\hline
\href{https://oeis.org/A141660}{A141660} & $2^k\eulerian{n}{k-1}$
& $\mathscr{A}(0,1,1; v\mapsto 2v)$ &
$\mathscr{N}\lpa{\lpa{2-\frac1{\log 2}}n, 
\lpa{\frac1{(\log 2)^2}-2}n;n^{-\frac12}}$ \\ 
\href{https://oeis.org/A142075}{A142075} & $2^k\eulerian{n}{k}$
& $\mathscr{A}(1,1,2; v\mapsto 2v)$ &
$\mathscr{N}\lpa{\lpa{2-\frac1{\log 2}}n, 
\lpa{\frac1{(\log 2)^2}-2}n;n^{-\frac12}}$ \\ 
\href{https://oeis.org/A156365}{A156365} & $2^k\eulerian{n}{k}$
& $\mathscr{A}(1,1,1; v\mapsto 2v)$ &
$\mathscr{N}\lpa{\lpa{2-\frac1{\log 2}}n, 
\lpa{\frac1{(\log 2)^2}-2}n;n^{-\frac12}}$  \\ 
\href{https://oeis.org/A156366}{A156366} & $3^k\eulerian{n}{k}$
& $\mathscr{A}(1,1,1; v\mapsto 3v)$ &
$\mathscr{N}\lpa{\lpa{\frac32-\frac1{\log 3}}n, 
\lpa{\frac1{(\log 3)^2}-\frac34}n;n^{-\frac12}}$ \\
\href{https://oeis.org/A142963}{A142963} & $(\star)$
& $\mathscr{A}(0,1,\frac12; v\mapsto 4v)$ &
$\mathscr{N}\lpa{\lpa{\frac43-\frac1{2\log 2}}n, 
\lpa{\frac1{4(\log 2)^2}-\frac49}n;n^{-\frac12}}$ \\ \hline
\end{tabular}	
\end{center}
Here $(\star) = [v^k](1-4v)^{n+\frac12}
(v\mathbb{D}_v)^n\frac1{\sqrt{1-4v}}$. 

Along another direction, the $\theta$-derivative polynomials 
\[
    P_n(v) := (1-sv)^{n+r}
	(v\mathbb{D}_v)^n(1-sv)^{-r}\qquad(s>0;r>0),
\]
are of type $\mathscr{A}(0,1,r; v\mapsto sv)$ and satisfy the CLT 
\eqref{mv-clt}. The same CLT holds for the coefficients 
$s^k\eulerian{n}{k}$ with $s>0$. 

\begin{figure}[!h]
\begin{center}\small
\begin{tabular}{c c c c}
\includegraphics[height=3.5cm]{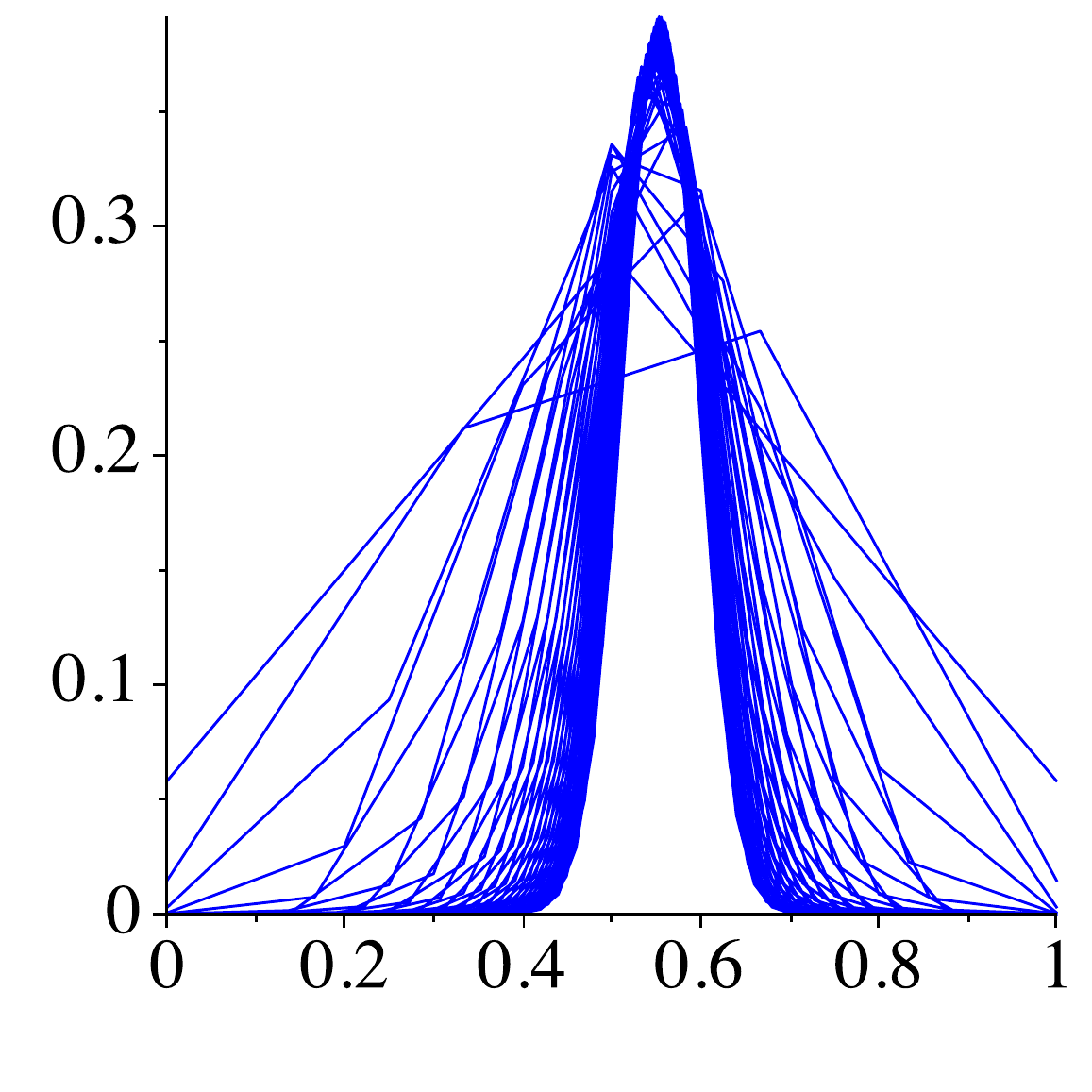} &
\includegraphics[height=3.5cm]{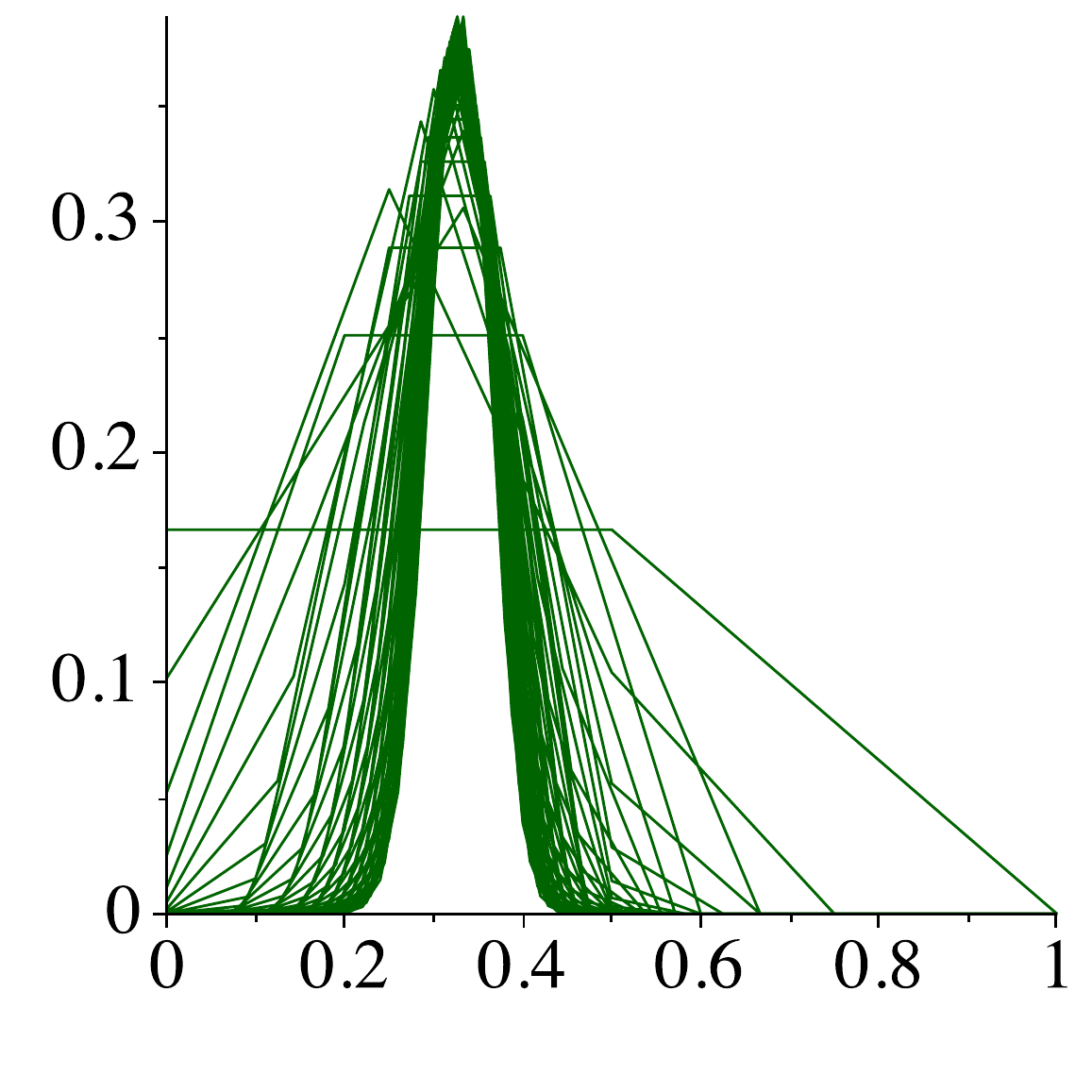} & 
\includegraphics[height=3.5cm]{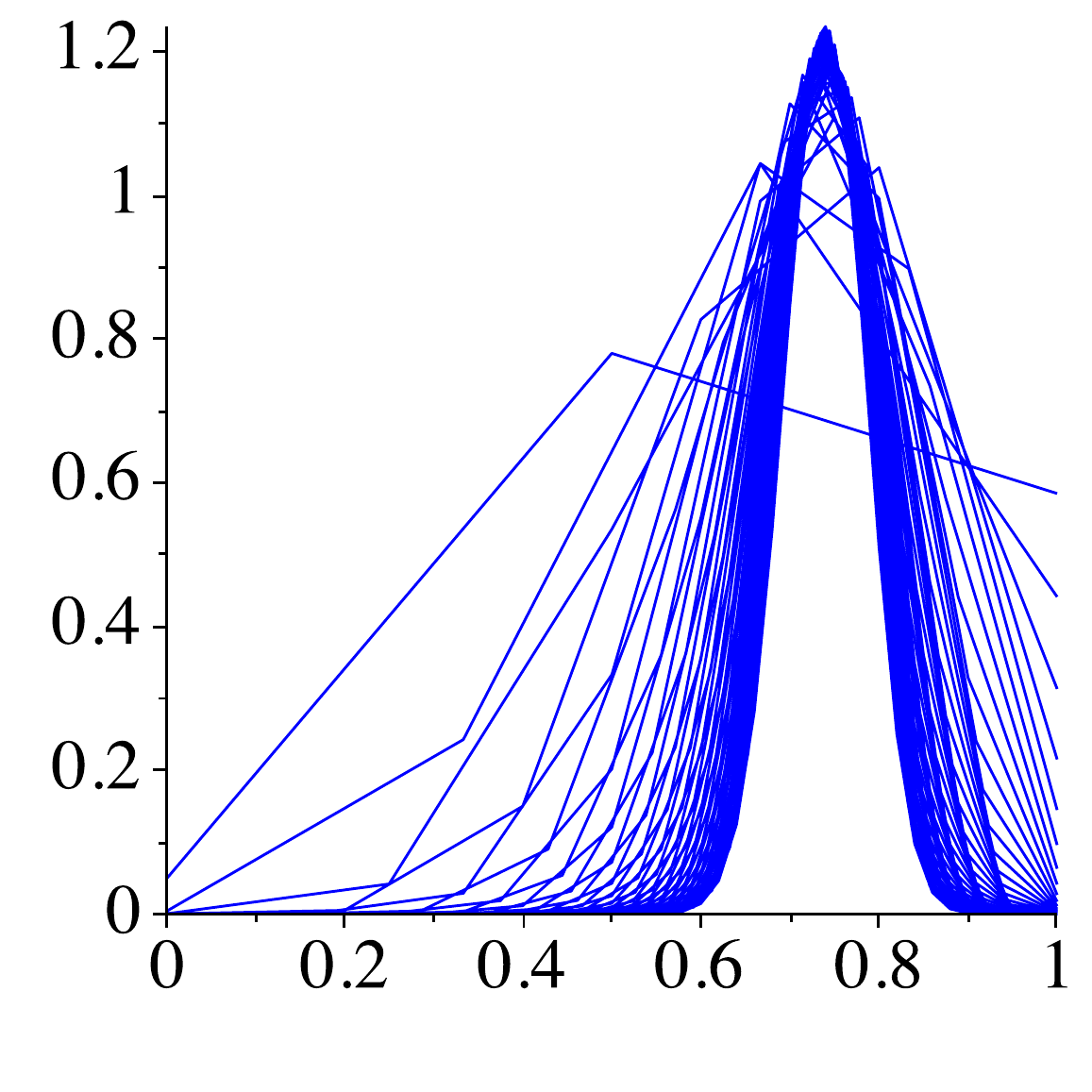} &
\includegraphics[height=3.5cm]{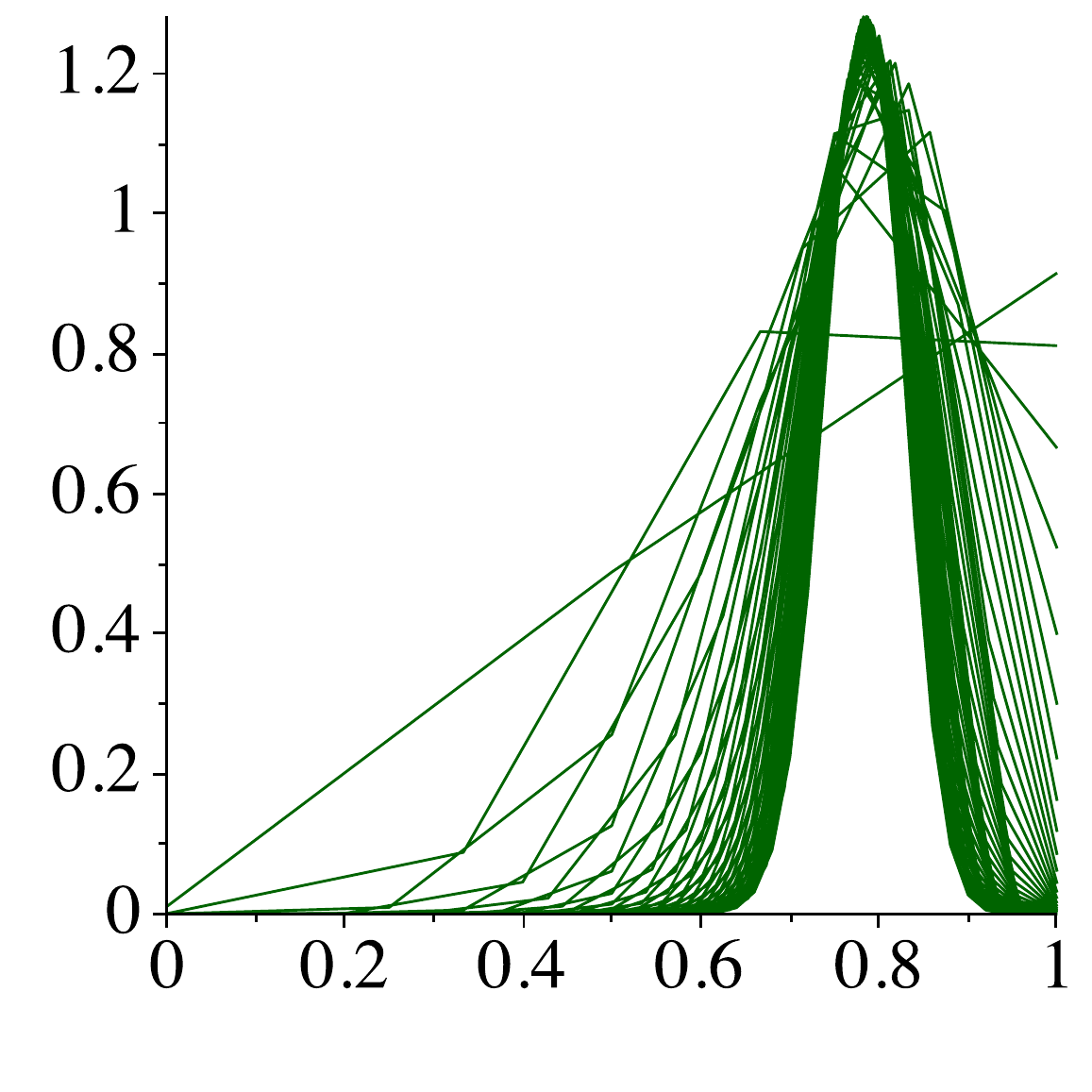} \\
\href{https://oeis.org/A156920}{A156920} &
\href{https://oeis.org/A055151}{A055151} &
\href{https://oeis.org/A290315}{A290315} & 
\href{https://oeis.org/A290316}{A290316} 
\end{tabular}    
\end{center}
\caption{Normalized histograms (by their standard deviations
or by $\sqrt{n}$) of \href{https://oeis.org/A156920}{A156920}, 
\href{https://oeis.org/A055151}{A055151}, 
\href{https://oeis.org/A290315}{A290315},  and 
\href{https://oeis.org/A290316}{A290316} in the unit interval 
(namely, $[v^{\theta n}]P_n(v)$ with $\theta\in[0,1]$).}
\label{fig:1msv}
\end{figure}

\subsubsection{From $\mathscr{N}\lpa{\frac14n,\frac1{16}n}$ to
$\mathscr{N}\lpa{\frac13n,\frac1{18}n}$}

Consider \href{https://oeis.org/A055151}{A055151}, which enumerates 
Motzkin paths of length $n$ with $k$ up steps. This sequence of 
polynomials satisfies the recurrence
\[
    (n+2)P_n(v) = ((1+4v)n+2-4v)P_{n-1}(v)
    +2v(1-4v)P_{n-1}'(v) \qquad(n\ge1),
\]
with $P_0(v)=1$. Changing $v\mapsto \frac14v$ and then considering
the reciprocal, we are led to the polynomials of type
$\mathscr{M}(0,2,\frac32)$ (see \S~\ref{ss-ck}), which has the CLT
$\mathscr{N}\lpa{\frac14n,\frac1{16}n;n^{-\frac12}}$. Reversing these
two steps back and then integrating twice (due to the factor
$(n+2)!$), we deduce that the OGF of $P_n$ is of the form
\[
    \sum_{n\ge0}P_n(v)z^n
	= \frac{1-z-\sqrt{(1-z)^2-4vz^2}}{2vz^2},
\]
yielding the CLT $\mathscr{N}\lpa{\frac13n,\frac n{18};n^{-\frac12}}$
by Theorem~\ref{thm-saqp} with $\rho(v) = \lpa{1+2\sqrt{v}}^{-1}$; 
see Figure~\ref{fig:1msv}. An essentially the same sequence is
\href{https://oeis.org/A080159}{A080159}, and the reciprocal of $P_n$
corresponds to \href{https://oeis.org/A107131}{A107131}.

\begin{center}
\begin{tabular}{lll}\hline
Up steps in Motzkin paths & \href{https://oeis.org/A055151}{A055151} 
($=$\href{https://oeis.org/A080159}{A080159}) & 
$\mathscr{N}\lpa{\frac13n,\frac1{18}n;n^{-\frac12}}$ \\
Reciprocal of \href{https://oeis.org/A055151}{A055151} & 
\href{https://oeis.org/A107131}{A107131} &
$\mathscr{N}\lpa{\frac 23n,\frac1{18}n;n^{-\frac12}}$\\ \hline
\end{tabular}	
\end{center}

Recurrences of a similar form can be found in 
\cite{Ma2013a,Ma2014,Ma2016}.

\subsubsection{From $\mathscr{N}\lpa{\frac23n,\frac19n}$ to
\\ $\mathscr{N}\left(\left(\frac{2q}{q-1}
-\frac{q-1}{q-1-\log q}\right)n,
\left(\frac{(q-1)^2-q(q-1-\log q)}{(q-1-\log q)^2} 
-\frac{2q}{(q-1)^2}\right)n\right)$}

The polynomials defined by (see Section~\ref{sec-barbero} for the 
class $\mathscr{T}$)
\[
    P_n(v) := n![z^n]\mathscr{T}\lpa{q^{-1},2,1; 
    v\mapsto qv,z\mapsto qz}\qquad(q\ge1),
\]
satisfy the recurrence
\[
    P_n(v) = (2q^2vn+1-q(q+1)v)P_{n-1}(v)
    +qv(1-qv)P_{n-1}'(v)\qquad(n\ge1),
\]
with $P_0(v)=1$. The coefficients are nonnegative when $q\ge1$. 
When $q=1$, the $P_n$'s generate the second order Eulerian 
numbers \href{https://oeis.org/A008517}{A008517}, and when $q=2, 3$, 
they correspond to \href{https://oeis.org/A290315}{A290315} and
\href{https://oeis.org/A290316}{A290316}, respectively, which 
appeared in \cite{Lang2017}. Since the EGF of $P_n$ equals (by 
\eqref{T-p2r})
\[
    \left(\frac{T_2\lpa{qve^{-qv+q(1-qv)^2z}}}{qv}\right)^{\frac1q}
    \frac{1-qv}{1-T_2\lpa{qve^{-qv+q(1-qv)^2z}}},
\]
we deduce, by Theorem~\ref{thm-saqp} with $\rho(v) = 
\frac{qv-1-\log qv}{q(qv-1)^2}$, the CLT (see Figure~\ref{fig:1msv})
\[
    \mathscr{N}\left(\left(\frac{2q}{q-1}
    -\frac{q-1}{q-1-\log q}\right)n,
    \left(\frac{(q-1)^2-q(q-1-\log q)}{(q-1-\log q)^2} 
    -\frac{2q}{(q-1)^2}\right)n;n^{-\frac12}\right),
\]
in contrast to the CLT 
$\mathscr{N}\lpa{\frac23n,\frac19n;n^{-\frac12}}$
for $\mathscr{T}\lpa{\frac1q,2,1}$. Note that $q>1$ need not to be an 
integer.
\begin{center}
\begin{tabular}{ll}\hline
\href{https://oeis.org/A290315}{A290315}  & 
$\mathscr{N}\lpa{\frac{3-4\log 2}{1-\log 2}n,
\frac{-5+10\log 2-4\log^22}{(1-\log 2)^2}n;n^{-\frac12}}$ \\ 
\href{https://oeis.org/A290316}{A290316} &
$\mathscr{N}\lpa{\frac{4-3\log 3}{2-\log 3}n,
\frac{-16+18\log 3-3\log^23}{2(2-\log 3)^2}n;n^{-\frac12}}$
\\ \hline
\end{tabular}	
\end{center}

\subsubsection{Non-normal limit laws}
Concrete examples with the factor ``$1-v$'' replaced by ``$1-sv$'' in
the derivative term of \eqref{Pnv-gen} and leading to non-normal
limit laws also exist and most of them are much simpler in nature.
For example, the following sequences all lead to geometric limit laws.
\begin{center}
\begin{tabular}{llll}
\multicolumn{1}{c}{OEIS} &
\multicolumn{1}{c}{$e_n$} & 
\multicolumn{1}{c}{Type} &
\multicolumn{1}{c}{$[v^k]P_n(v)$} \\ \hline	
\href{https://oeis.org/A059268}{A059268} & $n$ & $\ET{n+2v,-v(1-2v)}$ & $2^k$\\ 
\href{https://oeis.org/A152920}{A152920} & $n-1$ & $\ET{n+2v,-v(1-2v)}$
    & $(2n-k)2^{k-1}$\\
\href{https://oeis.org/A118413}{A118413} & $\frac{n(2n-1)}{2n+1}$ & 
$\ET{n+2v,-v(1-2v)}$ & $(2n-1)2^{k-1}$\\
\href{https://oeis.org/A233757}{A233757} & $\frac{n(2^n-1)}{2^{n+1}-1}$ & 
$\ET{n+2v,-v(1-2v)}$  & $(2^n-1)2^{k-1}$\\  
\href{https://oeis.org/A130128}{A130128} & $n$ & $\ET{n+1+2v,-v(1-2v)}$
    & $(n-k+1)2^{k-1}$\\       
\href{https://oeis.org/A100851}{A100851} & $\frac12n$ & $\ET{n+3v,-v(1-3v)}$ & $2^n3^k$\\ \hline
\end{tabular}    
\end{center}

\subsection{P\'olya urn models} \label{ss-polya}
P\'olya's urn schemes \cite{Mahmoud2008} are simple yet very useful 
in many modeling applications and are based on the ball-replacement 
matrix 
\begin{center}
\begin{tabular}{c|cc}
    ball color & white & black \\ \hline
    white & $a$ & $b$ \\
    black & $c$ & $d$    
\end{tabular}    
\end{center}
and the initial configuration with $s_0\ge1$ balls in the urn. At each
stage draw a ball uniformly at random from the urn, and then
return the ball together with $a$ white and $b$ black balls if its
color is white, or $c$ white and $d$ black balls if its color is
black. Repeat this procedure $n$ times and we are interested in the
number $X_n$ of white balls after stage $n$. Assume $a+b=c+d=q\ge1$
and $a\ne c$. Then the probability generating function $W_n(v)$ of
$X_n$ satisfies the recurrence
\[
    W_n(v) = v^c W_{n-1}(v) 
    +\frac{v^{a+1}-v^{c+1}}{s_0+q(n-1)}\, W_{n-1}'(v)
    \qquad(n\ge1),
\]
with $W_0(v) = v^{X_0}$, where $0\le X_0\le s_0$ is a constant. This 
fits into our framework \eqref{Pnv-gen} if we consider $P_n(v) = 
W_n(v) \prod_{0\le j<n}(s_0+qj)$, leading to the recurrence
\begin{align}\label{polya-Pnv}
    P_n(v) = v^c(s_0+q(n-1)) P_{n-1}(v) 
    +\lpa{v^{a+1}-v^{c+1}} P_{n-1}'(v)
    \qquad(n\ge1),
\end{align}
with $P_0(v) = v^{s_0}$, which, in terms of the notations of 
\eqref{Pnv-gen}, gives $\alpha(v)=qv^c$, $\beta(v)=
\frac{v^{a+1}-v^{c+1}}{1-v}$ and $\gamma(v) = (s_0-q)v^c$. To apply 
our Theorem~\ref{thm-clt} on normal limit laws, we require $\alpha(v),
\beta(v)$ and $\gamma(v)$ to be analytic in $|v|\le 1$, which forces 
$a,c\ge-1$. Then the condition \eqref{thm-clt-positivity} becomes
\[
    \alpha(1)+2\beta(1) = a+b+2(c-a) >0
    \quad \Longrightarrow\quad
    \frac{a-c}{a+b}<\frac12,
\]
and 
\[
    \sigma^2 = \frac{(a+b)bc(c-a)^2}
    {(b+c)^2(2c+b-a)}>0,
\]
which requires that $b, c\ge1$ (if both $b,c<0$, then $a>0$, which
would imply $2c+b-a<0$). Thus if 
\[
    a\ge0, a\ne c, a+b=c+d\ge1 \;\text{and} \;
    b,c, 2c+b-a\ge1, 
\]
then the number of white balls follows the CLT
\[
    \mathscr{N}\left(\frac{c(a+b)}{b+c}\,n,
    \frac{(a+b)bc(c-a)^2}
    {(b+c)^2(2c+b-a)}\,n \right).
\]
This result was derived in \cite{Bagchi1985} by the method of moments
but with a manipulation different from ours; see also
\cite{Freedman1965} for the case when $a=d$ and $b=c$. The condition
$a\ge0$ can be relaxed but then additional conditions are needed to
guarantee that $[v^k]P_n(v)\ge0$; see \cite{Bagchi1985} for details.

On the other hand, it is also possible to solve the PDE associated 
with the EGF of \eqref{polya-Pnv}, and we obtain, in particular, 
\[
    \rho(v) = \lpa{1-v^{c-a}}^{-\frac{a+b}{c-a}}
    \int_v^1 t^{-a-1}\lpa{1-t^{c-a}}^{\frac{a+b}{c-a}-1}
    \dd t.
\]
We then deduce not only the same CLT but also a convergence rate. 
See also \cite{Flajolet2006} for an analytic approach, 
\cite{Janson2004,Knape2014} for probabilistic approaches and
\cite{Mahmoud2008} for a general introduction and more information.

If $a=0$, $c=1$, then, with $r=X_0$, $P_n(v)$ is essentially (up to a 
factor $v^r$) of type $\mathscr{T}(r,q,r)$ (see 
Section~\ref{sec-barbero}); in particular, we obtain the Eulerian 
numbers when $q=r=1$. Many other cases (normal or non-normal) can be 
further examined; we omit the details here. 

\subsection{$P_n'(v) = (\alpha(v)n+\gamma(v))P_{n-1}(v)
+\beta(v)(1-v)P_{n-1}'(v)$}

When the left-hand side of the Eulerian recurrence \eqref{Pnv-gen} is
replaced by $P_n'(v)$ (together with some boundary conditions), the
same method of moments still applies, as already described in
\cite{Hitczenko2016}. Note that in such cases, the presence of the
crucial factor ``$1-v$'' in Eulerian recurrences is not essential for
the application of the method of moments. We briefly consider two
examples from \cite{Laborde-Zubieta2015} in the context of tree-like
tableaux; see also \cite{Hitczenko2018}. The first one is of the form
\begin{align}\label{lz}
    P_n'(v) = nP_{n-1}(v)
    +2(1-v)P_{n-1}'(v)\qquad(n\ge1),
\end{align}
with $P_0(v)=1$ and $P_n(1)=n!$, where $[v^k]P_n(v)$ equals the 
number of tree-like tableaux of size $n$ with $k$ occupied corners. 
By a direct calculation of the factorial moments, we see that 
\[
    P_n(v) = \sum_{0\le m\le \cl{\frac12n}}
    \binom{n-m+1}{m}(n-m)!(v-1)^m\qquad(n\ge0). 
\]
Thus the limit law of the coefficients is Poisson$(1)$ because
\[
    \frac{P_n(v)}{P_n(1)}\to e^{v-1}. 
\]

Another sequence of polynomials studied in 
\cite{Hitczenko2016, Laborde-Zubieta2015} is 
\[
    Q_n'(v) = 2vnQ_{n-1}(v)
    +2(1-v^2)Q_{n-1}'(v)\qquad(n\ge1),
\]
with $Q_0(v)=1$ and $Q_n(1)=2^nn!$. Since the $Q_n$'s all contain
even powers of $v$, we consider $R_n(v) = Q_n(\sqrt{v})$, which then 
satisfies the recurrence
\[
    R_n'(v) = nR_{n-1}(v)
    +2(1-v)R_{n-1}'(v)\qquad(n\ge1),
\]
with $R_0(v)=1$ and $R_n(1)=2^nn!$. This is of the same form as 
\eqref{lz} but with a different boundary condition. By the same 
method of moments, we can show that the distribution of the 
coefficients of $R_n$ is asymptotically Poisson$(\frac12)$. 

Interestingly, if we use the boundary condition $P_n(0)=0$ for 
$n\ge1$ instead of $P_n(1)=n!$, then by solving the PDE satisfied by 
the OGF of $P_n$
\[
    z^2 F_z' +zF = (1-2(1-v)z)F_v',
\]
we obtain 
\[
    F(z,v) = \frac{2(1-(1-v)z)}
	{1+\sqrt{1-4z+4(1-v)z^2}}.
\]
This leads instead to the CLT
$\lpa{\frac14n,\frac1{16}n;n^{-\frac12}}$ by Theorem~\ref{thm-saqp}
with $\rho(v)=\frac12\lpa{1+\sqrt{v}}^{-1}$. Also in this case,
$P_n(1)=\frac1{n+1}\binom{2n}n$. This coincides, up to a shift of
indices, \href{https://oeis.org/A091894}{A091894} (Touchard
distribution), which counts particularly the $231$-avoiding
permutations according to the number of peaks. Furthermore,
\[
    [v^kz^n]F(z,v) = \frac{2^{n+1-2k}}{k}
    \binom{n-1}{2k-2}\binom{2k-2}{k-1}
    \qquad(1\le k\le \cl{\tfrac n2}).
\]

\subsection{Extended Eulerian recurrences of Cauchy-Euler type}
\label{ss-ext-ce}
Similar to the Cauchy-Euler differential equations (see
\cite{Chern2002}), the same method of moments can be extended further
to the equi-dimensional Eulerian recurrence,
\begin{align}\label{ER-CE}
    e_nP_n(v) = \sum_{0\le j\le \ell}
    (1-v)^jP_{n-1}^{(j)}(v)\sum_{0\le i\le r-j}
    d_{j,i}(v) n^i\qquad(r\ge\ell\ge1).
\end{align}

Consider first \href{https://oeis.org/A091156}{A091156}, counting big 
ascents in Dyck paths of a given semilength:
\[
    P_n(v) = \frac1{n+1}\sum_{0\le k\le \tr{\frac12n}}
    \binom{n+1}{k}v^k\sum_{0\le j\le n-2k}\binom{k+j-1}{k-1}
    \binom{n+1-k}{n-2k-j}\qquad(n\ge1).
\]
Such $P_n$ satisfies the recurrence
\[
    P_n(v) = \frac{(1+3v)n+1-3v}{n+1}\,P_{n-1}(v)
    +\frac{2(4n-3)v(1-v)}{n(n+1)}P_{n-1}'(v)
    +\frac{4v(1-v)^2}{n(n+1)}\,P_{n-1}''(v),
\]
for $n\ge1$ with $P_0(v)=1$, which can be proved by the OGF 
\[
    \sum_{n\ge0}P_n(v)z^n=
	\frac{1-\sqrt{1-4z+4(1-v)z^2}}{2z(1-(1-v)z)}.
\]
We then deduce the CLT $\mathscr{N}\lpa{\frac14n,\frac1{16}n;
n^{-\frac12}}$ by Theorem~\ref{thm-saqp} with $\rho(v) = 
\frac12\lpa{1+\sqrt{v}}^{-1}$. 

We consider another example (with $\ell=r=2$ in \eqref{ER-CE}) from
Legendre-Stirling permutations \cite{Egge2010}, where an extension of
Eulerian numbers using Legendre-Stirling numbers \cite{Everitt2002} 
was studied:
\[
    D_n(v) := \sum_{j\ge0}d_n(j)v^j 
	= \frac{P_n(v)}{(1-v)^{3n+1}}
    = \frac{v}{1-v}\lpa{vD_{n-1}(v)}''.
\]
Here $d_n(j) = d_n(j-1)+j(j+1)d_{n-1}(j)$ with $d_0(j)=1$ for $j\ge0$ 
and $d_n(0)=0$ for $n\ge1$, and 
\begin{align*}
    P_n(v) &= v(3n-2)(3vn+2-3v)P_{n-1}(v)
	+2v(1-v)(3vn+1-3v)P_{n-1}'(v) \\ &\quad 
	+v^2(1-v)^2P_{n-1}''(v) \qquad(n\ge2),
\end{align*}
with $P_0(v)=1$ and $P_1(v)=2v$. A CLT $\mathscr{N}\lpa{\frac65
n,\frac{36}{175}n}$ for the coefficients of $P_n$ was derived in
\cite{Egge2010} by the real-rootedness approach. The same CLT can aso
be obtained by the method of moments; in particular, the mean and the
variance are given by
\[
    \mathbb{E}(X_n) = \frac{6n-1}{5}
    \quad\text{and}\quad
    \mathbb{V}(X_n) = \frac{9(n-1)(12n+11)}{175(3n-1)}
    \qquad(n\ge1).
\]
However, we have no Berry-Esseen bound because no solution is 
available for the PDE 
\[
    \left(v^{-2}\partial_z^3-z^2\partial_z^2-2z(1-v)
    \partial_{zv}^2 - (1-v)^2\partial_v^2
    -4z\partial_z-2(1-v)\partial_v-2\right)F=0,
\]
satisfied by the EGF $F(z,v) :=
v\sum_{n\ge0}\frac{P_n(v)}{(3n)!}\,z^{3n}$. Note that the 
real-rootedness approach used in \cite{Egge2010} can be refined to 
get the optimal Berry-Esseen bound. 

\subsection{A multivariate Eulerian recurrence}
\label{sec:multi}

Enumerating simultaneously the number of descents $X_n$ in a random 
permutation of $n$ elements and that $Y_n$ of its inverse leads to 
the recurrence for the probability generating function of $X_n$ and 
$Y_n$ (see \cite{Carlitz1966, Petersen2013, Visontai2013})
\begin{align*}
    P_n(v,w) &:= \mathbb{E}\lpa{v^{X_n}w^{Y_n}}\\
	&= \biggl( \frac{(n-1)(1-v)(1-w)}{n^2}
	+vw\left(1+\frac{1-v}n
	\,{\partial_v}\right)
	\left(1+\frac{1-w}{n}\,{\partial_w}
	\right)\biggr)P_{n-1}(v,w),
\end{align*}
for $n\ge1$, with $P_0(v,w)=1$. Recently, Chatterjee and Diaconis 
\cite{Chatterjee2017} proved the CLT $\mathscr{N}\lpa{n,\frac16n}$ 
for $X_n+Y_n$, the total number of descents of a permutation and its 
inverse:
\[
    P_n(v,v) = \mathbb{E}\lpa{v^{X_n+Y_n}}
	= \frac{(1-v)^{2n+2}}{n!}\sum_{j,l\ge0}\binom{jl+n-1}{n}
	v^{j+l}.
\]
This paper also mentions six different ways to prove the CLT for  
Eulerian numbers: sum of $2$-dependent random variables, sum of 
Uniform$[0,1]$ random variables, Harper's real-rootedness (sum of 
Bernoullis), Stein's method, Bender's analytic method and the method 
of moments, but none of the six applies to the coefficients of 
$[v^k]P_n(v,v)$; see also \cite{Ozdemir2019}. 

While a direct use of the method of moments fails, we show that it is
possible to extend the method to establish the CLT for $X_n+Y_n$; in
particular, we derive the asymptotics of the central moments
$\mathbb{E}(\bar{X}_n+\bar{Y}_n)^m$ through those of the joint
moments $\mathbb{E}(\bar{X}_n^j\bar{Y}_n^{m-j})$, where $\bar{X}_n :=
X_n-\frac{n+1}2$ and $\bar{Y}_n := Y_n-\frac{n+1}2$, so that
$\mathbb{E}(\bar{X}_n)=\mathbb{E}(\bar{Y}_n)=0$. For that purpose, we
define
\[
    Q_n(s,t) := \exp\left(-\frac{n+1}2\,s-\frac{n+1}2\,t\right)
	P_n(e^s,e^t),
\]
which satisfies the recurrence
\begin{equation}\label{Qn-rr}
\begin{split}
    Q_n(s,t) &= \biggl(\frac{4(n-1)}{n^2}\,\sinh\lpa{\tfrac12s}
	\sinh\lpa{\tfrac12t}\biggr)Q_{n-1}(s,t)\\
	&\quad 
	+\left(\cosh\lpa{\tfrac12s}-\frac{2}{n}\sinh\lpa{\tfrac12s}
	\,{\partial_s}\right)
	\left(\cosh\lpa{\tfrac12t}-\frac{2}{n}\sinh\lpa{\tfrac12t}
	\,{\partial_t}\right)Q_{n-1}(s,t),
\end{split}
\end{equation}
for $n\ge1$, with $Q_0(s,t)=e^{-\frac12s-\frac12t}$ and $Q_1(s,t)=1$. 
Write now 
\[
    Q_n(s,t) = 1+\sum_{m+l\ge2}Q_{n;m,l}
	\,\frac{s^mt^l}{m!l!}\,
	\qquad(n\ge1),
\]
with $Q_{0;m,l}=(-1)^{m+l}2^{-m-l}$. Then by the recurrence 
\eqref{Qn-rr} and induction, we see that 
\[
    Q_{n;l,2m+1-l}= 0 \qquad(n\ge1;0\le l\le 2m+1).
\]
To compute the asymptotics of $Q_{n;l,2m-l}$, we use the recurrence 
\[
    Q_{n;m,l} = \left(1-\frac{m}{n}\right)
    \left(1-\frac{l}{n}\right)Q_{n-1;m,l}
    +R_{n;m,l},
\]
where
\begin{equation}\label{rnml}
    \begin{split}
	R_{n;m,l}&=\frac{n-1}{n^2}
	\sum_{\substack{0\le i\le \tr{\frac12(m-1)}
	\\ 0\le j\le \tr{\frac12(l-1)}}} 
	\binom{m}{2i+1}\binom{l}{2j+1}
	2^{-2i-2j}Q_{n-1;m-2i-1,l-2j-1}\\
	&\quad + \sum_{\substack{0\le i\le \tr{\frac12m}
	\\ 0\le j\le \tr{\frac12l} \\ i+j\ge 1}} 
	\binom{m}{2i}\binom{l}{2j}
	2^{-2i-2j}Q_{n-1;m-2i,l-2j}
	\left(1-\frac{m-2i}{n(2i+1)}\right)
	\left(1-\frac{l-2j}{n(2j+1)}\right).
    \end{split}
\end{equation}
Then by induction, we show that 
\begin{align}\label{qml-ae}
    Q_{n;2m-l,l} \sim d(l)d(2m-l)\sigma_n^{2m}
    \qquad(0\le l\le 2m; m\ge0),
\end{align}
where $\sigma_n^2 := \frac1{12}n$, $d(2j+1)=0$ and 
$d(2j)=(2j-1)!!=\frac{(2j)!}{j!2^j}$. See Appendix~\ref{App-A} for 
details. By the expansion   
\[
    \mu_m := \mathbb{E}\lpa{\bar X_n+\bar Y_n}^m
    = \sum_{0\le l\le m}\binom{m}{l}Q_{n;m-l,l},
\]
we see that $\mu_{2m+1}=0$ because $Q_{n;2m+1-l,l}=0$; furthermore,
using the estimate \eqref{qml-ae}, we deduce that 
\[
    \mu_{2m} \sim \sigma_n^{2m} 
    \sum_{0\le l\le m}\binom{2m}{2l}d(2l)d(2m-2l)
    = d(2m) 2^m \sigma_n^{2m}.
\]
We then conclude that $X_n+Y_n\sim \mathscr{N}\lpa{n,\frac16n}$. 

A type $B$ analogue is given in \cite{Visontai2013} and the same CLT 
$\mathscr{N}\lpa{n,\frac16n}$ can be established by the same approach.

\section{The degenerate case: $\beta(v)\equiv0\Longrightarrow P_n(v) 
= a_n(v)P_{n-1}(v)$}\label{sec-beta-is-0}

For completeness, we briefly discuss a special class of Eulerian 
recurrences of the form (without derivative terms)
\begin{align}\label{beta-is-zero}
	P_n(v) = a_n(v)P_{n-1}(v)\qquad(n\ge1),
\end{align}
with $P_0(v)$ given. Typical examples include binomial coefficients
with $a_n(v) = 1+v$ and Stirling numbers of the first kind with
$a_n(v) = n-1+v$. Assume that $[v^k]a_n(v)\ge0$, $a_n(v)$ is analytic
in $|v|\le1$ and $a_n(1)>0$ for $k,n\ge0$. Define $X_n$ as in
\eqref{Xnk-general}. Then, with $a_0(v):= P_0(v)$, $X_n$ is
expressible as the sum of independent random variables:
\[
    X_n = \sum_{0\le j\le n}Y_{j},
    \quad\text{where}\quad 
    \mathbb{E}\lpa{v^{Y_{j}}}
    = \frac{a_j(v)}{a_j(1)}. 
\]
Thus $X_n$ is asymptotically normally distributed if the Lyapunov 
condition (see \cite{Fischer2011}) holds:
\[
    \sum_{0\le j\le n}\mathbb{E}|Y_j-\mathbb{E}(Y_j)|^3
    = o(\mathbb{V}(X_n)^{3/2}).
\]
This condition is not optimal but is simpler to use in a setting like 
ours. In particular, it holds when each $Y_j$ is bounded. 

\paragraph{A simple linear framework}
To be more precise, we consider the linear framework when 
$a_n(v) = \alpha(v)n+\gamma(v)$, where $\alpha$ and $\gamma$ are in 
most cases polynomials. Then we have
\[
    \mathbb{E}(X_n) = \frac{P_0'(1)}{P_0(1)}
	+\sum_{1\le j\le n}\frac{\alpha'(1)j+\gamma'(1)}
	{\alpha(1)j+\gamma(1)}.
\]
It follows that 
\begin{align*}
	\mathbb{E}(X_n) 
	= \begin{cases}
	    \mu(\alpha) n +\nu \log n +O(1), &
		\text{if }\mu(\alpha)>0;\\
		\frac{\gamma'(1)}{\alpha(1)}\log n+O(1),&
		\text{if } \mu(\alpha)=0, \alpha(1),\gamma'(1)>0;\\
	    \mu(\gamma) n +O(1),&
		\text{if }\mu(\alpha)=\alpha(1)=0,
		\mu(\gamma)>0,\\
	\end{cases}
\end{align*}
where 
\[
    \mu(f) := \frac{f'(1)}{f(1)}, \quad\text{and}\quad
    \nu := \frac{\alpha(1)\gamma'(1)-\alpha'(1)\gamma(1)}
    {\alpha(1)^2}. 
\]
For the variance, with the notation 
\[
    \sigma^2(f) := \frac{f'(1)}{f(1)}+\frac{f''(1)}{f(1)}
    -\left(\frac{f'(1)}{f(1)}\right)^2,
\]
we have
\begin{align*}
	\mathbb{V}(X_n) 
	= \begin{cases}
	    \sigma^2(\alpha) n 
		+O(\log n), &
		\text{if }\sigma^2(\alpha)>0;\\
		\varsigma\log n+O(1),&
		\text{if } \sigma^2(\alpha)=0, \alpha(1),\varsigma>0;\\
	    \sigma^2(\gamma) n +O(1),&
		\text{if }\sigma^2(\alpha)=\alpha(1)=0,
		\sigma^2(\gamma)>0,\\
	\end{cases}
\end{align*}
where 
\[
    \varsigma := \frac{\gamma'(1)+\gamma''(1)}{\alpha(1)}
    -\frac{2\alpha'(1)\gamma'(1)}{\alpha(1)^2}
    +\frac{\gamma(1)\alpha'(1)^2}{\alpha(1)^3}. 
\]
In all cases, the distribution of $X_n$ is asymptotically normal if 
$\mathbb{V}(X_n)\to\infty$:
\begin{align*}
	X_n\sim \begin{cases}
		\mathscr{N}\lpa{\mu(\alpha)n,\sigma^2(\alpha)n},&
		\text{if }\sigma^2(\alpha)>0;\\
		\mathscr{N}\lpa{\mu(\alpha)n+\nu\log n,
		\varsigma \log n},&
		\text{if }\sigma^2(\alpha)=0, \varsigma>0;\\
		\mathscr{N}\lpa{\mu(\gamma)n,
		\sigma^2(\gamma)n},&
		\text{if }\sigma^2(\alpha)=\alpha(1)=0,
		\sigma^2(\gamma)>0.
	\end{cases}
\end{align*}

\paragraph{Applications}
The literature and the database OEIS abound with examples satisfying
\eqref{beta-is-zero}, and they are mostly of a simpler nature when
compared with \eqref{Pnv-gen}. The prototypical example is binomial
coefficients $\binom{n}{k}$: \href{https://oeis.org/A007318}{A007318}
(or \href{https://oeis.org/A135278}{A135278}) for which $a_n(v) =
1+v$. We then obtain the CLT $\mathscr{N}\lpa{\frac12n,\frac14n}$, a
result first established by de Moivre in 1738 \cite{deMoivre1738}.
Another 80 OEIS sequences of the form \eqref{beta-is-zero} with
$a_n(v) = e_n(1+v)$ are collected in Appendix~\ref{App-Bino}, where
$e_n$ is either a constant or a sequence of $n$. We get the same CLT
$\mathscr{N}\lpa{\frac12n,\frac14n}$ for the coefficients.

We also identified another 182 sequences satisfying 
\eqref{beta-is-zero} with $a_n(v) = p+qv+rv^2$ with $p, q, r$
nonnegative integers. The corresponding coefficients follow the CLT 
\[
    \mathscr{N}\left(\frac{q+2r}{p+q+r}\,n,
    \frac{pq+4pr+qr}{(p+q+r)^2}\,n\right) ; 
\]
see Appendix~\ref{App-Bino} for the tables of these sequences.

Examples for which $\alpha(v), \sigma^2(\alpha)\neq0$ are scarce: 
\begin{center}
\begin{tabular}{ccc}\hline
\href{https://oeis.org/A059364}{A059364}$(n,k)$ $=\sum_{k\le j<n}\binom{j}{k}\stirling{n}{n-j}$
& $a_n(v)=(1+v)n+1$ & $\mathscr{N}\lpa{\frac12n,\frac14n}$\\
\href{https://oeis.org/A088996}{A088996}: reciprocal of \href{https://oeis.org/A059364}{A059364} & $(1+v)n-1$ & 
$\mathscr{N}\lpa{\frac12n,\frac14n}$  \\ 
\href{https://oeis.org/A322225}{A322225} & $(1+v^2)n+v$ & 
$\mathscr{N}\lpa{n,n}$\\
\href{https://oeis.org/A322235}{A322235} & $(1+2v^2)n+v$ & 
$\mathscr{N}\lpa{\frac43n,\frac89n}$\\
\hline
\end{tabular}	
\end{center} 
Here $\stirling{n}{k}$ denotes the unsigned Stirling numbers of the
first kind (\href{https://oeis.org/A132393}{A132393},
\href{https://oeis.org/A094638}{A094638},
\href{https://oeis.org/A130534}{A130534}), another prototypical 
example with log-variance CLT. 

We now group other examples with logarithmic variance according as
$\mu(\alpha)>0$ or $\mu(\alpha)=0$, respectively.

\paragraph{Polynomials with $\mu(\alpha)>0$ and 
$\sigma^2(\alpha)=0$, and $\varsigma>0$}

\begin{footnotesize}
\begin{center}
\begin{longtable}{clll }
\multicolumn{1}{c}{OEIS} &
\multicolumn{1}{c}{$a_n(v)$} &
\multicolumn{1}{c}{Initial} &
\multicolumn{1}{c}{CLT} \\ \hline
\href{https://oeis.org/A094638}{A094638} & $v n+1$ & $P_{0}(v)=1$ & \nom{n-\log n}{ \log n} \\
\href{https://oeis.org/A109692}{A109692} & $2 v n +1-v$ & $P_{0}(v)=1$ & 
\nom{n-\frac{1}{2}\log n}{\frac{1}{2} \log n} \\
\href{https://oeis.org/A145324}{A145324} & $v n+1+v$ & $P_{0}(v)=1$ & \nom{n-\log n}{ \log n} \\
\href{https://oeis.org/A196841}{A196841} & $v n+1+v$ & $P_{1}(v)=1+v$ & \nom{n-\log n}{ \log n} \\
\href{https://oeis.org/A196842}{A196842} & $v n+1+v$ & $P_{2}(v)=1+3 v+2v^2$ 
& \nom{n-\log n}{ \log n} \\
\href{https://oeis.org/A196843}{A196843} & $v n+1+v$ & $P_{3}(v)=1+6 v+11 v^2+6 v^3$ 
& \nom{n-\log n}{ \log n} \\
\href{https://oeis.org/A196844}{A196844} & $v n+1+v$ & $P_{4}(v)=1+10 v+35 v^2+50 v^3 + 24 v^4$ 
& \nom{n-\log n}{ \log n} \\
\href{https://oeis.org/A196845}{A196845} & $v n+1+2 v$ & $P_{0}(v)=1$ & \nom{n-\log n}{ \log n} \\
\href{https://oeis.org/A196846}{A196846} & $v n+1+2 v$ & $P_{2}(v)=1+3 v+2 v^2$ 
& \nom{n-\log n}{ \log n} \\
\href{https://oeis.org/A201949}{A201949} & $v n+1-v+v^2$ & $P_{0}(v)=1$ & \nom{n}{ \log n} \\
\href{https://oeis.org/A249790}{A249790} & $v n+1+v^2$ & $P_{0}(v)=1$ & \nom{n}{ \log n} \\
\href{https://oeis.org/A291845}{A291845} & $2 v n+1-v+v^2$ & $P_{0}(v)=1$ & \nom{n}{ \log n} \\ 
\href{https://oeis.org/A324960}{A324960} & $v n+1+2v+v^2$ & $P_{0}(v)=1$ & \nom{n}{2\log n} \\ 
\hline
\end{longtable}
\end{center}
\end{footnotesize} 

\paragraph{Polynomials with $\mu(\alpha)=\sigma^2(\alpha)=0$, and 
$\varsigma>0$}
\begin{footnotesize}
\begin{center}
\begin{longtable}{llll}
\multicolumn{1}{c}{OEIS} &
\multicolumn{1}{c}{$a_n(v)$} &
\multicolumn{1}{c}{Initial} &
\multicolumn{1}{c}{CLT} \\ \hline	
\href{https://oeis.org/A028338}{A028338} & $2 n-1+v$ & $P_{0}(v)=1$
 & \nom{\frac{1}{2} \log n}{\frac{1}{2} \log n} \\
\href{https://oeis.org/A125553}{A125553} & $n+2 v$ & $P_{0}(v)=2$
 & \nom{2 \log n}{2 \log n} \\
\href{https://oeis.org/A130534}{A130534} & $n+v$ & $P_{0}(v)=1$
 & \nom{ \log n}{ \log n} \\
\href{https://oeis.org/A132393}{A132393} & $n-1+v$ & $P_{0}(v)=1$
 & \nom{ \log n}{ \log n} \\
\href{https://oeis.org/A136124}{A136124} & $n+1+v$ & $P_{0}(v)=1$
 & \nom{ \log n}{ \log n} \\
\href{https://oeis.org/A137320}{A137320} & $n-1+2 v$ & $P_{0}(v)=1$
 & \nom{2 \log n}{2 \log n} \\
\href{https://oeis.org/A137339}{A137339} & $n-1+3 v$ & $P_{0}(v)=1$
 & \nom{3 \log n}{3 \log n} \\
\href{https://oeis.org/A143491}{A143491} & $n+1+v$ & $P_{0}(v)=1$
 & \nom{ \log n}{ \log n} \\
\href{https://oeis.org/A143492}{A143492} & $n+2+v$ & $P_{0}(v)=1$
 & \nom{ \log n}{ \log n} \\
\href{https://oeis.org/A143493}{A143493} & $n+3+v$ & $P_{0}(v)=1$
 & \nom{ \log n}{ \log n} \\
\href{https://oeis.org/A161198}{A161198} & $2 n-1+2 v$ & $P_{0}(v)=1$
 & \nom{ \log n}{ \log n} \\
\href{https://oeis.org/A180013}{A180013} & $\frac{1+n}{n} \left(n-1+v\right)$ & $P_{0}(v)=1$
 & \nom{ \log n}{ \log n} \\
\href{https://oeis.org/A204420}{A204420} & $ \left(2 n-1\right) \left(2 n-2+v\right)$ & $P_{0}(v)=1$
 & \nom{\frac{1}{2} \log n}{\frac{1}{2} \log n} \\
\href{https://oeis.org/A216118}{A216118} & $\frac{n+3}{n-1} \left(n+v\right)$ & $P_{1}(v)=1+v$
 & \nom{ \log n}{ \log n} \\
\href{https://oeis.org/A225470}{A225470} & $3 n-1+v$ & $P_{0}(v)=1$
 & \nom{\frac{1}{3} \log n}{\frac{1}{3} \log n} \\
 \href{https://oeis.org/A286718}{A286718} & $3 n-2+v$ & $P_{0}(v)=1$
 & \nom{\frac{1}{3} \log n}{\frac{1}{3} \log n} \\
\href{https://oeis.org/A225471}{A225471} & $4 n-1+v$ & $P_{0}(v)=1$
 & \nom{\frac{1}{4} \log n}{\frac{1}{4} \log n} \\
 \href{https://oeis.org/A290319}{A290319} & $4 n-3+v$ & $P_{0}(v)=1$
  & \nom{\frac{1}{4} \log n}{\frac{1}{4} \log n} \\
\href{https://oeis.org/A225477}{A225477} & $3 n-1+3 v$ & $P_{0}(v)=1$
 & \nom{ \log n}{ \log n} \\
\href{https://oeis.org/A225478}{A225478} & $4 n-1+4 v$ & $P_{0}(v)=1$
 & \nom{ \log n}{ \log n} \\
\href{https://oeis.org/A254881}{A254881} & $ \left(n-1+v\right) \left(n+v\right)$ & $P_{0}(v)=1$
 & \nom{2 \log n}{2 \log n} \\
\end{longtable}	
\end{center}	
\end{footnotesize} 

Historically, the Stirling numbers of the first kind numbers were 
found as early as the 17th century in Thomas Harriot's unpublished 
manuscripts in addition to James Stirling's book \emph{Methodus 
Differentialis} published in 1730; see \cite[p.\ 61]{Beery2009} and 
\cite{Knuth1992} for more historical notes. The CLT for 
$\stirling{n}{k}$ first appeared in Goncharov's 1942 paper 
\cite{Goncharov1942} (see also \cite{Feller1945,Goncharov1944}) in 
the form of cycles in permutations. 

\paragraph{$a_n(v)$ depending on the parity of $n$} 
\begin{footnotesize}
\begin{center}
\begin{longtable}{lllll}
\multicolumn{5}{c}{{}} \\
\multicolumn{1}{c}{OEIS} &
\multicolumn{1}{c}{$a_n(v)$ ($n$ odd)} &
\multicolumn{1}{c}{$a_n(v)$ ($n$ even)} &
\multicolumn{1}{c}{Initial} &
\multicolumn{1}{c}{CLT} \\ \hline
\endhead
\multicolumn{5}{c}{{Continued on next page}} \\
\endfoot
\endlastfoot
\href{https://oeis.org/A060523}{A060523} & $n$ & $n-1+v$ & $P_{0}(v)=1$ 
& \nom{\frac{1}{2}\log n}{\frac{1}{2}\log n} \\
\href{https://oeis.org/A064861}{A064861} & $1+2 v$ & $1+v$ & $P_{0}(v)=1$ 
& \nom{\frac{7}{12}n}{\frac{17}{72}n} \\ 
\href{https://oeis.org/A152815}{A152815} & $1$ & $1+v$ & $P_{0}(v)=1$ 
& \nom{\frac{1}{4}n}{\frac{1}{8}n} \\  
\href{https://oeis.org/A152842}{A152842} & $1+3 v$ & $1+v$ & $P_{0}(v)=1$ 
& \nom{\frac{5}{8}n}{\frac{7}{32}n} \\ 
\href{https://oeis.org/A188440}{A188440} & $1$ & $1+2 v$ & $P_{0}(v)=1$ 
& \nom{\frac{1}{3}n}{\frac{1}{9}n} \\ 
\href{https://oeis.org/A246117}{A246117} & $\frac12(n-1)+v$ & $\frac12n+v$ & $P_{0}(v)=1$ 
& \nom{2\log n}{2\log n} \\
\href{https://oeis.org/A274496}{A274496} & $2$ & $1+v$ & $P_{0}(v)=1$ 
& \nom{\frac{1}{4}n}{\frac{1}{8}n} \\  
\href{https://oeis.org/A274498}{A274498} & $3$ & $1+2 v$ & $P_{0}(v)=1$ 
& \nom{\frac{1}{3}n}{\frac{1}{9}n} \\ 
\href{https://oeis.org/A026519}{A026519} & $1+v+v^2$ & $1+v^2$ & $P_{0}(v)=1$ 
& \nom{n}{\frac{5}{6}n} \\  
\href{https://oeis.org/A026536}{A026536} & $1+v^2$ & $1+v+v^2$ & $P_{0}(v)=1$ 
& \nom{n}{\frac{5}{6}n} \\  
\href{https://oeis.org/A026552}{A026552} & $1+v^2$ & $1+v+v^2$ & $P_{1}(v)=1+v+v^2$ 
& \nom{n}{\frac{5}{6}n} \\  \hline
\end{longtable}	
\end{center}	
\end{footnotesize} 

\paragraph{Nonlinear $a_n(v)$} Let $p_n$ denote the $n$th prime and 
$f_n$ the $n$th Fibonacci number. Then by the prime number theorem it 
is known that $p_n \sim n\log n$; also $f_n \sim 
5^{-\frac12}\phi^{-n-1}$, where $\phi = \frac{\sqrt{5}-1}2$ is the 
golden ratio. Then the following CLTs follow from these estimates and 
Lyapunov's condition. 
\begin{center}
\begin{tabular}{llll}
\multicolumn{1}{c}{OEIS} &
\multicolumn{1}{c}{$a_n(v)$} &
\multicolumn{1}{c}{Initial} &
\multicolumn{1}{c}{CLT} \\ \hline
\href{https://oeis.org/A096294}{A096294} & $p_n-1+v$ & $P_0(v)=1$ &
$\mathscr{N}(\log\log n,\log\log n)$ \\
\href{https://oeis.org/A260613}{A260613} & $1+p_nv$ & $P_0(v)=1$ &
$\mathscr{N}(n-\log\log n,\log\log n)$ \\	
\href{https://oeis.org/A130405}{A130405} & $f_n+f_{n-1}v$ & $P_0(v)=1$ &
$\mathscr{N}\lpa{\frac{\phi}{1+\phi}n,
\frac{\phi}{(1+\phi)^2}n}$ \\ \hline
\end{tabular}	
\end{center}

\paragraph{A CLT $\mathscr{N}\lpa{\frac12n,\frac1{6}n}$}
Sequence \href{https://oeis.org/A220884}{A220884} is not of the type 
\eqref{beta-is-zero} but has a similar product form
\[
    P_n(v) = \prod_{1\le j<n}(jv+n+1-j),
\]
which leads to the CLT $\mathscr{N}\lpa{\frac12n,\frac1{6}n}$.

\paragraph{Non-normal limit laws} Non-normal limit laws arise when 
the variance remains bounded and the analysis is simple because the 
probability generating function (PGF) tends to a finite limit. 
Consider the case when $a_n(v) = e_n + v$, where 
$\sum_{j\ge1}e_j^{-1}$ is convergent. Then 
\[
    \mathbb{E}\lpa{v^{X_n}}
	= \frac{P_n(v)}{P_n(1)}
	= \prod_{1\le j\le n}\frac{e_j+v}
	{e_j+1}
	\to \prod_{j\ge1}\frac{1+\frac{v}{e_j}}
	{1+\frac{1}{e_j}}.
\]
When $a_n(v) = e_nv+1$, we consider $n-X_n$, and we get the same 
limit law. Some examples of these types are collected in the  
following table ($P_0(v)=1$ in all cases). 
\begin{footnotesize}
\begin{center}
\begin{tabular}{lll|lll}
\multicolumn{1}{c}{OEIS} &
\multicolumn{1}{c}{$a_n(v)$} &
\multicolumn{1}{c}{$\begin{array}{c} 
\text{PGF of the}\\
\text{limit law}	
\end{array}$} &
\multicolumn{1}{c}{OEIS} &
\multicolumn{1}{c}{$a_n(v)$} &
\multicolumn{1}{c}{$\begin{array}{c} 
\text{PGF of the}\\
\text{limit law}	
\end{array}$} 
\\ \hline	
\href{https://oeis.org/A008955}{A008955} & $vn^2+1$ & 
$\frac{\sinh\lpa{\pi\sqrt{v}}}{\sqrt{v}\sinh(\pi)}$ & 
\href{https://oeis.org/A008956}{A008956} & $v(2n-1)^2+1$	& 
$\frac{\cosh\lpa{\frac12\pi\sqrt{v}}}
{\cosh\lpa{\frac12\pi}}$ \\
\href{https://oeis.org/A108084}{A108084} & $2^n+v$ & 
$\frac{\prod_{j\ge1}\lpa{1+2^{-j}v}}
{\prod_{j\ge1}\lpa{1+2^{-j}}}$ & 
\href{https://oeis.org/A128813}{A128813} & $\frac12vn(n+1)+1$ &
$\frac{\cos\lpa{\frac12\pi\sqrt{1-8v}}}
{v\cosh\lpa{\frac{\sqrt{7}}2\pi}}$ \\
\href{https://oeis.org/A160563}{A160563} & $(2n-1)^2+v$ & 	
$\frac{\cosh\lpa{\frac{\pi}2\sqrt{v}}}
{\cosh\lpa{\frac{\pi}2}}$ &
\href{https://oeis.org/A173007}{A173007} & $3^n+v$ & $\frac{\prod_{j\ge1}\lpa{1+3^{-j}v}}
{\prod_{j\ge1}\lpa{1+3^{-j}}}$ \\
\href{https://oeis.org/A173008}{A173008} & $4^n+v$ & $\frac{\prod_{j\ge1}\lpa{1+4^{-j}v}}
{\prod_{j\ge1}\lpa{1+4^{-j}}}$ &
\href{https://oeis.org/A249677}{A249677} & $vn^3+1$ & $\frac{\prod_{j\ge1}\lpa{1+j^{-3}v}}
{\prod_{j\ge1}\lpa{1+j^{-3}}}$ \\
\href{https://oeis.org/A269944}{A269944} & $(n-1)^2+v$ & $\frac{\sinh\lpa{\pi\sqrt{v}}}
{\sqrt{v}\sinh(\pi)}$ &
\href{https://oeis.org/A269947}{A269947} & $(n-1)^3+v$ & $v\frac{\prod_{j\ge1}\lpa{1+j^{-3}v}}
{\prod_{j\ge1}\lpa{1+j^{-3}}}$ \\ \hline
\end{tabular}	
\end{center}	
\end{footnotesize}

\section{Conclusions}
\label{sec-conclusions}

In connecting Eulerian numbers to descents in permutations in the
preface of Petersen's book \cite{Petersen2015}, Richard Stanley
writes: ``\emph{Who could believe that such a simple concept would
have a deep and rich theory, with close connections to a vast number
of other subjects?}'' We demonstrated in this paper, through a
considerable number (more than 500) of examples from the literature
and the OEIS database, that not only have the Eulerian numbers been 
very fruitfully explored, but its simple extension to Eulerian
recurrences is very effective and powerful in modeling many different
laws---a prolific source of various phenomena indeed, although we
limited our study mostly to linear (in $n$) factors $a_n(v)$ and
$b_n(v)$. The combined use of an elementary approach (method of
moments) and an analytic one (notably Theorem~\ref{thm-saqp}) also
proved to be functional, handy and very successful. To see further
the modeling versatility of Eulerian recurrences, we conclude with a
few special Eulerian examples from OEIS of the recursive form
$\ET{a_n(v),b_n(v)}$, where $a_n(v)$ and $b_n(v)$ are quadratic
either in $n$ or in $v$.

\paragraph{A mixture of two Betas $\Longrightarrow$ Uniform} 
Writing all rational numbers $\frac{p}q\in(0,1)$ as ordered pairs 
$(p,q)$ gives sequence \href{https://oeis.org/A181118}{A181118} or 
the polynomials 
\[
    P_n(v) = \sum_{1\le k\le n}
    \lpa{kv^{2k}+(n+1-k)v^{2k-1}},
\]
which satisfy the recurrence 
\[
    P_n\in\GET{1}{\frac{2n^2+v^2n-v}{(n-1)(2n-1)},
    -\frac{nv(1+v)}{(n-1)(2n-1)};v+v^2}.
\]
The limit distribution is Uniform$[0,2]$ although the random variable
is a mixture of two Betas; see Figure~\ref{fig-A181118}. On the other 
hand, the sequence \href{https://oeis.org/A215655}{A215655} is twice 
\href{https://oeis.org/A181118}{A181118}. 

\paragraph{A mixture of two normals} The Eulerian recurrence is also
capable of describing the binomial distribution concatenated twice:
$P_n(v) := (1+v)^n(1+v^{n+1})$, which corresponds to
\href{https://oeis.org/A152198}{A152198} and satisfies 
\[
    P_n\in\EET{\frac{(1+2v-v^2)n-v(1-v)}{n}}{-\frac{v(1+v)}{n};1+v}. 
\]
On the other hand, the sequence 
\href{https://oeis.org/A188440}{A188440}
corresponds to the polynomials $\sum_{0\le k\le \tr{\frac
n2}}\binom{\tr{\frac12n}}{k}v^k$. If we concatenate the two 
polynomial rows with the same row number $\tr{\frac12n}$ and read 
them sequentially as one, we get 
\[
    P_n\in\EET{\frac{(1+4v-2v^2)n-2v(1-v)}{n}}
    {-\frac{v(1+2v)}{n};1+v},
\]
and the resulting distributions are similar to those of
\href{https://oeis.org/A152198}{A152198}.

\begin{figure}[!ht]
\begin{center}
\includegraphics[width=3.5cm]{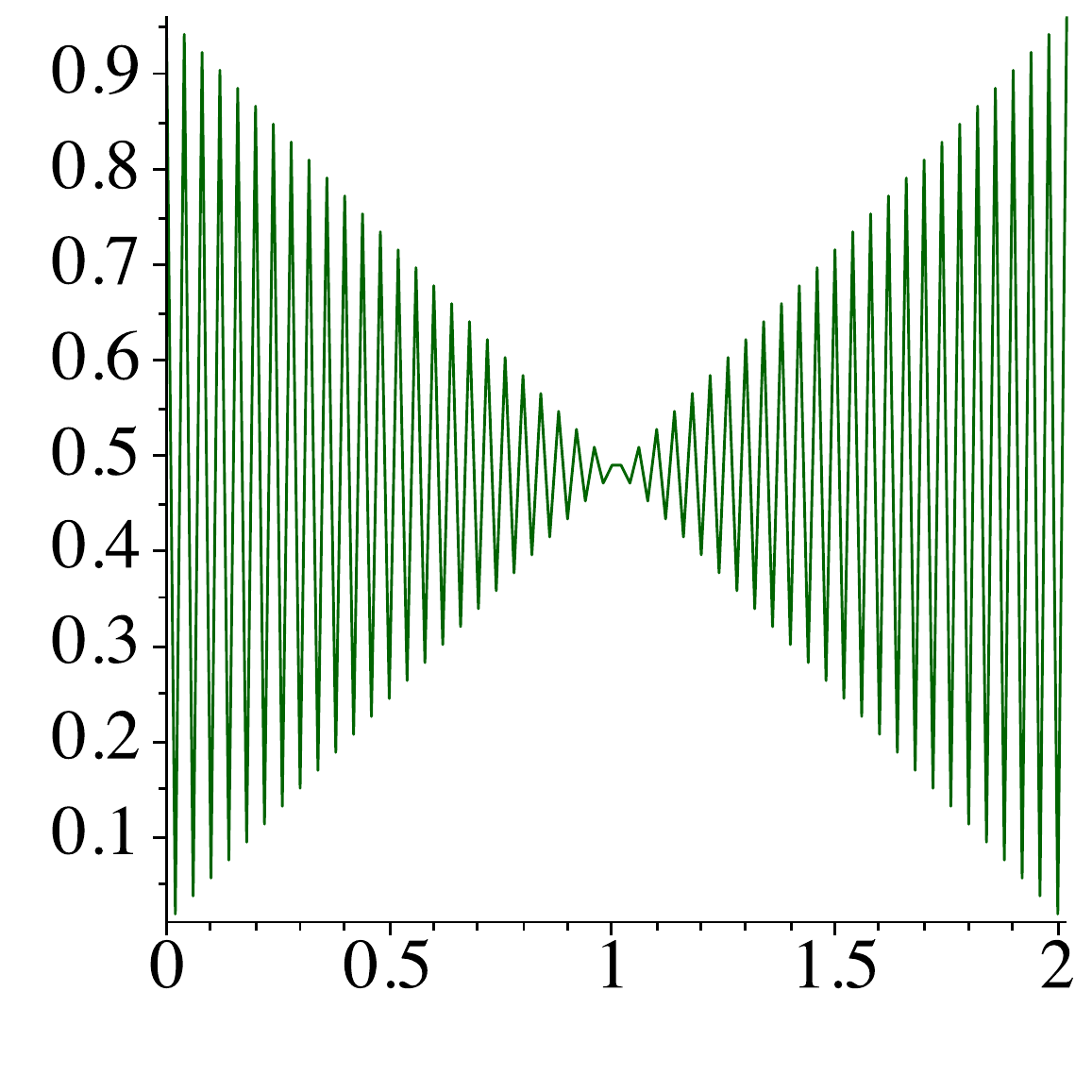} \;   
\includegraphics[width=3.5cm]{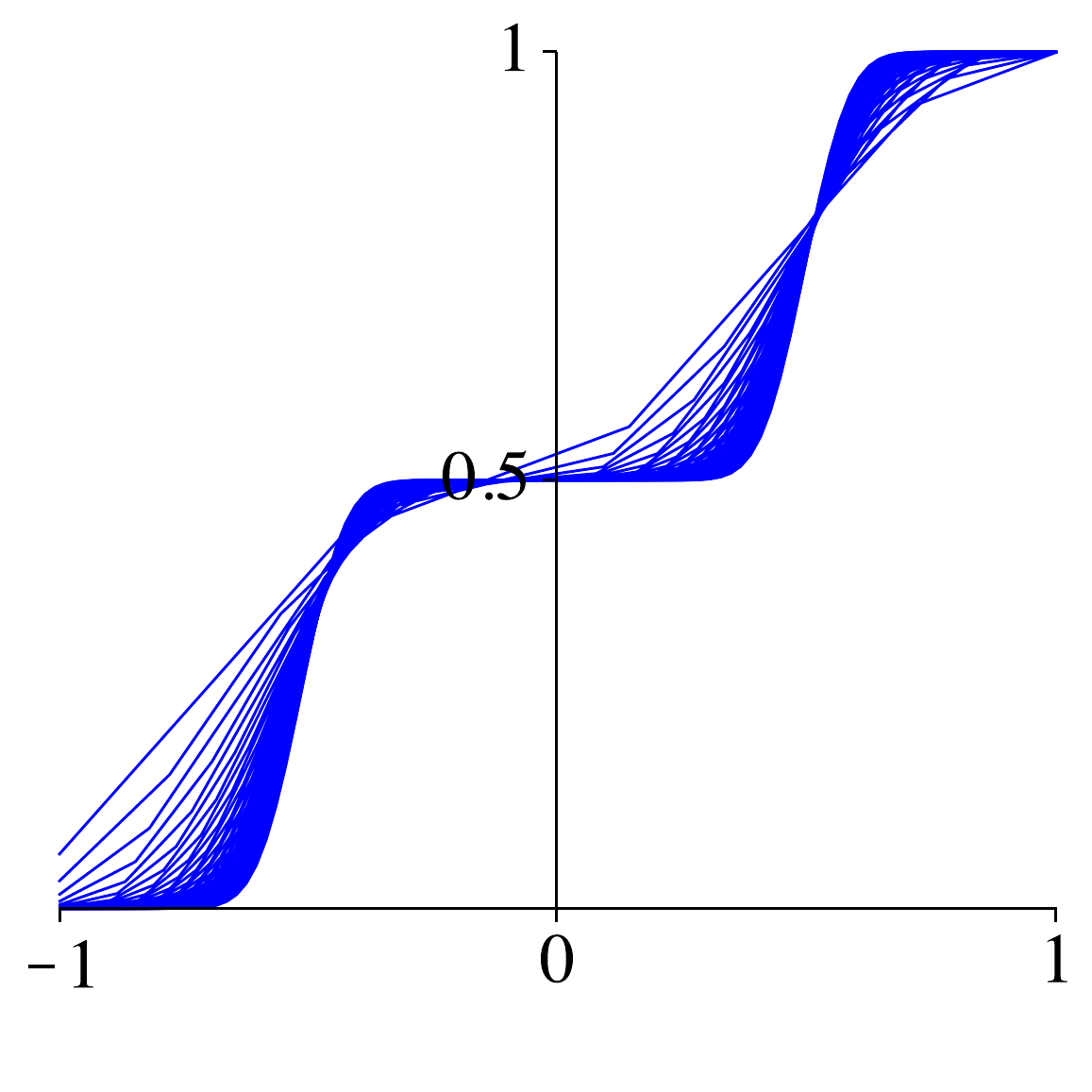} \;   
\includegraphics[width=3.5cm]{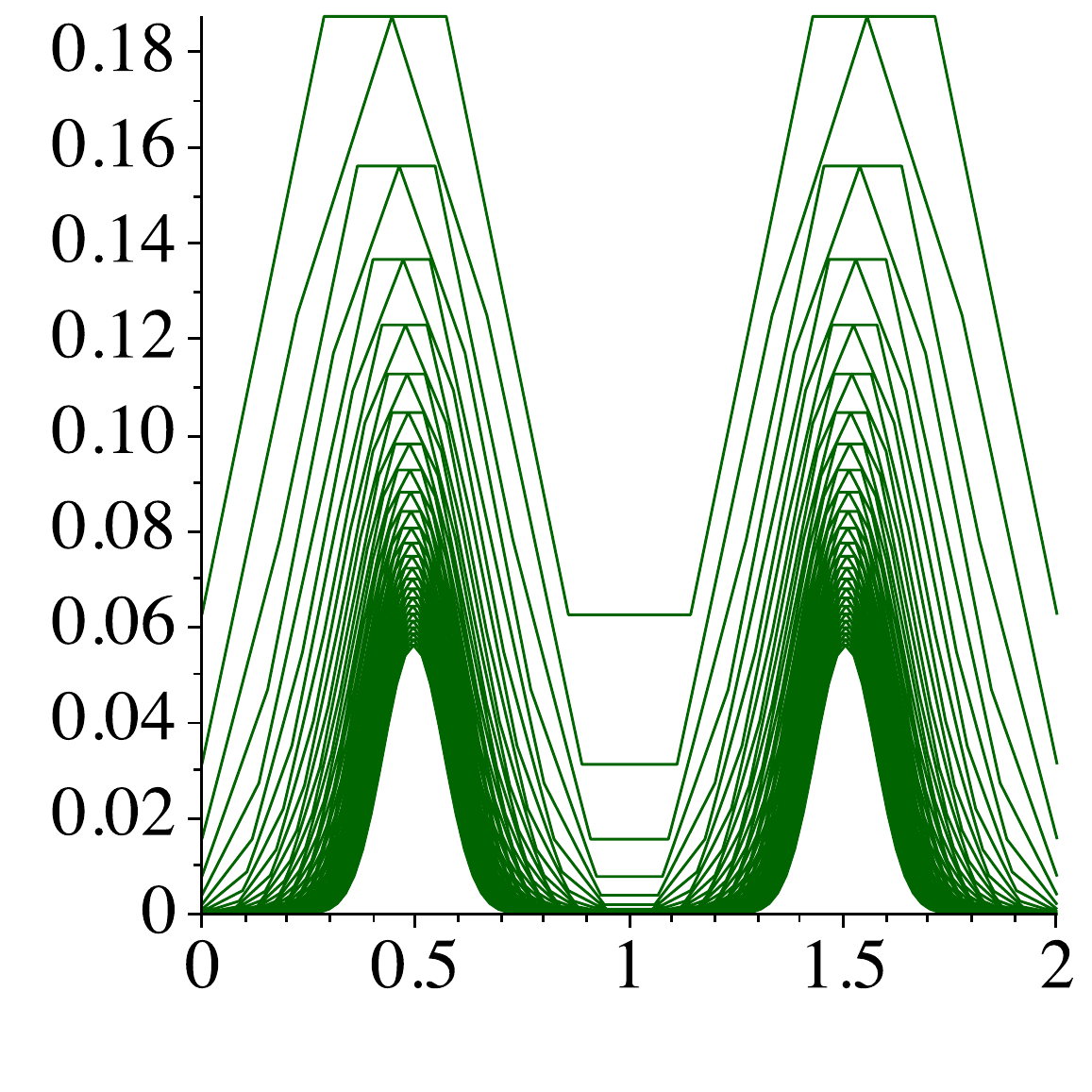}
\end{center} 
\caption{Histograms of \href{https://oeis.org/A181118}{A181118} when
$n=50$ (left) and of \href{https://oeis.org/A152198}{A152198} for
$n=3,\dots,50$ (right), and the (normalized) distribution functions
of \href{https://oeis.org/A152198}{A152198} (middle).}
\label{fig-A181118}
\end{figure}

\paragraph{Degenerate limit law} While the recurrence
$P_n\in\ET{\frac{n+v^2}{n},-\frac{v(1+v)}{2n};1}$ leads to a uniform
limit law (see Section~\ref{ss-ull}), changing the minus sign to a
positive one 
\[
    P_n\in\EET{\frac{n+v^2}{n}}{\frac{v(1+v)}{2n};1}
\] 
gives the closed-form solution $P_n(v) = 1+nv^2$. 

Similarly, the recurrence $P_n\in\ET{\frac{n+v}{n},-\frac{v}{n};1}$
leads to a uniform limit law
(\href{https://oeis.org/A000012}{A000012}, as we examined in 
\S~\ref{sec:ull}), but $P_n\in\ET{\frac{n+v}{n},\frac{v}{n};1}$ gives 
$P_n(v) = 1+nv$, and yields the degenerate limit law, which (when 
read sequentially) corresponds to
\href{https://oeis.org/A057979}{A057979}, and up to different initial
conditions, to \href{https://oeis.org/A133622}{A133622},
\href{https://oeis.org/A152271}{A152271} and
\href{https://oeis.org/A158416}{A158416}. Note that these four
sequences are not triangular sequences.

\paragraph{Another normal limit law}
For the examples examined in this section, if we have no a priori
information about the solution, then the method of moments still 
works well except for the normal mixtures. But the analytic method 
will generally become more messy as the PDEs involved will have higher
orders. To see this, we look briefly at another example 
\href{https://oeis.org/A136267}{A136267} (with normal limit law), 
which is defined via Narayana numbers (see Section~\ref{sec-3v2v}) by
\[
    P_n(v) = \frac1{1+v}\sum_{1\le k\le 2n+2}
    \binom{2n+1}{k-1}\binom{2n+2}{k-1}\frac{v^{k-1}}{k},
\]
a polynomial of degree $2n$. Such polynomials satisfy the rather 
cumbersome recurrence 
\begin{small}
\[
    P_n\in\mathscr{E}\left\langle\!\left\langle
    \frac{2(1+10v+5v^2)n^2+(5+24v+3v^2)n
    +3(1+2v-v^2)}{(n+1)(2n+3)}, 
    \frac{(4n+3)v(1+v)}{(n+1)(2n+3)};1
    \right\rangle\!\right\rangle.
\]
\end{small}
To prove this, we see first that the OGF of $P_n(v)$ satisfies the 
PDE 
\begin{align*}
    &2z^2(1-(1+10v+5v^2)z)\partial_z^2Y
    -4vz^2(1-v^2)\partial_z\partial_vY
    +z(7-(11+84v+33v^2)z)\partial_z Y\\
    &-7vz(1-v^2)\partial_v Y
    +(3-10(1+5v+v^2)z)Y-3=0,
\end{align*}
with $Y(0,v)=1$. Although this equation is not easy to solve, it is 
easy to check that the solution is given by 
\[
    Y(z,v) 
    := \frac{f\lpa{\sqrt{z},v}+f\lpa{-\sqrt{z},v}}{2v(1+v)z},
\]
where $f$ is the OGF for Narayana numbers; see \eqref{ogf-narayana}. 
By the recurrences in Section~\ref{ss-mv-gen}, the mean and the 
variance are
\[
    \mathbb{E}(X_n) = n, \quad\text{and}\quad
    \mathbb{V}(X_n) = \frac{n(n+1)}{4n+3}\qquad(n\ge1),
\]
and the asymptotic normality $\mathscr{N}\lpa{n,\frac14n}$ can either 
be derived by the method of moments or by the CLT for Narayana 
numbers. The complex-analytic approach (Theorem~\ref{thm-saqp}) also 
applies here with $\rho(v) = (1+v)^{-4}$, and we get an optimal 
convergence rate in the CLT $\mathscr{N}\lpa{n,\frac14n; 
n^{-\frac12}}$. Yet another approach is to apply Stirling's formula 
to $[v^k]P_n(v)$ and derive the corresponding LLT when 
$k=n+o\lpa{n^{\frac23}}$, but this approach is often limited to the 
situations when simple closed-form expression is available.

\paragraph{Perspectives}
A natural, fundamental question regarding more general Eulerian 
recurrence $P_n\in\ET{a_n(v),b_n(v)}$ is ``are there simple criteria 
(on $a_n(v)$ and $b_n(v)$) to guarantee the nonnegativity of the 
coefficients $[v^k]P_n(v)$?''

On the other hand, from a methodological point of view, how to
address the finer properties such as local limit theorems and large
deviations by a more systematic approach? Much remains to be
clarified.

B\'ona writes in \cite{Bona2004}: ``\emph{While Eulerian numbers have
been given plenty of attention during the last 200 years, most of the
research was devoted to analytic concepts.}'' Despite the large
literature on analytic aspects, a more complete compilation of the
Eulerian recurrences seems lacking and this paper also aims to
provide an attempt to gather more examples and types of Eulerian
recurrences, focusing on distributional aspect of the coefficients.
We believe that such an extensive compilation will also be helpful for
the study of other properties of Eulerian recurrences and related
structures.

Our method of moments relies crucially on the presence of the factor
``$1-v$'' in the derivative term in \eqref{Pnv-gen}; it fails when
$1-v$ is not there as we already saw many examples in
Section~\ref{ss-1-sv}. Such recurrences also occur frequently in
combinatorics and a systematic study of the corresponding
distributional properties of the coefficients will be given elsewhere.

Finally, from a computational point of view, the Eulerian recurrence
is a Markovian one in that the $n$th row of the polynomials $P_n$
depends only on $P_{n-1}$ and its derivative. This property not only
facilites the systematic computer search through all OEIS
sequences but also provides a good framework for mathematical
analysis; yet the total number (594) we worked out is still
relatively small compared with the 25,000+ nonnegative polynomial
sequences in OEIS (over a total of 327,000+). Although many such
polynomial sequences do not have combinatorial or structural
interpretations or are rather artificially constructed, they do
provide a very rich and valuable source for the study of various  
properties such as the distribution of the coefficients, and that of 
the zeros. A complete characterization of the corresponding limit 
laws is of special methodological and phenomenal interest but seems 
too early at this stage.

\section*{Acknowledgements}

The major part of this work was presented in several meetings,
seminars, and conferences in a few different cities (Taipei, Tokyo,
Vienna, Sydney, Vilnius, Strasbourg, Dalian, Changchun and
Qinhuangdao) during 2017--2018, and this paper has benefited from the
feedback and comments of the audience; we thank specially Luc
Devroye, Svante Janson, Christian Krattenthaler, Shi-Mei Ma, Bao-Xuan 
Zhu, Yeong-Nan Yeh, Thorsten Neuschel, a referee and the 
Editor-in-Chief Catherine Yan for very helpful comments and 
suggestions.

\appendix
\addcontentsline{toc}{section}{Appendices}

\section{Proof of \eqref{qml-ae}}\label{App-A}

By induction hypothesis, we see that the largest terms in the sum 
expression \eqref{rnml} of $R_{n;2m-l,l}$ occur when $(i,j)=(0,1)$ 
and $(i,j)=(1,0)$, giving
\[
    R_{n;2m-l,l} \sim \frac14\binom{2m-l}{2}Q_{n-1;2m-l-2,l}
	+\frac14\binom{l}{2}Q_{n-1;2m-l,l-2},
\]
where $Q$ with negative indices are interpreted as zero. By 
\eqref{qml-ae}
\begin{align}
	R_{n;2m-l,l} \sim C_{2m-l,l} \sigma_n^{2m-2},
\end{align}
where
\[
    C_{2m-l,l} := \frac{d(l)d(2m-l-2)}4\binom{2m-l}{2}
	+\frac{d(2m-l)d(l-2)}4\binom{l}{2}.
\]
Consider now the recurrence
\[
    x_n = \left(1-\frac{m}{n}\right)
	\left(1-\frac{l}{n}\right)x_{n-1}+y_n
	\qquad(n\ge n_0),
\]
with the given initial condition $x_{n_0}$, where $n_0 := 
\max\{m,l\}+1$. Then (with $y_{n_0} := x_{n_0}$) 
\[
    x_n = \frac{(n-m)!(n-l)!}{n!^2}
	\sum_{n_0\le j\le n}\frac{j!^2\, y_j}{(j-m)!(j-l)!}
	\qquad(n\ge m). 
\]
From this exact expression, we deduce the asymptotic transfer:  
\[
    \text{if }y_n \sim c n^\alpha,
    \text{ then }
	x_n\sim \frac{c}{m+l+\alpha+1}\, n^{\alpha+1}
	\qquad(m+l+\alpha>0). 
\]
Applying this transfer, we see that 
\[
    Q_{n;2m-l,l} \sim \frac{4C_{2m-l,l}}{m}\,\sigma_n^{2m}.
\]
Now the leading constant equals
\begin{align*}
	\frac{d(l)d(2m-l-2)}m\binom{2m-l}{2}
	+\frac{d(2m-l)d(l-2)}m\binom{l}{2}
	= d(l)d(2m-l),
\end{align*}
after a straightforward simplification. By induction, this proves 
\eqref{qml-ae}.

\section{Some OEIS sequences satisfying 
$P_n(v)=a_n(v)P_{n-1}(v)$} \label{App-Bino}

In this Appendix, we collect some OEIS sequences satisfying the 
recurrence $P_n(v)=a_n(v)P_{n-1}(v)$ and give their limit laws. 
For convenience, we use the notation $\GT{m}{a_n(v),0;B(v)}$ for an 
abbreviation of $\ET{a_n(v),0;P_m(v)=B(v)}$ (those without subscripts 
stand for $\GT{0}{a_n(v),0;B(v)}$ as above).

\paragraph{$a_n(v) = c \Longrightarrow 
\mathscr{N}\lpa{\frac12n,\frac14n}$, where $c$ is a constant}
\begin{footnotesize}
\begin{center}
\renewcommand{\arraystretch}{1.3}
\begin{longtable}{ll|ll}
\multicolumn{4}{c}{{}} \\
\multicolumn{1}{c}{OEIS} &
\multicolumn{1}{c}{Type} &
\multicolumn{1}{c}{OEIS} &
\multicolumn{1}{c}{Type} \\ \hline
\endhead
\multicolumn{4}{c}{{Continued on next page}} \\
\endfoot
\endlastfoot
\href{https://oeis.org/A007318}{A007318} & $\ET{1+v,0;1}$ & 
\href{https://oeis.org/A028262}{A028262} & $\GT{2}{1+v,0;1+3v+v^2}$ \\
\href{https://oeis.org/A028275}{A028275} & $\GT{2}{1+v,0;1+4v+v^2}$ & 
\href{https://oeis.org/A028313}{A028313} & $\GT{2}{1+v,0;1+5v+v^2}$ \\ 
\href{https://oeis.org/A028326}{A028326} & $\ET{1+v,0,2}$ & 
\href{https://oeis.org/A029600}{A029600} & $\GT{1}{1+v,0;2+3v}$ \\ 
\href{https://oeis.org/A029618}{A029618} & $\GT{1}{1+v,0;3+2v}$ & 
\href{https://oeis.org/A029635}{A029635} & $\GT{1}{1+v,0;1+2v}$ \\ 
\href{https://oeis.org/A029653}{A029653} & $\GT{1}{1+v,0;2+v}$ & 
\href{https://oeis.org/A038208}{A038208} & $\ET{2(1+v),0;1}$ \\ 
\href{https://oeis.org/A038221}{A038221} & $\ET{3(1+v),0;1}$ & 
\href{https://oeis.org/A038234}{A038234} & $\ET{4(1+v),0;1}$ \\ 
\href{https://oeis.org/A038247}{A038247} & $\ET{5(1+v),0;1}$ & 
\href{https://oeis.org/A038260}{A038260} & $\ET{6(1+v),0;1}$ \\ 
\href{https://oeis.org/A038273}{A038273} & $\ET{7(1+v),0;1}$ & 
\href{https://oeis.org/A038286}{A038286} & $\ET{8(1+v),0;1}$ \\ 
\href{https://oeis.org/A038299}{A038299} & $\ET{9(1+v),0;1}$ & 
\href{https://oeis.org/A038312}{A038312} & $\ET{10(1+v),0;1}$ \\ 
\href{https://oeis.org/A038325}{A038325} & $\ET{11(1+v),0;1}$ & 
\href{https://oeis.org/A038338}{A038338} & $\ET{12(1+v),0;1}$ \\ 
\href{https://oeis.org/A055372}{A055372} & $\GT{1}{2(1+v),0;1+v}$ & 
\href{https://oeis.org/A055373}{A055373} & $\GT{1}{3(1+v),0;1+v}$ \\ 
\href{https://oeis.org/A055374}{A055374} & $\GT{1}{4(1+v),0;1+v}$ & 
\href{https://oeis.org/A071919}{A071919} & $\GT{1}{1+v,0;1}$ \\ 
\href{https://oeis.org/A072405}{A072405} & $\GT{2}{1+v,0;1+v+v^2}$ & 
\href{https://oeis.org/A087698}{A087698} & $\GT{2}{1+v,0;1+v^2}$ \\ 
\href{https://oeis.org/A093560}{A093560} & $\GT{1}{1+v,0;3+v}$ & 
\href{https://oeis.org/A093561}{A093561} & $\GT{1}{1+v,0;4+v}$ \\ 
\href{https://oeis.org/A093562}{A093562} & $\GT{1}{1+v,0;5+v}$ & 
\href{https://oeis.org/A093563}{A093563} & $\GT{1}{1+v,0;6+v}$ \\ 
\href{https://oeis.org/A093564}{A093564} & $\GT{1}{1+v,0;7+v}$ & 
\href{https://oeis.org/A093565}{A093565} & $\GT{1}{1+v,0;8+v}$ \\ 
\href{https://oeis.org/A093644}{A093644} & $\GT{1}{1+v,0;9+v}$ & 
\href{https://oeis.org/A093645}{A093645} & $\GT{1}{1+v,0;10+v}$ \\ 
\href{https://oeis.org/A095660}{A095660} & $\GT{1}{1+v,0;1+3v}$ & 
\href{https://oeis.org/A095666}{A095666} & $\GT{1}{1+v,0;1+4v}$ \\ 
\href{https://oeis.org/A096940}{A096940} & $\GT{1}{1+v,0;1+5v}$ & 
\href{https://oeis.org/A096956}{A096956} & $\GT{1}{1+v,0;1+6v}$ \\ 
\href{https://oeis.org/A097805}{A097805} & $\GT{1}{1+v,0;v}$ & 
\href{https://oeis.org/A122218}{A122218} & $\GT{2}{1+v,0;1+v+v^2}$ \\ 
\href{https://oeis.org/A124459}{A124459} & $\GT{1}{1+v,0;3+2v}$ & 
\href{https://oeis.org/A129687}{A129687} & $\GT{2}{1+v,0;2+2v+v^2}$ \\ 
\href{https://oeis.org/A131084}{A131084} & $\GT{2}{1+v,0;2v+v^2}$ & 
\href{https://oeis.org/A132200}{A132200} & $\GT{1}{1+v,0;4+4v}$ \\ 
\href{https://oeis.org/A134058}{A134058} & $\GT{1}{1+v,0;2+2v}$ & 
\href{https://oeis.org/A134059}{A134059} & $\GT{1}{1+v,0;3+3v}$ \\
\href{https://oeis.org/A135089}{A135089} & $\GT{1}{1+v,0;5+5v}$ & 
\href{https://oeis.org/A144225}{A144225} & $\GT{2}{1+v,0;v}$ \\ 
\href{https://oeis.org/A147644}{A147644} & $\GT{3}{1+v,0;1+5v+5v^2+v^3}$ & 
\href{https://oeis.org/A159854}{A159854} & $\GT{2}{1+v,0;v^2}$ \\ 
\href{https://oeis.org/A172185}{A172185} & $\GT{1}{1+v,0;9+11v}$ & 
\href{https://oeis.org/A202241}{A202241} & $\GT{3}{1+v,0;4v+4v^2+v^3}$ \\ \hline
\end{longtable}    
\end{center}
\end{footnotesize}

\paragraph{$a_n(v) = d_n (1+v)\Longrightarrow 
\mathscr{N}\lpa{\frac12n,\frac14n}$, where $d_n$ is a sequence of $n$ 
and independent of $v$}
Here $f_n$ denotes the $n$th Fibonacci number (\href{https://oeis.org/A000045}{A000045}) and $B_n$ that of Bell 
numbers (\href{https://oeis.org/A000110}{A000110}). 

\begin{footnotesize}
\begin{center}
\renewcommand{\arraystretch}{1.5}
\begin{longtable}{cl|cl}
\multicolumn{4}{c}{{}} \\
\multicolumn{1}{c}{OEIS} &
\multicolumn{1}{c}{Type} &
\multicolumn{1}{c}{OEIS} &
\multicolumn{1}{c}{Type} \\ \hline
\endhead
\multicolumn{4}{c}{{Continued on next page}} \\
\endfoot
\endlastfoot
\href{https://oeis.org/A003506}{A003506} & 
$\ET{\frac{n+1}{n}(1+v),0;1}$ &  
\href{https://oeis.org/A016095}{A016095} & 
$\ET{\frac{f_{n+1}}{f_n}(1+v),0;1}$ \\  
\href{https://oeis.org/A055883}{A055883} & 
$\ET{\frac{B_{n}}{B_{n-1}}(1+v),0;1}$ &  
\href{https://oeis.org/A085880}{A085880} & 
$\ET{\frac{2(2n-1)}{n+1}(1+v),0;1}$ \\  
\href{https://oeis.org/A085881}{A085881} & 
$\ET{(2n-1)(1+v),0;1}$ &  
\href{https://oeis.org/A094305}{A094305} & 
$\ET{\frac{n+2}{n}(1+v),0;1}$ \\  
\href{https://oeis.org/A121547}{A121547} & 
$\GT{1}{\frac{n+2}{n-1}(1+v),0;v}$ &  
\href{https://oeis.org/A124860}{A124860} 
& 
$\ET{\frac{2^{n+1}-(-1)^{n+1}}{2^n-(-1)^n}(1+v),0;1}$ \\  
\href{https://oeis.org/A127952}{A127952} & 
$\GT{1}{\frac{n+1}{n}(1+v),0;2v}$ &  
\href{https://oeis.org/A129533}{A129533} & 
$\GT{2}{\frac{n}{n-2}(1+v),0;v}$ \\  
\href{https://oeis.org/A132775}{A132775} & 
$\ET{\frac{2n+1}{2n-1}(1+v),0;1}$ &  
\href{https://oeis.org/A134239}{A134239} & 
$\GT{1}{\frac{n+1}{n}(1+v),0;4+2v}$ \\  
\href{https://oeis.org/A134346}{A134346} & 
$\ET{\frac{2^{n+1}-1}{2^n-1}(1+v),0;1}$ &  
\href{https://oeis.org/A134400}{A134400} & 
$\GT{1}{\frac{n}{n-1}(1+v),0;1+v}$ \\  
\href{https://oeis.org/A135065}{A135065} & 
$\ET{\frac{(n+1)^2}{n^2}(1+v),0;1}$ &  
\href{https://oeis.org/A140880}{A140880} & 
$\ET{\frac{n+2}{n}(1+v),0;2}$ \\  
\href{https://oeis.org/A156992}{A156992} & 
$\ET{(n+1)(1+v),0;1}$ &  
\href{https://oeis.org/A164961}{A164961} & 
$\ET{(4n-2)(1+v),0;1}$ \\  
\href{https://oeis.org/A178820}{A178820} & 
$\ET{\frac{n+3}{n}(1+v),0;1}$ &  
\href{https://oeis.org/A178821}{A178821} & 
$\ET{\frac{n+4}{n}(1+v),0;1}$ \\  
\href{https://oeis.org/A178822}{A178822} & 
$\ET{\frac{n+5}{n}(1+v),0;1}$ &  
\href{https://oeis.org/A196347}{A196347} & 
$\ET{n(1+v),0;1}$ \\  
\href{https://oeis.org/A216973}{A216973} & 
$\GT{1}{\frac{n}{n-1}(1+v),0;1}$ &  
\href{https://oeis.org/A219570}{A219570} & 
$\GT{1}{(n-1)(1+v),0;1+v}$ \\  
\href{https://oeis.org/A237765}{A237765} & 
$\GT{2}{\frac{n}{n-2}(1+v),0;(1+v)^2}$ &  
\href{https://oeis.org/A249632}{A249632} 
& $\GT{1}{\frac{n^{n-2}}{(n-1)^{n-3}}(1+v),0;1+v}$ \\ 
\href{https://oeis.org/A253666}{A253666} & 
$\begin{cases}
\ET{\frac14n(1+v),0;1} \; n\text{ even}\\
\ET{\frac1n(1+v),0;1}\; n\text{ odd}
\end{cases}$  & 
\href{https://oeis.org/A258758}{A258758} & 
$\GT{1}{\frac{4n-2}{n}(1+v),0;1+v}$ \\ \hline
\end{longtable}    
\end{center}
\end{footnotesize}

\paragraph{$a_n(v) = p+qv\Longrightarrow 
\mathscr{N}\lpa{\frac{q}{p+q}n,\frac{pq}{(p+q)^2}n}$}
\begin{scriptsize}
\begin{center}
\renewcommand{\arraystretch}{1.3}
\begin{longtable}{cll|cll}
\multicolumn{6}{c}{{}} \\
\multicolumn{6}{c}{{}} \\
\multicolumn{1}{c}{OEIS} &
\multicolumn{1}{c}{Type} &
\multicolumn{1}{c}{CLT} &
\multicolumn{1}{c}{OEIS} &
\multicolumn{1}{c}{Type} &
\multicolumn{1}{c}{CLT} \\ \hline
\endhead
\multicolumn{6}{c}{{Continued on next page}} \\
\endfoot
\endlastfoot
\href{https://oeis.org/A013609}{A013609} & $\ET{1+2v,0;1}$ 
& \nom{\frac{2}{3}}{\frac{2}{9}n} &
\href{https://oeis.org/A013610}{A013610} & $\ET{1+3v,0;1}$ 
& \nom{\frac{3}{4}}{\frac{3}{16}n} \\
\href{https://oeis.org/A013611}{A013611} & $\ET{1+4v,0;1}$ 
& \nom{\frac{4}{5}}{\frac{4}{25}n}  &
\href{https://oeis.org/A013612}{A013612} & $\ET{1+5v,0;1}$ 
& \nom{\frac{5}{6}}{\frac{5}{36}n}\\
\href{https://oeis.org/A013613}{A013613} & $\ET{1+6v,0;1}$ 
& \nom{\frac{6}{7}}{\frac{6}{49}n} &
\href{https://oeis.org/A013614}{A013614} & $\ET{1+7v,0;1}$ 
& \nom{\frac{7}{8}}{\frac{7}{64}n}\\
\href{https://oeis.org/A013615}{A013615} & $\ET{1+8v,0;1}$ 
& \nom{\frac{8}{9}}{\frac{8}{81}n} &
\href{https://oeis.org/A013616}{A013616} & $\ET{1+9v,0;1}$ 
& \nom{\frac{9}{10}}{\frac{9}{100}n}\\
\href{https://oeis.org/A013617}{A013617} & $\ET{1+10v,0;1}$ 
& \nom{\frac{10}{11}}{\frac{10}{121}n} &
\href{https://oeis.org/A013618}{A013618} & $\ET{1+11v,0;1}$ 
& \nom{\frac{11}{12}}{\frac{11}{144}n}\\
\href{https://oeis.org/A013619}{A013619} & $\ET{1+12v,0;1}$ 
& \nom{\frac{12}{13}}{\frac{12}{169}n} &
\href{https://oeis.org/A013620}{A013620} & $\ET{2+3v,0;1}$ 
& \nom{\frac{3}{5}}{\frac{6}{25}n}\\
\href{https://oeis.org/A013621}{A013621} & $\ET{2+5v,0;1}$ 
& \nom{\frac{5}{7}}{\frac{10}{49}n} &
\href{https://oeis.org/A013622}{A013622} & $\ET{3+5v,0;1}$ 
& \nom{\frac{5}{8}}{\frac{15}{64}n}\\
\href{https://oeis.org/A013623}{A013623} & $\ET{2+7v,0;1}$ 
& \nom{\frac{7}{9}}{\frac{14}{81}n} &
\href{https://oeis.org/A013624}{A013624} & $\ET{3+7v,0;1}$ 
& \nom{\frac{7}{10}}{\frac{21}{100}n}\\
\href{https://oeis.org/A013625}{A013625} & $\ET{4+7v,0;1}$ 
& \nom{\frac{7}{11}}{\frac{28}{121}n} &
\href{https://oeis.org/A013626}{A013626} & $\ET{5+7v,0;1}$ 
& \nom{\frac{7}{12}}{\frac{35}{144}n}\\
\href{https://oeis.org/A013627}{A013627} & $\ET{6+7v,0;1}$ 
& \nom{\frac{7}{13}}{\frac{42}{169}n} &
\href{https://oeis.org/A013628}{A013628} & $\ET{4+5v,0;1}$ 
& \nom{\frac{5}{9}}{\frac{20}{81}n}\\
\href{https://oeis.org/A024462}{A024462} & $\GT{2}{1+3v,0;(1+v)^2}$ 
& \nom{\frac{3}{4}}{\frac{3}{16}n} &
\href{https://oeis.org/A027465}{A027465} & $\ET{3+v,0;1}$ 
& \nom{\frac{1}{4}}{\frac{3}{16}n}\\
\href{https://oeis.org/A027466}{A027466} & $\ET{7+v,0;1}$ 
& \nom{\frac{1}{8}}{\frac{7}{64}n} &
\href{https://oeis.org/A027467}{A027467} & $\ET{15+v,0;1}$ 
& \nom{\frac{1}{16}}{\frac{15}{256}n}\\
\href{https://oeis.org/A038195}{A038195} & $\GT{2}{22+v,0;(1+v)^2}$ 
& \nom{\frac{1}{3}}{\frac{2}{9}n} &
\href{https://oeis.org/A038207}{A038207} & $\ET{2+v,0;1}$ 
& \nom{\frac{1}{3}}{\frac{2}{9}n}\\
\href{https://oeis.org/A038210}{A038210} & $\ET{21+2v,0;1}$ 
& \nom{\frac{2}{3}}{\frac{2}{9}n} &
\href{https://oeis.org/A038212}{A038212} & $\ET{21+3v,0;1}$ 
& \nom{\frac{3}{4}}{\frac{3}{16}n}\\
\href{https://oeis.org/A038214}{A038214} & $\ET{21+4v,0;1}$ 
& \nom{\frac{4}{5}}{\frac{4}{25}n} &
\href{https://oeis.org/A038215}{A038215} & $\ET{2+9v,0;1}$ 
& \nom{\frac{9}{11}}{\frac{18}{121}n}\\
\href{https://oeis.org/A038216}{A038216} & $\ET{21+5v,0;1}$ 
& \nom{\frac{5}{6}}{\frac{5}{36}n} &
\href{https://oeis.org/A038217}{A038217} & $\ET{2+11v,0;1}$ 
& \nom{\frac{11}{13}}{\frac{22}{169}n}\\
\href{https://oeis.org/A038218}{A038218} & $\ET{21+6v,0;1}$ 
& \nom{\frac{6}{7}}{\frac{6}{49}n} &
\href{https://oeis.org/A038220}{A038220} & $\ET{3+2v,0;1}$ 
& \nom{\frac{2}{5}}{\frac{6}{25}n}\\
\href{https://oeis.org/A038222}{A038222} & $\ET{3+4v,0;1}$ 
& \nom{\frac{4}{7}}{\frac{12}{49}n} &
\href{https://oeis.org/A038224}{A038224} & $\ET{31+2v,0;1}$ 
& \nom{\frac{2}{3}}{\frac{2}{9}n}\\
\href{https://oeis.org/A038226}{A038226} & $\ET{3+8v,0;1}$ 
& \nom{\frac{8}{11}}{\frac{24}{121}n} &
\href{https://oeis.org/A038227}{A038227} & $\ET{31+3v,0;1}$ 
& \nom{\frac{3}{4}}{\frac{3}{16}n}\\
\href{https://oeis.org/A038228}{A038228} & $\ET{3+10v,0;1}$ 
& \nom{\frac{10}{13}}{\frac{30}{169}n} &
\href{https://oeis.org/A038229}{A038229} & $\ET{3+11v,0;1}$ 
& \nom{\frac{11}{14}}{\frac{33}{196}n}\\
\href{https://oeis.org/A038230}{A038230} & $\ET{31+4v,0;1}$ 
& \nom{\frac{4}{5}}{\frac{4}{25}n} &
\href{https://oeis.org/A038231}{A038231} & $\ET{4+v,0;1}$ 
& \nom{\frac{1}{5}}{\frac{4}{25}n}\\
\href{https://oeis.org/A038232}{A038232} & $\ET{22+v,0;1}$ 
& \nom{\frac{1}{3}}{\frac{2}{9}n} &
\href{https://oeis.org/A038233}{A038233} & $\ET{4+3v,0;1}$ 
& \nom{\frac{3}{7}}{\frac{12}{49}n}\\
\href{https://oeis.org/A038236}{A038236} & $\ET{22+3v,0;1}$ 
& \nom{\frac{3}{5}}{\frac{6}{25}n} &
\href{https://oeis.org/A038238}{A038238} & $\ET{41+2v,0;1}$ 
& \nom{\frac{2}{3}}{\frac{2}{9}n}\\
\href{https://oeis.org/A038239}{A038239} & $\ET{4+9v,0;1}$ 
& \nom{\frac{9}{13}}{\frac{36}{169}n} &
\href{https://oeis.org/A038240}{A038240} & $\ET{22+5v,0;1}$ 
& \nom{\frac{5}{7}}{\frac{10}{49}n}\\
\href{https://oeis.org/A038241}{A038241} & $\ET{4+11v,0;1}$ 
& \nom{\frac{11}{15}}{\frac{44}{225}n} &
\href{https://oeis.org/A038242}{A038242} & $\ET{41+3v,0;1}$ 
& \nom{\frac{3}{4}}{\frac{3}{16}n}\\
\href{https://oeis.org/A038243}{A038243} & $\ET{5+v,0;1}$ 
& \nom{\frac{1}{6}}{\frac{5}{36}n} &
\href{https://oeis.org/A038244}{A038244} & $\ET{5+2v,0;1}$ 
& \nom{\frac{2}{7}}{\frac{10}{49}n}\\
\href{https://oeis.org/A038245}{A038245} & $\ET{5+3v,0;1}$ 
& \nom{\frac{3}{8}}{\frac{15}{64}n} &
\href{https://oeis.org/A038246}{A038246} & $\ET{5+4v,0;1}$ 
& \nom{\frac{4}{9}}{\frac{20}{81}n}\\
\href{https://oeis.org/A038248}{A038248} & $\ET{5+6v,0;1}$ 
& \nom{\frac{6}{11}}{\frac{30}{121}n} &
\href{https://oeis.org/A038250}{A038250} & $\ET{5+8v,0;1}$ 
& \nom{\frac{8}{13}}{\frac{40}{169}n}\\
\href{https://oeis.org/A038251}{A038251} & $\ET{5+9v,0;1}$ 
& \nom{\frac{9}{14}}{\frac{45}{196}n} &
\href{https://oeis.org/A038252}{A038252} & $\ET{51+2v,0;1}$ 
& \nom{\frac{2}{3}}{\frac{2}{9}n}\\
\href{https://oeis.org/A038253}{A038253} & $\ET{5+11v,0;1}$ 
& \nom{\frac{11}{16}}{\frac{55}{256}n} &
\href{https://oeis.org/A038254}{A038254} & $\ET{5+12v,0;1}$ 
& \nom{\frac{12}{17}}{\frac{60}{289}n}\\
\href{https://oeis.org/A038255}{A038255} & $\ET{6+v,0;1}$ 
& \nom{\frac{1}{7}}{\frac{6}{49}n} &
\href{https://oeis.org/A038256}{A038256} & $\ET{23+v,0;1}$ 
& \nom{\frac{1}{4}}{\frac{3}{16}n}\\
\href{https://oeis.org/A038257}{A038257} & $\ET{32+v,0;1}$ 
& \nom{\frac{1}{3}}{\frac{2}{9}n} &
\href{https://oeis.org/A038258}{A038258} & $\ET{23+2v,0;1}$ 
& \nom{\frac{2}{5}}{\frac{6}{25}n}\\
\href{https://oeis.org/A038259}{A038259} & $\ET{6+5v,0;1}$ 
& \nom{\frac{5}{11}}{\frac{30}{121}n} &
\href{https://oeis.org/A038262}{A038262} & $\ET{23+4v,0;1}$ 
& \nom{\frac{4}{7}}{\frac{12}{49}n}\\
\href{https://oeis.org/A038263}{A038263} & $\ET{32+3v,0;1}$ 
& \nom{\frac{3}{5}}{\frac{6}{25}n} &
\href{https://oeis.org/A038264}{A038264} & $\ET{23+5v,0;1}$ 
& \nom{\frac{5}{8}}{\frac{15}{64}n}\\
\href{https://oeis.org/A038265}{A038265} & $\ET{6+11v,0;1}$ 
& \nom{\frac{11}{17}}{\frac{66}{289}n} &
\href{https://oeis.org/A038266}{A038266} & $\ET{61+2v,0;1}$ 
& \nom{\frac{2}{3}}{\frac{2}{9}n}\\
\href{https://oeis.org/A038268}{A038268} & $\ET{7+2v,0;1}$ 
& \nom{\frac{2}{9}}{\frac{14}{81}n} &
\href{https://oeis.org/A038269}{A038269} & $\ET{7+3v,0;1}$ 
& \nom{\frac{3}{10}}{\frac{21}{100}n}\\
\href{https://oeis.org/A038270}{A038270} & $\ET{7+4v,0;1}$ 
& \nom{\frac{4}{11}}{\frac{28}{121}n} &
\href{https://oeis.org/A038271}{A038271} & $\ET{7+5v,0;1}$ 
& \nom{\frac{5}{12}}{\frac{35}{144}n}\\
\href{https://oeis.org/A038272}{A038272} & $\ET{7+6v,0;1}$ 
& \nom{\frac{6}{13}}{\frac{42}{169}n} &
\href{https://oeis.org/A038274}{A038274} & $\ET{7+8v,0;1}$ 
& \nom{\frac{8}{15}}{\frac{56}{225}n}\\
\href{https://oeis.org/A038275}{A038275} & $\ET{7+9v,0;1}$ 
& \nom{\frac{9}{16}}{\frac{63}{256}n} &
\href{https://oeis.org/A038276}{A038276} & $\ET{7+10v,0;1}$ 
& \nom{\frac{10}{17}}{\frac{70}{289}n}\\
\href{https://oeis.org/A038277}{A038277} & $\ET{7+11v,0;1}$ 
& \nom{\frac{11}{18}}{\frac{77}{324}n} &
\href{https://oeis.org/A038278}{A038278} & $\ET{7+12v,0;1}$ 
& \nom{\frac{12}{19}}{\frac{84}{361}n}\\
\href{https://oeis.org/A038279}{A038279} & $\ET{8+v,0;1}$ 
& \nom{\frac{1}{9}}{\frac{8}{81}n} &
\href{https://oeis.org/A038280}{A038280} & $\ET{24+v,0;1}$ 
& \nom{\frac{1}{5}}{\frac{4}{25}n}\\
\href{https://oeis.org/A038281}{A038281} & $\ET{8+3v,0;1}$ 
& \nom{\frac{3}{11}}{\frac{24}{121}n} &
\href{https://oeis.org/A038282}{A038282} & $\ET{42+v,0;1}$ 
& \nom{\frac{1}{3}}{\frac{2}{9}n}\\
\href{https://oeis.org/A038283}{A038283} & $\ET{8+5v,0;1}$ 
& \nom{\frac{5}{13}}{\frac{40}{169}n} &
\href{https://oeis.org/A038284}{A038284} & $\ET{24+3v,0;1}$ 
& \nom{\frac{3}{7}}{\frac{12}{49}n}\\
\href{https://oeis.org/A038285}{A038285} & $\ET{8+7v,0;1}$ 
& \nom{\frac{7}{15}}{\frac{56}{225}n} &
\href{https://oeis.org/A038287}{A038287} & $\ET{8+9v,0;1}$ 
& \nom{\frac{9}{17}}{\frac{72}{289}n}\\
\href{https://oeis.org/A038288}{A038288} & $\ET{24+5v,0;1}$ 
& \nom{\frac{5}{9}}{\frac{20}{81}n} &
\href{https://oeis.org/A038289}{A038289} & $\ET{8+11v,0;1}$ 
& \nom{\frac{11}{19}}{\frac{88}{361}n}\\
\href{https://oeis.org/A038290}{A038290} & $\ET{42+3v,0;1}$ 
& \nom{\frac{3}{5}}{\frac{6}{25}n} &
\href{https://oeis.org/A038291}{A038291} & $\ET{9+v,0;1}$ 
& \nom{\frac{1}{10}}{\frac{9}{100}n}\\
\href{https://oeis.org/A038292}{A038292} & $\ET{9+2v,0;1}$ 
& \nom{\frac{2}{11}}{\frac{18}{121}n} &
\href{https://oeis.org/A038293}{A038293} & $\ET{33+v,0;1}$ 
& \nom{\frac{1}{4}}{\frac{3}{16}n}\\
\href{https://oeis.org/A038294}{A038294} & $\ET{9+4v,0;1}$ 
& \nom{\frac{4}{13}}{\frac{36}{169}n} &
\href{https://oeis.org/A038295}{A038295} & $\ET{9+5v,0;1}$ 
& \nom{\frac{5}{14}}{\frac{45}{196}n}\\
\href{https://oeis.org/A038296}{A038296} & $\ET{33+2v,0;1}$ 
& \nom{\frac{2}{5}}{\frac{6}{25}n} &
\href{https://oeis.org/A038297}{A038297} & $\ET{9+7v,0;1}$ 
& \nom{\frac{7}{16}}{\frac{63}{256}n}\\
\href{https://oeis.org/A038298}{A038298} & $\ET{9+8v,0;1}$ 
& \nom{\frac{8}{17}}{\frac{72}{289}n} &
\href{https://oeis.org/A038300}{A038300} & $\ET{9+10v,0;1}$ 
& \nom{\frac{10}{19}}{\frac{90}{361}n}\\
\href{https://oeis.org/A038301}{A038301} & $\ET{9+11v,0;1}$ 
& \nom{\frac{11}{20}}{\frac{99}{400}n} &
\href{https://oeis.org/A038302}{A038302} & $\ET{33+4v,0;1}$ 
& \nom{\frac{4}{7}}{\frac{12}{49}n}\\
\href{https://oeis.org/A038303}{A038303} & $\ET{10+v,0;1}$ 
& \nom{\frac{1}{11}}{\frac{10}{121}n} &
\href{https://oeis.org/A038304}{A038304} & $\ET{25+v,0;1}$ 
& \nom{\frac{1}{6}}{\frac{5}{36}n}\\
\href{https://oeis.org/A038305}{A038305} & $\ET{10+3v,0;1}$ 
& \nom{\frac{3}{13}}{\frac{30}{169}n} &
\href{https://oeis.org/A038306}{A038306} & $\ET{25+2v,0;1}$ 
& \nom{\frac{2}{7}}{\frac{10}{49}n}\\
\href{https://oeis.org/A038307}{A038307} & $\ET{52+v,0;1}$ 
& \nom{\frac{1}{3}}{\frac{2}{9}n} &
\href{https://oeis.org/A038308}{A038308} & $\ET{25+3v,0;1}$ 
& \nom{\frac{3}{8}}{\frac{15}{64}n}\\
\href{https://oeis.org/A038309}{A038309} & $\ET{10+7v,0;1}$ 
& \nom{\frac{7}{17}}{\frac{70}{289}n} &
\href{https://oeis.org/A038310}{A038310} & $\ET{25+4v,0;1}$ 
& \nom{\frac{4}{9}}{\frac{20}{81}n}\\
\href{https://oeis.org/A038311}{A038311} & $\ET{10+9v,0;1}$ 
& \nom{\frac{9}{19}}{\frac{90}{361}n} &
\href{https://oeis.org/A038313}{A038313} & $\ET{10+11v,0;1}$ 
& \nom{\frac{11}{21}}{\frac{110}{441}n}\\
\href{https://oeis.org/A038314}{A038314} & $\ET{25+6v,0;1}$ 
& \nom{\frac{6}{11}}{\frac{30}{121}n} &
\href{https://oeis.org/A038315}{A038315} & $\ET{11+v,0;1}$ 
& \nom{\frac{1}{12}}{\frac{11}{144}n}\\
\href{https://oeis.org/A038316}{A038316} & $\ET{11+2v,0;1}$ 
& \nom{\frac{2}{13}}{\frac{22}{169}n} &
\href{https://oeis.org/A038317}{A038317} & $\ET{11+3v,0;1}$ 
& \nom{\frac{3}{14}}{\frac{33}{196}n}\\
\href{https://oeis.org/A038318}{A038318} & $\ET{11+4v,0;1}$ 
& \nom{\frac{4}{15}}{\frac{44}{225}n} &
\href{https://oeis.org/A038319}{A038319} & $\ET{11+5v,0;1}$ 
& \nom{\frac{5}{16}}{\frac{55}{256}n}\\
\href{https://oeis.org/A038320}{A038320} & $\ET{11+6v,0;1}$ 
& \nom{\frac{6}{17}}{\frac{66}{289}n} &
\href{https://oeis.org/A038321}{A038321} & $\ET{11+7v,0;1}$ 
& \nom{\frac{7}{18}}{\frac{77}{324}n}\\
\href{https://oeis.org/A038322}{A038322} & $\ET{11+8v,0;1}$ 
& \nom{\frac{8}{19}}{\frac{88}{361}n} &
\href{https://oeis.org/A038323}{A038323} & $\ET{11+9v,0;1}$ 
& \nom{\frac{9}{20}}{\frac{99}{400}n}\\
\href{https://oeis.org/A038324}{A038324} & $\ET{11+10v,0;1}$ 
& \nom{\frac{10}{21}}{\frac{110}{441}n} &
\href{https://oeis.org/A038326}{A038326} & $\ET{11+12v,0;1}$ 
& \nom{\frac{12}{23}}{\frac{132}{529}n}\\
\href{https://oeis.org/A038327}{A038327} & $\ET{12+v,0;1}$ 
& \nom{\frac{1}{13}}{\frac{12}{169}n} &
\href{https://oeis.org/A038328}{A038328} & $\ET{26+v,0;1}$ 
& \nom{\frac{1}{7}}{\frac{6}{49}n}\\
\href{https://oeis.org/A038329}{A038329} & $\ET{34+v,0;1}$ 
& \nom{\frac{1}{5}}{\frac{4}{25}n} &
\href{https://oeis.org/A038330}{A038330} & $\ET{43+v,0;1}$ 
& \nom{\frac{1}{4}}{\frac{3}{16}n}\\
\href{https://oeis.org/A038331}{A038331} & $\ET{12+5v,0;1}$ 
& \nom{\frac{5}{17}}{\frac{60}{289}n} &
\href{https://oeis.org/A038332}{A038332} & $\ET{62+v,0;1}$ 
& \nom{\frac{1}{3}}{\frac{2}{9}n}\\
\href{https://oeis.org/A038333}{A038333} & $\ET{12+7v,0;1}$ 
& \nom{\frac{7}{19}}{\frac{84}{361}n} &
\href{https://oeis.org/A038334}{A038334} & $\ET{43+2v,0;1}$ 
& \nom{\frac{2}{5}}{\frac{6}{25}n}\\
\href{https://oeis.org/A038335}{A038335} & $\ET{34+3v,0;1}$ 
& \nom{\frac{3}{7}}{\frac{12}{49}n} &
\href{https://oeis.org/A038336}{A038336} & $\ET{26+5v,0;1}$ 
& \nom{\frac{5}{11}}{\frac{30}{121}n}\\
\href{https://oeis.org/A038337}{A038337} & $\ET{12+11v,0;1}$ 
& \nom{\frac{11}{23}}{\frac{132}{529}n} &
\href{https://oeis.org/A038763}{A038763} & $\GT{1}{1+3v,0;1+v}$ 
& \nom{\frac{3}{4}}{\frac{3}{16}n}\\
\href{https://oeis.org/A081277}{A081277} & $\GT{1}{1+2v,0;1+v}$ 
& \nom{\frac{2}{3}}{\frac{2}{9}n} &
\href{https://oeis.org/A120909}{A120909} & $\ET{1+2v,0;3}$ 
& \nom{\frac{2}{3}}{\frac{2}{9}n}\\
\href{https://oeis.org/A120910}{A120910} & $\ET{2+v,0;3}$ 
& \nom{\frac{1}{3}}{\frac{2}{9}n} &
\href{https://oeis.org/A123187}{A123187} & $\ET{1+13v,0;1}$ 
& \nom{\frac{13}{14}}{\frac{13}{196}n}\\
\href{https://oeis.org/A133371}{A133371} & $\ET{13+v,0;1}$ 
& \nom{\frac{1}{14}}{\frac{13}{196}n} &
\href{https://oeis.org/A136158}{A136158} & $\GT{1}{3+v,0;1+v}$ 
& \nom{\frac{1}{4}}{\frac{3}{16}n}\\
\href{https://oeis.org/A147716}{A147716} & $\ET{14+v,0;1}$ 
& \nom{\frac{1}{15}}{\frac{14}{225}n} &
\href{https://oeis.org/A183190}{A183190} & $\GT{1}{2+v,0;1}$ 
& \nom{\frac{1}{3}}{\frac{2}{9}n}\\
\href{https://oeis.org/A193722}{A193722} & $\GT{1}{1+3v,0;1+2v}$ 
& \nom{\frac{3}{4}}{\frac{3}{16}n} &
\href{https://oeis.org/A193723}{A193723} & $\GT{1}{3+v,0;2+v}$ 
& \nom{\frac{1}{4}}{\frac{3}{16}n}\\
\href{https://oeis.org/A193724}{A193724} & $\GT{1}{2+3v,0;1+v}$ 
& \nom{\frac{3}{5}}{\frac{6}{25}n} &
\href{https://oeis.org/A193725}{A193725} & $\GT{1}{3+2v,0;1+v}$ 
& \nom{\frac{2}{5}}{\frac{6}{25}n}\\
\href{https://oeis.org/A193726}{A193726} & $\GT{1}{2+5v,0;1+2v}$ 
& \nom{\frac{5}{7}}{\frac{10}{49}n} &
\href{https://oeis.org/A193727}{A193727} & $\GT{1}{5+2v,0;2+v}$ 
& \nom{\frac{2}{7}}{\frac{10}{49}n}\\
\href{https://oeis.org/A193728}{A193728} & $\GT{1}{4+3v,0;2+v}$ 
& \nom{\frac{3}{7}}{\frac{12}{49}n} &
\href{https://oeis.org/A193729}{A193729} & $\GT{1}{3+4v,0;1+2v}$ 
& \nom{\frac{4}{7}}{\frac{12}{49}n}\\
\href{https://oeis.org/A193730}{A193730} & $\GT{1}{2+3v,0;2+v}$ 
& \nom{\frac{3}{5}}{\frac{6}{25}n} &
\href{https://oeis.org/A193731}{A193731} & $\GT{1}{3+2v,0;1+2v}$ 
& \nom{\frac{2}{5}}{\frac{6}{25}n}\\
\href{https://oeis.org/A193734}{A193734} & $\GT{1}{1+4v,0;1+2v}$ 
& \nom{\frac{4}{5}}{\frac{4}{25}n} &
\href{https://oeis.org/A193735}{A193735} & $\GT{1}{4+v,0;2+v}$ 
& \nom{\frac{1}{5}}{\frac{4}{25}n}\\
\href{https://oeis.org/A200139}{A200139} & $\GT{1}{2+v,0;1+v}$ 
& \nom{\frac{1}{3}}{\frac{2}{9}n} &
\href{https://oeis.org/A201780}{A201780} & $\GT{2}{2+v,0;(1+v)^2}$ 
& \nom{\frac{1}{3}}{\frac{2}{9}n}\\
\href{https://oeis.org/A207628}{A207628} & $\GT{1}{1+2v,0;1+4v}$ 
& \nom{\frac{2}{3}}{\frac{2}{9}n} &
\href{https://oeis.org/A207636}{A207636} & $\GT{1}{2+v,0;3+2v}$ 
& \nom{\frac{1}{3}}{\frac{2}{9}n}\\
\href{https://oeis.org/A208659}{A208659} & $\GT{1}{1+2v,0;2+2v}$ 
& \nom{\frac{2}{3}}{\frac{2}{9}n} &
\href{https://oeis.org/A209149}{A209149} & $\GT{1}{2+v,0;3+v}$ 
& \nom{\frac{1}{3}}{\frac{2}{9}n}\\ \hline
\end{longtable}    
\end{center}
\end{scriptsize}

\paragraph{$a_n(v) = p+qv+rv^2
\Longrightarrow \mathscr{N}\lpa{\frac{q+2r}{p+q+r}n,
\frac{pq+4pr+qr}{(p+q+r)^2}n}$}
\begin{scriptsize}
\begin{center}
\renewcommand{\arraystretch}{1.3}
\begin{longtable}{cll|cll}
\multicolumn{6}{c}{{}} \\
\multicolumn{6}{c}{{}} \\
\multicolumn{1}{c}{OEIS} &
\multicolumn{1}{c}{Type} &
\multicolumn{1}{c}{CLT} &
\multicolumn{1}{c}{OEIS} &
\multicolumn{1}{c}{Type} &
\multicolumn{1}{c}{CLT} \\ \hline
\endhead
\multicolumn{6}{c}{{Continued on next page}} \\
\endfoot
\endlastfoot
\href{https://oeis.org/A152905}{A152905} & $\ET{1+v^2,0;1+v}$ 
& \nom{n}{n} &
\href{https://oeis.org/A249095}{A249095} & $\GT{1}{1+v^2,0;1+v+v^2}$ 
& \nom{n}{n}\\
\href{https://oeis.org/A260492}{A260492} & $\ET{1+v^2,0;1}$ 
& \nom{n}{n} &
\href{https://oeis.org/A249307}{A249307} & $\GT{1}{1+4v^2,0;1+2v+4v^2}$ 
& \nom{\frac{8}{5}n}{\frac{16}{25}n}\\
\href{https://oeis.org/A034870}{A034870} & $\ET{(1+v)^2,0;1}$ 
& \nom{n}{\frac{1}{2}n} &
\href{https://oeis.org/A096646}{A096646} & $\GT{1}{(1+v)^2,0;1+v+v^2}$ 
& \nom{n}{\frac{1}{2}n}\\
\href{https://oeis.org/A139548}{A139548} & $\ET{2(1+v)^2,0;1}$ 
& \nom{n}{\frac{1}{2}n} &
\href{https://oeis.org/A024996}{A024996} & $\GT{2}{1+v+v^2,0;1+2v^2+v^4}$ 
& \nom{n}{\frac{2}{3}n}\\
\href{https://oeis.org/A025177}{A025177} & $\GT{1}{1+v+v^2,0;1+v^2}$ 
& \nom{n}{\frac{2}{3}n} &
\href{https://oeis.org/A025564}{A025564} & $\GT{1}{1+v+v^2,0;1+2v+v^2}$ 
& \nom{n}{\frac{2}{3}n}\\
\href{https://oeis.org/A027907}{A027907} & $\ET{1+v+v^2,0;1}$ 
& \nom{\frac54n}{\frac{11}{16}n} &
\href{https://oeis.org/A084600}{A084600} & $\ET{1+v+2v^2,0;1}$ 
& \nom{n}{\frac{2}{3}n}\\
\href{https://oeis.org/A084602}{A084602} & $\ET{1+v+3v^2,0;1}$ 
& \nom{\frac{7}{5}n}{\frac{16}{25}n} &
\href{https://oeis.org/A084604}{A084604} & $\ET{1+v+4v^2,0;1}$ 
& \nom{\frac{3}{2}n}{\frac{7}{12}n}\\
\href{https://oeis.org/A084606}{A084606} & $\ET{1+2v+2v^2,0;1}$ 
& \nom{\frac{6}{5}n}{\frac{14}{25}n} &
\href{https://oeis.org/A084608}{A084608} & $\ET{1+2v+3v^2,0;1}$ 
& \nom{\frac{4}{3}n}{\frac{5}{9}n}\\
\href{https://oeis.org/A200536}{A200536} & $\ET{1+3v+2v^2,0;1}$ 
& \nom{\frac{7}{6}n}{\frac{17}{36}n} &
\href{https://oeis.org/A272866}{A272866} & $\ET{1+3v+v^2,0;1}$ 
& \nom{n}{\frac{2}{5}n}\\
\href{https://oeis.org/A272867}{A272867} & $\ET{1+4v+v^2,0;1}$ 
& \nom{n}{\frac{1}{3}n}\\ \hline
\end{longtable}    
\end{center}    
\end{scriptsize}

\footnotesize
\bibliographystyle{abbrv}
\bibliography{eulerian-rr}

\begin{thebibliography}{100}

\bibitem{Acan2016}
H.~Acan and P.~Hitczenko.
\newblock On a memory game and preferential attachment graphs.
\newblock {\em Adv. in Appl. Probab.}, 48(2):585--609, 2016.

\bibitem{Adin2001}
R.~M. Adin, F.~Brenti, and Y.~Roichman.
\newblock Descent numbers and major indices for the hyperoctahedral group.
\newblock {\em Adv. in Appl. Math.}, 27(2-3):210--224, 2001.

\bibitem{Andre1884}
D.~Andr{\'e}.
\newblock {\'E}tude sur les maxima, minima et s{\'e}quences des permutations.
\newblock {\em Ann. Sci. {\'E}c. Norm. Sup{\'e}r.}, 1:121--134, 1884.

\bibitem{Andre1895}
D.~Andr{\'e}.
\newblock M{\'e}moire sur les s{\'e}quences des permutations circulaires.
\newblock {\em Bull. Soc. Math. France}, 23:122--184, 1895.

\bibitem{Andre1906}
D.~Andr\'{e}.
\newblock M\'{e}moire sur les inversions \'{e}l\'{e}mentaires des permutations.
\newblock {\em Memorie della Pontificia accademia romana dei nuovi Lincei},
  24:189--223, 1906.

\bibitem{Aval2013}
J.-C. Aval, A.~Boussicault, and S.~Dasse-Hartaut.
\newblock The tree structure in staircase tableaux.
\newblock {\em S\'em. Lothar. Combin.}, 70:Art. B70g, 11, 2013.

\bibitem{Aval2013a}
J.-C. Aval, A.~Boussicault, and P.~Nadeau.
\newblock Tree-like tableaux.
\newblock {\em Electron. J. Combin.}, 20(4):Paper 34, 24, 2013.

\bibitem{Bagchi1985}
A.~Bagchi and A.~K. Pal.
\newblock Asymptotic normality in the generalized {P}\'olya-{E}ggenberger urn
  model, with an application to computer data structures.
\newblock {\em SIAM J. Algebraic Discrete Methods}, 6(3):394--405, 1985.

\bibitem{Barbero2014}
J.~F. Barbero~G., J.~Salas, and E.~J.~S. Villase\~nor.
\newblock Bivariate generating functions for a class of linear recurrences:
  general structure.
\newblock {\em J. Combin. Theory Ser. A}, 125:146--165, 2014.

\bibitem{Barbero2015}
J.~F. Barbero~G., J.~Salas, and E.~J.~S. Villase\~nor.
\newblock Generalized {S}tirling permutations and forests: higher-order
  {E}ulerian and {W}ard numbers.
\newblock {\em Electron. J. Combin.}, 22(3):Paper 3.37, 20, 2015.

\bibitem{Barry2013}
P.~Barry.
\newblock General {E}ulerian polynomials as moments using exponential {R}iordan
  arrays.
\newblock {\em J. Integer Seq.}, 16(9):Article 13.9.6, 15, 2013.

\bibitem{Barton1965}
D.~E. Barton and C.~L. Mallows.
\newblock Some aspects of the random sequence.
\newblock {\em Ann. Math. Statist.}, 36:236--260, 1965.

\bibitem{Beery2009}
J.~Beery and J.~Stedall, editors.
\newblock {\em Thomas {H}arriot's Doctrine of Triangular Numbers: the
  `{M}agisteria Magna'}.
\newblock European Mathematical Society (EMS), Z\"urich, 2009.

\bibitem{Bender1973}
E.~A. Bender.
\newblock Central and local limit theorems applied to asymptotic enumeration.
\newblock {\em J. Combinatorial Theory Ser. A}, 15:91--111, 1973.

\bibitem{Bergeron1992}
F.~Bergeron, P.~Flajolet, and B.~Salvy.
\newblock Varieties of increasing trees.
\newblock In {\em C{AAP} '92 ({R}ennes, 1992)}, volume 581 of {\em Lecture
  Notes in Comput. Sci.}, pages 24--48. Springer, Berlin, 1992.

\bibitem{Bienayme1874}
I.~J. Bienaym\'e.
\newblock Sur une question de probabilit\'es.
\newblock {\em Bull. Soc. Math. France}, 2:153--154, 1874.

\bibitem{Bienayme1875}
I.~J. Bienaym{\'e}.
\newblock Application d'un th{\'e}or{\`e}me nouveau du calcul des
  probabilit{\'e}s.
\newblock {\em C. R. Acad. Sci. Paris}, 81:417--423, 1875.

\bibitem{Bona2004}
M.~B\'{o}na.
\newblock {\em Combinatorics of Permutations}.
\newblock Chapman \& Hall/CRC, Boca Raton, FL, 2004.
\newblock With a foreword by Richard Stanley.

\bibitem{Bona2008}
M.~B\'ona.
\newblock Real zeros and normal distribution for statistics on {S}tirling
  permutations defined by {G}essel and {S}tanley.
\newblock {\em SIAM J. Discrete Math.}, 23(1):401--406, 2008/09.

\bibitem{Borowiec2016}
A.~Borowiec and W.~M{\l}otkowski.
\newblock New {E}ulerian numbers of type {$D$}.
\newblock {\em Electron. J. Combin.}, 23(1):Paper 1.38, 13, 2016.

\bibitem{Branden2015}
P.~Br\"and\'en.
\newblock Unimodality, log-concavity, real-rootedness and beyond.
\newblock In {\em Handbook of Enumerative Combinatorics}, pages 437--483. CRC
  Press, Boca Raton, FL, 2015.

\bibitem{Brenti1989}
F.~Brenti.
\newblock Unimodal, log-concave and {P}\'olya frequency sequences in
  combinatorics.
\newblock {\em Mem. Amer. Math. Soc.}, 81(413):viii+106, 1989.

\bibitem{Brenti1990}
F.~Brenti.
\newblock Unimodal polynomials arising from symmetric functions.
\newblock {\em Proc. Amer. Math. Soc.}, 108(4):1133--1141, 1990.

\bibitem{Brenti1994a}
F.~Brenti.
\newblock Log-concave and unimodal sequences in algebra, combinatorics, and
  geometry: an update.
\newblock In {\em Jerusalem Combinatorics '93}, volume 178 of {\em Contemp.
  Math.}, pages 71--89. Amer. Math. Soc., Providence, RI, 1994.

\bibitem{Brenti1994}
F.~Brenti.
\newblock {$q$}-{E}ulerian polynomials arising from {C}oxeter groups.
\newblock {\em European J. Combin.}, 15(5):417--441, 1994.

\bibitem{Canfield2015}
E.~R. Canfield.
\newblock Asymptotic normality in enumeration.
\newblock In {\em Handbook of Enumerative Combinatorics}, Discrete Math. Appl.
  (Boca Raton), pages 255--280. CRC Press, Boca Raton, FL, 2015.

\bibitem{Carlitz1958}
L.~Carlitz.
\newblock Eulerian numbers and polynomials.
\newblock {\em Math. Mag.}, 32:247--260, 1958/1959.

\bibitem{Carlitz1960}
L.~Carlitz.
\newblock Eulerian numbers and polynomials of higher order.
\newblock {\em Duke Math. J.}, 27:401--423, 1960.

\bibitem{Carlitz1965}
L.~Carlitz.
\newblock The coefficients in an asymptotic expansion.
\newblock {\em Proc. Amer. Math. Soc.}, 16:248--252, 1965.

\bibitem{Carlitz1973}
L.~Carlitz.
\newblock Enumeration of permutations by rises and cycle structure.
\newblock {\em J. Reine Angew. Math.}, 262/263:220--233, 1973.

\bibitem{Carlitz1978}
L.~Carlitz.
\newblock Some polynomials related to {F}ibonacci and {E}ulerian numbers.
\newblock {\em Fibonacci Quart.}, 16(3):216--226, 1978.

\bibitem{Carlitz1978b}
L.~Carlitz.
\newblock Some remarks on the {E}ulerian function.
\newblock {\em Univ. Beograd. Publ. Elektrotehn. Fak. Ser. Mat. Fiz.},
  (602-633):79--91 (1979), 1978.

\bibitem{Carlitz1979}
L.~Carlitz.
\newblock Degenerate {S}tirling, {B}ernoulli and {E}ulerian numbers.
\newblock {\em Utilitas Math.}, 15:51--88, 1979.

\bibitem{Carlitz1972}
L.~Carlitz, D.~C. Kurtz, R.~Scoville, and O.~P. Stackelberg.
\newblock Asymptotic properties of {E}ulerian numbers.
\newblock {\em Z. Wahrscheinlichkeitstheorie und Verw. Gebiete}, 23:47--54,
  1972.

\bibitem{Carlitz1953}
L.~Carlitz and J.~Riordan.
\newblock Congruences for {E}ulerian numbers.
\newblock {\em Duke Math. J.}, 20:339--343, 1953.

\bibitem{Carlitz1966}
L.~Carlitz, D.~P. Roselle, and R.~A. Scoville.
\newblock Permutations and sequences with repetitions by number of increases.
\newblock {\em J. Combinatorial Theory}, 1:350--374, 1966.

\bibitem{Carlitz1974}
L.~Carlitz and R.~Scoville.
\newblock Generalized {E}ulerian numbers: combinatorial applications.
\newblock {\em J. Reine Angew. Math.}, 265:110--137, 1974.

\bibitem{Carlitz1973a}
L.~Carlitz, R.~Scoville, and T.~Vaughan.
\newblock Enumeration of permutations and sequences with restrictions.
\newblock {\em Duke Math. J.}, 40:723--741, 1973.

\bibitem{Caro-Lopera2015}
F.~J. Caro-Lopera, G.~Gonz\'alez-Far\'{\i}as, and N.~Balakrishnan.
\newblock The generalized {P}ascal triangle and the matrix variate
  {J}ensen-logistic distribution.
\newblock {\em Comm. Statist. Theory Methods}, 44(13):2738--2752, 2015.

\bibitem{Chao1996}
C.-C. Chao, L.~Zhao, and W.-Q. Liang.
\newblock Estimating the error of a permutational central limit theorem.
\newblock {\em Probab. Engrg. Inform. Sci.}, 10(4):533--541, 1996.

\bibitem{Charalambides1982}
C.~A. Charalambides.
\newblock On the enumeration of certain compositions and related sequences of
  numbers.
\newblock {\em Fibonacci Quart.}, 20(2):132--146, 1982.

\bibitem{Charalambides1991}
C.~A. Charalambides.
\newblock On a generalized {E}ulerian distribution.
\newblock {\em Ann. Inst. Statist. Math.}, 43(1):197--206, 1991.

\bibitem{Charalambides2002a}
C.~A. Charalambides.
\newblock {\em Enumerative Combinatorics}.
\newblock Chapman \& Hall/CRC, Boca Raton, FL, 2002.

\bibitem{Charalambides2002}
C.~A. Charalambides.
\newblock The rook numbers of {F}errers boards and the related restricted
  permutation numbers.
\newblock {\em J. Statist. Plann. Inference}, 101(1-2):33--48, 2002.

\bibitem{Charalambides1993}
C.~A. Charalambides and M.~V. Koutras.
\newblock On a generalization of {M}orisita's model for estimating the habitat
  preference.
\newblock {\em Ann. Inst. Statist. Math.}, 45(2):201--210, 1993.

\bibitem{Chatterjee2017}
S.~Chatterjee and P.~Diaconis.
\newblock A central limit theorem for a new statistic on permutations.
\newblock {\em Indian J. Pure Appl. Math.}, 48(4):561--573, 2017.

\bibitem{Chebikin2008}
D.~Chebikin.
\newblock Variations on descents and inversions in permutations.
\newblock {\em Electron. J. Combin.}, 15(1):Research Paper 132, 34, 2008.

\bibitem{Chen2004}
L.~H.~Y. Chen, T.~N.~T. Goodman, and S.~L. Lee.
\newblock Asymptotic normality of scaling functions.
\newblock {\em SIAM J. Math. Anal.}, 36(1):323--346, 2004.

\bibitem{Chen2009}
W.~Y.~C. Chen, R.~L. Tang, and A.~F.~Y. Zhao.
\newblock Derangement polynomials and excedances of type {$B$}.
\newblock {\em Electron. J. Combin.}, 16(Research Paper 15):16, 2009.

\bibitem{Chern2002}
H.-H. Chern, H.-K. Hwang, and T.-H. Tsai.
\newblock An asymptotic theory for {C}auchy-{E}uler differential equations with
  applications to the analysis of algorithms.
\newblock {\em J. Algorithms}, 44(1):177--225, 2002.

\bibitem{Chow2003}
C.-O. Chow.
\newblock On the {E}ulerian polynomials of type {$D$}.
\newblock {\em European J. Combin.}, 24(4):391--408, 2003.

\bibitem{Chow2014a}
C.-O. Chow and S.-M. Ma.
\newblock Counting signed permutations by their alternating runs.
\newblock {\em Discrete Math.}, 323:49--57, 2014.

\bibitem{Chow2014b}
C.-O. Chow, S.-M. Ma, T.~Mansour, and M.~Shattuck.
\newblock Counting permutations by cyclic peaks and valleys.
\newblock {\em Ann. Math. Inform.}, 43:43--54, 2014.

\bibitem{Chow2012}
C.-O. Chow and T.~Mansour.
\newblock Asymptotic probability distributions of some permutation statistics
  for the wreath product {$C_r\wr S_n$}. 
\newblock {\em Online J. Anal. Comb.}, (7):14, 2012.

\bibitem{Chow1988}
Y.~S. Chow and H.~Teicher.
\newblock {\em Probability Theory}.
\newblock Springer Texts in Statistics. Springer-Verlag, New York, second
  edition, 1988.

\bibitem{Chui1992}
C.~K. Chui.
\newblock {\em An Introduction to Wavelets}.
\newblock Academic Press, Inc., Boston, MA, 1992.

\bibitem{Chung1949}
K.~L. Chung and W.~Feller.
\newblock On fluctuations in coin-tossing.
\newblock {\em Proc. Nat. Acad. Sci. U. S. A.}, 35:605--608, 1949.

\bibitem{Chuntee2017}
W.~Chuntee and K.~Neammanee.
\newblock Exponential bounds for normal approximation of the number of descents
  and inversions.
\newblock {\em Comm. Statist. Theory Methods}, 46(3):1218--1229, 2017.

\bibitem{Clark1998}
L.~Clark.
\newblock Asymptotic normality of the generalized {E}ulerian numbers.
\newblock {\em Ars Combin.}, 48:213--218, 1998.

\bibitem{Clark2002}
L.~Clark.
\newblock Central and local limit theorems for excedances by conjugacy class
  and by derangement.
\newblock {\em Integers}, 2:Paper A3, 9, 2002.

\bibitem{Comtet1974}
L.~Comtet.
\newblock {\em Advanced Combinatorics}.
\newblock D. Reidel Publishing Co., Dordrecht, enlarged edition, 1974.
\newblock The art of finite and infinite expansions.

\bibitem{Conger2007}
M.~Conger and D.~Viswanath.
\newblock Normal approximations for descents and inversions of permutations of
  multisets.
\newblock {\em J. Theoret. Probab.}, 20(2):309--325, 2007.

\bibitem{Conger2010}
M.~A. Conger.
\newblock A refinement of the {E}ulerian numbers, and the joint distribution of
  {$\pi(1)$} and {${\rm Des}(\pi)$} in {$S_n$}.
\newblock {\em Ars Combin.}, 95:445--472, 2010.

\bibitem{Conway1988}
J.~H. Conway and N.~J.~A. Sloane.
\newblock Voronoi cells of lattices and quantization errors.
\newblock In {\em Sphere Packings, Lattices and Groups}, pages 449--475.
  Springer, New York, NY, 1988.

\bibitem{Conway1997}
J.~H. Conway and N.~J.~A. Sloane.
\newblock Low-dimensional lattices. {VII}. {C}oordination sequences.
\newblock {\em Proc. Roy. Soc. London Ser. A}, 453(1966):2369--2389, 1997.

\bibitem{Corcino2018}
C.~B. Corcino, R.~B. Corcino, I.~Mez{\H{o}}, and J.~L. Ram{\'\i}rez.
\newblock Some polynomials associated with the r-{W}hitney numbers.
\newblock {\em Proc. Indian Acad. Sci. (Math. Sci.)}, 128(3):27, Jun 2018.

\bibitem{Corless1996}
R.~M. Corless, G.~H. Gonnet, D.~E.~G. Hare, D.~J. Jeffrey, and D.~E. Knuth.
\newblock On the {L}ambert {$W$} function.
\newblock {\em Adv. Comput. Math.}, 5(4):329--359, 1996.

\bibitem{Curry1966}
H.~B. Curry and I.~J. Schoenberg.
\newblock On {P}\'{o}lya frequency functions. {IV}. {T}he fundamental spline
  functions and their limits.
\newblock {\em J. Analyse Math.}, 17:71--107, 1966.

\bibitem{Dais2001}
D.~I. Dais.
\newblock On the string-theoretic {E}uler number of a class of absolutely
  isolated singularities.
\newblock {\em Manuscripta Math.}, 105(2):143--174, 2001.

\bibitem{Dale1988}
M.~R.~T. Dale and J.~W. Moon.
\newblock Statistical tests on two characteristics of the shapes of cluster
  diagrams.
\newblock {\em J. Classification}, 5(1):21--38, 1988.

\bibitem{Dasse-Hartaut2013}
S.~Dasse-Hartaut and P.~Hitczenko.
\newblock Greek letters in random staircase tableaux.
\newblock {\em Random Structures Algorithms}, 42(1):73--96, 2013.

\bibitem{David1962}
F.~N. David and D.~E. Barton.
\newblock {\em Combinatorial Chance}.
\newblock Hafner Publishing Co., New York, 1962.

\bibitem{deMoivre1738}
A.~de~Moivre.
\newblock {\em The Doctrine of Chances}.
\newblock Woodfall, London, second edition, 1738.

\bibitem{Diaconis1988}
P.~Diaconis.
\newblock {\em Group Representations in Probability and Statistics}, volume~11
  of {\em Institute of Mathematical Statistics Lecture Notes---Monograph
  Series}.
\newblock Institute of Mathematical Statistics, Hayward, CA, 1988.

\bibitem{Diaconis1987}
P.~Diaconis and B.~Efron.
\newblock Probabilistic-geometric theorems arising from the analysis of
  contingency tables.
\newblock In {\em Contributions to the Theory and Application of Statistics},
  pages 103--125. Academic Press, Boston, MA, 1987.

\bibitem{Diaconis2009}
P.~Diaconis and J.~Fulman.
\newblock Carries, shuffling, and symmetric functions.
\newblock {\em Adv. in Appl. Math.}, 43(2):176--196, 2009.

\bibitem{Dillon1968}
J.~F. Dillon and D.~P. Roselle.
\newblock Eulerian numbers of higher order.
\newblock {\em Duke Math. J.}, 35:247--256, 1968.

\bibitem{Dominici2011}
D.~Dominici, K.~Driver, and K.~Jordaan.
\newblock Polynomial solutions of differential-difference equations.
\newblock {\em J. Approx. Theory}, 163(1):41--48, 2011.

\bibitem{Drmota1997}
M.~Drmota and M.~Soria.
\newblock Images and preimages in random mappings.
\newblock {\em SIAM J. Discrete Math.}, 10(2):246--269, 1997.

\bibitem{Dubeau1995}
F.~Dubeau and J.~Savoie.
\newblock On the roots of orthogonal polynomials and {E}uler-{F}robenius
  polynomials.
\newblock {\em J. Math. Anal. Appl.}, 196(1):84--98, 1995.

\bibitem{Dwass1973}
M.~Dwass.
\newblock The number of increases in a random permutation.
\newblock {\em J. Combinatorial Theory Ser. A}, 15:192--199, 1973.

\bibitem{Dwyer1940}
P.~S. Dwyer.
\newblock The cumulative numbers and their polynomials.
\newblock {\em Ann. Math. Statistics}, 11:66--71, 1940.

\bibitem{Egge2010}
E.~S. Egge.
\newblock Legendre-{S}tirling permutations.
\newblock {\em European J. Combin.}, 31(7):1735--1750, 2010.

\bibitem{Elizalde2003}
S.~Elizalde and M.~Noy.
\newblock Consecutive patterns in permutations.
\newblock {\em Adv. in Appl. Math.}, 30(1-2):110--125, 2003.

\bibitem{Entringer1969}
R.~C. Entringer.
\newblock Enumeration of permutations of {$(1,\cdots,n)$} by number of maxima.
\newblock {\em Duke Math. J.}, 36:575--579, 1969.

\bibitem{Erdelyi1981}
A.~Erd{\'e}lyi, W.~Magnus, F.~Oberhettinger, and F.~G. Tricomi.
\newblock {\em Higher Transcendental Functions. {V}ol. {I}}.
\newblock Robert E. Krieger Publishing Co. Inc., Melbourne, Fla., 1981.

\bibitem{Eriksen2000}
N.~Eriksen, H.~Eriksson, and K.~Eriksson.
\newblock Diagonal checker-jumping and {E}ulerian numbers for color-signed
  permutations.
\newblock {\em Electron. J. Combin.}, 7:Research Paper 3, 11, 2000.

\bibitem{Esseen1985}
C.-G. Esseen.
\newblock On the application of the theory of probability to two combinatorial
  problems involving permutations.
\newblock In {\em Proceedings of the Seventh Conference on Probability Theory
  ({B}ra\c sov, 1982)}, pages 137--147. VNU Sci. Press, Utrecht, 1985.

\bibitem{Eu2014}
S.-P. Eu, T.-S. Fu, and Y.-J. Pan.
\newblock A refined sign-balance of simsun permutations.
\newblock {\em European J. Combin.}, 36:97--109, 2014.

\bibitem{Euler1741}
L.~Euler.
\newblock Methodus universalis series summandi ulterius promota.
\newblock {\em Commentarii Academiae Scientiarum Petropolitanae}, 8:147--158,
  1741.
\newblock (first presented to the St. Petersburg Academy on September 17,
  1736).

\bibitem{Euler1755}
L.~Euler.
\newblock {\em Institutiones calculi differentialis cum eius usu in analysi
  finitorum ac Doctrina serierum}.
\newblock Academiae Imperialis Scientiarum Petropolitanae, St. Petersbourg,
  1755.

\bibitem{Evans2000}
G.~Evans, J.~Blackledge, and P.~Yardley.
\newblock {\em Numerical Methods for Partial Differential Equations}.
\newblock Springer-Verlag London, Ltd., London, 2000.

\bibitem{Everitt2002}
W.~N. Everitt, L.~L. Littlejohn, and R.~Wellman.
\newblock Legendre polynomials, {L}egendre-{S}tirling numbers, and the
  left-definite spectral analysis of the {L}egendre differential expression.
\newblock {\em J. Comput. Appl. Math.}, 148(1):213--238, 2002.

\bibitem{Feller1945}
W.~Feller.
\newblock The fundamental limit theorems in probability.
\newblock {\em Bull. Amer. Math. Soc.}, 51:800--832, 1945.

\bibitem{Fischer2011}
H.~Fischer.
\newblock {\em A History of the Central Limit Theorem: From Classical to Modern
  Probability Theory}.
\newblock Springer, New York, 2011.

\bibitem{Flajolet2006}
P.~Flajolet, P.~Dumas, and V.~Puyhaubert.
\newblock Some exactly solvable models of urn process theory.
\newblock In {\em Fourth {C}olloquium on {M}athematics and {C}omputer {S}cience
  {A}lgorithms, {T}rees, {C}ombinatorics and {P}robabilities}, Discrete Math.
  Theor. Comput. Sci. Proc., AG, pages 59--118, 2006.

\bibitem{Flajolet1997}
P.~Flajolet, X.~Gourdon, and C.~Mart\'\i~nez.
\newblock Patterns in random binary search trees.
\newblock {\em Random Structures Algorithms}, 11(3):223--244, 1997.

\bibitem{Flajolet1990}
P.~Flajolet and A.~Odlyzko.
\newblock Singularity analysis of generating functions.
\newblock {\em SIAM J. Discrete Math.}, 3(2):216--240, 1990.

\bibitem{Flajolet2009}
P.~Flajolet and R.~Sedgewick.
\newblock {\em Analytic Combinatorics}.
\newblock Cambridge University Press, Cambridge, 2009.

\bibitem{Flajolet1993}
P.~Flajolet and M.~Soria.
\newblock General combinatorial schemas: {G}aussian limit distributions and
  exponential tails.
\newblock {\em Discrete Math.}, 114(1-3):159--180, 1993.

\bibitem{Foata1970}
D.~Foata and M.-P. Sch\"{u}tzenberger.
\newblock {\em Th\'eorie g\'eom\'etrique des polyn\^omes eul\'eriens}.
\newblock Lecture Notes in Mathematics, Vol. 138. Springer-Verlag, Berlin-New
  York, 1970.

\bibitem{Foulkes1980}
H.~O. Foulkes.
\newblock Eulerian numbers, {N}ewcomb's problem and representations of
  symmetric groups.
\newblock {\em Discrete Math.}, 30(1):3--49, 1980.

\bibitem{Franssens2006}
G.~R. Franssens.
\newblock On a number pyramid related to the binomial, {D}eleham, {E}ulerian,
  {M}ac{M}ahon and {S}tirling number triangles.
\newblock {\em J. Integer Seq.}, 9(4):Article 06.4.1, 34, 2006.

\bibitem{Frechet1931}
M.~Fr\'echet and J.~Shohat.
\newblock A proof of the generalized second-limit theorem in the theory of
  probability.
\newblock {\em Trans. Amer. Math. Soc.}, 33(2):533--543, 1931.

\bibitem{Freedman1965}
D.~A. Freedman.
\newblock Bernard {F}riedman's urn.
\newblock {\em Ann. Math. Statist}, 36:956--970, 1965.

\bibitem{Frobenius1910}
G.~Frobenius.
\newblock \"{U}ber die {B}ernoullischen {Z}ahlen und die {E}ulerschen
  {P}olynome.
\newblock {\em Sitzungsberichte Berliner Akademie der Wissenschaften}, pages
  809--847, 1910.

\bibitem{Fulman2004}
J.~Fulman.
\newblock Stein's method and non-reversible {M}arkov chains.
\newblock In {\em Stein's Method: Expository Lectures and Applications},
  volume~46 of {\em IMS Lecture Notes Monogr. Ser.}, pages 69--77. Inst. Math.
  Statist., Beachwood, OH, 2004.

\bibitem{Fulman2019}
J.~Fulman, G.~B. Kim, and S.~Lee.
\newblock Central limit theorem for peaks of a random permutation in a fixed
  conjugacy class of ${S}_n$.
\newblock {\em arXiv preprint arXiv:1902.00978}, 2019.

\bibitem{Gao1992}
Z.~Gao and L.~B. Richmond.
\newblock Central and local limit theorems applied to asymptotic enumeration.
  {IV}. {M}ultivariate generating functions.
\newblock {\em J. Comput. Appl. Math.}, 41(1-2):177--186, 1992.

\bibitem{Gautschi1959}
W.~Gautschi.
\newblock Exponential integral $\int_1^\infty e^{-xt} t^{-n}dt$ for large
  values of $n$.
\newblock {\em J. Res. Nat. Bur. Standards}, 62:123--125, 1959.

\bibitem{Gawronski2013}
W.~Gawronski and T.~Neuschel.
\newblock Euler-{F}robenius numbers.
\newblock {\em Integral Transforms Spec. Funct.}, 24(10):817--830, 2013.

\bibitem{Gessel1978}
I.~Gessel and R.~P. Stanley.
\newblock Stirling polynomials.
\newblock {\em J. Combinatorial Theory Ser. A}, 24(1):24--33, 1978.

\bibitem{Gessel1977}
I.~M. Gessel.
\newblock {\em Generating Functions and Enumeration of Sequences.}
\newblock PhD thesis, MIT, 1977.

\bibitem{Gessel1992}
I.~M. Gessel.
\newblock Super ballot numbers.
\newblock {\em J. Symbolic Comput.}, 14(2-3):179--194, 1992.

\bibitem{Giladi1994}
E.~Giladi and J.~B. Keller.
\newblock Eulerian number asymptotics.
\newblock {\em Proc. Roy. Soc. London Ser. A}, 445(1924):291--303, 1994.

\bibitem{Goncharov1942}
V.~Goncharov.
\newblock Sur la distribution des cycles dans les permutations.
\newblock {\em C. R. (Doklady) Acad. Sci. URSS (N.S.)}, 35:267--269, 1942.

\bibitem{Goncharov1944}
V.~Goncharov.
\newblock Du domaine de l'analyse combinatoire.
\newblock {\em Bull. Acad. Sci. URSS S\'er. Math. [Izvestia Akad. Nauk SSSR]},
  8:3--48, 1944.

\bibitem{Gorenflo2014}
R.~Gorenflo, A.~A. Kilbas, F.~Mainardi, and S.~V. Rogosin.
\newblock {\em Mittag-Leffler Functions, Related Topics and Applications}.
\newblock Springer-Verlag, Berlin, 2014, 2014.

\bibitem{Goulden1983}
I.~P. Goulden and D.~M. Jackson.
\newblock {\em Combinatorial Enumeration}.
\newblock John Wiley \&\ Sons, Inc., New York, 1983.

\bibitem{Graham1994}
R.~L. Graham, D.~E. Knuth, and O.~Patashnik.
\newblock {\em Concrete Mathematics}.
\newblock Addison-Wesley, Reading, MA, second edition, 1994.

\bibitem{Hackl2018}
B.~Hackl and H.~Prodinger.
\newblock The necklace process: a generating function approach.
\newblock {\em Statist. Probab. Lett.}, 142:57--61, 2018.

\bibitem{Hald1998}
A.~Hald.
\newblock {\em A History of Mathematical Statistics from 1750 to 1930}.
\newblock John Wiley \& Sons, Inc., New York, 1998.

\bibitem{Harper1967}
L.~H. Harper.
\newblock Stirling behavior is asymptotically normal.
\newblock {\em Ann. Math. Statist.}, 38:410--414, 1967.

\bibitem{Harris1994}
B.~Harris and C.~Park.
\newblock A generalization of the {E}ulerian numbers with a probabilistic
  application.
\newblock {\em Stat. Probab. Lett.}, 20(1):37--47, 1994.

\bibitem{Hayman1956}
W.~K. Hayman.
\newblock A generalisation of {S}tirling's formula.
\newblock {\em J. Reine Angew. Math.}, 196:67--95, 1956.

\bibitem{Hensley1982}
D.~Hensley.
\newblock Eulerian numbers and the unit cube.
\newblock {\em Fibonacci Quart.}, 20(4):344--348, 1982.

\bibitem{Heyde1977}
C.~C. Heyde and E.~Seneta.
\newblock {\em I. {J}. {B}ienaym\'e. {S}tatistical Theory Anticipated}.
\newblock Springer-Verlag, New York-Heidelberg, 1977.
\newblock Studies in the History of Mathematics and Physical Sciences, No. 3.

\bibitem{Hitczenko2014}
P.~Hitczenko and S.~Janson.
\newblock Weighted random staircase tableaux.
\newblock {\em Combin. Probab. Comput.}, 23(6):1114--1147, 2014.

\bibitem{Hitczenko2016}
P.~Hitczenko and A.~Lohss.
\newblock Corners in tree-like tableaux.
\newblock {\em Electron. J. Combin.}, 23(4):Paper 4.26, 18, 2016.

\bibitem{Hitczenko2018}
P.~Hitczenko and A.~Lohss.
\newblock Probabilistic consequences of some polynomial recurrences.
\newblock {\em Random Structures Algorithms}, 53(4):652--666, 2018.

\bibitem{Hoeffding1948}
W.~Hoeffding and H.~Robbins.
\newblock The central limit theorem for dependent random variables.
\newblock {\em Duke Math. J.}, 15:773--780, 1948.

\bibitem{Hsu1999}
L.~C. Hsu and P.~J.-S. Shiue.
\newblock On certain summation problems and generalizations of {E}ulerian
  polynomials and numbers.
\newblock {\em Discrete Math.}, 204(1-3):237--247, 1999.

\bibitem{Hwang1994}
H.-K. Hwang.
\newblock {\em Th\'{e}or\`{e}mes limites pour les structures combinatoires et
  les fonctions arithm\'{e}tiques}.
\newblock PhD thesis, LIX, Ecole polytechnique, 1994.

\bibitem{Hwang1998}
H.-K. Hwang.
\newblock On convergence rates in the central limit theorems for combinatorial
  structures.
\newblock {\em European J. Combin.}, 19(3):329--343, 1998.

\bibitem{Hwang2003}
H.-K. Hwang.
\newblock Second phase changes in random {$m$}-ary search trees and generalized
  quicksort: convergence rates.
\newblock {\em Ann. Probab.}, 31(2):609--629, 2003.

\bibitem{Hwang2019-2}
H.-K. Hwang, H.-H. Chern, and G.-H. Duh.
\newblock Limit laws of the degenerate {E}ulerian recurrences.
\newblock \url{http://algo.stat.sinica.edu.tw/eulerian/b0.html}, 2019.

\bibitem{Hwang2019}
H.-K. Hwang, H.-H. Chern, and G.-H. Duh.
\newblock Limit laws of the {E}ulerian recurrences
  {$P_n(v)=a_n(v)P_{n-1}(v)+b_n(v)(1-v)P_{n-1}'(v)$}.
\newblock \url{http://algo.stat.sinica.edu.tw/eulerian/main.html}, 2019.

\bibitem{Hwang2015}
H.-K. Hwang and V.~Zacharovas.
\newblock Limit distribution of the coefficients of polynomials with only unit
  roots.
\newblock {\em Random Structures Algorithms}, 46(4):707--738, 2015.

\bibitem{Ikollo-Ndoumbe2016}
M.~Ikollo~Ndoumbe.
\newblock Une preuve de la formule g\'en\'eralis\'ee d'{E}uler-{F}robenius.
\newblock {\em IMHOTEP J. Afr. Math. Pures Appl.}, 1(1):1--6, 2016.

\bibitem{Janardan1988}
K.~G. Janardan.
\newblock Relationship between {M}orisita's model for estimating the
  environmental density and the generalized {E}ulerian numbers.
\newblock {\em Ann. Inst. Statist. Math.}, 40(3):439--450, 1988.

\bibitem{Janardan1993}
K.~G. Janardan.
\newblock Some properties of the generalized {E}ulerian distribution.
\newblock {\em J. Statist. Plann. Inference}, 34(2):159--169, 1993.

\bibitem{Janson2004}
S.~Janson.
\newblock Functional limit theorems for multitype branching processes and
  generalized {P}\'{o}lya urns.
\newblock {\em Stochastic Process. Appl.}, 110(2):177--245, 2004.

\bibitem{Janson2008}
S.~Janson.
\newblock Plane recursive trees, {S}tirling permutations and an urn model.
\newblock In {\em Fifth {C}olloquium on {M}athematics and {C}omputer
  {S}cience}, Discrete Math. Theor. Comput. Sci. Proc., AI, pages 541--547.
  Assoc. Discrete Math. Theor. Comput. Sci., Nancy, 2008.

\bibitem{Janson2013}
S.~Janson.
\newblock Euler-{F}robenius numbers and rounding.
\newblock {\em Online J. Anal. Comb.}, 8:1--34, 2013.

\bibitem{Janson2011}
S.~Janson, M.~Kuba, and A.~Panholzer.
\newblock Generalized {S}tirling permutations, families of increasing trees and
  urn models.
\newblock {\em J. Combin. Theory Ser. A}, 118(1):94--114, 2011.

\bibitem{Johnson1992}
N.~L. Johnson, S.~Kotz, and A.~W. Kemp.
\newblock {\em Univariate Discrete Distributions}.
\newblock John Wiley \& Sons, Inc., New York, second edition, 1992.
\newblock A Wiley-Interscience Publication.

\bibitem{Kaplansky1946}
I.~Kaplansky and J.~Riordan.
\newblock The problem of the rooks and its applications.
\newblock {\em Duke Math. J.}, 13:259--268, 1946.

\bibitem{Kermack1938}
W.~O. Kermack and A.~G. McKendrick.
\newblock Some properties of points arranged at random on a m{\"o}bius surface.
\newblock {\em Math. Gaz.}, pages 66--72, 1938.

\bibitem{Kim2019}
G.~B. Kim.
\newblock Distribution of descents in matchings.
\newblock {\em Ann. Comb.}, 23(1):73--87, 2019.

\bibitem{Kim2018}
G.~B. Kim and S.~Lee.
\newblock Central limit theorem for descents in conjugacy classes of {$S_n$}.
\newblock {\em arXiv preprint arXiv:1803.10457}, 2018.

\bibitem{Knape2014}
M.~Knape and R.~Neininger.
\newblock P\'{o}lya urns via the contraction method.
\newblock {\em Combin. Probab. Comput.}, 23:1148--1186, 2014.

\bibitem{Knuth1992}
D.~E. Knuth.
\newblock Two notes on notation.
\newblock {\em Amer. Math. Monthly}, 99(5):403--422, 1992.

\bibitem{Knuth1998}
D.~E. Knuth.
\newblock {\em The Art of Computer Programming. {V}olume 3: Sorting and
  Searching}.
\newblock Addison-Wesley, Reading, MA, 1998.
\newblock Second edition.

\bibitem{Koutras1994}
M.~V. Koutras.
\newblock Eulerian numbers associated with sequences of polynomials.
\newblock {\em Fibonacci Quart.}, 32(1):44--57, 1994.

\bibitem{Laborde-Zubieta2015}
P.~Laborde-Zubieta.
\newblock Occupied corners in tree-like tableaux.
\newblock {\em S\'em. Lothar. Combin.}, 74:Art. B74b, 14, 2015.

\bibitem{Lang2017}
W.~Lang.
\newblock On generating functions of diagonals sequences of {S}heffer and
  {R}iordan number triangles.
\newblock arXiv:1708.01421, 2017.

\bibitem{Laplace1777}
P.-S. Laplace.
\newblock M\'{e}moire sur l'usage du calcul aux diff\'{e}rences partielles dans
  la th\'{e}orie des suites.
\newblock {\em M\'{e}moires de l'Acad\'{e}mie royale des sciences de Paris},
  pages 313--335, 1777.

\bibitem{Laplace1812}
P.-S. Laplace.
\newblock {\em Th\'{e}orie analytique des probabilit\'{e}s}, volume
  Pierre-Simon Laplace: \OE uvres compl\`{e}tes, Tome 7 (645 pages).
\newblock Courcier, Paris, 1812.

\bibitem{Lehmer1985}
D.~H. Lehmer.
\newblock Interesting series involving the central binomial coefficient.
\newblock {\em Amer. Math. Monthly}, 92(7):449--457, 1985.

\bibitem{Li1867}
S.~Li.
\newblock {\em Duoji Bilei (Analogical Categories of Discrete Accumulations)}.
\newblock 1867.

\bibitem{Liagre1855}
J.-B. Liagre.
\newblock Sur la probabilit{\'e} de l'existence d'une cause d'erreur
  r{\'e}guli\`{e}re dans une s{\'e}rie d'observations.
\newblock {\em Bull. l'Acad. R. Sci. Lett. Beaux Arts de Belg.}, 22:9--13,
  15--54, 1855.

\bibitem{Liese2010}
J.~Liese and J.~Remmel.
\newblock {$Q$}-analogues of the number of permutations with {$k$}-excedances.
\newblock {\em Pure Math. Appl. (PU.M.A.)}, 21(2):285--320, 2010.

\bibitem{Liu2007}
L.~L. Liu and Y.~Wang.
\newblock A unified approach to polynomial sequences with only real zeros.
\newblock {\em Adv. in Appl. Math.}, 38(4):542--560, 2007.

\bibitem{Liu2015}
L.~L. Liu and B.-X. Zhu.
\newblock Strong {$q$}-log-convexity of the {E}ulerian polynomials of {C}oxeter
  groups.
\newblock {\em Discrete Math.}, 338(12):2332--2340, 2015.

\bibitem{Luo1982}
J.~J. Luo.
\newblock The study of the {S}tirling numbers and the {E}uler numbers by {L}i
  {J}enshoo.
\newblock {\em J. Math. Res. Exposition}, 2(4):173--182, 1982.

\bibitem{Luschny2013}
P.~Luschny.
\newblock Eulerian polynomials, Webpage, 2013.

\bibitem{Ma2012a}
S.-M. Ma.
\newblock Derivative polynomials and enumeration of permutations by number of
  interior and left peaks.
\newblock {\em Discrete Math.}, 312(2):405--412, 2012.

\bibitem{Ma2012}
S.-M. Ma.
\newblock An explicit formula for the number of permutations with a given
  number of alternating runs.
\newblock {\em J. Combin. Theory Ser. A}, 119(8):1660--1664, 2012.

\bibitem{Ma2013a}
S.-M. Ma.
\newblock Enumeration of permutations by number of alternating runs.
\newblock {\em Discrete Math.}, 313(18):1816--1822, 2013.

\bibitem{Ma2013}
S.-M. Ma.
\newblock A family of two-variable derivative polynomials for tangent and
  secant.
\newblock {\em Electron. J. Combin.}, 20(1):Paper 11, 12, 2013.

\bibitem{Ma2013b}
S.-M. Ma.
\newblock Some combinatorial arrays generated by context-free grammars.
\newblock {\em European J. Combin.}, 34(7):1081--1091, 2013.

\bibitem{Ma2014}
S.-M. Ma.
\newblock On {$\gamma$}-vectors and the derivatives of the tangent and secant
  functions.
\newblock {\em Bull. Aust. Math. Soc.}, 90(2):177--185, 2014.

\bibitem{Ma2016}
S.-M. Ma and H.-N. Wang.
\newblock Enumeration of a dual set of {S}tirling permutations by their
  alternating runs.
\newblock {\em Bull. Aust. Math. Soc.}, 94(2):177--186, 2016.

\bibitem{Ma2015}
S.-M. Ma and Y.-N. Yeh.
\newblock Derivative polynomials and enumeration of permutations by their
  alternating descents.
\newblock arXiv: 1504.02372, 2015.

\bibitem{Ma2017}
S.-M. Ma and Y.-N. Yeh.
\newblock {E}ulerian polynomials, {S}tirling permutations of the second kind
  and perfect matchings.
\newblock {\em Electron. J. Combin.}, 24(4):Paper 4.27 (18 pages), 2017.

\bibitem{MacMahon1908}
P.~A. MacMahon.
\newblock Second memoir on the compositions of numbers.
\newblock {\em Philos. Trans. Roy. Soc. A.}, 207:65--134, 1908.

\bibitem{MacMahon1921}
P.~A. MacMahon.
\newblock The divisors of numbers.
\newblock {\em Proc. Lond. Math. Soc.}, 2(1):305--340, 1921.

\bibitem{Magagnosc1980}
D.~Magagnosc.
\newblock Recurrences and formulae in an extension of the {E}ulerian numbers.
\newblock {\em Discrete Math.}, 30(3):265--268, 1980.

\bibitem{Mahmoud2008}
H.~Mahmoud.
\newblock {\em P{\'o}lya Urn Models}.
\newblock CRC press, 2008.

\bibitem{Mahmoud1993}
H.~M. Mahmoud, R.~T. Smythe, and J.~Szyma\'nski.
\newblock On the structure of random plane-oriented recursive trees and their
  branches.
\newblock {\em Random Structures Algorithms}, 4(2):151--176, 1993.

\bibitem{Mallows2008}
C.~Mallows and L.~Shepp.
\newblock The necklace process.
\newblock {\em J. Appl. Probab.}, 45(1):271--278, 2008.

\bibitem{Mann1945}
H.~B. Mann.
\newblock On a test for randomness based on signs of differences.
\newblock {\em Ann. Math. Statistics}, 16:193--199, 1945.

\bibitem{Mantaci1993}
R.~Mantaci.
\newblock Sur la distribution des anti-exc\'edances dans le groupe sym\'etrique
  et dans ses sous-groupes.
\newblock {\em Theoret. Comput. Sci.}, 117(1-2):243--253, 1993.

\bibitem{Martzloff2006}
J.-C. Martzloff.
\newblock {\em A History of {C}hinese Mathematics}.
\newblock Springer-Verlag, Berlin, english edition, 2006.

\bibitem{Mezo2014}
I.~Mez\H{o}.
\newblock Recent developments in the theory of {S}tirling numbers.
\newblock In H.~Nagoshi, editor, {\em Proceedings of Analytic Number
  Theory---Distribution and Approximation of Arithmetic Objects}, pages 58--80.
  RIMS, 2014.

\bibitem{Mezo2016}
I.~Mez\H{o} and J.~L. Ram\'{\i}rez.
\newblock Some identities of the {$r$}-{W}hitney numbers.
\newblock {\em Aequationes Math.}, 90(2):393--406, 2016.

\bibitem{Minai1993}
A.~A. Minai and R.~D. Williams.
\newblock On the derivatives of the sigmoid.
\newblock {\em Neural Networks}, 6(6):845--853, 1993.

\bibitem{Montgomery1990}
D.~C. Montgomery, L.~A. Johnson, and J.~S. Gardiner.
\newblock {\em Forecasting and Time Series Analysis}.
\newblock McGraw-Hill Companies, 1990.

\bibitem{Moore1943}
G.~H. Moore and W.~A. Wallis.
\newblock Time series significance tests based on signs of differences.
\newblock {\em J. Amer. Statist. Assoc.}, 38:153--164, 1943.

\bibitem{Morisita1971}
M.~Morisita.
\newblock Measuring of habitat value by the ``environmental density'' method.
\newblock In G.~P. Patil, E.~C. Pielou, and W.~E. Waters, editors, {\em
  Statistical Ecology}, volume~1, pages 379--401, 1971.

\bibitem{Morley1897}
F.~Morley.
\newblock A generating function for the number of permutations with an assigned
  number of sequences.
\newblock {\em Bull. Amer. Math. Soc.}, 4(1):23--28, 1897.

\bibitem{Myint-U2007}
T.~Myint-U and L.~Debnath.
\newblock {\em Linear Partial Differential Equations for Scientists and
  Engineers}.
\newblock Birkh\"auser Boston, Inc., Boston, MA, fourth edition, 2007.

\bibitem{Nakano2014}
F.~Nakano and T.~Sadahiro.
\newblock A generalization of carries processes and {E}ulerian numbers.
\newblock {\em Adv. in Appl. Math.}, 53:28--43, 2014.

\bibitem{Netto1901}
E.~Netto.
\newblock {\em Lehrbuch der Combinatorik}.
\newblock Verlag von B. G. Teubner, Leipzig, 1901.

\bibitem{Norton2013}
E.~Norton.
\newblock Symplectic reflection algebras in positive characteristic as ore
  extensions.
\newblock {\em arXiv preprint arXiv:1302.5411}, 2013.

\bibitem{Oden2006}
N.~L. Oden and M.~J. McIntosh.
\newblock Exact moments and probabilities for {W}ei's urn randomization model.
\newblock {\em Statist. Probab. Lett.}, 76(15):1694--1700, 2006.

\bibitem{Odlyzko1995}
A.~M. Odlyzko.
\newblock Asymptotic enumeration methods.
\newblock In {\em Handbook of Combinatorics, {V}ol.\ 1,\ 2}, pages 1063--1229.
  Elsevier Sci. B. V., Amsterdam, 1995.

\bibitem{Ozdemir2019}
A.~Y. {\"O}zdemir.
\newblock Martingales and descent statistics.
\newblock {\em arXiv preprint arXiv:1901.01719}, 2019.

\bibitem{Petersen2013}
T.~K. Petersen.
\newblock Two-sided {E}ulerian numbers via balls in boxes.
\newblock {\em Math. Mag.}, 86(3):159--176, 2013.

\bibitem{Petersen2015}
T.~K. Petersen.
\newblock {\em Eulerian Numbers}.
\newblock Birkh\"{a}user/Springer, New York, 2015.

\bibitem{Pita-Ruiz-V.2017}
C.~d.~J. Pita Ruiz~V.
\newblock Weighted sums of squares via generalized {E}ulerian polynomials.
\newblock {\em Fibonacci Quart.}, 55(5):149--165, 2017.

\bibitem{Pitman1997}
J.~Pitman.
\newblock Probabilistic bounds on the coefficients of polynomials with only
  real zeros.
\newblock {\em J. Combin. Theory Ser. A}, 77(2):279--303, 1997.

\bibitem{Postnikov2008}
A.~Postnikov, V.~Reiner, and L.~Williams.
\newblock Faces of generalized permutohedra.
\newblock {\em Doc. Math.}, 13:207--274, 2008.

\bibitem{Rakotondrajao2007}
F.~Rakotondrajao.
\newblock On {E}uler's difference table.
\newblock In {\em Proceedings of FPSAC'07, Tianjin, 2007}, 2007.

\bibitem{Ramirez2018}
J.~L. Ram\'{i}rez, S.~N. Villamar\'{i}n, and D.~Villamizar.
\newblock Eulerian numbers associated with arithmetical progressions.
\newblock {\em Electron. J. Combin.}, 25(1):Paper 1.48, 12, 2018.

\bibitem{Renyi1967}
A.~R\'enyi.
\newblock Probabilistic methods in analysis. {I}.
\newblock {\em Mat. Lapok}, 18:5--35, 1967.

\bibitem{Renyi1967a}
A.~R\'enyi.
\newblock Probabilistic methods in analysis. {II}.
\newblock {\em Mat. Lapok}, 18:175--194, 1967.

\bibitem{Riordan1958}
J.~Riordan.
\newblock {\em An Introduction to Combinatorial Analysis}.
\newblock John Wiley \& Sons, Inc., New York; Chapman \& Hall, Ltd., London,
  1958.

\bibitem{Rzadkowski2019}
G.~Rz\c{a}dkowski and M.~Urli\'{n}ska.
\newblock Some applications of the generalized {E}ulerian numbers.
\newblock {\em J. Combin. Theory Ser. A}, 163:85--97, 2019.

\bibitem{Saalschutz1893}
L.~Saalsch\"utz.
\newblock {\em Vorlesungen \"uber die {B}ernoullischen {Z}ahlen}.
\newblock Verlag von Julius Springer, Berlin, 1893.

\bibitem{Samadi2004}
S.~Samadi, M.~O. Ahmad, and M.~N.~S. Swamy.
\newblock Characterization of {B}-spline digital filters.
\newblock {\em IEEE Trans. Circuits Syst. I. Regul. Pap.}, 51(4):808--816,
  2004.

\bibitem{Sandor2004}
J.~S\'andor and B.~Crstici.
\newblock {\em Handbook of Number Theory. {II}}.
\newblock Kluwer Academic Publishers, Dordrecht, 2004.

\bibitem{Savage2012}
C.~D. Savage and G.~Viswanathan.
\newblock The {$1/k$}-{E}ulerian polynomials.
\newblock {\em Electron. J. Combin.}, 19(1):Paper 9, 21, 2012.

\bibitem{Schmidt1997}
F.~Schmidt and R.~Simion.
\newblock Some geometric probability problems involving the {E}ulerian numbers.
\newblock {\em Electron. J. Combin.}, 4(2):Research Paper 18, approx. 13, 1997.
\newblock The Wilf Festschrift (Philadelphia, PA, 1996).

\bibitem{Schoenberg1973}
I.~J. Schoenberg.
\newblock {\em Cardinal Spline Interpolation}.
\newblock SIAM, Philadelphia, Pa., 1973.

\bibitem{Shareshian2017}
J.~Shareshian and M.~L. Wachs.
\newblock Gamma-positivity of variations of eulerian polynomials.
\newblock {\em arXiv preprint arXiv:1702.06666}, 2017.

\bibitem{Sharon2007}
E.~Sharon, S.~Litsyn, and J.~Goldberger.
\newblock Efficient serial message-passing schedules for {LDPC} decoding.
\newblock {\em IEEE Trans. Inform. Theory}, 53(11):4076--4091, 2007.

\bibitem{Shur2003}
W.~Shur.
\newblock Two game-set inequalities.
\newblock {\em J. Integer Seq.}, 6(4):Article 03.4.1, 12, 2003.

\bibitem{Simpson1756}
T.~Simpson.
\newblock A letter to the {R}ight {H}onorable {G}eorge {E}arl of
  {M}acclesfield, {P}resident of the {R}oyal {S}ociety, on the advantage of
  taking the mean of a number of observations, in practical astronomy.
\newblock {\em Philos. Trans. R. Soc. Lond.}, 49:82--93, 1756.

\bibitem{Simpson1757}
T.~Simpson.
\newblock {\em Miscellaneous Tracts on Some Curious, and Very Interesting
  Subjects in Mechanics, Physical-Astronomy, and Speculative Mathematics}.
\newblock Nourse, London, 1757.

\bibitem{Sobolev1997}
S.~L. Sobolev and V.~L. Vaskevich.
\newblock {\em The {T}heory of {C}ubature {F}ormulas}.
\newblock Kluwer Academic Publishers Group, Dordrecht, 1997.

\bibitem{Stanley1977}
R.~P. Stanley.
\newblock Eulerian partitions of a unit hypercube.
\newblock In {\em Higher Combinatorics (Proc. NATO Advanced Study Inst.,
  Berlin, 1976)}, page~49. Reidel, Dordrecht, 1977.

\bibitem{Stanley1989}
R.~P. Stanley.
\newblock Log-concave and unimodal sequences in algebra, combinatorics, and
  geometry.
\newblock In {\em Graph theory and its applications: {E}ast and {W}est
  ({J}inan, 1986)}, volume 576 of {\em Ann. New York Acad. Sci.}, pages
  500--535. New York Acad. Sci., New York, 1989.

\bibitem{Stanley1999}
R.~P. Stanley.
\newblock {\em Enumerative Combinatorics. Volume 2}.
\newblock Cambridge University Press, Cambridge, NY, 1999.

\bibitem{Stanley2012}
R.~P. Stanley.
\newblock {\em Enumerative Combinatorics. {V}olume 1}.
\newblock Cambridge University Press, Cambridge, second edition, 2012.

\bibitem{Steingrimsson1994}
E.~Steingr\'{\i}msson.
\newblock Permutation statistics of indexed permutations.
\newblock {\em European J. Combin.}, 15(2):187--205, 1994.

\bibitem{Stembridge1994}
J.~R. Stembridge.
\newblock Some permutation representations of {W}eyl groups associated with the
  cohomology of toric varieties.
\newblock {\em Adv. Math.}, 106(2):244--301, 1994.

\bibitem{Stigler1986}
S.~M. Stigler.
\newblock {\em The History of Statistics: The Measurement of Uncertainty before
  1900}.
\newblock The Belknap Press of Harvard University Press, Cambridge, MA, 1986.

\bibitem{Strasser2011}
G.~Strasser.
\newblock Generalisations of the {E}uler adic.
\newblock {\em Math. Proc. Cambridge Philos. Soc.}, 150(2):241--256, 2011.

\bibitem{Sulanke1999}
R.~A. Sulanke.
\newblock Constraint-sensitive {C}atalan path statistics having the {N}arayana
  distribution.
\newblock {\em Discrete Math.}, 204(1-3):397--414, 1999.

\bibitem{Takacs1979}
L.~Tak\'acs.
\newblock A generalization of the {E}ulerian numbers.
\newblock {\em Publ. Math. Debrecen}, 26(3-4):173--181, 1979.

\bibitem{Tanimoto2006}
S.~Tanimoto.
\newblock A study of {E}ulerian numbers for permutations in the alternating
  group.
\newblock {\em Integers}, 6:A31, 12, 2006.

\bibitem{Tanny1973}
S.~Tanny.
\newblock A probabilistic interpretation of {E}ulerian numbers.
\newblock {\em Duke Math. J.}, 40:717--722, 1973.

\bibitem{Visontai2013}
M.~Visontai.
\newblock Some remarks on the joint distribution of descents and inverse
  descents.
\newblock {\em Electron. J. Combin.}, 20(1):Paper 52, 12, 2013.

\bibitem{Schrutka1941}
L.~von Schrutka.
\newblock Eine neue {E}inteilung der {P}ermutationen.
\newblock {\em Math. Ann.}, 118:246--250, 1941.

\bibitem{Wallner2016}
M.~{Wallner}.
\newblock {A half-normal distribution scheme for generating functions}.
\newblock {\em ArXiv e-prints}, Oct. 2016.

\bibitem{Warren1999}
D.~Warren.
\newblock The {F}robenius-{H}arper technique in a general recurrence model.
\newblock {\em J. Appl. Probab.}, 36(1):30--47, 1999.

\bibitem{Warren1996}
D.~Warren and E.~Seneta.
\newblock Peaks and {E}ulerian numbers in a random sequence.
\newblock {\em J. Appl. Probab.}, 33(1):101--114, 1996.

\bibitem{Wilf2004}
H.~S. {Wilf}.
\newblock {The method of characteristics, and ``problem 89'' of Graham, Knuth
  and Patashnik}.
\newblock {\em ArXiv Mathematics e-prints}, June 2004.

\bibitem{Wilson2013}
R.~Wilson and J.~J. Watkins, editors.
\newblock {\em Combinatorics: {A}ncient and {M}odern}.
\newblock Oxford University Press, Oxford, 2013.

\bibitem{Wolfowitz1944}
J.~Wolfowitz.
\newblock Asymptotic distribution of runs up and down.
\newblock {\em Ann. Math. Statistics}, 15:163--172, 1944.

\bibitem{Worpitzky1883}
J.~Worpitzky.
\newblock Studien \"{u}ber die {B}ernoullischen und {E}ulerischen {Z}ahlen.
\newblock {\em J. Reine Angew. Math.}, 94:203--232, 1883.

\bibitem{Wright1940}
E.~M. Wright.
\newblock The generalized {B}essel function of order greater than one.
\newblock {\em Quart. J. Math., Oxford Ser.}, 11:36--48, 1940.

\bibitem{Xiong2013}
T.~Xiong, H.-P. Tsao, and J.~I. Hall.
\newblock General {E}ulerian numbers and {E}ulerian polynomials.
\newblock {\em J. Math.}, pages Art. ID 629132, 9, 2013.

\bibitem{Xu2011}
Y.~Xu and R.-H. Wang.
\newblock Asymptotic properties of {$B$}-splines, {E}ulerian numbers and cube
  slicing.
\newblock {\em J. Comput. Appl. Math.}, 236(5):988--995, 2011.

\bibitem{Zhang1939}
Y.~Zhang.
\newblock {Duoji Bilei Shuzheng (The Explication of Duoji Bilei)}.
\newblock {\em Kexue (Science)}, 23:647--663, 1939.

\bibitem{Zhuang2016}
Y.~Zhuang.
\newblock Counting permutations by runs.
\newblock {\em J. Combin. Theory Ser. A}, 142:147--176, 2016.

\end{thebibliography}

\end{document}